\documentclass[11pt,twoside]{preprint}
\usepackage[full]{textcomp}
\usepackage[osf]{newtxtext}
\usepackage{comment}

\usepackage{amssymb}
\usepackage{mathtools}
\usepackage{hyperref}
\usepackage{breakurl}
\usepackage{mhenvs}
\usepackage{mhequ}
\usepackage{mhsymb}
\usepackage{booktabs}
\usepackage{tikz}
\usepackage{mathrsfs}
\usepackage{longtable}

\usepackage{microtype}
\usepackage{comment}
\usepackage{wasysym}
\usepackage{centernot}
\usepackage{enumitem}
\usepackage{bm}
\usepackage{stackrel}

\makeatletter
\newcommand{\globalcolor}[1]{%
  \color{#1}\global\let\default@color\current@color
}
\makeatother

\newif\ifdark
\IfFileExists{dark}{\darktrue}{\darkfalse}

\ifdark

\definecolor{darkred}{rgb}{0.9,0.2,0.2}
\definecolor{darkblue}{rgb}{0.7,0.3,1}
\definecolor{darkgreen}{rgb}{0.1,0.9,0.1}
\definecolor{pagebackground}{rgb}{.15,.21,.18}
\definecolor{pageforeground}{rgb}{.84,.84,.85}
\pagecolor{pagebackground}
\AtBeginDocument{\globalcolor{pageforeground}}
\colorlet{symbols}{black!50}

\else

\definecolor{darkred}{rgb}{0.7,0.1,0.1}
\definecolor{darkblue}{rgb}{0.4,0.1,0.8}
\definecolor{darkgreen}{rgb}{0.1,0.7,0.1}
\definecolor{pagebackground}{rgb}{1,1,1}
\definecolor{pageforeground}{rgb}{0,0,0}
\colorlet{symbols}{black!50}

\fi

\DeclareMathAlphabet{\mathbbm}{U}{bbm}{m}{n}

\DeclareFontFamily{U}{BOONDOX-calo}{\skewchar\font=45 }
\DeclareFontShape{U}{BOONDOX-calo}{m}{n}{
  <-> s*[1.05] BOONDOX-r-calo}{}
\DeclareFontShape{U}{BOONDOX-calo}{b}{n}{
  <-> s*[1.05] BOONDOX-b-calo}{}
\DeclareMathAlphabet{\mcb}{U}{BOONDOX-calo}{m}{n}
\SetMathAlphabet{\mcb}{bold}{U}{BOONDOX-calo}{b}{n}

\setlist{noitemsep,topsep=4pt}

\makeatletter
\def\DeclareSymbol#1#2#3{%
	\expandafter\gdef\csname MH@symb@#1\endcsname{%
	\tikz[baseline=#2,scale=0.15,draw=symbols,line join=round]{#3}}%
	\expandafter\gdef\csname MH@symb@#1s\endcsname{\scalebox{0.75}{%
	\tikz[baseline=#2,scale=0.15,draw=symbols,line join=round]{#3}}}%
	\expandafter\gdef\csname MH@symb@#1ss\endcsname{\scalebox{0.65}{%
	\tikz[baseline=#2,scale=0.15,draw=symbols,line join=round]{#3}}}%
	}
\def\<#1>{\ifthenelse{\boolean{mmode}}{\mathchoice{\csname MH@symb@#1\endcsname}{\csname MH@symb@#1\endcsname}{\csname MH@symb@#1s\endcsname}{\csname MH@symb@#1ss\endcsname}}{\csname MH@symb@#1\endcsname}}
\makeatother

\DeclareSymbol{0}{-2.4}{\node[dot] {};}
\DeclareSymbol{1}{0}{\draw[white] (-.5,0) -- (.5,0); \draw (0,0)  -- (0,1.5) node[sdot] {};}
\DeclareSymbol{11}{0}{\draw (-0.5,1.2) node[sdot] {} -- (0,0) -- (0.5,1.2) node[sdot] {};}
\DeclareSymbol{111}{0}{\draw (0,0) -- (0,1.2) node[sdot] {}; \draw (-.7,1) node[sdot] {} -- (0,0) -- (.7,1) node[sdot] {};}
\DeclareSymbol{131}{-3}{\draw (0,0) -- (0,-1) -- (1,0) node[sdot] {}; \draw (0,0) -- (0,-1) -- (-1,0) node[sdot] {}; \draw (0,0) -- (0,1.2) node[sdot] {}; \draw (-.7,1) node[sdot] {} -- (0,0) -- (.7,1) node[sdot] {};}
\DeclareSymbol{11}{0}{\draw (-0.5,1.2) node[sdot] {} -- (0,0) -- (0.5,1.2) node[sdot] {};}
\DeclareSymbol{12}{-3}{\draw (-.8,1) node[sdot] {} -- (0,0) -- (0,-1); \draw (1,0) node[sdot] {} -- (0,-1) -- (-1,0) node[sdot] {};}
\DeclareSymbol{30}{-3}{\draw (0,0) -- (0,-1); \draw (0,0) -- (0,1.2) node[sdot] {}; \draw (-.7,1) node[sdot] {} -- (0,0) -- (.7,1) node[sdot] {};}
\DeclareSymbol{10}{-3}{\draw (-.8,1) node[sdot] {} -- (0,0) -- (0,-1);}
\DeclareSymbol{22}{-3}{\draw (0,0.3) -- (0,-1) -- (1,0) node[sdot] {}; \draw (0,0.3) -- (0,-1) -- (-1,0) node[sdot] {};\draw (-.7,1) node[sdot] {} -- (0,0.3) -- (.7,1) node[sdot] {};}
\DeclareSymbol{211}{0}{\draw (-0.5,2.4) node[sdot] {} -- (-1,1.6);\draw (-1.5,2.4) node[sdot] {} -- (-0.5,0.8); \draw (0,1.6) node[sdot] {} -- (-0.5,0.8) -- (0,0) -- (0.5,0.8) node[sdot] {};}

\DeclareSymbol{Xi22}{0.5}{\draw (0,0) node[xi] {} -- (-1,1) node[not] {} -- (0,2) node[xi] {};}

\DeclareSymbol{Xi2}{-2}{\draw (0,-0.25) node[xi] {} -- (-1,1) node[xi] {};}
\DeclareSymbol{Xi3}{0}{\draw (0,0) node[xi] {} -- (-1,1) node[xi] {} -- (0,2) node[xi] {};}
\DeclareSymbol{Xi4}{2}{\draw (0,0) node[not] {} -- (-1,1) node[xi] {} -- (0,2) node[xi] {} -- (-1,3) node[xi] {};}
\DeclareSymbol{Xi2X}{-2}{\draw (0,-0.25) node[xi] {} -- (-1,1) node[xix] {};}
\DeclareSymbol{XXi2}{-2}{\draw (0,-0.25) node[xix] {} -- (-1,1) node[xi] {};}

\DeclareSymbol{IXi^2}{-1}{\draw (-1,1) node[xi] {} -- (0,0);
\draw[kernels2] (1,1) node[xi] {} -- (0,0) node[not] {};}

\DeclareSymbol{IXi3}{-1}{\draw[kernel1] (-2,2) node[] {$\bullet$} -- (0,0);
\draw[kernel1] (2,2) node[] { $ \bullet $ } -- (0,0) node[] {$ \bullet $};
\draw[kernel1] (0,2) node[] { $ \bullet $ } -- (0,0) ;}

\DeclareSymbol{IXi2}{-1}{\draw[kernel1] (-1,2) node[] {$\bullet$} -- (0,0);
\draw[kernel1] (1,2) node[] { $ \bullet $ } -- (0,0) node[] {$ \bullet $};}

\DeclareSymbol{IXi3b}{-1}{\draw[kernel1] (-2,2) node[] {$\bullet$} -- (0,0);
\draw[kernel1] (2,2) node[] { $ \bullet $ } -- (0,0) node[] {$ \bullet $};
\draw[kernel1] (0,2) node[] { $ \bullet $ } -- (0,0) ;
\draw (0.75,-0.5) node[] {\tiny{$\mfl$}};
\draw (-3,2.5) node[] {\tiny{$ \mfi $}};
\draw (3,2.5) node[] {\tiny{$ \mfq $}};
\draw (1,2.5) node[] {\tiny{$ \mfj $}};
}

\DeclareSymbol{IXi4}{-1}{\draw[kernel1] (-1,2) node[] {$\bullet$} -- (0,0);
\draw[kernel1] (1,2) node[] { $ \bullet $ } -- (0,0) node[] {$ \bullet $};
\draw[kernel1] (-1,2) node[] { $ \bullet $ } -- (-1,4) node[] { $ \bullet $ };}

\DeclareSymbol{IXi4b}{-1}{\draw[kernel1] (-1,2) node[] {$\bullet$} -- (0,0);
\draw[kernel1] (1,2) node[] { $ \bullet $ } -- (0,0) node[] {$ \bullet $};
\draw[kernel1] (-1,2) node[] { $ \bullet $ } -- (-1,4) node[] { $ \bullet $ };
\draw (0.75,-0.5) node[] {\tiny{$\mfl$}};
\draw (-2,2.5) node[] {\tiny{$ \mfj $}};
\draw (2,2.5) node[] {\tiny{$ \mfq $}};
\draw (-2,4) node[] {\tiny{$ \mfi $}};
}

\DeclareSymbol{XiX}{-2.8}{\node[xibx] {};}
\DeclareSymbol{tauX}{-2.8}{ \draw[kernels2] (0,0) node[xibx] {};}
\DeclareSymbol{Xi}{-2.8}{\node[xib] {};}
\DeclareSymbol{XiY}{-2.8}{\node[xie] {};}
\DeclareSymbol{XiZ}{-2.8}{\node[xid] {};}
\DeclareSymbol{IXiX}{0}{\draw (0,-0.25) node[not] {} -- (0,1.5) node[xix] {};}

\newcommand{\cT}{\mcb{T}}

\newcommand{\cut}{\mathfrak{C}}

\def\Poly{\mathscr{P}}

\newcommand{\mcA}{\mathcal{A}}

\newcommand{\mcR}{\mathcal{R}}
\newcommand{\mcC}{\mathcal{C}}

\newcommand{\mcS}{\mathcal{S}}

\newcommand{\mcT}{\mathcal{T}}
\newcommand{\mcI}{\mathcal{I}}

\newcommand{\mcN}{\mathcal{N}}
\newcommand{\mcK}{\mathcal{K}}

\newcommand{\mcW}{\mathcal{W}}
\newcommand{\mcP}{\mathcal{P}}

\newcommand{\mcQ}{\mathcal{Q}}

\newcommand{\jets}{\mathscr{U}}
\newcommand{\G}{\mcb{Q}}

\newcommand{\bu}{\mathbf{u}}

\newcommand{\T}{\mathbf{T}}

\def\Labe{\mathfrak{e}}

\def\Labn{\mathfrak{n}}

\def\Labhom{\mathfrak{t}}

\def\Deltam{\Delta^{\!-}}

\def\Deltap{\Delta^{\!+}}

\def\Span#1{\mathop{\mathrm{Span}}{#1}}

\def\${|\!|\!|}
\def\DD{\mathscr{D}}

\def\id{\mathrm{id}}

\def\graft#1{\curvearrowright_{#1}}

\def\hgraft#1{\hat\curvearrowright_{#1}}
\def\bgraft#1{\bar\curvearrowright_{#1}}
\def\Labo{\mathfrak{o}}
\def\root{\mathrm{root}}
\def\poly{\mathrm{poly}}
\def\nonroot{\mathrm{non-root}}

\def\Gauss{\mathrm{Gauss}}

\def\scal#1{{\langle#1\rangle}}

\def\CS{\mathcal{S}}
\def\CR{\mathcal{R}}
\def\CF{\mathcal{F}}
\def\CT{\mcb{T}}
\def\bone{\mathbf{1}}

\def\noise{\mathop{\mathrm{noise}}}

\def\homplus#1{|#1|_{+}}

\newcommand{\mfT}{\mathfrak{T}}

\newcommand{\mfo}{\mathfrak{o}}
\newcommand{\mfe}{\mathfrak{e}}

\newcommand{\mfL}{\mathfrak{L}}

\newcommand{\mfc}{\mathfrak{c}}

\newcommand{\mfR}{\mathfrak{R}}

\newcommand{\mft}{\mathfrak{t}}
\newcommand{\mfi}{\mathfrak{i}}
\newcommand{\mfj}{\mathfrak{j}}

\newcommand{\mfm}{\mathfrak{m}}

\newcommand{\mfp}{\mathfrak{p}}

\newcommand{\mfl}{\mathfrak{l}}

\newcommand{\mfq}{\mathfrak{q}}
\newcommand{\mfb}{\mathfrak{b}}
\newcommand{\mff}{\mathfrak{f}}
\newcommand{\mfD}{\mathfrak{D}}

\newcommand{\Y}{\mcb{X}}

\def\wwnorm#1{|\!|\!| #1 |\!|\!|}

\def\cC{\mathscr{C}}

\def\can{\mbox{\tiny{can}}}
\def\BPHZ{\mbox{\tiny \textsc{bphz}}}

\def\mnoise{\mfm^{\scriptscriptstyle \Xi}}
\def\mpoly{\mfm^{\scriptscriptstyle X}}
\def\barmnoise{\bar{\mfm}^{\scriptscriptstyle \Xi}}
\def\barmpoly{\bar{\mfm}^{\scriptscriptstyle X}}

\newcommand{\ex}{\mathrm{ex}}

\def\combplus[#1,#2,#3,#4]{\binom{#1\ {\scriptstyle #4} }{#2\ #3}}

\def\singlescalegenvert[#1,#2]{\hat{H}^{#2}_{#1}}
\def\multiscalegenvert[#1,#2]{H^{#2}_{#1}}

\def\nr[#1]{\tilde{N}[#1]} 
\def\inn[#1]{\mathring{N}[#1]}
\def\nrinn[#1]{\hat{N}_{#1}} 
\def\nrmod[#1,#2]{\tilde{N}_{#1}(#2)}
\def\nrinnmod[#1,#2]{\hat{N}_{#1}(#2)}

\def\ident[#1]{\underline{#1}}

\def\mylink#1#2{\mathrel{\vbox{\offinterlineskip\ialign{%
    \hfil##\hfil\cr
    $\scriptscriptstyle#1$\cr
    \noalign{\kern0.1ex}
    $#2$\cr
}}}}
\def\mysublink[#1]#2#3{\mathrel{\vbox{\offinterlineskip\ialign{%
    \hfil##\hfil\cr
    $\scriptscriptstyle#2$\cr
    \noalign{\kern0.1ex}
    $#3$\cr
    \noalign{\kern-0.2ex}
    \smash{\raisebox{-\height}{\hbox{$\scriptscriptstyle #1$}}}\cr
    \noalign{\kern0.2ex}
}}}}

\def\fon[#1]{\cC_{#1}}

\def\mincompproj[#1]{\mfp_{#1}}

\def\Proj_#1{\mathop{\mathrm{Proj}_{#1}}}

\def\negrenorm[#1]{\mfR_{#1}}
\def\topnegrenorm[#1]{\overline{\mfR}_{#1}}

\def\quotedge[#1]{E^{q}_{#1}}

\def\posrenorm[#1]{\mcC_{#1}}
\def\topposrenorm[#1]{\overline{\mcC_{#1}}}
\def\cutsmod[#1]{\mathbb{C}_{+,#1}}

\def\fullcutsmod[#1]{\cut_{#1}}

\def\rem{\mathop{\mathrm{rem}}}
\def\clas{\mathop{\mathrm{clas}}}

\def\emptyset{{\centernot\ocircle}}

\colorlet{testcolor}{green!60!black}

\colorlet{redkernel}{red!80}

\def\symbol#1{\textcolor{symbols}{#1}}
\def\1{\mathbf{\symbol{1}}}

\usetikzlibrary{shapes.misc}
\usetikzlibrary{shapes.symbols}
\usetikzlibrary{shapes.geometric}
\usetikzlibrary{decorations}
\usetikzlibrary{decorations.markings}

\usetikzlibrary{calc}

\newcommand*{\mathcolor}{}
\def\mathcolor#1#{\mathcoloraux{#1}}
\newcommand*{\mathcoloraux}[3]{%
  \protect\leavevmode
  \begingroup
    \color#1{#2}#3%
  \endgroup
}

\makeatletter
\pgfdeclareshape{crosscircle}
{
  \inheritsavedanchors[from=circle] 
  \inheritanchorborder[from=circle]
  \inheritanchor[from=circle]{north}
  \inheritanchor[from=circle]{north west}
  \inheritanchor[from=circle]{north east}
  \inheritanchor[from=circle]{center}
  \inheritanchor[from=circle]{west}
  \inheritanchor[from=circle]{east}
  \inheritanchor[from=circle]{mid}
  \inheritanchor[from=circle]{mid west}
  \inheritanchor[from=circle]{mid east}
  \inheritanchor[from=circle]{base}
  \inheritanchor[from=circle]{base west}
  \inheritanchor[from=circle]{base east}
  \inheritanchor[from=circle]{south}
  \inheritanchor[from=circle]{south west}
  \inheritanchor[from=circle]{south east}
  \inheritbackgroundpath[from=circle]
  \foregroundpath{
    \centerpoint%
    \pgf@xc=\pgf@x%
    \pgf@yc=\pgf@y%
    \pgfutil@tempdima=\radius%
    \pgfmathsetlength{\pgf@xb}{\pgfkeysvalueof{/pgf/outer xsep}}%
    \pgfmathsetlength{\pgf@yb}{\pgfkeysvalueof{/pgf/outer ysep}}%
    \ifdim\pgf@xb<\pgf@yb%
      \advance\pgfutil@tempdima by-\pgf@yb%
    \else%
      \advance\pgfutil@tempdima by-\pgf@xb%
    \fi%
    \pgfpathmoveto{\pgfpointadd{\pgfqpoint{\pgf@xc}{\pgf@yc}}{\pgfqpoint{-0.707107\pgfutil@tempdima}{-0.707107\pgfutil@tempdima}}}
    \pgfpathlineto{\pgfpointadd{\pgfqpoint{\pgf@xc}{\pgf@yc}}{\pgfqpoint{0.707107\pgfutil@tempdima}{0.707107\pgfutil@tempdima}}}
    \pgfpathmoveto{\pgfpointadd{\pgfqpoint{\pgf@xc}{\pgf@yc}}{\pgfqpoint{-0.707107\pgfutil@tempdima}{0.707107\pgfutil@tempdima}}}
    \pgfpathlineto{\pgfpointadd{\pgfqpoint{\pgf@xc}{\pgf@yc}}{\pgfqpoint{0.707107\pgfutil@tempdima}{-0.707107\pgfutil@tempdima}}}
  }
}
\makeatother

\definecolor{connection}{rgb}{0.7,0.1,0.1}

\tikzset{
root/.style={circle,fill=black!50,inner sep=0pt, minimum size=3mm},
        dot/.style={circle,fill=black,inner sep=0pt, minimum size=1.2mm},
        sdot/.style={circle,fill=black,inner sep=0pt,minimum size=.5mm},
        dotred/.style={circle,fill=black!50,inner sep=0pt, minimum size=2mm},
        var/.style={circle,fill=black!10,draw=black,inner sep=0pt, minimum size=3mm},
        kernel/.style={semithick,shorten >=2pt,shorten <=2pt},
        kernel1/.style={draw=black,thick},
        kernels/.style={snake=zigzag,shorten >=2pt,shorten <=2pt,segment amplitude=1pt,segment length=4pt,line before snake=2pt,line after snake=5pt,},
        rho/.style={densely dashed,semithick,shorten >=2pt,shorten <=2pt},
           testfcn/.style={dotted,semithick,shorten >=2pt,shorten <=2pt},
           tau/.style={circle,inner sep=1pt,draw=black,fill=white,text=black,thin},
        renorm/.style={shape=circle,fill=white,inner sep=1pt},
        labl/.style={shape=rectangle,fill=white,inner sep=1pt},
        xic/.style={very thin,circle,fill=symbols,draw=black,inner sep=0pt,minimum size=1.2mm},
        xi/.style={very thin,circle,fill=blue!10,draw=black,inner sep=0pt,minimum size=1.2mm},
        xix/.style={crosscircle,fill=blue!10,draw=black,inner sep=0pt,minimum size=1.2mm},
	xib/.style={very thin,circle,fill=blue!10,draw=black,inner sep=0pt,minimum size=1.6mm},
	xie/.style={very thin,circle,fill=green!50!black,draw=black,inner sep=0pt,minimum size=1.6mm},
	xid/.style={very thin,circle,fill=symbols,draw=black,inner sep=0pt,minimum size=1.6mm},
	xibx/.style={crosscircle,fill=blue!10,draw=black,inner sep=0pt,minimum size=1.6mm},
	kernels2/.style={very thick,draw=connection,segment length=12pt},
	not/.style={thin,circle,fill=symbols,draw=connection,fill=connection,inner sep=0pt,minimum size=0.5mm},
	>=stealth,
  }

\newtheorem{assumption}[lemma]{Assumption}

\let\comm\comment

\let\D\CD
\def\s{\mathfrak{s}}

\def\K{\mathfrak{K}}

\def\RR{\mathfrak{R}}

\def\${|\!|\!|}

\def\?{{\color{red}?}}

\def\id{\mathrm{id}}

\def\restr{\mathord{\upharpoonright}}

\def\proj{\mathbf{p}}
\def\reg{\mathop{\mathrm{reg}}}
\def\ireg{\mathop{\mathrm{ireg}}}

\def\PPi{\boldsymbol{\Pi}}

\def\id{\mathrm{id}}

\def\Labhom{\mathfrak{t}}
\def\Labe{\mathfrak{e}}
\def\Labn{\mathfrak{n}}

\def\BB{\mathcal{B}}
\def\BBspan{\mcb{B}}

\def\BBbig{\check{\BB}}

\def\VV{\mathcal{V}}
\def\VVspan{\mcb{V}}

\def\nodes{\mathsf{N}}
\def\edges{\mathsf{E}}
\def\decor{\mathsf{D}}

\def\trunc{L}

\def\branch{P}
\def\trunk{R}
\def\slash{\leavevmode\unskip\kern0.18em/\penalty\exhyphenpenalty\kern0.18em}
\def\dash{\leavevmode\unskip\kern0.18em--\penalty\exhyphenpenalty\kern0.18em}

\def\one{\mathbbm{1}}

{\theorembodyfont{\rmfamily}
\newtheorem{example}[lemma]{Example}}

\let\basepoint\logof
\def\logof{\mathord{{\basepoint}}} 

\title{Renormalising SPDEs in regularity structures}
\author{Y.~Bruned$^1$, A.~Chandra$^2$, I.~Chevyrev$^3$, and M.~Hairer$^2$}

\institute{University of Edinburgh, UK \and Imperial College London, UK \and University of Oxford, UK}

\date{\today}
\begin{document}
\maketitle
\begin{abstract}
The formalism recently introduced in \cite{BHZalg} allows one to assign a regularity structure,
as well as a corresponding ``renormalisation group'', to any subcritical system of semilinear
stochastic PDEs. Under very mild additional assumptions, it was then shown in \cite{CH} that
large classes of driving noises exhibiting the relevant small-scale behaviour can be lifted 
to such a regularity structure in a robust way, following a renormalisation procedure reminiscent
of the BPHZ procedure arising in perturbative QFT. 

The present work completes this programme by constructing an action of the renormalisation group
onto a suitable class of stochastic PDEs which is intertwined with its action on the corresponding
space of models. This shows in particular that solutions constructed from the BPHZ lift of a smooth
driving noise coincide with the classical solutions of a modified PDE. 
This yields a very general black box type local existence and stability theorem for a wide class 
of singular nonlinear SPDEs. 
\end{abstract}
\setcounter{tocdepth}{2}

\tableofcontents

\section{Introduction}\label{sec: intro}

This article is part of the ongoing programme initiated in \cite{Regularity} aiming to develop a 
robust existence and approximation theory for a wide class of semilinear parabolic stochastic 
partial differential equations (SPDEs). The problem we tackle here is that of showing that 
when such equations are ``renormalised'' using the procedure given in \cite{BHZalg,CH}, the resulting
process is again the solution to a modified equation containing counterterms that only depend
in a local way on the solution itself.
%

A similar situation to the one dealt with here already arises in the 
classical theory of stochastic integration. There, one is faced with the problem of defining integrals with
respect to Brownian motion which, on a pathwise level, has insufficient regularity for the classical Riemann-Stieltjes integral to be well-defined.  
When establishing the convergence of discrete approximations one must take advantage of probabilistic cancellations in order to overcome this pathwise irregularity. 
Moreover, one sees that different classes of approximations that would have had the same limit for the case of regular drivers actually lead to different limiting integrals with different properties when working with irregular stochastic drivers\dash in the limit, one can obtain either the It{\^o} or Stratonovich integral, or any of a one-parameter
family of theories of stochastic integration which contains these as special cases \cite{GregAndrew}. 
In practice this choice is informed either by phenomenological considerations or by a desire 
for the integral to satisfy a given mathematical property. 

In the theory of parabolic, locally subcritical SPDEs, both the design of approximations and the framework for showing convergence of these approximations become more involved.
A rigorous solution\slash integration theory in this case was fairly intractable until just a few years ago\dash now there are 
several frameworks available that provide rigorous descriptions of what it means to be a (local) solution to these 
SPDEs: the theory of regularity structures \cite{Regularity}, 
the theory of paracontrolled distributions \cite{Paracontrol}, a Wilsonian renormalisation group approach \cite{Kupiainen2016}, 
and most recently the approach of \cite{OW}. Although these approaches differ in 
their technical details and their scope of application, the
solutions constructed with all of them do coincide for those examples in which more than one approach
applies.

As an example, suppose that one wants to develop a notion of solution for the Cauchy problem
associated to the system of SPDEs on $\R_+ \times \T^d$
\begin{equ}\label{eq: formal spde}
(\partial_{t} - \Delta)\phi_{j} = F_{j}(\phi,\nabla \phi) + \xi_j\;,
\end{equ}
where $(F_j)_{j=1}^m$ is a collection of local nonlinearities given by smooth functions. 
One can take the vector of ``drivers'' $\xi = (\xi_{j})_{j=1}^{m}$ to be a family of generalised random fields which are stationary, jointly Gaussian, and have covariances $\E[\xi_{j}(z)\xi_{k}(\bar{z})]$ which are smooth 
as long as $\bar{z} \not = z$ but behave like a homogeneous distribution of some negative degree near 
the diagonal  $z = \bar{z}$.
A sufficient condition for~\eqref{eq: formal spde} to be \emph{locally subcritical} is that, via power-counting considerations, the nonlinear term $F_{j}(\phi, \nabla \phi)$ is expected to be of better regularity than the driving noise $\xi_{j}$.\footnote{See Definition~\ref{def: obey for blackbox} and the remark following it for a formal definition.}

In many cases of interest one cannot solve \eqref{eq: formal spde} using classical deterministic methods since the lack of regularity 
of $\xi_j$ may force some $\phi_{j}$ to live in a space of functions\slash distributions on which $F_{j}(\phi,\nabla \phi)$ has no canonical meaning. 
A naive way to obtain a well-defined approximation to~\eqref{eq: formal spde} is to replace $\xi_j$ with $\xi_j^{(\eps)} = \xi_j \ast \rho_{\eps}$ where $\eps > 0$ and $\rho_{\eps}$ is a smooth approximation of the identity with $\lim_{\eps \downarrow 0} \rho_{\eps} = \delta$.
Then one has classical solutions $\phi_{\eps} = (\phi_{j,\eps})_{j=1}^{m}$ for the system of equations 
\begin{equ}\label{eq: mollified spde}
(\partial_{t} - \Delta)\phi_{j,\eps} = F_{j}(\phi_{\eps},\nabla \phi_{\eps}) + \xi_j^{(\eps)}\;,
\end{equ}
Unfortunately, in the generic situation, $\phi_{j,\eps}$ will either fail to converge as $\eps \downarrow 0$, or
 converge to a trivial limit as in \cite{Ryser}, so that simply replacing $\xi_j$ with $\xi_j^{(\eps)}$ 
does not allow one to 
define solutions to \eqref{eq: formal spde} via a limiting procedure. 

Upon studying the $\eps \downarrow 0$ behaviour of formal perturbative expansions for $\phi_{\eps}$,
one is naturally led to a more sophisticated approximation procedure. 
In general, one expects to find $n_{1}, \dots, n_{m} \in \N$ and, for each $j=1,\ldots, m$, a family of constants $\{c_{(j,i)}[\rho_\eps]\}_{i=1}^{n_{j}}$,
typically divergent as $\rho_{\eps} \rightarrow \delta$, as well as a family of functions 
$\{P_{(j,i)}(\cdot,\cdot)\}_{i=1}^{n_{j}}$, such that the classical solutions $\hat{\phi}_{\eps} = (\hat{\phi}_{j,\eps})_{j=1}^{m}$ to the system of equations
\begin{equ}\label{eq: renormalised spde}
(\partial_{t} - \Delta)\hat{\phi}_{j,\eps} = F_{j}(\hat{\phi}_{\eps},\nabla \hat{\phi}_{\eps})
-
\sum_{i=1}^{n_{j}}c_{(j,i)}[\rho_\eps]P_{(j,i)}(\hat{\phi}_{\eps},\nabla \hat{\phi}_{\eps}) + \xi_j^{(\eps)}\;,
\end{equ}
converge in probability as $\eps \downarrow 0$ to a tuple $\phi$ of limiting random distributions, 
which can be viewed as ``a solution''  to the system of equations~\eqref{eq: formal spde} and which
\textit{does not depend on the specific choice of approximation} $\xi_j^{(\eps)}$.
This process of using approximations where one regularises at a certain scale and then modifies the nonlinearity in a way that depends on this regularisation scale is called \emph{renormalisation}.
\begin{remark}\label{remark on renormalisation}
One may worry about the fact that~\eqref{eq: renormalised spde} no longer seems to relate to 
the ``real'' equation~\eqref{eq: formal spde} due to the presence of the additional counterterms $P_{(j,i)}$.
From a physical perspective however, this is not as unnatural as it may seem. Indeed, 
what one can typically ``guess'' from
physical arguments is not the specific system of equations \eqref{eq: formal spde}, but rather the generic form
of such a system, with the nonlinearities $F_j$ involving a priori unknown parameters (`coupling constants')
that then need to be determined a posteriori by matching predictions with experiments. From this
perspective, \eqref{eq: renormalised spde} is actually also of the form \eqref{eq: formal spde} and
simply corresponds to an $\eps$-dependent reparametrisation of the family of equations under
consideration. One way of interpreting this is that the whole \textit{family} of solutions
given by \eqref{eq: formal spde} and indexed by a suitable finite-dimensional collection of possible
nonlinearities $F$ converges to a limiting family of solutions as $\eps \to 0$, but the collection of
nonlinearities has to be suitably reparametrised in the process. A trivial but analogous situation is
the following. For any fixed $\eps$, consider the subset $A_\eps \subset \R^2$ parametrised by
$\R$ and given by $A_\eps = \{(x_\eps(t), y_\eps(t))\,:\, t\in \R\}$, where
\begin{equ}
x_\eps(t) = \eps t + \frac{2}{\eps} \;,\qquad y_\eps(t) = \eps \cos(t)\;.
\end{equ}
While it is clear that $A_\eps \to A_0$ with $A_0 = \R \times \{0\}$, $x_\eps$ and $y_\eps$
do not converge to a parametrisation of $A_0$, although they do if we perform the $\eps$-dependent
reparametrisation $t \mapsto t/\eps - 2/\eps^2$ and write instead
$A_\eps = \{(\hat x_\eps(t), \hat y_\eps(t))\,:\, t\in \R\}$ with
\begin{equ}
\hat x_\eps(t) = t \;,\qquad \hat y_\eps(t) = \eps \cos \Bigl(\frac{t}{\eps} - \frac{2}{\eps^2}\Bigr)\;.
\end{equ}
In this analogy, $t$ plays the role of $F$, $(x_\eps,y_\eps)$ plays the role of the solution map 
$\phi_\eps$, while $(\hat x_\eps,\hat y_\eps)$ plays the role of the ``renormalised''
solution map $\hat \phi_\eps$. 
\end{remark}

While perturbative methods can shed light on the mechanics of renormalisation, they are limited 
to proving statements about the term by term behaviour of formal expansions for $\hat{\phi}_{\eps}$ 
which one does not expect to be summable. 
The jump from knowing how to set up approximations such as \eqref{eq: renormalised spde} to showing that the solutions $\hat{\phi}_{\eps}$ 
do actually converge as $\eps \downarrow 0$ requires fundamentally new ideas and is the main achievement of the methods 
developed in \cite{Regularity,Paracontrol,Kupiainen2016,OW}.
\subsection{A review of the theory of regularity structures}
The setting of the current work is the theory of regularity structures \cite{Regularity}, 
so we quickly present the theory's central ideas. 
Those seeking more pedagogical expositions are encouraged to look at \cite{FrizHairer,CW17,CDM}. 
The approach of the theory can be illustrated by the diagram shown in Figure~\ref{fig:diagram}.
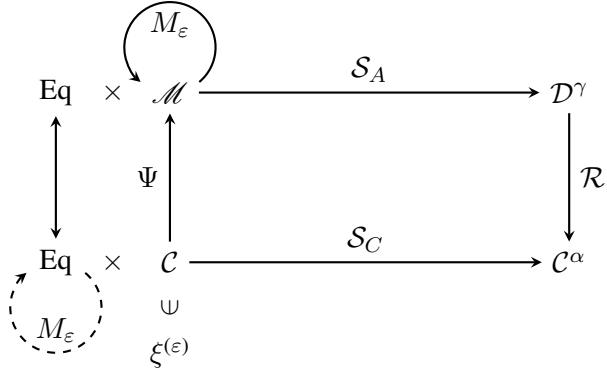
\begin{figure}
\begin{center}
\begin{tikzpicture}[thick,scale=1.5]
  \node (Fr) at ( -1,1.5) [] {$\mathrm{Eq}$};
  \node (X) at ( 0,1.5) [] {$\mathscr{M}$};
  \node at ( 0,2.1) [] {$M_{\eps}$};
  \draw[->] (-50:0.4) ++(0,1.9) arc (-50:230:0.4);
  \node at ( -0.5,1.5) [] {$\times$};
  \node (end) at ( 3.5,1.5) [] {$\CD^{\gamma}$};

  \node (F) at ( -1,0) [] {$\mathrm{Eq}$};
  \draw[dashed,->] (50:0.4) ++(-1,-0.4) arc (50:-230:0.4);
  \node at ( -1,-0.6) [] {$M_{\eps}$};
  \node at ( -0.5,0) [] {$\times$};
  \node (O) at ( 0,0) [] {$\CC$};
  \node at ( 0,-0.8) [] {$\xi^{(\eps)}$};
  \node at ( 0,-0.4) [rotate=90] {$\in$};

  \node (end1) at ( 3.5,0) [] {$\CC^{\alpha}$};
  \draw [->] (end) --node [right] {$\CR$} (end1);
  \draw [->] (O) -- node [left] {$\Psi$} (X);
  \draw [<->] (F) -- (Fr);
  \draw [->] (X) --node [above] {$\CS_{A}$} (end);
  \draw [->] (O) --node [above] {$\CS_{C}$}  (end1);
\end{tikzpicture}
\caption{Mechanism of renormalisation}\label{fig:diagram}
\end{center}
\end{figure}

In this figure, $\mathrm{Eq}$ denotes a space of possible equations. While the instances of $\mathrm{Eq}$ on the top and bottom lines can 
be thought of as the same, we will see them as playing different roles. The choice of an element in $\mathrm{Eq}$ on the bottom line 
will be called a \emph{concrete} equation and the choice of an element in $\mathrm{Eq}$ on the top line will be called an \emph{abstract} equation. 

Continuing on the bottom line, we denote by $\CC$ a space of continuous (or sufficiently smooth) 
functions defined on the underlying space-time\dash this is where regularised realisations of our driving noise live, but this space
is typically much too small to contain instances of the limiting noise $\xi$.
The space $\CC^{\alpha}$ is a H\"older-type space of space-time functions\slash distributions where the solution to the equation at hand will live. 
Given a concrete equation and a regularised driving noise $\xi^{(\eps)}$ the classical solution map $\CS_{C}$ returns the solution to the specified concrete equation starting from $0$ (or some other specified initial condition) and driven by $\xi^{(\eps)}$.

While the map $\CS_{C}$ is well-defined when the driving noise is drawn from $\CC$, it lacks sufficient continuity in this argument to be 
well-defined on any of the distributional spaces in which the convergence $\xi = \lim_{\eps \downarrow 0} \xi^{(\eps)}$ takes place. 
This is actually already the case for stochastic ordinary differential equations, see \cite{TerryPath}.
The theory of regularity structures follows the philosophy of (controlled) rough paths 
\cite{Lyons,MaxControl,MR2314753,MR2667703,FrizHairer} and builds a continuous solution map $\CS_{A}$ at the price of defining it on a richer space\dash one must feed into the map $\CS_{A}$ not just a realisation of the driving noise but also a suitable 
``enhancement'', which encodes various multilinear functionals of the driving noise that are a priori ill-defined. 

Such a collection of data is referred to as a \emph{model} in the terminology of regularity structures, with
the space of models $\mathscr{M}$ being a fairly complicated nonlinear metric space.
The multilinear functionals one must define in order to specify an element of $\mathscr{M}$ are such that they can be defined canonically when evaluated on regularised instances of the noise but have no such interpretation when evaluated on an un-regularised realisation\dash this is because one encounters ill-defined pointwise products of rough functions and distributions.  
Consequently, on the space of regularised realisations of the noise $\CC $, one has a canonical lift $\Psi : \CC \rightarrow \mathscr{M}$ but this lift does not extend continuously 
to typical realisations of $\xi$. 

One can also define a bundle\footnote{Strictly speaking, the space $\CD^{\gamma}$ doesn't satisfy the axioms of a 
vector bundle because fibres corresponding to different models are not isomorphic in general. At an 
algebraic level, the object in each fibre is always a ``jet''-valued function, but the analytic requirements
we impose on this do depend on the underlying model and can be very different, even for nearby fibres.} $\CD^{\gamma}$ of H{\"o}lder-type spaces of abstract jets over $\mathscr{M}$ 
where the abstract equation can be formulated as a well-posed fixed point problem.
The fixed point yields a solution map $\CS_A$ that is a continuous section of the bundle $\CD^{\gamma}$:
given $Z \in \mathscr{M}$,  $\CS_{A}[Z]$ belongs to the fibre over the model $Z$. 

The map $\CR$ appearing on the very right is the \emph{reconstruction operator} which is a continuous map from the bundle $\CD^{\gamma}$ to some H{\"o}lder space $\CC^{\alpha}$ of space-time functions\slash distributions. 
The key point of the diagram above is that the square commutes, namely $\CR \circ \CS_{A} \circ \Psi= \CS_{C}$. This 
factorisation of $\CS_{C}$ separates difficulties: the map $\Psi$ is discontinuous, but has the advantage of being given 
explicitly, while the map $\CS_A$ is given as the solution to a fixed point problem but has the advantage of being continuous.

The incorporation of renormalisation in the abstract setting is done by 
replacing the canonical lift $\Psi: \CC \rightarrow \mathscr{M}$ by a different lift which is allowed to break the
usual definition of a product. The space of those deformations of the product that are
allowed and that preserve stationarity is itself rather small.
In particular, one can exhibit a finite-dimensional Lie group $\mfR$ acting on $\mathscr{M}$
(in this case by a right action or equivalently by a left action of the adjoint) which 
parametrises all ``natural'' lifts of the noise.
The art of renormalisation then involves remembering that $\xi^{(\eps)}$ is random, and choosing, for each $\eps > 0$, a 
\textit{deterministic} element $M_{\eps} \in \mfR$, determined by the law of $\xi^{(\eps)}$, such that the random models $M_{\eps}^* \circ \Psi [\xi^{(\eps)}]$ converge in probability as $\eps \to 0$. 
If this can be done, then thanks to the pathwise continuity of $\CR$ and $\CS_{A}$, one concludes 
that 
\begin{equ}[e:renormSol]
\hat \CS_C^{(\eps)}[\xi^{(\eps)}] \eqdef \CR \circ \CS_{A} \circ M_{\eps}^* \circ \Psi[ \xi_{\eps}]
\end{equ}
also converges in probability as $\eps \to 0$ to some limiting ``renormalised solution map'' 
$\hat \CS_C$, which is only defined almost surely with respect to the law of $\xi$.

The overall framework of the theory of regularity structures was set forth in \cite{Regularity}. 
The theory was designed to be robust and fairly automated in that it does not need to be modified on an equation by equation basis but three of the above steps were left to the person applying the theory in general: 
\begin{enumerate}
\item[(i)] the construction of a Lie group $\mfR$ rich enough to contain $\{M_{\eps}\}_{\eps > 0}$, 
\item[(ii)] proving the convergence of the renormalised 
models $M_{\eps}^* \circ \Psi [\xi^{(\eps)}]$, 
\item[(iii)] showing that $\hat \CS_C^{(\eps)}$
actually coincides with the classical (not renormalised) solution map, but for a \textit{modified} equation.
\end{enumerate}
Robust theorems which automate the first two of these steps were recently obtained 
in \cite{BHZalg} and \cite{CH}, respectively. 
The aim of the current article is to give a general proof of the last step.

Note that the action of $M_{\eps}$ on the top line of Figure~\ref{fig:diagram} doesn't change the 
fixed point problem
used to build the map $\CS_A$, it only changes the model which is used as an input to this map. 
This deformation of the canonical lift generates a discrepancy between how we interpret products 
on the top and bottom lines of our diagram and as a result 
$\hat \CS_C^{(\eps)} \not =  \CS_{C}$. 
The purpose of the present article is to describe a 
corresponding action of $\mfR$ on a suitable space of equations $\mathrm{Eq}$ so that 
the identity
\begin{equ}
(\CR \circ \CS_{A})(F, M^* \Psi (\xi)) = \CS_C( M F, \xi)\;,
\end{equ}
holds for every smooth noise $\xi$, every right hand side $F \in \mathrm{Eq}$, and 
every $M \in \mfR$.

In~\cite{BCFP} the authors identified such an action in the simpler setting of regularity structures arising from branched rough paths, which gave rise to a natural morphism of \emph{pre-Lie algebras}. 
The approach in~\cite{BCFP} inspired that of the present work, however the setting here is quite a bit more complex. 

The problem of identifying the action of the renormalisation group on the equation is also found in perturbative quantum field theory (QFT) where one checks that the counterterms one would like to insert in order to make individual Feynman diagrams finite can be generated order by order by changing the coupling constants in the Lagrangian that was used to generate these terms in the first place. 
The fact that the Lagrangian can be modified in this way can usually be checked quite easily on a case by case basis\dash examples can be found in any textbook on QFT. 
However, we have not been able to find work analogous to the present work in the perturbative QFT literature\dash this would be a theorem which gives an explicit and model-independent formalism for deriving renormalised Lagrangians. 
\begin{remark}\label{second remark on renormalisation}
Continuing the thread of Remark~\ref{remark on renormalisation}, we can view our full space of equations as being parameterised by a family of coupling constants $\vec{c} = (c_{(j,i)}) \in \R^{K}$, where we range across $
1 \le j \le m$, 
$1 \le i \le n_{j}$,
and $K = \sum_{j=1}^{m} n_{j}$.  
The correspondence between the coupling constants $\vec{c}$ and the equation is given by~\eqref{eq: renormalised spde}.
The dual action of the renormalisation group on the equation is then just a representation of $\mfR$ on $\R^{K}$.

As shown in \cite{CH}, we can choose the sequence $M_{\eps}$ to depend on our choice of sequence $\rho_{\eps}$ of approximate identities in such a way 
that the limit $M_{\eps}[\rho_{\eps}]^* \circ \Psi[\xi \ast \rho_{\eps}]$ is independent of the choice
of $\rho_\eps$. 
Once one has obtained one limiting model $\lim_{\eps \downarrow 0} M_{\eps}^* \Psi[\xi^{(\eps)}]$, then an 
entire family of models is obtained via
\begin{equ}\label{family of limiting models}
\Big\{ \lim_{\eps \downarrow 0} M^* M_{\eps}^* \Psi[\xi^{(\eps)}]: M \in \mfR \Big\}\;.
\end{equ}
Once one fixes an initial choice $\vec{c}$ of coupling constants, every model in \eqref{family of limiting models} gives rise to a notion of solution which can be obtained as the $\eps \downarrow 0$ limit of the classical solution to \eqref{eq: renormalised spde} driven by $\xi^{(\eps)}$ and with coupling constants given by $M M_{\eps}\vec{c}$.

We stress that there is in general not a \emph{canonical} model or solution theory that can be pointed out in the 
family~\eqref{family of limiting models}. 
This is because, even though the BPHZ lift constructed in \cite{BHZalg,CH} seems canonical to a certain extent, 
it depends in general on
an arbitrary choice of (scale $1$) cutoff in the Green's function for the linear system.
Different choices of cutoff yield solutions
that differ by the action of an element of $\RR$, but no single choice of cutoff is more canonical than the 
others in general.
We reiterate however that 
\begin{claim}
\item[(i)] If a specific solution is required for modelling purposes, then its parameters do have to be 
determined by comparisons with data\slash experiments anyway. This will then determine a unique element of 
the family of solutions, which is independent of the parametrisation of the family that is being used. 
\item[(ii)] The exact same sensitivity\slash indeterminacy in the notion of solution is present already in the case of 
the theory of integration against Brownian motion.\footnote{This is not just an analogy as integration against 
Brownian motion falls within the framework of regularity structures and the choice between It\^o and 
Stratonovich integrals, or some interpolation of the two, is parameterised by the corresponding 
renormalisation group. The difference is that, in the case of SDEs, no renormalisation 
is in principle required and every smooth regularisation of Brownian motion yields the Stratonovich solution in the
limit. On the other hand, the BPHZ renormalisation procedure, which is the natural way of centring random models
used in the present article, always yields the It\^o solution in the limit. (In the special case of SDEs, the effect of 
the cutoff of the Green's function happens to vanish in the limit.)
}
\end{claim}
\end{remark}
\subsection{Outline of the paper}
Section~\ref{sec: black box} introduces the bare minimum in order to state an existence result, namely Theorem~\ref{thm: THE black box}, which is applicable to a wide class of semilinear SPDEs. 
This section can be read without any prior knowledge of the theory of regularity structures and, with the exception of Section~\ref{subsec: preliminary notation}, it can be skipped by those who are 
more interested in learning the method of proof for the main results of this paper.
In Section~\ref{sec: examples} we illustrate two applications of Theorem~\ref{thm: THE black box} to the generalised KPZ equation and the dynamical $\Phi^4_{4-\delta}$ model for any $\delta > 0$.

In the early parts of Section~\ref{sec: alg and main theorem} we recall some of the basic algebraic definitions from the theory of regularity structures and describe how we specialize them for our purposes. 
In Section~\ref{sec: gen construction of inner products} we introduce a formalism that allows us to efficiently deal with some of the combinatorial symmetry factors that appear when we work with spaces of combinatorial decorated trees.

After this preliminary work, we introduce the notion of \emph{coherence} in Section~\ref{subsec:coherence}, which plays a central role in the paper. 
Given a PDE determined by some right hand side $F$, we first define a function 
$\Upsilon^{F}[\cdot]$ on the trees of the corresponding regularity structure.  
In the case of a single scalar equation, we then say that a linear combination of trees $U$ is coherent if the coefficient of every tree of the form $\mcb{I}[\tau]$ in the expansion of $U$ is given by $\Upsilon^{F}[\tau]$ evaluated on the coefficients of the polynomial part of $U$. 
After introducing this concept, we present the first of two key lemmas, Lemma~\ref{lemma: coherence identity}, which states that coherence of $U$ with $F$ is equivalent to $U$ satisfying a fixed point problem determined by $F$.  The notion of coherence is close in the spirit to the B-series in numerical analysis. 
Indeed, B-series are numerical stepping methods for ODEs represented by a tree expansion whose coefficients are given by an analogue of the map $  \Upsilon^{F}[\cdot]  $. They were originally introduced to describe Runge-Kutta methods and have proven to be a 
powerful tool for classifying various numerical methods see  \cite{MR0305608,hairer74,Murua2006,MR2657947,MR2803804}.

In Section~\ref{subsec: main thm}, we describe how $\Upsilon^{\bullet}$ allows us to define an action of the renormalization group $\mathfrak{R}$ on the space of $F$'s which we write $F \mapsto MF$ for $M \in \mathfrak{R}$.
We can then present our second key lemma, Lemma~\ref{lem: renormalisation of upsilon}, which states that any renormalisation operator $M \in \mathfrak{R}$ takes expansions coherent with respect to $F$ to expansions coherent with respect to a new nonlinearity $MF$.
We conclude Section~\ref{sec: alg and main theorem} by presenting the main theorem of this paper, 
Theorem~\ref{thm: algebraic main theorem}, which gives the general form of the action of
the renormalisation group onto a suitable space of nonlinearities, 
and show how this theorem follows from 
Lemmas~\ref{lemma: coherence identity} and~\ref{lem: renormalisation of upsilon}. 

Sections~\ref{sec:SpacesOfTrees} and \ref{sec: grafting and renormalisation} are devoted to developing an 
algebraic\slash combinatorial framework in which we can prove Lemmas~\ref{lemma: coherence identity} and \ref{lem: renormalisation of upsilon}. 
In Section~\ref{sec:SpacesOfTrees} we introduce a new collection of trees which carry more data through additional decorations which greatly facilitates the proof of Lemma~\ref{lemma: coherence identity}. 
In Section~\ref{sec: grafting and renormalisation} we define various ``grafting'' operations on trees. 
A key result here is Proposition~\ref{prop:generate}, which states that a certain space of trees is the ``universal free object'' corresponding to our grafting operators. 
We then also state lemmas showing that the maps $\Upsilon^{\bullet}$ and the renormalisation operators $M \in \mfR$ all have ``morphism'' properties with respect to these grafting operators.
This, when combined with Proposition~\ref{prop:generate}, allows us to prove Lemma~\ref{lem: renormalisation of upsilon}. 

Section~\ref{sec: analytic aspects and DPD} is the analytic part of our paper which is needed to prove Theorem~\ref{thm: THE black box}. 
Sections~\ref{subsec: admissible models},\ref{subsec:space of jets}, and \ref{subsec: modelled distributions} recall many analytic objects in the theory of regularity structures and describe how we will specialize them for our purposes.
In Section~\ref{subsec: coherence of fp} we state and prove Theorem~\ref{thm:renormalised_equation}, which 
is obtained by combining Theorem~\ref{thm: algebraic main theorem} with the analytic theory given in the earlier parts of Section~\ref{sec: analytic aspects and DPD}.
One novel aspect of Theorem~\ref{thm: THE black box} is that it states, with full generality, to what degree one can expect to ``restart'' solutions to the class of SPDE under consideration and consequently what a natural notion of ``maximal solution'' should be. 
To facilitate this, we develop a new argument which could be loosely described as an analogue of the Da Prato--Debussche trick \cite{DPD2} in the space of modelled distributions\dash this is the content of Section~\ref{subsec: gen DPD}. 

In Appendix~\ref{subsec:GPreserveProof} we state a technical result describing how the renormalisation of nonlinearities influences the trees they generate via Duhamel expansion.  
In Appendix~\ref{proof lemma Upsilon sum identity} we give a multivariate Faa Di Bruno formula which allows us to show that work performed on the ``richer'' space of trees introduced in Section~\ref{sec:SpacesOfTrees} collapses appropriately to the smaller trees which populate the regularity structure. 
In Appendix~\ref{subsec:coInteractProofs} we give the details of the proofs describing how renormalisation interacts 
with the grafting operations, which leverages the co-interaction property of~\cite{BHZalg}. 
In Appendix~\ref{subsec:freelyGenProof} we describe how the techniques of \cite{Chapoton01} can be used to prove Proposition~\ref{prop:generate}.
In Appendix~\ref{sec: proof of blackbox} we describe how our abstract result can be combined with the framework of regularity structures to prove Theorem~\ref{thm: THE black box}.

The reader need not have any familiarity with \cite{CH}, but some familiarity with 
the frameworks of \cite{Regularity} and \cite{BHZalg} is assumed throughout the paper. 
\subsection*{Acknowledgements}
{\small
We thank Hao Shen for pointing out an error in the published version of this paper, see Remark~\ref{rem:motivating_newTheta}. 

MH gratefully acknowledges financial support from the
 Leverhulme Trust through a leadership award
 and from the European Research Council through a consolidator grant, project 615897.
IC is funded by a Junior Research Fellowship of St John's College, Oxford.
AC gratefully acknowledges financial support from the Leverhulme Trust via an Early Career Fellowship, ECF-2017-226.
}
\section{A black box theorem for local well-posedness of SPDEs}\label{sec: black box}

\subsection{Preliminary notation}\label{subsec: preliminary notation}
Throughout this article, we adopt the standard conventions $\sup \emptyset \eqdef -\infty$ and $\inf \emptyset \eqdef +\infty$.
We freely use multi-index notation. 
For any set $A$, $a \in A$, and $\theta \in \N^{A}$ we usually write $\theta[a]$ for the $a$-component of $\theta$,  $|\theta| \eqdef \sum_{a \in A} \theta[a]$, and $\theta! \eqdef \prod_{a \in A} \theta[a]!$. 
For a vector of commuting indeterminates (or real numbers) $x = (x_{a})_{a \in A}$, 
we similarly write $x^{\theta} = \prod_{a \in A} (x_{a})^{\theta[a]}$. 
For any $a \in A$ we define $e_{a} \in \N^{A}$ by setting $e_{a}[b] \eqdef \mathbbm{1}\{a = b\}$\label{ei page ref} for $b \in A$.

In \cite{BHZalg} and in this paper one often uses the notion of \emph{multisubsets} of some fixed set $A$. 
By saying $B$ is a multisubset of $A$ we are indicating that $B$ can contain certain elements of $A$ with multiplicity. 
In \cite{BHZalg} the collection of all multisubsets of $A$, denoted by $\hat{\mcb{P}}(A)$,\label{multiset pageref} is given by $\bigsqcup_{n \ge 0}[A]^n$ where $[A]^{n}$ is $A^{n}$ quotiented by permutation of entries.
We will implicitly identify $\hat{\mcb{P}}(A)$ with $\N^{A}$.
In particular, for $b \in \N^{A}$ we adopt the notational convention that
\[
\sum_{a \in b}
x_{a}
\eqdef
\sum_{a \in A}
b[a]x_{a}\;.
\]
Similarly, for $b_{1}, b_{2} \in \N^{A}$, writing the interpretation as ``multi-sets'' like in \cite{BHZalg} on the left and multi-indices on the right, one has
\begin{equs}[2]
b_{1} \cap b_{2} &= \min\{b_{1},b_{2}\}\;,\quad&\quad
b_{1} \cup b_2 &= \max\{b_1,b_2\}\;,\\ 
b_1 \sqcup b_2 &= b_1 + b_2\;,\quad&\quad
b_1 \subset b_2 &\Leftrightarrow b_1 \le b_2\;.
\end{equs}
We fix for the rest of the paper a dimension of space\footnote{The case $d=0$ corresponds to working with stochastic differential equations.} $d \ge 0$. 
We define our space-time to be $\Lambda \eqdef \R \times \T^{d}$ with the first component being 
referred to as ``time''.\label{introduction of Lambda} 
We write $\{ \d_i \}_{i=0}^{d}$ for the corresponding partial derivatives with respect to space-time. 
We will also sometimes identify functions on $\Lambda$ with functions on $\R^{d+1}$ by periodic continuation.
A space-time scaling\label{space-time scaling introduced} $\s$ is a tuple $\s = (\s_{i})_{i=0}^{d} \in [1,\infty)^{d+1}$ with non-vanishing components.
Given a space-time scaling $\s$ we set $|\s| \eqdef \sum_{i=0}^{d} \s_{i}$ and, for any $\eps > 0$, we define a scale 
transformation $S^{\eps}_{\s}$\label{scale transform page ref} on functions $\rho: \R^{d+1} \rightarrow \R$ by setting $(S^{\eps}_{\s}\rho)(z_{0},\dots,z_{d}) \eqdef \eps^{-|\s|}\rho(\eps^{-\s_{0}}z_{0},\dots,\eps^{-\s_{d}}z_{d})$.

We also introduce a notion of $\s$-degree $|k|_{\s}$ of a multi-index $k \in \N^{d+1}$ by setting $|k|_{\s} = \sum_{i=0}^{d} k[i] \s_{i}$. 
This gives a corresponding notion of $\s$-degree for polynomials and $\s$-degree of partial derivatives. 
We define  a scaled distance $|\cdot|_{\s}$ on $\Lambda$ as usual by setting 
$
|z|_{\s} 
\eqdef
\sum_{i=0}^{d}
|z_{i}|^{1/\s_{i}}
$.
We often write $\bar{\s} = (\s_{i})_{i=1}^{d}$ for the associated space-scaling on $\R^{d}$.\label{space scaling introduced} 
Clearly, one has natural analogues of all the above notation when working with distances and degrees of 
polynomials\slash derivatives on $\R^{d}$ with respect to the scaling $\bar \s$ and we use these in what follows.
\subsection{H{\"o}lder-Besov Spaces}\label{sec: holder-besov}
While the bulk of this paper is algebraic\slash combinatorial, our statement and proof of the main theorem of Section~\ref{black box section} requires us to reference scaled H{\"o}lder-Besov spaces.
This subsection can be skipped by those readers who are more interested in our main result Theorem~\ref{thm:renormalised_equation} as opposed to Theorem~\ref{thm: THE black box}. 

We first specialize to the case of $\alpha \in (0,\infty) \setminus \N$. 
We define, for every compact set $\mathfrak{K} \subset \Lambda$ and any function $f:\mathfrak{K} \rightarrow \R$, 
\[
\|f\|_{\alpha,\mathfrak{K}}
\eqdef
\sup_{
\substack{
z \in \mathfrak{K}\\
|k|_{\s} \le \lfloor \alpha \rfloor }
}
\inf
\Big\{
\sup_{
\substack{
\bar{z} \in \mathfrak{K}\\
|\bar{z} - z|_{\s} \le 1}
}
\frac{
\big|\big(\d^{k}f\big)(\bar{z}) - \big(\d^{k} P\bigr)(\bar{z} - z)\big|}
{|z-\bar{z}|_{\s}^{\alpha - |k|_{\s}}}:
\deg_\s P \le  \lfloor \alpha \rfloor
\Big\}\;,
\]
where $\deg_\s P$ denotes the $\s$-degree of the polynomial $P$.
We then define $\CC_{\s}^{\alpha}(\Lambda)$ to be the collection of $f: \Lambda \rightarrow \R$ with $\|f\|_{\alpha,\mathfrak{K}} < \infty$ for every compact $\mathfrak{K} \subset \Lambda$, and we equip it with a metric induced by these seminorms. 
If $d = 1$ and $\s = 1$ then $\CC_{\s}^{\alpha}(\Lambda)$ just corresponds to functions that admit $\lfloor \alpha \rfloor$ continuous derivatives and whose $\lfloor \alpha \rfloor$-th derivative is H{\"o}lder continuous of 
index $\alpha - \lfloor \alpha \rfloor$.

We now define $\CC_{\s}^{\alpha}(\Lambda)$ for $\alpha \in (-\infty,0)$. 
First, for every $r \in \N$ and $z \in \Lambda$ we define $\mathcal{B}_{z,r}$ to be the collection of all smooth functions $\omega: \Lambda \rightarrow \R$ which are supported on the ball $\{\bar{z} \in \Lambda:\ |\bar{z} - z|_{\s} \le 1\}$ and satisfy $\sup_{\bar{z} \in \Lambda} |\partial^{k}\omega(\bar{z})| \le 1$ for every $k \in \N^{d}$ with $|k|_{\s} \le r$. 
For any compact $\mathfrak{K} \subset \Lambda$ and distribution $f \in \mcb{S}'(\Lambda)$, we then set
\begin{equ}\label{Holder Besov seminorm}
\|f\|_{\alpha,\mathfrak{K}}
\eqdef
\sup
\Bigl\{
\lambda^{-\alpha} |( f, S^{\lambda}_{\s}\omega )| \,:\,
z \in \mathfrak{K},
\omega \in B_{z, \lceil - \alpha \rceil}, \lambda \in (0,1]
\Bigr\}\;.
\end{equ}
and we define $\CC_{\s}^{\alpha}(\Lambda)$ to be the collection of distributions $f$ with $\|f\|_{\alpha,\mathfrak{K}} < \infty$ for every compact $\mathfrak{K} \subset \Lambda$. As before, we equip $\CC_{\s}^{\alpha}(\Lambda)$ with a metric induced by these seminorms. 
The spaces $\CC^{\alpha}_{\bar{\s}}(\T^{d})$, $\alpha \in \R$, are defined in the analogous way.
\subsection{Types, nonlinearities, and functional derivatives}
We fix a finite set $\mfL_{-}$\label{set of negative types} which will index the set of rough driving noises that appear in our system of SPDEs and a finite set $\mfL_{+}$\label{set of positive types} which will index the set of components of our system of SPDEs. 
We also fix a \emph{degree assignment $|\cdot|_{\s}$} on $\mfL \eqdef \mfL_{-} \sqcup \mfL_{+}$ which takes strictly negative values on $\mfL_{-}$ and strictly positive ones on $\mfL_{+}$.
For $\mfl \in \mfL_{-}$ the value of $|\mfl|_{\s}$ represents an assumption on the regularity of the corresponding driving noise and for $\mft \in \mfL_{+}$ the value $|\mft|_{\s}$ represents an assumption on the regularizing properties of the inverse of the linear differential operator appearing in the $\mft$ equation.  
For any multi-set $A$ of elements of $\mfL$ we define $|A|_{\s} \eqdef \sum_{\mft \in A} |\mft|_{\s}$.

We define an indexing set $\mcb{O} \eqdef \mfL_{+} \times \N^{d+1}$\label{mcbO page ref} and write elements of this set as $(\mfb,q) \in \mcb{O}$ with $\mfb \in \mfL_{+}$ and $q \in \N^{d+1}$. For $(\mfb,q) \in \mcb{O}$ we write $|(\mfb,q)|_{\s} \eqdef |\mfb|_{\s} - |q|_{\s}$, where $|q|_\s = \sum_{i=0}^d q_i \s_i$. 
One should think of $\mcb{O}$ as indexing all the solutions and derivatives of solutions of our system of SPDEs.
We also assume that we have a partition $\mcb{O} \eqdef \mcb{O}_{+} \sqcup \mcb{O}_{-}$. 
This partition corresponds to an a priori assumption that the elements of $\mcb{O}_{-}$ will index space-time distributions of negative regularity, while $\mcb{O}_{+}$ will index functions of positive regularity,\footnote{Note that $|(\mfb,q)|_{\s}$ is \emph{not} an estimate on the regularity of the $q$-th derivative of the solution to the $\mfb$ equation and thus says nothing about whether $(\mfb,q) \in \mcb{O}_{+}$ or $\mcb{O}_{-}$.} see Section~\ref{subsec: subcrit}.
We introduce a family of commuting indeterminates $\Y \eqdef (\Y_{o})_{o \in \mcb{O}}$\label{Y page ref}. 
The indeterminate $\Y_{o}$ will, depending on context, serve as a placeholder for a portion of the abstract expansion corresponding to the (derivative of the) component of the solution indexed by $o$ or for 
a reconstruction of that expansion.

We write $\mcb{C}_{\mcb{O}}$ for the real algebra of smooth functions on $\R^{\mcb{O}}$ which depend on only 
finitely many components, which we henceforth identify with functions of $\Y$. 
Given $F \in \mcb{C}_{\mcb{O}}$ we write $\mcb{O}(F)$ for the minimal subset of $\mcb{O}$ such that $F$ does not depend on any components of $\mcb{O}$ outside of $\mcb{O}(F)$.

We introduce two types of differential operators on $\mcb{C}_{\mcb{O}}$. 
For $o \in \mcb{O}$ we write $D_{o}:\mcb{C}_{\mcb{O}} \rightarrow \mcb{C}_{\mcb{O}}$ for the operation of differentiation with respect to $\Y_{o}$. 
We also define, for every $0 \le i \le d$,
and every $(\mft,p) \in \mcb{O}$, $\d_i \Y_{(\mft,p)} = \Y_{(\mft,p + e_{i})}$ and we impose the chain rule
\begin{equ}\label{chain rule}
\partial_{i}F
\eqdef
\sum_{o \in \mcb{O}}
\d_i \Y_{o}\,
D_{o}F\;.
\end{equ}
\begin{remark}
Let us pause for a moment to help the reader follow our use of derivative and multi-index notation throughout the paper.
We will make use of the convention of Section~\ref{subsec: preliminary notation} for multi-indices, viewing each of the families $\{D_{o}\}_{o \in \mcb{O}}$ and $\{\partial_{i}\}_{i=0}^{d}$ as indeterminates.

We also point out that we are overloading the notation $\partial^{k}$ here since it clashes its 
the more standard use in Section~\ref{sec: holder-besov}. 
However, the meaning of $\partial^{k}$ should be clear from context, as when it is acting on a function of the indeterminates $\{ \Y_{\mcb{O}}: o \in \mcb{O}\}$ one should use the definition~\eqref{chain rule}.

For a more detailed motivation for~\eqref{chain rule}, we refer the reader to~\eqref{meaning of derivatives}.
In a nutshell, one should think of the variable $\Y_{(\mft,p)}$ as representing the $p$th derivative
(in space-time) of the component $u_\mft$ of the solution. In this sense, the notation \eqref{chain rule} is then consistent
with the traditional usage of the symbol.
\end{remark} 
We define $\Poly$\label{poly page ref} to be the sub-algebra of $\mcb{C}_{\mcb{O}}$ consisting of all elements $F \in \mcb{C}_{\mcb{O}}$ for which one can write
\begin{equ}\label{expansion of nonlinearity}
F(\Y)
=
\sum_{j=1}^{m}
F_{j}(\Y) \Y^{\alpha_{j}}\;,
\end{equ}
where $\alpha_{1},\dots,\alpha_{m} \in \N^{\mcb{O}}$ are distinct, supported on $\mcb{O}_{-}$, and have only finitely many non-zero components and, for all $1 \le j \le m$, the element $F_{j}(\Y)$ is not identically 
$0$ and $\mcb{O}(F_{j}) \subset \mcb{O}_{+}$.
Note that in the representation \eqref{expansion of nonlinearity} we put, in each term, all of the dependence of the positive degree parts in the factor $F_{j}(\Y)$ (even if this dependence is itself polynomial).

Once one has a representation of the form \eqref{expansion of nonlinearity} it is unique (modulo permutations of the index $j$).
An equivalent definition of $\Poly$ is the collection of all $F \in \mcb{C}_{\mcb{O}}$ such that there exists $\alpha \in \N^{\mcb{O}}$ supported on $\mcb{O}_{-}$ with $D^{\alpha}F = 0$.
 
The nonlinear terms appearing in our systems of equations will be viewed as elements of $\Poly$, namely we are restricting ourselves to the case where any dependence on the rough distributions appearing in the equation is polynomial.   

\begin{lemma}\label{lem:partialK F vanishes}
Let $F \in \mcb{C}_{\mcb{O}}$ and $k \in \N^{d+1}\setminus\{0\}$. Then $\partial^k F = 0$ if and only if $F$ is constant.
\end{lemma}
\begin{proof}
If $F$ is constant, then evidently $\partial^k F = 0$.
Conversely, suppose $F$ is not constant.
Then there exists a (not necessarily unique) element $(\mft,p) \in \mcb{O}$ 
for which $D_{(\mft,p)} F \neq 0$ but $D_{(\mft,\bar p)} F = 0$ for all $\bar p > p$.
Using Definition~\ref{chain rule}, the fact that $D_{o}$ and $D_{\bar o}$ commute, and that $D_o\Y_{\bar o} = \delta_{o,\bar o}$, we see that $D_{(\mft,p+e_i)}\partial_i F = D_{(\mft,p)}F$ for every $i \in \{0,\ldots, d\}$, whence we 
conclude that $\partial_i F$ is not constant either. The conclusion then readily follows by induction.
\end{proof}
\subsection{Example: the generalised KPZ equation}\label{sec: example of genKPZ part 0}
We make a small aside to clarify the abstract notation we have introduced above by looking at how a concrete example can be recast in this setting. Some of the particular choices we make here are motivated in Section~\ref{sec: example of genKPZ}. 

We consider the generalised KPZ equation (as described in~\cite{KPZg}), a natural stochastic evolution on loop space.
We set $d=1$, $\s = (2,1)$, and fix $n > m \ge 1$. 
We are studying the evolution of a loop on an $m$-dimensional manifold in local coordinates, it follows that we will have a system of $m$ scalar equations and so we fix some set $\mfL_{+}$ with $|\mfL_{+}| = m$.  
Our dynamics will be driven by $n$ noises so we fix a set $\mfL_{-}$ with $|\mfL_{-}| = n$. 
Recall that both $\mfL_{+}$ and $\mfL_{-}$ are just abstract sets we use for indexing.
Our system of equations is then given by 
\begin{equ}\label{eq: gen KPZ}
\d_{t}  
u_{\mft}
=
(\d_x^2 -1)
u_{\mft}
+
u_{\mft}
+
\sum_{\mfj,\mfq \in \mfL_{+}} \Gamma^{\mft}_{\mfj,\mfq}(u)\,(\d_x u_{\mfj} )
(\d_x u_{\mfq})
+ 
\sum_{\mfl \in \mfL_{-}}
\sigma^{\mfl}_{\mft}(u)\xi_{\mfl}\;, \quad \mft \in \mfL_{+}\;,
\end{equ} 
where we write $(t,x)$ rather than $(z_0,z_1)$ for the space-time coordinates.
Here the $(\xi_{\mfl})_{\mfl \in \mfL_{-}}$ are independent space-time white noises and 
the $ (\Gamma^{\mft}_{\mfj,\mfq} )$ are the Christoffel symbols of the underlying manifold. 
For each $\mfl$ one should think of $\sigma^{\mfl}$ as a smooth vector field on $\R^{\mfL_{+}} = \R^{m}$. 
One chooses the collection of smooth vector fields $(\sigma^{\mfl})_{\mfl \in \mfL_{-}}$ so that they generate the metric, that is
$\sum_{\mfl \in \mfL_{-}} (L_{\sigma^\mfl})^2 = \Delta $ where $\Delta$ is the Laplace-Beltrami operator on our manifold and $ L_{\sigma^{\mfl}} $ is the Lie derivative in the direction of $ \sigma^{\mfl}$.
 Note that these vector fields and Christoffel symbols only depend on $(u_{\mft})_{\mft \in \mfL_{+}}$ itself, not on any space-time derivatives of the $u_{\mft}$.

We set $|\mft|_{\s} \eqdef 2$ for every $\mft \in \mfL_+$, which encodes 
the fact that the Green's function of $(\partial_{x}^{2} - 1)$ increases the regularity of the solution by
two degrees of differentiability in the parabolic scaling. 
We also fix some $\kappa \in (0,\frac{1}{6})$ and, for every $\mfl \in \mfL_{-}$, we set $|\mfl|_{\s} = -\frac{3}{2} - \kappa$. This encodes the path-wise parabolic regularity estimate on the driving space-time noises $(\xi_{\mfl})_{\mfl \in \mfL_{-}}$ which guarantees that each $\xi_{\mfl}$ belongs almost surely to $\CC_\s^{-\frac{3}{2}-\kappa}$.

We now turn to defining the partition $\mcb{O} = \mcb{O}_{+} \sqcup \mcb{O}_{-}$. 
The choice of this split is determined by an assumption\footnote{This assumption is encoded in the map $\reg$ described in Section~\ref{subsec: subcrit}} on the regularity of the $(u_{\mft})_{\mft \in \mfL_{+}}$. 
We will later\footnote{See the continuation of this example in Section~\ref{sec: example of genKPZ}} see that the 
 assumption that $ u_{\mft} \in \CC_\s^{\frac{1}{2}-3\kappa}$ for every $\mft \in \mfL_{+}$ is a self-consistent one.
 It is also the case that space-time derivation with multi-index $k \in \N^{d+1}$ maps $\CC^{\alpha}_{\s}$ into $\CC_\s^{\alpha - |k|_{\s}}$. 
Thus the $(u_{\mft})$ are of positive regularity but any space-time derivative of them is of negative regularity. 
Accordingly, we set $\mcb{O}_{+} \eqdef \{ (\mft,0)\}_{\mft\in\mfL_+}$ and  $ \mcb{O}_{-} \eqdef \{ (\mft,p), \, p \neq 0  \}_{\mft \in \mfL_+}$. 

The nonlinearity $F = (F^{\mfl}_{\mft}: \mft \in \mfL_{+}, \mfl \in \mfL_{-} \sqcup \{\mathbf{0}\})$ is given by setting, for $\mft \in \mfL_{+}$, 
\begin{equ}
F^{\mfl}_{\mft}
\eqdef
\begin{cases}
\Y_{(\mft,0)} +
\sum_{\mfj,\mfq \in \mfL_{+}}
\Gamma^{\mft}_{\mfj,\mfq}(\Y)
\Y_{(\mfj,(0,1))}\Y_{(\mfq,(0,1))}& 
\textnormal{if $\mfl = \mathbf{0}$,}\\
\sigma^{\mfl}_{\mft}(\Y) & \textnormal{otherwise.}
\end{cases}
\end{equ}
We note that $\Gamma^{\mft}_{\mfj,\mfq}(\Y)$ and $\sigma^{\mfl}_{\mft}(\Y)$ are only functions of $(\Y_{(\mft,0)}: \mft \in \mfL_{+})$\dash there is no dependence on derivatives as mentioned earlier. 
\subsection{Regularity pairs and subcriticality}\label{subsec: subcrit}
Given any function $f$ from $\mfL$ (resp. $\mfL_+$) to $\R$, we extend $f$ canonically
to $\mfL \times \N^{d+1}$ (resp. $\mfL_+ \times \N^{d+1}$) by setting $f(\mft,p) \eqdef f(\mft) - |p|_{\s}$. 
Fix then a map $\reg: \mfL \sqcup \{\mathbf{0}\} \rightarrow \R$\label{reg page ref} for which the following hold.
\begin{enumerate}
\item $\reg(\mathbf{0}) = 0$.
\item One has $o \in \mcb{O}_{\pm}$ if and only if $(\pm) \reg(o) > 0$.
\item For every $\mfl \in \mfL_{-}$ one has $\reg(\mfl) < |\mfl|_{\s}$. (Recall that $|\mfl|_{\s}<0$ in this case.)
\end{enumerate}
For $\mft \in \mfL_{+}$ one should think of $\reg(\mft)$ as an estimate of the space-time 
regularity of the distribution\slash function associated to $\mft$.
For $\mfl \in \mfL_{-}$ the quantity $\reg(\mfl)$ can be taken arbitrarily close to but strictly smaller than $|\mfl|_{\s}$\dash this doesn't really encode any new information, but such a convention will be convenient later on to gain a little
bit of ``wriggle room''.
\begin{definition}\label{def: obey for blackbox} Suppose we are given a tuple $F = (F^{\mfl}_{\mft})_{\mft,\mfl}$, where $\mft$ ranges over $\mfL_{+}$, $\mfl$ ranges over $\mfL_{-} \sqcup \{\mathbf{0}\}$, and for each $\mft,\mfl$ one has $F^{\mfl}_{\mft} \in \Poly$. We say $F$ obeys $\reg$ if the following condition holds. 

For every $\mft \in \mfL_{+}$ and $\mfl \in \mfL_{-}\sqcup\{\mathbf{0}\}$, if one expands $F^{\mfl}_{\mft}$ as in \eqref{expansion of nonlinearity} then for every exponent $\alpha \in \N^{\mcb{O}}$ appearing in the expansion of $F^{\mfl}_{\mft}$  one has
\begin{equ}\label{eq: subcriticality}
\reg(\mft)
< 
|\mft|_{\s}
+
\reg(\mfl)
+
\sum_{o \in \alpha}
\reg(o)\;.
\end{equ}
We define $\widetilde{\G}$\label{tilde G page ref} to be the set of all tuples $F$ which obey $\reg$.
\end{definition}
Condition~\eqref{eq: subcriticality} enforces that the assumptions on regularity encoded by $\reg$ are self-consistent when checked on an equation with right hand side determined by $F$, namely the system of SPDEs formally given by
\begin{equ}[e:SPDE]
\d_t \phi_\mft = \mathscr{L}_\mft \phi_\mft + F^{\mathbf{0}}_{\mft}(\boldsymbol{\phi}) + \sum_{\mfl \in \mfL_-} F^{\mfl}_{\mft}(\boldsymbol{\phi})\,\xi_\mfl\;.
\end{equ}

\begin{remark}\label{rem:kappa}
As we did in Section~\ref{sec: example of genKPZ part 0}, we will use the parameter $\kappa$ to denote the aforementioned ``wriggle room'' in our power counting and regularity estimates.
It will, for instance, account for the difference between the left and right hand sides of~\eqref{eq: subcriticality}, as well as the difference between $\reg(\mfl)$ and $|\mfl|_\s$.
\end{remark}

Here and for the remainder of this subsection, we use the convention that for 
any collection $\phi = (\phi_{\mft})_{\mft \in \mfL_{+}}$ of smooth functions $\phi_{\mft}\colon \R_+ \times \T^{d} \rightarrow \R$ and $z \in \Lambda$ we write 
\begin{equ}
\boldsymbol{\phi}(z) \eqdef 
\big(\d^p\phi_{\mft}(z):\ (\mft,p) \in \mcb{O} \big)
 \in \R^{\mcb{O}}\;.
\label{eq:bold phi def}
\end{equ}
Condition~\eqref{eq: subcriticality} also guarantees that the SPDE associated to $F$ can be algebraically formulated using a regularity structure built in~\cite{BHZalg}.
If there exists a function $\reg$ such that $F$ obeys $\reg$, then $F$ is said to be \emph{locally subcritical}.

To guarantee the existence of local solutions however, extra assumptions are needed.
These additional assumptions will be formulated in terms of a function $\ireg \colon \mfL_+ \to \R$,\label{ireg page ref}
which one should think of as the regularity of the initial condition for the ``remainder'' 
part of $\phi_\mft$ for our SPDE~\eqref{e:SPDE} in a decomposition reminiscent of the Da Prato--Debussche 
trick (see Section~\ref{subsubsec:loc exist thm}).

For $\mft \in \mfL_+$, $\mfl \in \mfL_{-}\sqcup\{\mathbf{0}\}$, and $F^\mfl_\mft$ as in~\eqref{expansion of nonlinearity} with multisets $\alpha_j^\mfl$
for $j \le m_\mfl$ where $m_\mfl$ is the corresponding value of $m$ in~\eqref{expansion of nonlinearity}, define
\begin{equ}\label{eq:nlt def}
n^\mfl_\mft \eqdef |\mfl|_\s + \min_{1 \leq j \leq m_\mfl} \min_f \sum_{o \in \alpha_j^\mfl} f_o(o)\;,
\end{equ}
where $|0|_\s \eqdef 0$
and where $\min_f$ is taken over all assignments $f \colon o \mapsto f_o \in \{\reg,\ireg\}$ which,
for those $j$ such that $(F^\mfl_\mft)_j$ is identically constant, satisfy $f^{-1}(\ireg) \cap \alpha_j^\mfl 
\neq \emptyset$.
The quantity $n^\mfl_\mft$ should be thought of as an estimate on the blow-up rate of the term $F^{\mfl}_{\mft}(\boldsymbol{\phi})\,\xi_\mfl$ (with $\xi_{\mathbf{0}} \equiv 1$) at the hyperplane $t=0$ (cf. Lemma~\ref{lem:DgammaMult}).
We remark that this rate is determined by the regularity of the initial condition for the ``remainder'' part together with the space-time regularity of the ``stationary'' part (see once more Section~\ref{subsubsec:loc exist thm}); this is the reason for the somewhat complicated definition of $n^\mfl_\mft$.
\begin{remark}\label{remark:F const}
If $F^\mfl_\mft$ itself is identically constant, then $m_\mfl=1$ 
and $\alpha_1^\mfl = 0$,
so that $\min_f$ in~\eqref{eq:nlt def} is taken over the empty set.
Due to our convention $\inf\emptyset \eqdef +\infty$, it follows that $n^\mfl_\mft = +\infty$ in this case.
\end{remark}
Define further
\begin{equ}
n_\mft \eqdef \min_{\mfl \in \mfL_-\sqcup\{\mathbf{0}\}} n^\mfl_\mft\;.
\end{equ}
%
\begin{assumption}\label{assump:Schauder1}
For every $o \in \mcb{O}_{+}$, it holds that $0 \leq \ireg(o) \leq \reg(o)$. Moreover, for every $\mft \in \mfL_+$, it holds that
$n_\mft > -\s_0$ and $n_\mft + |\mft|_\s > \ireg(\mft)$.
\end{assumption}
The first condition of Assumption~\ref{assump:Schauder1} is required to deal with the composition of solutions with smooth functions, while the second condition is required to reconstruct products of singular modelled distributions which appear in the abstract fixed point map associated to our equation.
\begin{remark}\label{remark: warning}
In practice one starts with a specific system of equations, then fixes a scaling $\s$, computes the 
regularisation of the kernels $\{|\mft|_{\s}\}_{\mft \in \mfL_{+}}$ and regularity of the noises 
$\{|\mfl|_{\s}\}_{\mfl \in \mfL_{-}}$, encodes the nonlinearities that appear in terms of a rule, 
and then tries to determine the functions $\reg$ and $\ireg$.
We introduce notions in a different order because, from the outset, we always want to
consider a whole family of equations on which the renormalisation group will then be able
to act.
\end{remark} 
\subsection{Kernels on the torus}\label{sec: kernels on torus}

We make the following standing assumption regarding the linear part of our equation.

\begin{assumption}\label{ass:kernels}
For each $\mft \in \mfL_{+}$, we are given a differential operator $\mathscr{L}_{\mft}$\label{Lmft page ref} involving 
only the spatial derivatives $\{ \d_i \}_{i=1}^{d}$ which satisfies the following properties.
\begin{itemize}
\item $\d_{0} - \mathscr{L}_{\mft}$ admits a Green's function $G_{\mft}: \Lambda \setminus \{0\} \rightarrow \R$\label{Gmft page ref} which is a kernel of order $|\mft|_{\s}$ in the sense of \cite[Ass.~5.1]{Regularity} with respect to $\s$. 

\item For any $\eta \in (-\infty,0) \setminus \N$, $u \in \mcC^{\eta}_{\bar{\s}}(\T^{d})$, $k \in \N^{d+1}$, 
and for some $\chi > 0$, one has the bounds
\begin{equs}\label{eq:initialBound}
\sup_{t \in (0,1]}
t^{-(\eta - |k|_{\s})/\s_{0}}
\sup_{x \in \T^{d}}
\Big|
\int_{\T^{d}} dy\ D^{k}G_{\mft}(t,x-y) u(y)
\Big|
&< 
\infty\;,\quad\\
\label{eq:exponentialBound}
\sup_{|t| > 1}
\sup_{x \in \T^{d}}
e^{\chi|t|}
\cdot
\big|
(D^{k}G_{\mft})(t,x)
\big|
&< 
\infty\;.
\end{equs}
\end{itemize}
\end{assumption}
\begin{example}
If $\mathscr{L}_{\mft} = Q(\nabla_x)-1$ for a homogeneous polynomial $Q$ of even degree $2q$ with respect to a scaling $\bar\s$, then~\cite[Lem.~7.4]{Regularity} implies that we can take $|\mft|_\s=2q$ for the scaling $\s = (2q,\bar\s_1,\ldots,\bar\s_d)$.
In particular, the heat operator with unit mass falls into this framework: here $d \geq 1$, $\s = (2,1,\dots,1)$,  
$\mathscr{L}_{\mft} = \Delta - 1$, where $\Delta \eqdef \sum_{j=1}^{d} \d_j^{2}$, and $|\mft|_{\s} = 2$.
However, while this is the most common example, other non-trivial choices are possible: if one sets $d = 2$ and $\s = (4,2,1)$ then one can take $\mathscr{L}_{\mft} \eqdef \d_1^{2} - \d_2^4 - 1$, and $|\mft|_{\s} = 4$.
\end{example}
\begin{remark}
One can sharpen the second condition by assuming that for each $\mft \in \mfL_+$, there exists $\kappa_\mft > 0$ such that~\eqref{eq:initialBound} holds with $t^{-(\eta-|k|_\s)\kappa_\mft}$ in place of $t^{-(\eta-|k|_\s)/\s_{0}}$.
This could allow, in certain cases, for a lower regularity of initial data and\slash or driving terms.\footnote{In particular, one could sharpen Lemma~\ref{lem:initialCond} below.}
However, since $\kappa_\mft = 1/\s_{0}$ is optimal in most cases of interest, and since the current assumptions are already quite involved, we refrain from making this generalisation.
\end{remark}
\subsection{The local well-posedness theorem}\label{black box section}
We fix some quantities and objects just for the remainder of this subsection so we can state the aforementioned result. 

We introduce a family of rooted decorated combinatorial trees $\mathring{\mathscr{T}}$.\label{basic trees page ref} 
An element $\tau \in \mathring{\mathscr{T}}$ consists of an underlying combinatorial rooted tree $T$ with node set $N_{T}$, edge set $E_{T}$, an edge decoration $\mff: E_{T} \rightarrow \mcb{O}$, and a node decoration $\mfm = (\mnoise, \mpoly): N_{T} \rightarrow (\mfL_{-} \sqcup \{\mathbf{0}\}) \times \N^{d+1}$.
We also write $\rho_{T} \in N_{T}$ for the root node
and we sometimes write $\tau = T^{\mfm}_{\mff}$.
Observe that for every $\tau \in \mathring{\mathscr{T}}$ there exist unique $n \geq 0$, $o_1, \ldots, o_n \in \mcb{O}$, and $\tau_1,\dots,\tau_n \in \mathring{\mathscr{T}}$ such that
$\tau$ is obtained by attaching each 
$\tau_j$ to the root of $\tau$ (which has some decoration $(\mfl,k)$) using edges with decoration $o_j$.
In this case we adopt the symbolic notation
\begin{equ}\label{eq:generic T}
\tau = \mathbf{X}^{k} \Xi_{\mfl}\Big( \prod_{j=1}^n \mcb{I}_{o_j}[\tau_j]\Big)\;,\quad
k \eqdef \mpoly(\rho_T)\;,\quad \mfl \eqdef \mnoise(\rho_T)\;.
\end{equ}
In particular, when $n=0$, so that $\tau$ consists of only the root node, 
we write $\tau = \mathbf{X}^{k} \Xi_{\mfl}$.
Also, if $\mnoise(\rho_T)=0$ or $\mpoly(\rho_T) = 0$, then we omit the corresponding symbol $\Xi_{\mathbf{0}}$ or $\mathbf{X}^0$ respectively. 
Every tree of the form $\tau = \mcb{I}_o[\bar \tau]$ for some $o$ and some $\bar \tau$, we call \emph{planted}.

For every $T_{\mff}^{\mfm} \in \mathring{\mathscr{T}}$, we set 
\[
|T_{\mff}^{\mfm}|_{\s}
\eqdef 
\sum_{e \in E_{T}}
|\mff(e)|_{\s}
+
\sum_{u \in N_{T}}
\big(
|\mnoise(u)|_{\s} 
+
|\mpoly(u)|_{\s}
\big)\;,
\] 
where $|\mathbf{0}|_{\s} \eqdef 0$.

For any $\tau$ and $F \in \widetilde{\G}$ we define 
$\Upsilon^F[\tau] \eqdef  (\Upsilon^F_{\mft}[\tau])_{\mft \in \mfL_{+}}\in \Poly^{\mfL_{+}}$ inductively as follows. 
For $\tau$ given by~\eqref{eq:generic T} for some $n \geq 0$, we define for every $\mft \in \mfL_{+}$\label{upsilon page ref}
\begin{equ}\label{eq: first Upsilon}
\Upsilon^{F}_{\mft}[\tau]
\eqdef
\Big(
\prod_{j=1}^{n}
\Upsilon^{F}_{\mft_{j}}\big[
\tau_j
\big]
\Big)\cdot
\Big(
\partial^{k}
\prod_{j=1}^{n}D_{o_j}
\Big)
F_{\mft}^{\mfl}\;,
\end{equ}
where $o_{j} = (\mft_{j},p_{j})$.

\begin{remark}\label{rem:B-series}
Our definition of $ \Upsilon^{F}[\cdot] $ as described in \eqref{eq: first Upsilon} also appears, in a simpler form, 
in the theory of B-series for ODEs \cite{MR0305608,hairer74}. 
Given an ODE
\begin{equs}\label{eq: ode}
\partial_t y = F(y), \quad y(0) = y_0 \in \R^{d}\;,
\end{equs}
a B-series associated to it is a discrete time-stepping method $ y_k \mapsto y_{k+1}  $ given by an expansion over rooted trees:
\begin{equs}\label{eq: B-series formula}
y_{k+1} - y_k = \sum_{\tau \in \mathbf{T}}  h^{|\tau|} \alpha(\tau) \frac{\Upsilon^{F}[\tau](y_k)}{S(\tau)},
\end{equs}
where $\mathbf{T}$ is the set of rooted combinatorial trees. 
For an ODE like~\eqref{eq: ode} these trees don't need any decorations:
there is only one type of `noise', which represents the constant $1$ (which we called $\Xi_{\mathbf{0}}$ earlier) and there are no abstract 
polynomials, so no node decorations are required. 
Furthermore, even though $y$ may have several components, there is only one type of edge (and thus no need for edge decorations) because each component of $y$ plays the same role
(as there is only one integration operator) and we don't see the appearance of spatial derivatives in such an ODE.

In the right hand side of~\eqref{eq: B-series formula}, $h > 0$ is the step size, $ |\tau| $ denotes the number of nodes of $\tau$, $ S(\tau) $ denotes its symmetry factor,\footnote{See~\eqref{def: symmetry factor in section 2} and Remark~\ref{rem:B-series 2}} and $ \alpha : \mathbf{T} \rightarrow \mathbf{R}$ is a function which determines\footnote{
One choice of $\alpha$ is setting $\alpha(\tau) = 1/\gamma(\tau)$ where $\gamma(\tau)$ is the ``density'' given by $\gamma(\bullet) = 1$ and  $\gamma(\tau) = |\tau| \prod_{j=1}^{n}\gamma(\tau_{j})$. Here we are writing $\tau$ recursively as in \eqref{eq:generic T}. 
With this choice, the formula given in~\eqref{eq: B-series formula} coincides with the Taylor expansion  of the exact solution to the initial value problem.} the time stepping method.
Similarly to above, $ \Upsilon^{F}[\bullet] = F$ and, for $ \tau =  \prod_{j=1}^n \mcb{I}[\tau_j] $ with $\mcb{I}$ the operation of grafting the tree onto a new root, it is defined  by
\begin{equs} 
\Upsilon^{F}[\tau](y)
\eqdef
F^{(n)}(y)
\prod_{j=1}^{n} \Upsilon^{F}\big[
\tau_j
\big](y)\;.
\end{equs}
For example, we have
\begin{equs}
\Upsilon^{F}[\bullet](y) & = F(y), \quad \Upsilon^{F}[\<IXi3>](y) = F^{(3)}(y)(F(y))^{3}, \\ 
\Upsilon^{F} [\<IXi4> ](y) & =  F^{(2)}(y)F(y)F^{(1)}(y)F(y). 
\end{equs}
In~\cite{Gubinelli2010693}, similar tree expansions have also been used for solving rough differential equations of the form
\begin{equation}
\label{eq:CODE}
	dY_t = \sum_{\mfl=0}^m F_{\mfl}(Y_t)dX^{\mfl}_t,
\end{equation}
where the $F_\mft$ are vector fields on $\R^{d}$ and $X:[0,T] \rightarrow \R^{m+1}$ is a driving signal 
(we can incorporate a non-noisy term into this framework by using the convention that $X_{0}(t) = t$). 
The solution to~\eqref{eq:CODE} is again given by a tree expansion
\begin{equation}
\label{eq:CODEsolution}
	Y_t =  Y_s + \sum_{\tau \in \mathbf{T}} \frac{\Upsilon^{F}[\tau](Y_s)}{S(\tau)} X^{\tau}_{st}\;,
\end{equation}
at least in the sense of asymptotic series.
Here $\mathbf{T}$ is the collection of rooted trees with node decorations in $\{0,\ldots,m\}$ 
(corresponding to the different components of $X$) and, as before, only one type of edge.  
The $X^{\tau}_{st}$ are multiple integrals of the driving signal. 
In this context one inductively sets, for a tree $ \tau = \Xi_{\mfl} \prod_{j=1}^n \mcb{I}[\tau_j] $, 
\begin{equs}
\Upsilon^{F}[\tau](y)
\eqdef
F^{(n)}_{\mfl}(y)
\prod_{j=1}^{n}
\Upsilon^{F}\big[
\tau_j
\big](y)\;.
\end{equs}
As an example we compute:
\begin{equs}
\Upsilon^{F}[\bullet_{\mfl}](y) & = F_{\mfl}(y), 
\quad 
\Upsilon^{F}\big[\<IXi3b>\big](y) = F_{\mfl}^{(3)}(y)F_{\mfi}(y)F_{\mfj}(y)F_{\mfq}(y), \\ 
\Upsilon^{F} \big[\<IXi4b> \big](y) & =  F^{(2)}_{\mfl}(y)F_{\mfq}(y)F^{(1)}_{\mfj}(y)F_{\mfi}(y).
\end{equs}
The only difference between~\eqref{eq: ode} and~\eqref{eq:CODE} is that the latter had multiple drivers so we decorated the vertices with the labels of these drivers.  

In the setting of this article, we allow for several additional decorations on our trees: (i) we have edge decorations which keep track of the 
different components of our systems of equations (different components can have different integration operators associated to them), as well
as their derivatives, (ii) in addition to storing data about drivers in our node decorations as in~\eqref{eq:CODEsolution}, we also 
store abstract classical monomials $\mathbf{X}^{k}$ which is required because our Green's functions have non-vanishing derivatives,
unlike the Heaviside function, which is the Green's function of the one-dimensional derivative. 
Moreover, our evaluation operator $\Upsilon^{F}_{\mft}$ is not in general 
just a function of the solution's value at a point but can also depend on its derivatives.
\end{remark}

We next introduce a notion which greatly restricts the types of trees we need to consider.

\begin{definition}\label{def: mft non-vanishing}
Given $\mft \in \mfL_{+}$, $F \in \widetilde{\G}$, and $\tau \in \mathring{\mathscr{T}}$ of the form~\eqref{eq:generic T}, we say that $\tau$ is $\mft$-non-vanishing for $F$ if
$\big(
\partial^{k}
\prod_{j=1}^{n}D_{o_j}
\big)
F_{\mft}^{\mfl}
\not = 0
$, and $\tau_j$ is $\mft_j$-non-vanishing for $F$ for all $j \in [n]$.
Let $\mathring{\mathscr{T}}_{\mft}[F]$\label{basic trees non-vanish page ref} denote the set of all $\tau \in \mathring{\mathscr{T}}$ 
that are $\mft$-non-vanishing for $F$. 
We also write $\mathring{\mathscr{T}}_{\mft,-}[F] \subset \mathring{\mathscr{T}}_{\mft}[F]$ \label{basic neg trees page ref}
for those elements $\tau$ for which $|\tau|_\s < 0$.
\end{definition}
We note that, for every $\gamma \in \R$, and $\mft \in \mfL_+$, the set $\{\tau \in \mathring{\mathscr{T}}_{\mft}[F] : |\tau|_\s < \gamma\}$ is finite due to the local subcriticality of $F$ (see the proof of Theorem~\ref{thm: THE black box}). 
In particular, $\mathring{\mathscr{T}}_{\mft,-}[F]$ is a finite set.
Note also that if $\tau$ is not $\mft$-non-vanishing, then $\Upsilon^F_\mft[\tau] =0$ but the converse implication is not true in general.\footnote{This is because one can find $f_{1}, f_{2} \in \mcb{C}_{\mcb{O}} \setminus \{0\}$ with $f_{1}f_{2} =0$ because their supports are disjoint. However, the converse does hold for polynomial nonlinearities.}

With these notations, we introduce several structural assumptions on a nonlinearity $F \in \widetilde{\G}$ and the sets $\mathring{\mathscr{T}}_{\mft}[F]$.

\begin{assumption}\label{assump:planted trees}
For every $\mft \in \mfL_+$ and $ T^{\mfm}_{\mff} \in \mathring{\mathscr{T}}_{\mft}[F]$ and any strict subtree\footnote{\label{foot:subtree}By a subtree, we mean that $\bar T^{\bar\mfm}_{\bar\mff}$ is a tree whose node and edge sets are subsets of those of $T^{\mfm}_{\mff}$ and whose decorations satisfy $\bar\mff(e) = \mff(e)$ for all $e \in E_{\bar T}$, and $\barmnoise(x) = \mnoise(x)$ and $\barmpoly(x) \leq \mpoly(x)$ for all $x \in N_{\bar T}$.
By a \emph{strict} subtree, we mean $\bar T^{\bar\mfm}_{\bar\mff} \neq T^{\mfm}_{\mff}$.
Note that if $T^{\mfm}_{\mff}$ is $\mft$-non-vanishing, then every subtree $ \bar T^{\bar\mfm}_{\bar\mff}$ with $ \rho_T = \rho_{\bar T}$ is also $\mft$-non-vanishing.
}
$ \bar T^{\bar\mfm}_{\bar\mff}$ of $ T $ with $ \rho_T = \rho_{\bar T}$, one has $|\bar T^{\bar\mfm}_{\bar\mff}|_\s > -
(|\mft|_\s \wedge \s_0)$.
\end{assumption}
\begin{remark}
Assumption~\ref{assump:planted trees} is used in Section~\ref{subsec: gen DPD} to ensure that certain subspaces of a regularity structure form sectors.
For readers familiar with the notion of rules in~\cite{BHZalg}, it may appear surprising that we do not simply restrict to trees which conform to a given rule.
While it holds that if $\tau$ is $\mft$-non-vanishing, then $\tau$ must conform to a rule naturally associated with $F$ (Proposition~\ref{prop:obeyEquivalence}), the converse does not hold in general.
For example, using notation from Section~\ref{subsec:Phi44-delta} for the $\Phi^4_{4-\delta}$ model, taking $i\in\{1,\ldots, 4\}$ and $\tau \eqdef \mathbf{X}^{e_i} \mcb{I}_{(\mft,0)}[\Xi_\mfl]^3$, it holds that $\tau$ obeys the rule naturally associated with the equation but is not $\mft$-non-vanishing.
In fact, removing $\mathbf{X}^{e_i}$ from the root of $\tau$ gives a subtree $\bar\tau$ for which $|\bar\tau|_\s < -(|\mft|\wedge\s_0)=-2$ whenever $\delta \leq \frac23$, which would violate Assumption~\ref{assump:planted trees} if $\tau$ were included in $\mathring{\mathscr{T}}_{\mft}[F]$.

On the other hand, it may appear simpler to keep only those $\tau$ for which $\Upsilon^F_\mft[\tau] \neq 0$.
This definition, however, turns out to not be stable under the structure group, and does not yield sectors of the regularity structures in which we can solve for fixed points.
\end{remark}
For the explanation behind the following two assumptions, see Remark~\ref{remark: conditions on trees for blackbox}.
\begin{assumption}\label{assump:remove_noise_positive}
For every $\mft \in \mfL_+$ and $T^{\mfm}_{\mff} \in \mathring{\mathscr{T}}_\mft[F]$ with $E_{T} \not = \emptyset$, one has
\begin{equ}
|T^{\mfm}_{\mff}|_{\s}
-
\max_{
\substack{
x \in N_{T}\\
\mnoise(x) \not = 0}
}
|\mnoise(x)|_{\s}
>
0\;.
\end{equ}
\end{assumption}

\begin{assumption}\label{assump:no_diverging_variances}
For every $\mft \in \mfL_+$ and $T^{\mfm}_{\mff} \in \mathring{\mathscr{T}}_\mft[F]$ with $E_{T} \not = \emptyset$, one has
\begin{equ}
|T^{\mfm}_{\mff}|_{\s} > -\frac{|\s|}{2} \quad \text{ and } \quad
|T^{\mfm}_{\mff}|_{\s} + |\s| + \min_{\mfl \in \mfL_{-}} |\mfl|_{\s} > 0\;.
\end{equ}
\end{assumption}
\subsubsection{The random driving terms}
Given $\eps > 0$, a smooth function $\rho: \Lambda \rightarrow \R$, and $\psi \in \mcb{S}'(\Lambda)$, we write
\begin{equ}\label{mollification}
\psi^{(\rho,\eps)} \eqdef \psi \ast (S^{\eps}_{\s}\rho)\;.
\end{equ}
Next we describe the class of driving noises included in our main theorem.
\begin{definition} \label{definition noises}
We define $\mathrm{Gauss}$ to be the collection of all tuples $\xi = (\xi_{\mfl})_{\mfl \in \mfL_{-}}$ of jointly Gaussian,
stationary, centred, random elements of $\mcb{S}'(\Lambda)$ which satisfy the following regularity properties for every $\mfl, \mfl' \in \mfL_{-}$ .
\begin{enumerate}
\item \label{point:noise1} There exist distributions $C_{\mfl,\mfl'} \in \mcb{S}'(\Lambda)$ whose singular support is contained in $\{0\}$ and with the property that for every $f,g \in \mcb{S}(\Lambda)$, 
\[
\E[ \xi_{\mfl}(f) \xi_{\mfl'}(g)]
=
C_{\mfl,\mfl'} \Big( \int_{\R \times \T^{d}} dz f(z - \cdot)g(z) \Big)\;.
\]
\item \label{point:noise2} Writing $z \mapsto C_{\mfl,\mfl'}(z)$ for the smooth function which determines $C_{\mfl,\mfl'}$ away from $0$, one has, for any $g \in \mcb{S}(\Lambda)$ satisfying $D^{k}g(0) = 0$ for all $k \in \N^{d+1}$ with $|k|_{\s} < -|\s| - |\mfl|_{\s} - |\mfl'|_{\s}$, 
\[
C_{\mfl,\mfl'}[g]
=
\int dz\, C_{\mfl,\mfl'}(z) g(z)\;.
\]
\item \label{point:noise3} There exists $\kappa > 0$ such that for any $k \in \N^{d+1}$
\[
\sup_{
0 < |z|_{\s} \le 1
}
|D^{k} C_{\mfl,\mfl'}(z)| \cdot |z|_{\s}^{- |\mfl|_{\s} - |\mfl'|_{\s} + |k|_{\s} - \kappa} < \infty\;. 
\]
\end{enumerate}
\end{definition}
It follows from Kolmogorov's continuity theorem combined with items~\ref{point:noise2} and~\ref{point:noise3} above
that every $\xi \in \Gauss$ admits a version which is a random element of 
$\CC^{\noise} \eqdef \bigoplus_{\mfl \in \mfL_{-}} \CC_{\s}^{|\mfl|_{\s}}(\Lambda)$.
\label{Cnoise page ref}
\subsubsection{The local existence theorem}
\label{subsubsec:loc exist thm}
We fix $F \in \widetilde{\G}$ for the remainder of this subsection.
For every $\xi \in \Gauss$, our main result yields a (local in time) solution theory for the initial value problem
\begin{equ}\label{ivp}
\forall \mft \in \mfL_{+}, \; \;
\d_{t} \phi_{\mft} 
=
\mathscr{L}_{\mft}\phi_{\mft}
+
F_{\mft}^{\mathbf{0}}(\boldsymbol{\phi})
+
\sum_{\mfl \in \mfL_{-}}
F_{\mft}^{\mfl}(\boldsymbol{\phi})\xi_{\mfl}
\;,
\end{equ}
with a suitably chosen initial condition $\phi_{\mft}(0,\cdot)$.

One annoying technical problem is that, in general, our solutions may not accommodate evaluation at fixed times.
(This is actually already the case for the $\Phi^4_3$ model and was taken care of in an ad hoc manner in
\cite[Sec.~9]{Regularity}.) 
To circumvent this, we introduce a decomposition of our solution into a sum of explicit space-time distributions coming 
from perturbation theory, along with a remainder which is actually a function.
We first give an intuitive but slightly imprecise statement of the result before introducing the necessary spaces 
required for its precise formulation in Theorem~\ref{thm: THE black box}.

\begin{theorem}[Metatheorem]
Suppose all the assumptions of this section are satisfied.
Fix a collection of Gaussian fields $\xi \in \Gauss$, a mollifier $\rho \in \mcC^\infty(\R^{d+1})$, and a collection of sufficiently nice ``initial conditions'' $(\psi_\mft)_{\mft \in \mfL_+}$.
Then, for every $\eps > 0$, there exist two families of smooth functions $(\CS^-_{\rho,\eps,\mft})_{\mft \in \mfL_+}$ and $(\CS^+_{\rho,\eps,\mft})_{\mft \in \mfL_+}$ such that
\begin{itemize}
\item for every $\mft\in\mfL_+$, $\CS^-_{\rho,\eps,\mft}$ depends on $\xi,\rho$, and $\eps$, and converges to a space-time distribution on $\Lambda$ as $\eps\downarrow0$;
\item for every $\mft\in\mfL_+$, $\CS^+_{\rho,\eps,\mft}$ depends on $\xi,\rho,\eps,$ and $\psi$, and converges in a space of functions up to a (random) blow-up time as $\eps\downarrow0$;
\item $(\phi_{\mft,\eps})_{\mft\in\mfL_+} \eqdef (\CS^-_{\rho,\eps,\mft} + \CS^+_{\rho,\eps,\mft})_{\mft \in \mfL_+}$ solves a renormalised PDE given by~\eqref{eq: renormalised equation} with initial conditions $(\psi_\mft + \CS^-_{\rho,\eps,\mft}(0,\cdot))_{\mft \in \mfL_+}$.
\end{itemize}
\end{theorem}

We now introduce formally where our solutions will live.
	The explicit stationary part will live in the space 
\[
\CC^{\reg,-} \eqdef \bigoplus_{\mft \in \mfL_{+}} 
\begin{cases}
\CC^{\reg(\mft)}_{\s}(\Lambda)& 
\textnormal{ if } \reg(\mft) < 0\;,\\
\{0\} & \textnormal{ otherwise.}
\end{cases}
\]\label{Creg- page ref}
In order to describe the remainder, we first set, for any $\mft \in \mfL_{+}$,
\[
\widetilde{\reg}(\mft)
\eqdef
|\mft|_{\s}
+
\inf
\big\{ 
|\tau|_{\s}:\ 
\tau \textnormal{ is $\mft$-non-vanishing and }
|\tau|_{\s} > -(|\mft|_{\s}\wedge \s_0)
\big\}\;.
\]
We then define, for any $T \in (0,\infty]$, $\CC^{\reg,+}_{T} \eqdef \bigoplus_{\mft \in \mfL_{+}} \CC^{\widetilde{\reg}(\mft)}_{\s}((0,T) \times \T^{d})$. 

We also define the spaces\label{Cireg page ref}
\begin{equ}
\CC^{\ireg} \eqdef \bigoplus_{\mft \in \mfL_{+}} \CC_{\bar{\s}}^{\ireg(\mft)\wedge\widetilde{\reg}(\mft)}(\T^d)
\end{equ}
and $\widehat{\CC^{\ireg}}
\eqdef
\CC^{\ireg} \sqcup \{\infty\}$. 
More precisely, we view $\CC^{\ireg}$ as a Banach space with norm $\|\cdot\|_{\CC^{\ireg}}$ and define the topological space $\widehat{\CC^{\ireg}}$ by including a point at infinity $\infty$ and determining the topology by starting with the basis of open balls in $\CC^{\ireg}$ and adding sets of the form $\{ g \in \CC^{\ireg}: \|g\|_{\CC^{\ireg}} \ge N \} \sqcup \{\infty\}$ for any $N > 0$. 
We adopt the notational convention that $\|\infty\|_{\CC^{\ireg}} = +\infty$. 

Given $f \in \CC(\R_{+},\widehat{\CC^{\ireg}})$, we define $T[f] = \inf\{ t \ge 0: f(t) = \infty\}$. 
The space where our ``maximal remainders'' will live is then given by 
\[
\CC^{\rem}
\eqdef
\Big\{
f \in \CC(\R_{+},\widehat{\CC^{\ireg}}):\ 
\begin{array}{c}
\forall t > T[f],\ f(t) = \infty\\
f\restr_{(0,T[f])} \in \mcC^{\reg,+}_{T[f]}
\end{array}
\Big\}\;.\label{Creg page ref}
\]
To state our convergence results we now want to put a topology on $\CC^{\rem}$ which requires some care (see Remark~\ref{rem:motivating_newTheta} below). 

For any $f \in \CC^{\rem}$, we define its running supremum by 
\[
S_f(t) \eqdef \sup_{s \le t}  \|f(s)\|_{\CC^{\ireg}} \in [0,\infty] \;.
\]
Given a parameter $L>0$ and a fixed smooth mollifier 
$\phi\colon \R \to [0,\infty)$ integrating to $1$ and with closed support equal to $[0,1]$, we define a smoothened version of $S_f$ by setting
\begin{equ}
S_f^L(t) = \tan \int_0^1 \tan^{-1}(S_f(t+s/L)) \phi(s)\,ds\;. 
\end{equ}
Note that $S_f^L$ is increasing, $S_f^L(t) \ge S_f(t)$, and that
$\inf\{t\geq 0\,:\, S_f^L(t) = \infty\} = \inf\{t\geq 0\,:\, S_f(t) = \infty\}$.
We then further fix a smooth decreasing function $\psi \colon \R \to [0,1]$ with
derivative supported in $[1,2]$ and $\psi(1) = 1$ and $\psi(2) = 0$. 

Then, given any $f \in \CC^{\rem}$ and $L \in \N$ we define $\Theta_L(f) \in \mcC^{\reg,+}_{\infty}$ by setting 
\begin{equ}
\Theta_L(f)(t) = \psi(S_f^L(t)/L) f(t)\;.
\end{equ}
Note that $\psi(S_f^L(t)/L)$ is smooth in $t$ since  $S_{f}^{L}(\cdot)$ is smooth on $\{t \ge 0: S_{f}^{L}(t) < 3L\}$ and so we indeed have $\Theta_L(f) \in \mcC^{\reg,+}_{\infty}$.
Furthermore, $\sup_{t\geq 0}\|\Theta_L(f)(t)\|_{\CC^{\ireg}} \leq 2L$.


We equip $\CC^{\rem}$ with a metric $d(\cdot,\cdot) \eqdef \sum_{L=1}^{\infty} 2^{-L} d_{L}(\cdot,\cdot)$, where for $f,g \in \mcC^{\rem}$
\[
d_L(f,g) \eqdef 1\wedge \Big[\sup_{t \in [0,L]} \|\Theta_L(f)(t) - \Theta_L(g)(t)\|_{\mcC^{\ireg}} + \|\Theta_L(f) - \Theta_L(g)\|_{\mcC^{\reg,+}_{L}}\Big]\;.
\]
The following lemma tells us that the metric $d(\cdot,\cdot)$ has continuity properties that are compatible with the stability properties of maximal solutions to semi-linear PDE.  
We define, for any $L \in \N$ and $f \in \CC^{\rem}$, $T^L(f) = \inf\{t \geq 0 \,:\, \|f(t)\|_{\CC^{\ireg}} \ge L\}$. 
\begin{lemma}\label{lemma:metric_stability}
Let $f,f_{1},f_{2},\dots \in \CC^{\rem}$ satisfy the following condition: for every $L \in \N$ one has  
\begin{equ}\label{eq:convergence_standard}
\lim_{n\rightarrow \infty} 
\sup_{ t \in [0,T(L,n)]}
\| f_{n}(t) - f(t) \|_{\mcC^{\ireg}}
+
\|f_{n} - f\|_{\mcC^{\reg,+}_{T(L,n)}}
= 
0\;,
\end{equ}
where $T(L,n) = L \wedge T^{L}(f) \wedge T^{L}(f_{n})$. 
Then it follows that $d(f_{n},f) \rightarrow 0$ as $n \rightarrow \infty$. 
\end{lemma}
\begin{proof}
The key point is that, for any fixed $L \in \N$, one has the uniform convergence of the functions $ \psi(S^{L}_{f_n}(\cdot)/L)$ to $ \psi(S^{L}_{f}(\cdot)/L)$, along with their derivatives, on the interval $[0,L]$ as $n \rightarrow \infty$.

Here one uses the observation that $\tan^{-1}(S_{f_n}(\cdot))$ converges point-wise to $\tan^{-1}(S_{f}(\cdot))$ as $n \rightarrow \infty$ \dash note that it is good enough that $\tan$ is continuous on $[0,\pi/2)$ since, for any $g$, $\psi(S^{L}_{g}(s)/L)$ vanishes for any $s \ge \inf \{ t \ge 0 : S^{L}_{g}(t) \ge 2L\}$ which means we never directly encounter a blow-up, even if $T[g] < L$.  
\end{proof}

\begin{remark}\label{rem:motivating_newTheta}
Our particular method of ``stopping $f$  at threshold $L$'' to define $\Theta_{L}(f)$ was designed to guarantee (i) that  $\Theta_{L}(f)$ has time regularity at least as good as that of $f$, and  (ii) $d(\cdot,\cdot)$ should have good continuity properties as described in Lemma~\ref{lemma:metric_stability}. 

Constraint (i) prevents us from using a sharp time cut-off while constraint (ii) prevents us from using $T^{L}(f)$ from determining how to localize $f$ in time since $T^{L}(\cdot)$ does not have good enough continuity properties. 
Constraint (ii) is crucial for us to have continuity in the initial condition and the driving model, compare Lemma~\ref{lemma:metric_stability} with \cite[Cor.~7.12]{Regularity}. 

Note that the published version of this paper contained an error since there we performed a smooth time localisation using $T^{L}(\cdot)$ which means that constraint (i) was satisfied but we violated constraint (ii). 
\end{remark}
We now set 
\[
\CC^{\clas}
\eqdef
\Big\{
f \in \CC(\R_{+},\widehat{\CC^{\ireg}}):\ 
\begin{array}{c}
\forall t > T[f],\ f(t) = \infty\\
f\restr_{(0,T[f])} \in \mcC^{\infty}((0,T[f]) \times \T^{d})
\end{array}
\Big\}\;.
\]
Finally, consider a smooth function $\rho:\R^{d+1} \rightarrow \R$ supported on the ball $|z|_{\s} \le 1$ with $\int \rho = 1$, as well as a family of constants
\begin{equs}\label{eq:familyOfConst}
\{c_{\rho,\eps}^\tau \in \R:\  \tau \in \mathring{\mathscr{T}}_{\mft,-}[F]\ \textnormal{ for some } \mft \in \mfL_+,\ \eps > 0 \}\;.
\end{equs}
We then denote by 
\begin{equ}
\CS_{\rho,\eps}\colon\CC^{\noise} \times \CC^{\ireg} \rightarrow \CC^{\clas}\;,\qquad
\CS_{\rho,\eps}\colon (\xi,\psi) \mapsto \phi_{\eps} = (\phi_{\mft,\eps})_{\mft \in \mfL_{+}}
\end{equ}
the classical solution map of the following system of initial value problems for $\phi_{\eps} = (\phi_{\mft,\eps}: \mft \in \mfL_{+})$:
\begin{equ}\label{eq: renormalised equation}
\d_{t} \phi_{\mft,\eps} 
= 
\mathscr{L}_{\mft}
\phi_{\mft,\eps}
+
F_{\mft}^{\mathbf{0}}(\boldsymbol{\phi_{\eps}})
+
\sum_{\mfl \in \mfL_{-}}
F_{\mft}^{\mfl}(\boldsymbol{\phi_{\eps}})\xi^{(\rho,\eps)}_{\mfl}
+\sum_{\tau \in \mathring{\mathscr{T}}_{\mft,-}[F]}
c_{\rho,\eps}^\tau
\frac{\Upsilon_{\mft}^{F}[\tau](\boldsymbol{\phi_{\eps}})}{S(\tau)}
\;,
\end{equ}
with initial data $\phi_{\mft,\eps}(0,\cdot)
\eqdef \psi_{\mft}(\cdot)$, where the mollified noises $\xi^{(\rho,\eps)}_{\mfl}$ are defined as in \eqref{mollification}.
The combinatorial symmetry factor $S(\tau)$ appearing in this identity is defined as follows.
For any tree $\tau$ written as
\begin{equ}\label{eq: display of tree}
\tau = \mathbf{X}^{k} \Xi_{\mfl}\Big( \prod_{j=1}^m \mcb{I}_{o_j}[\tau_j]^{\beta_j}\Big)\;,
\end{equ}
where we group terms (uniquely)
in such a way that $(o_i,\tau_i) \neq (o_j,\tau_j)$ for $i \neq j$, we inductively set
\begin{equ}\label{def: symmetry factor in section 2}
S(\tau)
\eqdef
k!
\Big(
\prod_{j=1}^{m}
S(\tau_{j})^{\beta_{j}}
\beta_{j}!
\Big)\;.
\end{equ}

\begin{remark}\label{rem:B-series 2} 
Continuing the thread of Remark~\ref{rem:B-series}, the formulas~\eqref{eq: display of tree} and~\eqref{def: symmetry factor in section 2} can be adapted to the cases of~\eqref{eq: B-series formula} and~\eqref{eq:CODEsolution} as follows. 
For~\eqref{eq: display of tree}, in both scenarios one only has a single type of edge so the $o_{j}$ are always all the same. 
Furthermore, in the case of~\eqref{eq: ode} one forces $k=0$ and $\mfl = \mathbf{0}$, while for~\eqref{eq:CODE} one always sets $k=0$ and forces $\mfl \in \{\mathbf{0}\} \cup \{1,\dots,m\}$.
In this way,  both~\eqref{eq: B-series formula} and~\eqref{eq:CODEsolution} are special cases 
of~\eqref{def: symmetry factor in section 2}. 
\end{remark}

As discussed earlier, to formulate the main result of this section, we decompose the map $\CS_{\rho,\eps}$ into a ``stationary'' part
\begin{equ}
\CS_{\rho,\eps}^-\colon\CC^{\noise}\rightarrow \CC^{\reg,-} \cap\CC^\infty\;,
\end{equ}
independent of the initial condition, as well as a ``remainder'' part
\begin{equ}
\mathcal{S}_{\rho,\eps}^{+}:\CC^{\noise} \times \CC^{\ireg} \rightarrow \CC^{\clas}\;,
\end{equ}
which does depend on the initial condition.
These two maps will be chosen in such a way that one has the identity
\begin{equ}[e:idenSol]
\CS_{\rho,\eps}\Big(\xi,\psi+\mathcal{S}_{\rho,\eps}^{-}(\xi)(0,\cdot)\Big) 
= 
\mathcal{S}_{\rho,\eps}^{+}(\xi,\psi)
+ \mathcal{S}_{\rho,\eps}^{-}(\xi)\;.
\end{equ}
Here $\mathcal{S}_{\rho,\eps}^{-}(\xi)(0,\cdot)$ is a function of space obtained by restricting $\mathcal{S}_{\rho,\eps}^{-}(\xi)$ to the time $0$ hyperplane.
We also remark that addition between an element of $\CC^{\clas}$ and an element of $\CC^\infty$ naturally yields 
again an element of $\CC^{\clas}$. 

The precise definitions of $\mathcal{S}_{\rho,\eps}^{\pm}$ will be given in \eqref{e:defSpm} 
below, based on the construction of Section~\ref{subsec: gen DPD} below,
but do not matter much at this stage. Suffices to say that this decomposition should be thought of as 
a higher order version of the classical Da Prato--Debussche trick \cite{MR1941997,DPD2}. 
With all of these preliminaries in place, our general convergence result 
can be formulated as follows.

\begin{theorem}\label{thm: THE black box}
Suppose that Assumptions~\ref{assump:Schauder1},~\ref{ass:kernels},~\ref{assump:planted trees},~\ref{assump:remove_noise_positive}, and~\ref{assump:no_diverging_variances} hold.
Let $\xi \in \Gauss$, viewed as a $\CC^{\noise}$-valued random variable.
Then the system~\eqref{ivp} admits
maximal solutions in the following sense.
There exist maps $\mathcal{S}^{-}:\CC^{\noise} \rightarrow \CC^{\reg,-}$ and 
$\mathcal{S}^{+}:\CC^{\noise} \times \CC^{\ireg} \rightarrow \CC^{\rem}$ with the following properties.
\begin{itemize}
\item The maps $\mathcal{S}^{\pm}$ are measurable.
\item Almost surely, the map $\psi \mapsto T[\mathcal{S}^{+}(\xi,\psi)]$ is a strictly positive lower semicontinuous function
and $\psi \mapsto \mathcal{S}^{+}(\xi,\psi)$ is continuous from $\CC^{\ireg}$ into $\mcC^{\rem}$.
\item For any smooth function $\rho:\R^{d+1} \rightarrow \R$ supported on the unit ball with $\int \rho = 1$, there exists a choice of constants~\eqref{eq:familyOfConst} such that, as $\eps \downarrow 0$, $\mathcal{S}_{\rho,\eps}^{-}$ converges to $\mathcal{S}^{-}$ in probability as random elements of $\CC^{\reg,-}$, and, for fixed $\psi \in \CC^{\ireg}$, $\CS^{+}_{\rho,\eps}(\xi,\psi)$ converges in probability to $\CS^{+}(\xi,\psi)$ as random elements of $\CC^{\rem}$.
\end{itemize}
\end{theorem}
\begin{remark}
One does not, in general, have unique (or even canonical) choices for the maps $\mcS^+$ and $\mcS^-$; this is already evident on the level of SDEs where the choice of It{\^o} or Stratonovich integration (or any interpolation thereof) 
leads to different solutions.
However, in the language of regularity structures, non-uniqueness is captured entirely by the choice of model above $\xi$.
In particular, as the formulation of the theorem indicates, the model can be chosen independently of the mollifier $\rho$, and the maps $\mcS^+$ and $\mcS^-$ become canonical once this choice is made.
\end{remark}
\begin{remark}\label{rem:choice_for_constants}
A possible choice for the constants \eqref{eq:familyOfConst} is given by
\begin{equs} \label{formula_constant}
 c_{\rho,\eps}^\tau
  = \E\big[\PPi^{\varrho,\varepsilon} \tilde{\CA}^{\ex}_{-} \tau(0)\big]\;,
 \end{equs}
where $ \tilde{\CA}^{\ex}_{-} $ is the twisted antipode defined in \cite[Prop.~6.17]{BHZalg} and 
$ \PPi^{\varrho,\varepsilon} $ is the canonical lift of $(\xi_\mfl^{(\varepsilon)})_{\mfl\in\mfL_-}$ defined in~\cite[Sec.~6.2]{BHZalg}.
In particular, $c_{\varepsilon,\rho}^\tau$ can be taken as zero whenever $\tau$ is planted or 
 of the form \eqref{eq:generic T} with $k\neq 0$.
\end{remark}

\begin{remark}\label{remark: conditions on trees for blackbox}
Assumption~\ref{assump:remove_noise_positive} is in fact automatic for locally subcritical systems of SPDEs where all driving noises have the same regularity.
In the general case, this condition is more for convenience than a fundamental necessity. 
This assumption was also made in~\cite{CH} to ease the presentation of the proof.
If this condition fails, one can always rewrite the system under consideration in such a way
that it is satisfied for the rewritten (equivalent) system. 
We also mention that our formulation of the renormalised equation in~\eqref{eq: renormalised equation} is based on this assumption. 
Our main result, Theorem~\ref{thm: algebraic main theorem}, which is a combinatorial\slash algebraic result, does not require this condition and in the more general case one may see new terms involving components of the noises $\xi$ in the renormalised equation\dash see Section~\ref{thm: algebraic main theorem}.

In Assumption~\ref{assump:no_diverging_variances}, the first condition  guarantees that none of the stochastic objects we need to control have diverging variances. 
Diverging variances cannot be cancelled by the subtraction of renormalisation constants and thus fall outside of our framework. 
The difficulty of dealing with this scenario was already observed in \cite{MR1883719,MR2667703,Hoshino2016,HuNualart,CH}
and one cannot expect the conclusions of Theorem~\ref{thm: THE black box} to hold in this case.
The second condition likewise prevents the occurrence of divergences we cannot renormalise.
\end{remark}
\begin{remark}\label{rem:stationarity}
The statement of Theorem~\ref{thm: THE black box} is more convoluted than a classical maximal existence theorem
due to our splitting of the solution map into maps $\CS^-$ and $\CS^+$. 
This is because our method of proof is to solve an equation for the remainder term of a truncated perturbative expansion at stationarity.
This truncated expansion is given by $\mathcal{S}^{-}$, which can be written explicitly\footnote{See the proof of Theorem~\ref{thm: THE black box} in Section~\ref{sec: proof of blackbox}.} as a finite sum of renormalised multilinear functionals of $\xi$.
Each such functional makes sense as a global in time object and there is one such functional for every $\mft$-non-vanishing tree $T^{\mfm}_{\mff}$ with $|T^{\mfm}_{\mff}|_{\s} \leq -(|\mft|_\s\wedge \s_0)$.
The remainder, which may blow up in finite time, is then given by $\mathcal{S}^{+}(\psi)$.

The generality allowed by the assumptions of Theorem~\ref{thm: THE black box} means that this notion of maximal solution is the best one can hope for. 
This is also needed for treating equations with scaling behaviour like the dynamical 
$\Phi^4_d$ problem with $d \ge 3$. 
Indeed, our result then applies as stated for arbitrary (non-integer) $d < 4$ for which it is not possible
to find a function space $\CB$ containing typical realisations of the solutions and such that even
the \textit{deterministic} Allen-Cahn equation is well-posed for arbitrary initial data in $\CB$.
\end{remark}
\subsection{Applications}\label{sec: examples}
\subsubsection{The generalised KPZ equation}\label{sec: example of genKPZ}
We apply Theorem~\ref{thm: THE black box} to the generalised KPZ equation as described in Section~\ref{sec: example of genKPZ part 0}.
We shall see below that the convergence statement of Theorem~\ref{thm: THE black box} simplifies greatly in this example since one can readily check that there are no $\mft$-non-vanishing trees with $|T^\mfm_\mff|_\s \leq -(|\mft|_\s\wedge \s_0)$, so by Remark~\ref{rem:stationarity}, one can take $\mathcal{S}_{\varrho,\varepsilon}^{-}(\xi) = \mathcal{S}^{-}(\xi) = 0$.

For every $\mft \in \mfL_+$, we set $\mathscr{L}_{\mft} \eqdef \partial_{1}^{2}-1$. 
Note that we chose $\mathscr{L}_{\mft}$ to satisfy~\eqref{eq:exponentialBound}, and added the term $\Y_{(\mft,0)}$ to $F_{\mft}^{\mathbf{0}}$ accordingly.

We define $\reg : \mfL \to \R$ and $\ireg : \mfL_+ \to \R$ by $\reg(\mft) = \ireg(\mft) = \frac{1}{2} - 3 \kappa$ for $\mft \in \mfL_{+}$, and $\reg(\mfl) = -3/2-2\kappa$ for $\mfl \in \mfL_{-}$.
Then it is straightforward to check that $F$ obeys $\reg$ in the sense of Definition~\ref{def: obey for blackbox}.

We turn to checking the assumptions of Theorem~\ref{thm: THE black box}.
\begin{itemize}
\item Assumption~\ref{assump:Schauder1} is readily verified upon noting that
\begin{equ}
n_\mft = |\mfl|_\s \wedge (2\ireg(\mft)-2) = \Big(-\frac{3}{2}-\kappa\Big)\wedge (-1-6\kappa)\;.
\end{equ}
\item One can readily check that the $\mft$-non-vanishing trees with lowest degree are of the form $\Xi_\mfl$, for which $|\Xi_\mfl|_\s = -\frac32 - \kappa$.
Hence, as we already pointed out, there are no $\mft$-non-vanishing trees with $|T^\mfm_\mff|_\s \leq -(|\mft|_\s\wedge \s_0)=-2$, from which Assumption~\ref{assump:planted trees} follows (see footnote~\ref{foot:subtree}).
\item As mentioned in Remark~\ref{remark: conditions on trees for blackbox}, Assumption~\ref{assump:remove_noise_positive} follows from the 
fact that $F$ obeys $\reg$ and $|\cdot|_{\s}$ is constant on $\mfL_{-}$.
\item Assumption~\ref{assump:no_diverging_variances} is an immediate consequence of the easily proven fact that for any $T^{\mfm}_{\mff} \in \mathring{\mathscr{T}}_{\mft_i,-}[F]$ with $|E_{T}| > 0$ one has $|T^{\mfm}_{\mff}|_{\s} \ge -1 - 2 \kappa > -\frac32 + \kappa$.
\end{itemize}

It follows that one can apply Theorem~\ref{thm: THE black box} to the generalised KPZ equation with the further simplification that $\mathcal{S}_{\varrho,\varepsilon}^{-}(\xi) = \mathcal{S}^{-}(\xi) = 0$ as claimed.

We finish this subsection by performing explicit computations of some of the terms 
$\Upsilon^{F}_{\mft}[T^{\mfm}_{\mff}](\mathbf{u_{\eps}})$ and constants $c_{\rho,\eps}[T^{\mfm}_{\mff}]$ appearing on the RHS of the renormalised equation \eqref{eq: renormalised equation} for $u_{\eps} = (u_{\mft,\eps})_{\mft \in \mfL_{+}}$\dash here $u_{\eps}$ is playing the role of $\phi_{\eps}$.
As in \eqref{eq:bold phi def}, we set $ \mathbf{u_{\eps}} $ for $ \big(\d^pu_{\mft,\eps}(z):\ (\mft,p) \in \mcb{O} \big)
 \in \R^{\mcb{O}}\;. $
There is a degree of freedom in choosing $c_{\rho,\eps}[\cdot]$ as given in \eqref{formula_constant} in that in order to specify $\PPi^{\varrho,\eps}$ one must fix a choice of truncation of $(\partial_{t} - \mathscr{L}_{\mft})^{-1}$ for each $\mft \in \mfL_{+}$.
For convenience, since all these kernels coincide in our case, just fix a single kernel 
$K(z)$ which is a smooth function on $\Lambda \setminus \{0\}$ that is of compact support, agrees with $(\partial_{t} - \mathscr{L}_{\mft})^{-1}$ for $|z|_{\s} \le 1$, and integrates to $0$. 

Before presenting it we describe some notational conventions we use in the computation.
In order to lighten notation we will drop the $\eps$ from notation, writing $u = (u_{\mft}: \mft \in \mfL_{+}) $ and $\mathbf{u}= \big(\d^pu_{\mft}(z):\ (\mft,p) \in \mcb{O} \big)
 \in \R^{\mcb{O}}\;. $
 The functions  $\sigma_{\mathfrak{a}}^{\mathfrak{d}}(u)$ and $\Gamma^{\mathfrak{a}}_{\mathfrak{b},\mathfrak{c}}(u)$, $\mathfrak{a},\mfb,\mfc \in \mfL_{+}$, $\mathfrak{d} \in \mfL_{-}$, are always written as functions of $u$, not $\bu$, to make it clear that they depend on $(u_{\mft'}: \mft' \in \mfL_{+})$, but not on any derivatives of these functions. Moreover, when we write an expression like $D_{\mft'}\sigma_{\mft}^{\mfl}(u)$ we are taking a derivative in the $u_{\mft'}$ argument. 
We also recall that in this example we use $z = (t,x)$ notation for points in space-time and so we use the shorthand $\partial_{x} \eqdef \partial^{(0,1)}$.

Let us consider the trees given by $\mcb{I}_{(\mfj,0)}(\mathbf{X}^{(0,1)} \Xi_{\mfl}) \Xi_{\mfl} $ and $ \mcb{I}_{(\mfj,0)} (\Xi_{\mfl}) \mcb{I}_{(\mfq,(0,1))}(\Xi_{{\mfl}}) $, where $\mfj, \mfq \in \mfL_{+}$ and $\mfl \in \mfL_{-}$ as in the notation of~\eqref{eq:generic T}. 
Both trees are of degree $-2\kappa$
and we can depict them graphically as in \cite{FrizHairer,WongZakai} by
\begin{equs}
\mcb{I}_{(\mfj,0)}(\mathbf{X}^{(0,1)} \Xi_{\mfl}) \Xi_{\mfl} = \<Xi2X>, \enskip
\mcb{I}_{(\mfj,0)} (\Xi_{\mfl}) \mcb{I}_{(\mfq,(0,1))}(\Xi_{\mfl})  
= \<IXi^2>\;.
\end{equs}
We suppress the indices in the graphical notation for the sake of conciseness but for what follows $\mfj$, $\mfq$, and $\mfl$ have been fixed. 
Circles
represent instances of $\Xi_{\mfl}$, the cross represents the factor $\mathbf{X}^{(0,1)}$.
We first walk through the computation of $\Upsilon^{F}_{\mft}[\<Xi2X>](\bu)$ for some arbitrary $\mft \in \mfL_{+}$.

\begin{equs}
\Upsilon_{\mft}^{F}[\<Xi>](\bu) & = F_{\mft}^{\mfl}(\bu) = \sigma^{\mfl}_{\mft}(u)
\\
\Upsilon_{\mft}^{F}[\<XiX>] (\bu) & = 
\big( \partial^{(0,1)} \Upsilon_{\mft}^{F}[\<Xi>] \big)(\bu) 
=
\big( \partial^{(0,1)} F_{\mft}^{\mfl} \big)(\bu) 
= 
\sum_{\mft' \in \mfL_{+}} (\partial_x u_{\mft'}) (D_{\mft'} \sigma_{\mft}^{\mfl})(u)
\\ 
\Upsilon^{F}_{\mft}[\<Xi2X>](\bu) 
& = \left( \Upsilon_{\mfj}^{F}[\<XiX>]
  D_{(\mfj,0)}  F_{\mft}^{\mfl}  \right)(\bu)
\\ & = 
(D_{j} \sigma_{\mft}^\mfl)(u)
\sum_{\mft' \in \mfL_{+}} (\partial_x u_{\mft'}) (D_{\mft'} \sigma^{\mfj}_{\mfl})(u)\;. 
\end{equs}
Using formula~\eqref{formula_constant} we get:
\begin{equs}
c_{\rho,\eps}[ \<Xi2X> ]
& =  \E \big[ \big(\PPi^{\varrho,\varepsilon} \tilde{\CA}^{\ex}_{-} \<Xi2X>\big)(0)\big] = - \E \big[\big( \PPi^{\varrho,\varepsilon}\<Xi2X>\big)(0)\big] \label{eq:first const}
\\ & = 
\int_{\Lambda^{2}}
\rho_{\eps}(z-z') \rho_{\eps}(-z')z_{1}K(-z)\,dz\,dz' 
\;.
\end{equs}
For the tree $\<IXi^2>$ we have
\begin{equs}
\Upsilon^{F}_{\mft}[\<IXi^2>](\bu)
&=
\big[
\Upsilon^{F}_{\mfj}[\<Xi>]\ 
\Upsilon^{F}_{\mfq}[\<Xi>]\ 
D_{(\mfj,0)} D_{(\mfq,(0,1))} F^{\mathbf{0}}_{\mft}
\big]
(\bu)   
\\
{}&=  \sigma^{\mfl}_{\mfj}(u)\sigma^{\mfl}_{\mfq}(u) 
\sum_{\mft' \in \mfL_{+}} (\partial_x u_{\mft'}) \big[ D_{\mfj} (\Gamma^{\mft}_{\mft',\mfq} +   \Gamma^{\mft}_{\mfq,\mft'}) \big](u)
\end{equs}
and, writing $ K^{(\varrho,\varepsilon)}  \eqdef K \ast \varrho_{\varepsilon} $,
\begin{equs}
c_{\rho,\eps}[\<IXi^2>]
&=  
\E 
\big[
\big(\PPi^{\varrho,\varepsilon}
\tilde{\CA}^{\ex}_{-}        
\<IXi^2>\big) (0) \big]
 =  - \E
\big[ 
\big(\PPi^{\varrho,\varepsilon}  
\<IXi^2>\big)(0)
\big] \\ 
{}&=
\int_{\Lambda}
K^{(\varrho,\varepsilon)}(z) (\partial_{x}K^{(\varrho,\varepsilon)})(z)\,dz\;. \label{eq:second const}
\end{equs}
Note that~\eqref{eq:second const} vanishes, and, for spatially symmetric $\rho$ and $K$, so does~\eqref{eq:first const};
in fact, this remains true for any choice of noise $\xi \in \Gauss$.
\subsubsection{The dynamical $\Phi^4_{4-\delta}$ model}
\label{subsec:Phi44-delta}
We consider in this subsection the equation
\begin{equs}
\partial_t \phi = (\Delta-1)\phi - \phi^3 + \xi\;.\label{eq:Phi44-delta}
\end{equs}
We work here in $1+4$ dimensions $\Lambda \eqdef \R \times \T^4$ and use the parabolic scaling $\s = (2,1,1,1,1)$.
We fix some $\delta > 0$ and consider $\xi$ as a Gaussian noise which satisfies the conditions of Definition~\ref{definition noises} for every $|\mfl|_\s < -3+\frac{\delta}{2}$ (constructed, for example, by the convolution of white noise on $\Lambda$ with a slightly regularising kernel).
Note that, in terms of scaling properties, the cases $\delta=2$ and $\delta=1$ behave  like the usual $\Phi^4_2$~\cite{DPD2} and $\Phi^4_3$~\cite{CC13,Regularity,HX16,MourratWeber16} equations respectively, while the case $\delta = 0$ corresponds to the critical regime.

In this example, we demonstrate a situation where one is unable to start the equation from initial data of the ``natural regularity'', i.e., of the same regularity as the solution, and thus requires the full power of Assumptions~\ref{assump:Schauder1} and~\ref{assump:planted trees} and the decomposition of $\mcS$ into $\mcS^{\pm}$ in Theorem~\ref{thm: THE black box}.
As we shall see below, for every $\delta > 0$, one must take $\varphi_{\varepsilon}(0,\cdot) = \psi + \mcS^-_{\varrho,\varepsilon}(0,\cdot)$ where $\psi \in \mcC^\eta_{\bar{\s}}(\T^d)$ with $\eta > (-\frac{2}{3}) \vee (-\delta)$, and $\mcS^-_{\varrho,\varepsilon}$ is the explicit stationary part which converges in $\mcC^{-1+\delta/2-\kappa}_\s(\Lambda)$ in probability as $\varepsilon \to 0$.
In particular, rougher noise forces us to start the equation from a smoother initial condition for the remainder (which can be interpreted as starting the equation closer to equilibrium).
\begin{remark}
One can see directly the necessity of the lower bound $\eta > -\frac{2}{3}$ by recalling that the deterministic map which sends $\phi(0,\cdot)$ to the solution~\eqref{eq:Phi44-delta} with zero noise $\xi = 0$ is continuous only for $\varphi(0,\cdot) \in \mcC^\eta_{\bar{\s}}(\T^d)$ with $\eta > -\frac{2}{3}$, cf.~\cite[Rem.~9.9]{Regularity}.
\end{remark}
We fix some $\delta > 0$ and $\kappa \in (0,\frac{\delta}{6})$ for the rest of the example. 
First, a computation shows that the renormalised equation takes the form
\begin{equ}
\partial_t \phi_\varepsilon = (\Delta-1) \phi_\varepsilon - \phi^3_\varepsilon + C_{\varepsilon,2} \phi^2_\varepsilon + C_{\varepsilon,1} \phi_\varepsilon + C_{\varepsilon,0} + \sum_{i=1}^4 C_\varepsilon^{(i)}\partial_i \phi_\varepsilon + \xi^{(\varepsilon)}\;.\label{eq:renormalised Phi eq}
\end{equ}
Indeed, here $\mfL_+ = \{\mft\}$ and $\mfL_- = \{\mfl\}$ are singletons, $|\mft|_\s = 2$, and $|\mfl|_\s \eqdef -3+\frac{\delta}{2} - \kappa$.
The corresponding nonlinearity is the cubic function $F(\Y) = \Y_{(\mft,0)}^3$.
Suppose that $\tau = T^\mfm_\mff \in \mathring{\mathscr{T}}_{\mft,-}[F]$ has at least one edge, i.e.\ it is not of the form 
$\Xi_\mfl$ (note that, for $k \neq 0$, $\mathbf{X}^k\Xi_\mfl$ is not $\mft$-non-vanishing and thus does not belong to $\mathring{\mathscr{T}}_{\mft,-}[F]$).
The following can readily be deduced by an inductive argument:
if $x \in N_T$ is a leaf, then $\mnoise(x) = \mfl$; if $x \in N_T$ is not a leaf, then $\mnoise(x) = 0$; if $e \in E_T$, then $\mff(e) = (\mft,0)$.
Moreover, every node $x \in N_T$ must have $\mpoly(x) = 0$ and, if $x$ is not a leaf, must have three outgoing edges with the following possible exceptions:
\begin{enumerate}
\item \label{pt:phi2I} there is exactly one node with one outgoing edge; in this case $|\tau|_\s = -1+\frac{\delta}{2}-\kappa$ if $\tau$ is the planted tree $\mcb{I}_{(\mft,0)}[\Xi_\mfl]$, and $|\tau|_\s \geq -1+\frac{3\delta}{2}-3\kappa$ otherwise,
\item \label{pt:partial_phi} there is exactly one node $x$ with two outgoing edges and $\mpoly(x) = e_i$ for some $i=1,\ldots, 4$; in this case $|\tau|_\s = -1+\delta-2\kappa$ if $\tau=\mathbf{X}^{e_i}\mcb{I}_{(\mft,0)}[\Xi_\mfl]^2$, and $|\tau|_\s \geq -1+2\delta-4\kappa$ otherwise,
\item there are exactly two nodes with two outgoing edges each; in this case $|\tau|_\s \geq -1+\frac{3\delta}{2}-3\kappa$,
\item there is exactly one node with two outgoing edges; in this case $|\tau|_\s \geq -2+\delta-2\kappa$,
\item no exceptions; in this case $|\tau|_\s \geq -3+\frac{3\delta}{2}-3\kappa$. 
\end{enumerate}
An inductive argument tells us that the counterterms associated to the above possibilities are given respectively by (up to combinatorial factors):
\[
\textnormal{1. }
\phi^2_\varepsilon,
\enskip 
\textnormal{2. }
\partial_i \phi_\varepsilon,
\enskip 
\textnormal{3. }
\phi^2_\varepsilon,
\enskip
\textnormal{4. }
\phi_\varepsilon,
\enskip
\textnormal{and}
\enskip
\textnormal{5. a numeric constant.}
\]
Furthermore, recall from Remark~\ref{rem:choice_for_constants} that planted trees and trees with non-zero polynomial decorations at the root do not contribute to the counterterms, and thus the first sub-cases in cases~\ref{pt:phi2I} and~\ref{pt:partial_phi} can be ignored when determining the renormalised equation.
It follows that $C_{\varepsilon,2}\phi_\varepsilon^2$ and $C_\varepsilon^{(i)}\partial_i \phi_\varepsilon$ do not appear in~\eqref{eq:renormalised Phi eq} whenever $\delta > \frac{2}{3}$ and $\delta > \frac{1}{2}$ respectively, which explains their absence in the usual $\Phi^4_3$ equation.
Similarly, $C_{\varepsilon,1}\phi_\varepsilon$ and $C_{\varepsilon,0}$ do not appear whenever $\delta > 2$, precisely the values for which \eqref{eq:Phi44-delta} is classically well-posed.
Note further that, due to the symmetry $\phi \mapsto -\phi$, one can in fact take $C_{\varepsilon,2}=C_{\varepsilon,0}=0$ since our noise has vanishing odd moments.
\begin{remark}\label{rem:not_complete}
Equation~\eqref{eq:Phi44-delta} also demonstrates an example where the naive rule constructed from the corresponding nonlinearity is not complete in the sense of~\cite{BHZalg}.
Indeed, the rule $R(\mfl) = \{()\}, R(\mft)=\{(\Xi_\mfl),([\mcb{I}_{(\mft,0)}]_\ell)\}_{\ell =0,\ldots, 3}$ ceases to be complete for $\delta \leq \frac{1}{2}$,
and its completion~\cite[Def.~5.21]{BHZalg} is given by adding $\{(\mcb{I}_{(\mft,e_i)})\}_{i=1,\ldots 4}$ to $R(\mft)$.
While the consideration of rules is not necessary to compute the renormalised equation, we note that $R$ fails to be complete for the same reason as the counterterm $\sum_{i=1}^4 C_\varepsilon^{(i)}\partial_i \phi_\varepsilon$ appears in~\eqref{eq:renormalised Phi eq}, which are consequences of the (negative) renormalisation procedure.
\end{remark}
Continuing on, we define $\reg : \mfL \to \R$ by $\reg(\mfl) \eqdef -3+\frac{\delta}{2}-2\kappa$ and $\reg(\mft) \eqdef -1 + \frac{\delta}{2} - 3\kappa$.
Since $\reg(\mft) > -1$, we see that $F$ obeys $\reg$, i.e., the equation~\eqref{eq:Phi44-delta} is subcritical.
Furthermore, choosing any $\ireg:\mfL_+ \to \R$ such that $\ireg(\mft) > (-\frac{2}{3})\vee(-\delta+6\kappa)$, we see that Assumption~\ref{assump:Schauder1} is satisfied.
Indeed, the first condition in Assumption~\ref{assump:Schauder1} is trivial since $\mcb{O}_+ = \emptyset$.
Furthermore, using Remark~\ref{remark:F const}, we have that $n_\mft = |0|_\s + (3\ireg(\mft)) \wedge (\ireg(\mft) + 2\reg(\mft))$.
Hence the bounds in the second condition of Assumption~\ref{assump:Schauder1} are respectively equivalent to
\begin{equs}
(3\ireg(\mft)) \wedge (\ireg(\mft) + 2\reg(\mft)) > -2
\; &\Leftrightarrow \;
\ireg(\mft) > (-2/3)\vee (-\delta+6\kappa)\;,
\\
(3\ireg(\mft)) \wedge (\ireg(\mft) + 2\reg(\mft)) + 2 > \ireg(\mft)
\; &\Leftrightarrow \;
\ireg(\mft) \wedge \reg(\mft) > -1\;,
\end{equs}
both of which are satisfied with the above choices.
Note that Assumption~\ref{assump:remove_noise_positive} is again automatic by Remark~\ref{remark: conditions on trees for blackbox}, while Assumptions~\ref{assump:no_diverging_variances} and~\ref{assump:planted trees}, are readily checked using the above classification of $\mathring{\mathscr{T}}_{\mft,-}[F]$.
We thus meet all the criteria to apply Theorem~\ref{thm: THE black box}.

To summarise, it follows from Theorem~\ref{thm: THE black box} that for any fixed $\delta > 0$ there exists
\begin{itemize}
\item a choice of constants $C_{\varepsilon,2},C_{\varepsilon,1},C_{\varepsilon,0}, C_{\varepsilon}^{(i)}$, $i=1,\dots, 4$, and
\item a function of the noise $\mcS^-_{\varrho,\varepsilon}(\xi)$ which is smooth for every $\varepsilon > 0$,
\end{itemize}
such that, as $\varepsilon \downarrow 0$,
\begin{itemize}
\item $\mcS^-_{\varrho,\varepsilon}(\xi) \to \mcS^-(\xi)$ in $\mcC^{-1+\delta/2-\kappa}_\s(\Lambda)$ in probability, and
\item for any $\eta > (-\frac{2}{3})\vee (-\delta)$, the solution to the renormalised equation~\eqref{eq:renormalised Phi eq} with initial condition $\phi_\varepsilon(0,\cdot) = \psi + (\mcS^-_{\varrho,\varepsilon}(\xi))(0,\cdot)$, for fixed $\psi \in \mcC^\eta_{\bar\s}(\T^d)$,
converges in probability to local solutions in the sense dictated by the theorem.
\end{itemize}
%
\begin{remark}
In the case $\delta \in (\frac{2}{3},2]$, the stationary part $\CS^-$ takes on the simple form $\CS^-_{\varrho,\eps}(\xi) \eqdef G*\xi^{(\varrho,\eps)} \to \CS^-(\xi) \eqdef G*\xi$, where the convergence moreover happens in the space $\CC([0,T], \CC^{-1+\delta/2-}_{\bar\s}(\T^4))$ (cf.~\cite[Sec.~9.4]{Regularity}).
One can leverage this fact to build a more explicit (though equivalent) solution theory for the $\Phi^4_3$ equation than that given by Theorem~\ref{thm: THE black box}.
\end{remark}
\subsubsection{Renormalisation of non-linearities}
  We close this section with a computation that shows how, with a pre-processing trick, the algebraic framework described in the next two sections also allows us to identify an action of the renormalisation group on reasonable local non-linearities even when the non-linearity in question doesn't appear on the right hand side of~\eqref{ivp}. 
This example is different in flavour from Sections~\ref{sec: example of genKPZ}
 and~\ref{subsec:Phi44-delta} since it is more of an advertisement of the generality of the upcoming Theorem~\ref{thm: algebraic main theorem} rather than the applicability of the self-contained Theorem~\ref{thm: THE black box}. 

To illustrate this trick we work with the equation
\begin{equs}\label{eq:gpam}
\partial_t  u_{\mfb} = (\Delta - 1)u_{\mfb} + u_{\mfb} + g(u_{\mfb}) \xi_{\mfl}\;,
\end{equs}
in $1+2$ dimensions, on $\Lambda \eqdef \R \times \T^2$, where $g:\R \rightarrow \R$ is a smooth function and $\xi_{\mfl}$ is spatial white noise, that is $\xi_{\mfl}$ is Gaussian with covariance structure $\E[\xi_{\mfl}(t,x) \xi_{\mfl}(s,y)] = \delta(x-y)$. 
Equation~\eqref{eq:gpam} is sometimes called the generalised parabolic Anderson model (gPAM). 
We use the parabolic scaling $\s = (2,1,1)$ and, since we have a scalar equation with a single driving noise,
both $\mfL_{+} = \{\mfb\}$ and $\mfL_{-} = \{\mfl\}$ are singletons.
We fix $\mathscr{L}_{\mfb} = \Delta - 1$ so we have $F^{\mfl}_{\mfb}(\mathbf{u}) = g(u_{\mfb})$ and $F^{\mathbf{0}}_{\mfb}(\mathbf{u}) = u_{\mfb}$.  
We can set $|\mft|_{\s} = 2$ and $|\mfl|_{\s} = -1 - \kappa$ for any $\kappa > 0$ and to keep the number of symbols of negative degree to a minimum we assume that $\kappa \in (0,\frac{1}{3})$. 
We then have $\mathring{\mathscr{T}}_{\mfb,-}[F] = \{ \<Xi2> \}$ where $\<Xi2> = \Xi_{\mfl}\mcb{I}_{(\mfb,0)}[\Xi_{\mfl}]$. 
One can then check that $\Upsilon_{\mfb}^{F}[\<Xi2>](\mathbf{u}) = g'(u_{\mfb})g(u_{\mfb})$. 
We leave the specification of $\reg$, $\ireg$, and the verification of the assumptions of Theorem~\ref{thm: THE black box} to the reader\dash the conclusion is that for any appropriate mollifier $\rho$ there exist constants $c^{\<Xi2>}_{\rho,\eps}$, for $\eps \in (0,1]$, such that, started from appropriate initial data, the local solutions of 
\begin{equs}\label{eq:renorm gpam}
\partial_t  u_{\mfb,\eps} = (\Delta - 1)u_{\mfb,\eps} + u_{\mfb,\eps} + g(u_{\mfb,\eps}) \xi_{\mfl,\eps} +  g'(u_{\mfb,\eps})g(u_{\mfb,\eps}) c^{\<Xi2>}_{\rho,\eps},\;
\end{equs}
where $\xi_{\mfl,\eps} \eqdef \xi_{\mfl} \ast \rho_{\eps}$ as before, converge to a limit $u$ as $\eps \downarrow 0$.

Readers familiar with the theory of regularity structures will know that the limit $u_{\mfb}$ can be obtained directly as a solution \eqref{eq:gpam} when one interprets the equation as ``being driven'' by a particular rough model $\hat{Z}$.  
Then the convergence of the $u_{\mfb,\eps}$ to $u_{\mfb}$ is a consequence of the fact that suitably renormalised smooth models $\hat{Z}_{\eps}$ converge to $\hat{Z}$ as $\eps \downarrow 0$, where the $u_{\mfb,\eps}$ arise as the solutions to~\eqref{eq:gpam} driven by $\hat{Z}_{\eps}$. 
As mentioned before, the contribution of this article is to show that the $u_{\mfb,\eps}$ are themselves solutions to classical PDEs driven by $\xi_{\mfl,\eps}$ (for the case of gPAM, this can be verified by hand as in~\cite{Regularity}).

The theory of regularity structures also allows one to make sense of classically ill-defined non-linearities of $u_{\mfb}$, such as $f(u_{\mfb})\xi_{\mfl}$ for some smooth function $f$, in a ``renormalised sense''. One can use the modelled distribution expansion of $u_{\mfb}$ to obtain a modelled distribution for $f(u_{\mfb}) \xi_{\mfl}$ and then applying to this the reconstruction operator associated to $\hat{Z}$.

It is natural to ask if $f(u_{\mfb})\xi_{\mfl}$ constructed in the above manner can be obtained as a limit of $h_{\eps}(u_{\mfb,\eps},\xi_{\mfl,\eps})$ for appropriate $h_{\eps}(\cdot,\cdot)$ as $\eps \downarrow 0$. 
For the case of gPAM, this can be done by hand fairly easily, but the method we give below is more robust and much more tractable for analogous computations for more complicated equations like~\eqref{eq: gen KPZ} \dash see \cite{BGHZ} 
for a situation where one needs such a result. 
  
The idea is simple and involves working with a larger system of equations. 
We set $\mfL_{+} \eqdef \{ \mfb, \mfq \}$ and look at 
\begin{equs}
\ u_{\mfb} =& P_{\mfb} 
\left[ 
u_{\mfb} + g(u_{\mfb}) \xi_{\mfl}
\right] \qquad\textnormal{and}\qquad
u_{\mfq} 
=& 
P_{\mfq}
\left[
f(u_{\mfb})\xi_{\mfl} 
\right]\;.
\end{equs}
Here we have written the system as a fixed point problem, where $P_{\mfb} \eqdef (\partial_{1} - \mathscr{L}_{\mfb})^{-1}$. The new operator $P_{\mfq}$ just corresponds to the identity or some approximate identity and has been inserted so that we can encode the problem of obtaining $f(u)\xi$ as solving a system of equations. 
We set $|\mfq|_{\s} = 0$ \dash 
we are slightly out of the scope of the assumptions of Theorem~\ref{thm: THE black box} but the proof still carries through.\footnote{Adopting the notation of later sections, a simple way to make this rigorous is to set $U_\mfq \eqdef F^\mfl_\mfq(U_\mfb)\Xi_\mfl$ (or more precisely $U_\mfq \eqdef \mcI_{\mfq}[F^\mfl_\mfq(U_\mfb)\Xi_\mfl]$ so that the trees in $U_\mfq$'s expansion are planted and thus fit into the framework of Section~\ref{sec: alg and main theorem} \dash  here $\mcI_{\mfq}$ acts ``trivially'' for our models) and observe that the proof of Theorem~\ref{thm:renormalised_equation} implies $u_\mfq(x) \eqdef \hat\mcR U_\mfq (x) = \sum_{\hat\mfl}(MF)^{\hat\mfl}_\mfq(u_\mfb(x))\xi_{\hat\mfl}(x)$, which is precisely what we want.}

In this new system we have $F_{\mfb}^{\mathbf{0}}(\mathbf{u}) = u_{\mfb}$, $F_{\mfb}^{\mfl}(\mathbf{u}) = g(u_{\mfb})$,
$F_{\mfq}^{\mathbf{0}} = 0$, and $F_{\mfq}^{\mfl}(\mathbf{u}) = f(u_{\mfb})$. 
For our larger system we now have new symbols which have an edge of type $\mfq$, but all of these new symbols vanish under $\Upsilon^{F}_{\mft}[\cdot]$, for $\mft \in \{\mfb,\mfq\}$, since $u_{\mfq}$ doesn't appear in any component of $F$. 
Again, the sole symbol contributing counterterms is $\<Xi2>$ and we observe that the system gets renormalised to
\begin{equs}
\ u_{\mfb,\eps} =& P_{\mfb} 
\left[ 
u_{\mfb,\eps} + g(u_{\mfb,\eps}) \xi_{\mfl,\eps}
+
c^{\<Xi2>}_{\rho,\eps}
\Upsilon^{F}_{\mfb}[\<Xi2>](\mathbf{u}_{\eps})
\right]\\
u_{\mfq,\eps} 
=& 
P_{\mfq}
\left[
f(u_{\mfb,\eps})\xi_{\mfl,\eps}
+
c^{\<Xi2>}_{\rho,\eps}
\Upsilon^{F}_{\mfq}[\<Xi2>](\mathbf{u}_{\eps})
\right]\;
\end{equs}
The $\mfb$ component works out the same as before, that is $\Upsilon^{F}_{\mfb}[\mathbf{u}] = g'(u_{\mfb})g(u_{\mfb})$.
For the $\mfq$ component we have 
$\Upsilon^{F}_{\mfq}[\mathbf{u}] = f'(u_{\mfb})g(u_{\mfb})$ which indicates that the function $h_{\eps}$ is given by 
\[
h_{\eps}(u_{\mfb,\eps},\xi_{\mfl,\eps})
\eqdef
f(u_{\mfb,\eps})\xi_{\mfl,\eps}
+
c^{\<Xi2>}_{\rho,\eps}
f'(u_{\mfb,\eps})g(u_{\mfb,\eps})\;.
\]
\section{Algebraic theory and main theorem}\label{sec: alg and main theorem}
In this section we will state the main theorem of this paper,  Theorem~\ref{thm: algebraic main theorem}. 
 In order to precisely state and prepare to prove this theorem, we will spend much of this section recalling and specializing many definitions from the theory of regularity structures.  
Before presenting what might be an intimidating level of formalism, we take a moment to describe the content of this theorem at a conceptual level. 

In the theory of regularity structures, fixed point problems are first posed and solved in a space of modelled distributions.
Recall that these modelled distributions are jets that should be viewed as generalised Taylor expansions\dash they are functions from space-time into a linear span of abstract monomials.
These abstract monomials can be split into two categories: (i) the ``classical'' monomials $\mathbf{X}^{k}$, $k \in \N^{d}$, which are  placeholders for polynomials and (ii) ``planted trees'' $\mcb{I}_{(\mft,0)}[\tau]$, which are placeholders for multi-linear functionals of the driving noise.

Since this section is more algebraic than analytic, we will not invoke the actual definition of modelled distributions (which also enforce some space-time regularity on these jets) but just investigate what it means for such a jet to solve the fixed point problem at an algebraic level, namely that the planted tree part of the jet is recovered when the full jet is inserted into the non-linearity.
We identify a necessary and sufficient condition for this to hold which we call ``coherence with the non-linearity''\dash see Definition~\ref{def: coherent} and Lemma~\ref{lemma: coherence identity}.
In words, coherence with the non-linearity forces the coefficients of planted trees in a jet to be given by particular explicit functions of the coefficients of the classical monomials.

Next we investigate how renormalisation and coherence interact. 
Recall that the renormalisation group of regularity structures can be interpreted as a group of linear operators acting on the span of abstract monomials. 
For every element of the renormalisation group, we define in~\eqref{def:MP} a map on the space of non-linearities.
In Lemma~\ref{lem: renormalisation of upsilon}, we show that this procedure induces a {\it bona fide} group action of the renormalisation group, and describe this action through the adjoint of the renormalisation maps acting on monomials. 
This then allows us to state Proposition~\ref{main prop}, which says that the combined action of the renormalisation operators on non-linearities and jets preserves coherence.
With this in hand we can then apply Lemma~\ref{lemma: coherence identity} to obtain Theorem~\ref{thm: algebraic main theorem}, which states that if a jet satisfies an algebraic fixed point problem, then the renormalised jet satisfies a renormalised algebraic fixed point problem.

We now return to formulating the setting for the precise statement and proof of Theorem~\ref{thm: algebraic main theorem}.
\subsection{Set-up of the regularity structure and renormalisation group}\label{subsec: set-up}
We freely use the notion of rules and the notations from \cite[Sec.~5.2]{BHZalg}.
We start by fixing a normal complete rule $R$\label{rule page ref} which is subcritical with respect to $\reg:\mfL \rightarrow \R$; namely
\begin{equ}
\reg(\mft)
< 
|\mft|_{\s}
+
\inf_{\mcN \in R(\mft)}
\sum_{(\mfb,p) \in \mcN}
\reg(\mfb,p)\;, \quad \forall \mft \in \mfL\;.
\end{equ}
We denote by $\mathscr{T} \eqdef (\mcb{T}^{\ex},G)$\label{reg st page ref} the (untruncated) extended regularity structure corresponding to $R$ as defined in \cite{BHZalg}.
We write $\mcT^{\ex}$\label{mcT page ref}
for the collection of decorated trees which span $\mcb{T}^{\ex}$\label{cT page ref}.
These decorated trees are of the form $ T^{\Labn,\Labo}_{\Labe} $ where $ T $ is a rooted tree endowed with a type map $ \Labhom : E_T \rightarrow \mfL  $, an edge decoration $ \Labe : E_T \rightarrow \N^{d+1} $ and two node decorations $ \Labn : 
N_T \rightarrow \N^{d+1} $, $ \Labo : N_T \rightarrow \Z^{d+1} \oplus \Z(\mfL) $. We assign to any such tree two degrees $|\cdot|_-$ and $|\cdot|_+$ by
\begin{equs}[e:degrees]
| T^{\Labn,\Labo}_{\Labe}|_{-} 
& \eqdef 
\sum_{e \in E_{T}}
(|\mft(e)|_\s - |\Labe(e)|_{\s})
+
\sum_{x \in N_{T}}
|\Labn(x)|_{\s}\;,  \\
| T^{\Labn,\Labo}_{\Labe}|_{+} 
& \eqdef 
\sum_{e \in E_{T}}
(|\mft(e)|_\s - |\Labe(e)|_{\s})
+
\sum_{x \in N_{T}}
(|\Labn(x)|_{\s} + |\mfo(x)|_\s)
\;.
\end{equs}
Given a rooted tree $ T $, we endow $ N_T $ with the partial order $ \leq $ where $ x \leq y $ if and only if $ x $ is on the unique path connecting $ y $ to the root.

We write $ \mfR $\label{RG page ref} for the corresponding renormalisation group,\footnote{At this stage we really mean the full renormalisation group $\mfR$ which gives complete freedom on how to treat extended labels, we are not restricting to the subgroups mentioned in \cite[Rem.~6.25]{BHZalg}.}
i.e., the set of linear maps $ M : \cT^\ex \to \cT^\ex$ of the form $ M_g $ for some $ g \in \mcb{G}^\ex_-$, see \cite[Sec.~6.3]{BHZalg}.

We now make a deviation from the notation of~\cite{BHZalg}.
In the sequel, we restrict ourselves to a subspace of $ \mcb{T}^{\ex} $, which, by an abuse of notation, we denote with the same symbol.
This subspace is spanned by all trees $T^{\Labn,\Labo}_{\Labe}$ for which $\Labn(x) = 0$ for all leaves $x\in N_T$ which are connected to the tree with an edge of type $ \mfl \in  \mfL_{-} $.
This means that $T^{\Labn,\Labo}_{\Labe}$ has no edges of the form $\Xi^\mfl_{k,\ell} \eqdef \mcb{I}^\mfl_k[X^\ell]$ with $\mfl\in \mfL_-$ and $\ell\neq 0$, but $X^\ell \Xi^\mfl_k \eqdef X^\ell\mcb{I}^\mfl_k[X^0]$ for example is allowed.
While both symbols appear in the regularity structure defined in~\cite{BHZalg}, it is the latter which we will identify with $\mathbf{X}^\ell \Xi_{(\mfl,k)}$ in Section~\ref{Notations for trees} (and the former will not have a meaning here).
In a similar way, we henceforth let $\mcT^\ex$ denote the set of all such trees.
Observe that $\mcb{T}^\ex$ forms a sector of our regularity structure which is closed under the action of $\mfR$.

\begin{remark}\label{rem:difference_of_notation}
The trees in $\mcT^\ex$ are of a different form to those in Section~\ref{black box section}, namely noises are treated as edges in $\mcT^\ex$ but as node decorations in $\mathring{\mathscr{T}}$.
We explain in Section~\ref{Notations for trees} how to reconcile these two formalisms, and why we choose the latter in this article (see Remark~\ref{rem:noise_node_vs_edge}).
\end{remark}
\subsection{Drivers}\label{subsec: drivers}
Our equations will be written in a mild formulation, where we ask for the components of the solution to be equal to an integral kernel operator acting on a linear combination of elements of $\Poly$ multiplied by ``driving terms''. 
The family of possible driving terms includes the noises and their derivatives $\{D^{\mfe}\xi_{\mft}: \mft \in \mfL_{-}, \mfe \in \N^{d+1} \}$, products of such terms, as well as the constant function $1$.  
In order to incorporate information on how an equation has been renormalised, it is natural to allow for some degeneracy in the set of drivers in our abstract formulation, so 
we introduce some notation for this. 
First, we define $\widetilde{\mfD}
\eqdef
\{ 
(\mfl,\mfe): \mfl \in \mfL_{-},\ 
\mfe \in \N^{d+1}
\}$.
Then we define 
\[
\hat{\mfD}
\eqdef
\Big\{
\hat{\mfl} 
\in \N^{\widetilde{\mfD}}:\ 
\exists \mft \in \mfL_{+} 
\textnormal{ with } \Xi_{\hat\mfl} \in R(\mft)
\Big\}\;,\label{drivers hat page ref}
\] 
where, for $\hat{\mfl} = \big\{ (\mfl_{1}, \mfe_{1}),\dots, (\mfl_{k},\mfe_{k}) \big\} \in \hat{\mfD}$, we set
\[
\Xi_{\hat{\mfl}} 
\eqdef
D^{\mfe_{1}}\Xi_{\mfl_{1}}
\cdots
D^{\mfe_{k}}\Xi_{\mfl_{k}}
\in \mcT^{\ex}\;,
\]
with the product being the tree product in $\mcT^{\ex}$
(in particular, $\Xi_{\mathbf{0}} = \bone$).
Note that subcriticality of $R$ implies that $\hat{\mfD}$ is finite and that, by completeness, one has 
$\mathbf{0} \in \hat{\mfD}$.

We also define, as in \cite[Def.~5.23]{BHZalg}, a set $D(\mft,N) \subset \Z^{d +1} \oplus \Z(\mfL)$ of extended decorations for every $\mft \in \mfL_{+}$ and $N \in R(\mft)$. We extend this definition by setting $D(\mft,N) = \emptyset$ for $\mft \in \mfL_{+}$ and $N \in \mcb{N} \setminus R(\mft)$ where $\mcb{N}$ is the set of 
all possible node types as in \cite[Sec.~5.2]{BHZalg}.

For each $\mft \in \mfL_{+}$ we define a corresponding set of drivers $\mfD_{\mft}$ \label{drivers for t} via
\begin{equ}[e:defDt']
\mfD_\mft 
\eqdef \big\{ 
(\hat{\mfl}, \mfo) \in \hat{\mfD} \times (\Z^{d+1} \oplus \Z(\mfL)):\ 
\mfo \in D(\mft,\hat{\mfl})
\big\}\;.
\end{equ}
We also write $\mfD \eqdef \bigcup_{\mft \in \mfL_{+}} \mfD_{\mft}$\label{drivers page ref}.
For $\mfl = (\hat{\mfl},\mfo) \in \mfD$ we use as above the shorthand
\begin{equ}[e:drivers]
\Xi_{\mfl}
\eqdef
\Xi_{\hat{\mfl}} \cdot \bullet^{0,\mfo}
\in \mcT^{\ex}, \; \; \homplus{\mfl} \eqdef \homplus{\Xi_{\mfl}} \;.
\end{equ}
The set $ \mfD $ contains all the drivers needed in order to formulate one of the main results of this paper, Theorem~\ref{thm: algebraic main theorem}.
We add all possible extended decorations $ \mfo $ such that these drivers are stable under the action of the renormalisation group.

\begin{remark}
As we will see in Examples~\ref{ex:B-series} and~\ref{ex:gKPZ}, the set $\hat\mfD$ can often be identified simply with $\mfL_-\sqcup \{\mathbf{0}\}$ (see also Section~\ref{Notations for trees}).
This happens whenever the SPDE, before and after renormalisation, contains no terms involving products or derivatives of noises;
this is the case, in particular, in the setting of the Section~\ref{black box section} due to Assumption~\ref{assump:remove_noise_positive}.
It might therefore be surprising that we choose to accommodate renormalisation procedures which
\begin{itemize}
\item have a dependence on the extended decoration $\mfo$,
\item produce a counterterm which involves derivatives or products of noises.
\end{itemize}
For the first item, such renormalisation procedures are included in~\cite{BHZalg} and the statement of Theorem~\ref{thm: algebraic main theorem} below becomes more natural if these are included here.

We give an example of a subcritical SPDE where the second item occurs.
 Set $\mfL_{+} \eqdef \{\mft\}$ and $\mfL_{-} 
\eqdef \{\mfl_{1},\mfl_{2}\}$ with $|\mft|_{\s} = 10$, $|\mfl_{1}|_{\s} = -1$, and $|\mfl_{2}|_{\s} = -16$. 
Our equation is
\[
\partial_{t} u_{\mft} 
= 
\mathscr{L}_{\mft}u_{\mft}
+ u_{\mft}^{2} + u_{\mft}^{2} \zeta_{\mfl_{1}} + \zeta_{\mfl_{2}}\;.
\]
Then the renormalisation counterterm corresponding to 
\[
\tau 
\eqdef 
\mcb{I}_{(\mft,0)}
\Big[
\mcb{I}_{(\mft,0)}[\Xi_{\mfl_2}]
\Big]
\mcb{I}_{(\mft,0)}[\Xi_{\mfl_2}]
\]
includes terms involving derivatives of $\zeta_{\mfl_{1}}$ and products of $\zeta_{\mfl_{1}}$. 
One can calculate the corresponding $\Upsilon^{F}_{\mft}[\tau]$ for our choice of $F$ as described in Section~\ref{subsec:coherence}. Alternatively, one can perform the renormalisation contraction of $\tau$ inside of the tree 
\[
\mcb{I}_{(\mft,0)}\Big[\Xi_{\mfl_{1}}\mcb{I}_{(\mft,0)}[\Xi_{\mfl_2}]
\Big]
\mcb{I}_{(\mft,0)}[\Xi_{\mfl_2}]\Xi_{\mfl_{1}}\;.
\]
\end{remark}

We give two examples which illustrate the sets $\hat\mfD$ and $\mfD_\mft$, as well as all possible values of $|\mfo|_\s$ for $(\hat{\mfl},\mfo) \in \mfD$.

\begin{example}\label{ex:B-series}
In the case of B-series (see Remark~\ref{rem:B-series}), one has a single (constant) driver $\Xi_\mathbf{0}$ and no products or derivatives of noises, nor extended decorations.
In this setting, we simply have $\mfD = \{(\mathbf{0}, 0)\}$.
In the setting of rough differential equations (see \cite[Sec.~6]{BCFP}), one again has no derivatives or products of noises, but now $\mfL_- = \{\mfl_1,\ldots,\mfl_m\}$ and $\hat\mfD = \mfL_- \sqcup \{\mathbf{0}\}$, where we have used an abuse of notation to identify $\mfl \in \mfL_-$ with an element of $\N^{\widetilde\mfD}$.
We also have
\begin{equs}
\mfD_{\mft} = \big\{ (\mfl,0) :\ \mfl \in\mfL_-  \big\} \sqcup  \big\{ 
(\mathbf{0}, \mfo) :\ 
\mfo \in D(\mft,\mathbf{0})
\big\}\;.
\end{equs}
Setting the degree of the drivers to $|\mfl|_\s \eqdef \alpha \in (-1,0)$, and recalling that edges increase degree by one, we have
\begin{equs}
  \{ | \mfo  |_{\s} :\
   \mfo \in  D(\mft,\mathbf{0}) \} = \{ k\alpha+(k-1) \in (-1,0) : \ k \in \N \}	 \;.
  \end{equs}
  (Note, however, that in this special setting, extended decorations can be ignored since they do not affect the renormalisation procedure, see~\cite[Rem.~45]{BCFP}.)
\end{example}

\begin{example}\label{ex:gKPZ}
In the case of the generalised KPZ equation introduced in Section~\ref{sec: example of genKPZ part 0}, one has $ \hat \mfD = \mfL_{-}\sqcup \{\mathbf{0}\}$
because no derivatives and products of noises appear on the right hand of the equation, nor do they appear after renormalisation.
Then $ \mfD_{\mft} $ is again given by
\begin{equs}
\mfD_{\mft} = \big\{ (\mfl,0) :\ \mfl \in \mfL_-  \big\} \sqcup  \big\{ 
(\mathbf{0}, \mfo) :\ 
\mfo \in D(\mft,\mathbf{0})
\big\}\;.
\end{equs}
Moreover, we can give an explicit expression of the elements $ | \mfo  |_{\s}, \mfo \in  D(\mft,\mathbf{0})$ which are exactly the negative degrees which appear among the trees generated by the rules. One has 
\begin{equs}
  \big\{ | \mfo  |_{\s} :\
   \mfo \in  D(\mft,\mathbf{0}) \big\} = \left\lbrace  -3/2- \kappa, \, -1 - 2 \kappa, \, -1/2 -\kappa, \, -4 \kappa, -2 \kappa, 0  \right\rbrace\;. 
  \end{equs}
\end{example}
\subsection{Inner product spaces of trees} 
\label{sec: gen construction of inner products}
We introduce a very general prescription for building inner products of rooted decorated trees which is designed to have the advantage of automatically encoding symmetry factors of trees.\footnote{The situation focused on in this paper is the ``fully commutative'', in particular we see the degree $1$ polynomials $(\mathbf{X}^{e_{i}})_{i=0}^{d}$ as commuting and also do not distinguish different planar embeddings of trees (i.e., we assume $\mcb{I}_{o}(\tau)$ and $\mcb{I}_{o'}(\tau')$ also commute).
One key advantage of the formalism we adopt here is that it allows our work to be more easily translated to situations where some of this commutativity is lost, for instance~\cite{Mate2}, by a simple tweak of the construction given here.}

Saying that a set of rooted trees is \emph{decorated} means that there are different species of nodes and edges appearing in our trees. We thus assume that we are given a set $\nodes$ of possible node species and another set $\edges$ of possible edge species. We write $\decor = (\nodes,\edges)$ for the tuple of decorations.
We also assume we have been given an inner product $\langle \cdot, \cdot \rangle$ on the free vector space generated by $\nodes$.

We will generate a corresponding set of rooted decorated trees $\mcT(\decor)$ and an inner product space built from the free vector space generated by $\mcT(\decor)$ and extending $\langle \bullet \rangle$ which we will call $\mathfrak{T}(\decor)$.
We write $Y$ for an element of $\nodes$.  
We write $I$ for an element of $\edges$, and each such element will be thought of as operator on $\mcT(\decor)$.
In particular we view the full set $\mcT(\decor)$ as being generated from the set of nodes by taking products and applying the edge operators. 
We recall how our symbolic definition of a rooted decorated tree corresponds to the naive one. 
Given $n \ge 0$, $I_{1},\dots,I_{n} \in \edges$, a collection of previously defined rooted decorated trees $\tau_{1},\dots,\tau_{n}$, and $Y \in \nodes$, the rooted decorated tree
\begin{equ}\label{general tree formula}
\tau = Y \prod_{i=1}^{n} I_{i}(\tau_{i})
\end{equ}
is obtained as follows: 
\begin{itemize}
\item Start with the trees $\tau_{1},\dots,\tau_{n}$ and add a new node of type $Y$.
\item  For each $1 \le i \le n$ connect the new node to the root of $\tau_{i}$ with an $I_{i}$ edge.
\item Make the new node the root.
\end{itemize} 
We treat the product over $[n]$ appearing in \eqref{general tree formula} as commutative.

We will define $\mcT(\decor) \eqdef \bigsqcup_{k=0} \mcT_{k}(\decor)$ and now define the sets on the RHS. 
For $k = 0 $ we set $\mcT_{0}(\decor) \eqdef \nodes$.
Then for $k \ge 1$ and $l \ge 0$ we define $\mcT_{k}^{(l)}(\decor)$ inductively by setting $\mcT_{k}^{(0)}(\decor) = \emptyset$ and then setting, for $l \ge 1$,\ $\mcT_{k}^{(l)}(\decor)$ to be given by all elements $\tau$ of the form \eqref{general tree formula} where one takes $n=k$ and requires $\tau_{1},\dots,\tau_{k} \in \mcT^{(l-1)}(\decor)$ where we set 
\[
\mcT^{(l-1)}(\decor) \eqdef \mcT_{0}(\decor) \sqcup 
\Big( 
\bigsqcup_{k \ge 1} \mcT^{(l-1)}_{k}(\decor)
\Big)\;.
\]
Finally, we set $\mcT_{k}(\decor) \eqdef \bigcup_{l \ge 0} \mcT^{(l)}_{k}(\decor)$. 
Similarly, we define 
\[
\mathfrak{T}(\decor) \eqdef 
\bigoplus_{k \ge 0} 
\mathfrak{T}_{k}(\decor)
\]
where, for each the $k \ge 0$, $\mathfrak{T}_{k}(\decor)$ is an inner product space with its underlying vector space being the free vector space generated by $\mathcal{T}_{k}(\decor)$.

\begin{remark}\label{rem:trees}
The space $\mathring{\mathscr{T}}$ is of the form $\mathfrak{T}(\decor)$ with 
$\nodes = (\mfL_{-}\sqcup \{\mathbf{0}\}) \times \N^{d+1}$ and $\edges = \mcb{O}$. 
(In our notations, $\nodes$ is identified with $\{\mathbf{X}^{k} \Xi_{\mfl}\,:\, k \in \N^{d+1},\, \mfl \in \mfL_{-}\sqcup \{\mathbf{0}\}\}$.)
\end{remark}
 
For $k=0$ the inner product for $\mathfrak{T}_{0}(\decor)$ is given by the one given as input for our construction. 
For $k \ge 1$ we inductively set, for any $\tau, \bar{\tau} \in \mathcal{T}_{k}(\decor)$,
\[
\langle \tau, \bar{\tau} \rangle 
\eqdef
\langle Y, \bar{Y} \rangle
\sum_{s \in S_{k}}
\prod_{j=1}^{k}
\delta_{I_{j}, \bar{I}_{s(j)}}
\langle \tau_{j}, \bar{\tau}_{s(j)} \rangle
\]
where $S_{k}$ is the set of permutations on $[k]$ and we are using $Y$ (resp. $\bar Y$), $I_{j}$ (resp. $\bar{I}_{j}$), and $\tau_{j}$ (resp. $\bar{\tau}_{j}$) as those appearing in \eqref{general tree formula} for the expansion of $\tau$ (resp. $\bar{\tau}$).

One should remember that $\mcT(\decor)$ is an orthogonal but \emph{not orthonormal} basis for $\mfT(\decor)$.
We often write expansions of $\sigma \in \mfT(\decor)$ in the dual basis $( \langle \tau, \tau \rangle^{-1}\tau: \tau \in \mcT(\decor))$ as $\sigma = \sum_{\tau \in \mcT(\decor)} \tau \langle\sigma,\tau\rangle\langle\tau,\tau\rangle^{-1}$.

The construction above also enjoys some natural functorial properties. 

\begin{lemma}\label{lem: adjoint identity}
Suppose we are given two sets of decorations $\mathsf{D} = (\mathsf{N},\mathsf{E})$ and $\mathsf{D}' = (\mathsf{N}',\mathsf{E})$ along with inner products on $\scal{\mathsf{N}}$ and $\scal{\mathsf{N}}$.
Then for any linear operator $A: \langle \mathsf{N} \rangle \rightarrow \langle \mathsf{N}' \rangle$ we define a linear operator $\mathfrak{T}_{\mathsf{D},\mathsf{D}'}(A): \mathfrak{T}(\decor) \to \mathfrak{T}(\decor')$ as follows. 
For any $\tau \in \mcT(\mathsf{D})$, we inductively set
\[
\mathfrak{T}_{\mathsf{D},\mathsf{D}'}(A)\tau
=
(AY) \prod_{i=1}^{n} I_{i}( \mathfrak{T}_{\mathsf{D},\mathsf{D}'}(A)  \tau_{i})\;,
\]
where on the RHS we have used the expansion \eqref{general tree formula}.

Then if we denote by $A^{\ast}: \langle \mathsf{N}' \rangle \rightarrow \langle \mathsf{N} \rangle$ the adjoint of $A$ (defined with respect to the given inner products on $ \langle \mathsf{N} \rangle $ and $\langle \mathsf{N}' \rangle$) then $\mathfrak{T}_{\mathsf{D},\mathsf{D}'}(A)^{\ast}$, the adjoint of $\mathfrak{T}_{\mathsf{D},\mathsf{D}'}(A)$, is given by $\mathfrak{T}_{\mathsf{D}',\mathsf{D}}(A^{\ast})$. 
\end{lemma} 
\subsection{The trees of \texorpdfstring{$\VV$}{V}}
\label{Notations for trees}
We introduce a new notation for rooted decorated combinatorial trees, closely related to the 
notation of Section~\ref{black box section}.
Let $\VV$\label{VV page ref} be the set of all decorated trees of the form $ T_{\mff}^{\mfm} $ where $ \mfm = (\mnoise,\mpoly) : N_T \rightarrow  \mfD  \times \N^{d+1} $ and $ \mff :
 E_T \rightarrow \mcb{O}$ are arbitrary maps.

As in Remark~\ref{rem:trees}, we formulate this in the language of Section~\ref{sec: gen construction of inner products}.
We set $\nodes
\eqdef
\mfD \times \N^{d+1}$ and $\edges \eqdef \mcb{O}$, so that $\VV = \mcT(\decor)$ with $\decor = (\nodes,\edges)$. 
We also give an inner product on $\langle \nodes \rangle$ by setting (using the same notational identifications as above),
\[
\langle \Xi_{\mfl} \mathbf{X}^{k}, \Xi_{\bar{\mfl}}  \mathbf{X}^{\bar{k}} \rangle
=
\delta_{\mfl, \bar{\mfl}} \delta_{k,\bar{k}}k!\;.
\]
A tree $\tau \in \VV$ is of the form
\begin{equ}\label{eq:general tree of V}
\tau = 
\Xi_{\mfl} \mathbf{X}^{k} \prod_{i=1}^{n} \mcb{I}_{(\mft_{i},p_{i})}(\tau_{i})\;,
\end{equ}
with $\mfl \in \mfD$, $k \in \N^{d+1}$, $n \ge 0$, $\tau_{1},\dots,\tau_{n} \in \VV$, and $(\mft_{1},p_{1}),\dots,(\mft_{n},p_{n}) \in \mcb{O}$. 
We then set $\VVspan \eqdef \mathfrak{T}(\decor)$\label{VVspan page ref} and denote by $\langle \cdot, \cdot \rangle$ the induced inner product on $\VVspan$ as described in Section~\ref{sec: gen construction of inner products}.

It is not hard to see that this inner product just keeps track of the symmetry factors analogous to that defined in~\eqref{def: symmetry factor in section 2}. 
Extending the definition of $S(\cdot)$ in the natural way one has, for any $\tau,\bar{\tau} \in \VV$, $\langle \tau,\bar{\tau} \rangle = \delta_{\tau,\bar{\tau}} S(\tau)$.

Since $\mfL_- \sqcup \{\mathbf{0}\}$ can be identified with a subset of $\mfD$
by identifying $\mathbf{0}$ with $(\mathbf{0},0)$ and $\mfl$ with 
$(\hat \mfl, 0)$ where $\hat\mfl = \{(\mfl,0)\} \in \hat\mfD$,
 $ \mathring{\mathscr{T}}$ can (and will) be identified with the corresponding subset of $\VV$.

We also identify $\mcT^{\ex} $ with a subset of $ \VV $ as follows.
 To a decorated tree $ T^{\Labn,\Labo}_{\Labe} $ equipped with a type map $ \Labhom : E_T \rightarrow \mfL $, we associate the decorated tree $ \bar T^{\mfm}_{\mff} $
where $ \bar T $ is obtained from $ T $ by removing all the edges with type in $ \mfL_{-} $. (This is indeed again a tree
since normal rules forbid to attach any further edge to an edge with a label in $\mfL_-$.)
The decoration $ \mfm $ is given by  $ \mfm = (\mnoise,\mpoly) = ((\hat{\mfl},\Labo),\Labn) $ where for every $ x \in N_T $, $ \hat{\mfl}(x) $ is equal to $ \lbrace (\Labhom(e),\Labe(e)): \;e \in E_{x}^{-} \rbrace $ where $ E_x^{-} $ are the edges incident to $ x $ with type belonging to $ \mfL_- $. The edge decoration $ \mff $ is defined by $ \mff  = (\Labhom,\Labe)  $.
For the rest of the paper, we use the notation
$ T^{\mfm}_{\mff} $ and we revert to the notation $ T^{\Labn,\Labo}_{\Labe} $ only when we need to rely on 
some results from \cite{BHZalg}, e.g.\ for the proofs given in Appendix~\ref{subsec:coInteractProofs}.

\begin{remark}\label{rem:noise_node_vs_edge}
We choose to treat drivers as part of the node decoration, rather than as collections of edges as done in~\cite{BHZalg}, for two reasons.
First, it is more natural from the definition of the map $\Upsilon$ in which all edges are treated as differential operators, see~\eqref{eq: first Upsilon}.
Second, it yields a more natural form of trees for the pre-Lie structures appearing in Section~\ref{sec:proof_of_thm}. 
In particular, the set $\{\mathbf{X}^k\Xi_\mfl \, : \, k \in \N^{d+1},\;\mfl\in\mfD\}$ forms a generating set 
for $\VVspan$ equipped with a family of grafting operators, which makes this set a natural choice for node decorations.
\end{remark}

\begin{remark}
Formally, the difference between $\mcT^\ex$ and $\VV$ is that $\VV$ does not enforce the restrictions that trees should conform to the rule $R$ and that extended decorations should be compatible with edge types as dictated in~\cite[Def.~5.24]{BHZalg}: the only role played by $R$ in the definition of $\VV$ is through the definition of the label set $\mfD$.
\end{remark}
We will often use the symbolic notation~\eqref{eq:general tree of V} as in~\cite[Sec.~8]{Regularity}
and~\cite[Sec.~4.3]{BHZalg}. 
In particular, the drivers $ \Xi_{\mfl} $, $ \mfl \in \mfD $  are given by \eqref{e:drivers}.
For $o \in \mcb{O}$, we also define an operator $ \mcb{I}_{o} : \VV \rightarrow \VV $ 
as suggested by~\eqref{eq:general tree of V}: given
$ T^{\mfm}_{\mff} \in \VV $, $ \mcb{I}_{o}(T^{\mfm}_{\mff}) $ is the decorated tree obtained by adding a new root with node decoration equal to zero and joining this new root to the root of $ T $ with an edge 
decorated by $o$.
\begin{remark}
Note that as in \cite{BHZalg} we do allow symbols of the form $\mcb{I}_{(\mft,p)}[X^{k}]$. 
\end{remark}

%

\subsection{A class of allowable equations}\label{subsubsec:AllowableEqs}
Recall that in the theory of regularity structures one lifts a concrete fixed point problem to an abstract fixed point problem in a space of modelled distributions. 

We define $\mathring{\G}$\label{mathring G page ref} to consist of all tuples $(F_{\mft}^{\mfl})_{\mft,\mfl}$ where $\mft$ ranges over $\mfL_{+}$, $\mfl$ ranges over $\mfD_{\mft}$ and for each such $\mft$ and $\mfl$ one has $F^{\mfl}_{\mft} \in \Poly$. 
There is a restriction on the equations we can work with in that they must be compatible with the rule $R$ used to construct our regularity structure\dash we now describe a subset $\G \subset \mathring{\G}$ which enforces this constraint. 
First, define $\mcb{N}_{+} \subset \mcb{N}$\label{pos node-types} to be collection of all node-types whose elements are all members of $\mfL_{+} \times \N^{d+1}$.
We then define a map $\hat{\mcb{N}}:\Poly \rightarrow \mcb{P}(\mcb{N}_{+})$ by setting, for $F$ given by \eqref{expansion of nonlinearity}, 
\[
\hat{\mcb{N}}(F)
\eqdef
\bigcup_{
1 \le j \le m
}
\{ \alpha \sqcup \beta\,:\, \alpha \le \alpha_j,\, \beta \in \hat {\mcb{P}}(\mcb{O}(F_{j}))\}\;.
\]
\begin{definition}\label{def:obey} We say that $F \in \mathring{\G}$ \emph{obeys} our fixed rule $R$ if for every $\mft \in \mfL_{+}$, $(\hat{\mfl}, \mfo) \in \mfD$, and $ N \in \hat{\mcb{N}}(F_{\mft}^{(\hat{\mfl},\mfo)})$, one has $\mfo \in D( \mft , N \sqcup \hat{\mfl})$.
We denote by $\G$\label{G page ref} the set of all $F \in \mathring{\G}$ which obey $R$.
\end{definition}

\begin{remark}At first glance the definition above may seem to just enforce conditions on the labels $\mfo$, but recall that $D(\mft,N) = \emptyset$ if $N \not \in R(\mft)$, so that it implies in particular that
$\hat{\mcb{N}}(F_{\mft}^{\mfl}) \subset R(\mft)$ for every $\mfl \in \mfD$ and $\mft \in \mfL_{+}$.
\end{remark} 
Up to now, we have not required any additional properties on our rule $R$ beyond completeness and subcriticality with respect to $\reg$. 
We now introduce an additional non-degeneracy assumption.
\begin{assumption}\label{assump:RregComplete}
For every $\mft \in \mfL_+$, $N \in R(\mft)$, and $o \in \mcb{O}_{+}$, 
one has $N \sqcup \{o\}\in R(\mft)$.
\end{assumption}
Note that any subcritical rule $R$ can be trivially extended to satisfy Assumption~\ref{assump:RregComplete} while remaining subcritical with respect $\reg$, so this is really just a condition
guaranteeing that we are considering a sufficiently large class of SPDEs.
We give a simple equivalent definition of $\G$ under Assumption~\ref{assump:RregComplete}.
\begin{proposition}\label{prop:obeyEquivalence}
Let $F \in \mathring{\G}$. Consider the following statements.
\begin{enumerate}[label=\upshape(\roman*\upshape)]
\item \label{point:obey1} $F \in \G$.
\item \label{point:obey2} For all $\mfl = (\hat\mfl,\mfo) \in \mfD$, $\mft \in \mfL_+$, and $\alpha \in \N^{\mcb{O}}$ such that $\mfo \notin D(\mft,\alpha \sqcup \hat{\mfl})$, one has $D^\alpha F^\mfl_\mft \equiv 0$.
\end{enumerate}
Then~\ref{point:obey1}~$\Rightarrow$~\ref{point:obey2}. If Assumption~\ref{assump:RregComplete} holds, then~\ref{point:obey1}~$\Leftrightarrow$~\ref{point:obey2}.
\end{proposition}
\begin{proof}
\ref{point:obey1}~$\Rightarrow$~\ref{point:obey2}: Let $F \in \mathring{\G}$ and let $\mfl, \mft, \alpha$ be as in point~\ref{point:obey2}.
Then necessarily $\alpha \notin \hat{\mcb{N}}(F^\mfl_\mft)$, from which it readily follows that $D^\alpha F^\mfl_\mft \equiv 0$, which proves~\ref{point:obey2}.

For $N \in \N^{\mcb{O}}$, define $N^- \eqdef N \mathbbm{1}_{\reg(\mft,p) < 0} \in \N^{\mcb{O}}$ (i.e., considering $N$ as an element of $\mcb{N}$, $N^-$ is obtained by removing all edge types $(\mft,p)$ for which $\reg(\mft,p) \geq 0$).
Observe that under Assumption~\ref{assump:RregComplete}, it follows readily from the definition of $D(\mft,N)$ that $D(\mft, N^-\sqcup\hat{\mfl}) \subset D(\mft,N\sqcup\hat{\mfl})$.

Suppose now that Assumption~\ref{assump:RregComplete} holds. Let $\mft \in \mfL_+$, $\mfl = (\hat\mfl,\mfo) \in \mfD$, and $N \in \hat{\mcb{N}}(F^\mfl_\mft)$. 
To prove~\ref{point:obey1}, it suffices to show that $\mfo \in D(\mft,N^-\sqcup\hat{\mfl})$. Observe that the expansion~\eqref{expansion of nonlinearity} and definition of $\hat{\mcb{N}}(F^\mfl_\mft)$ implies that $D^{N^-} F^\mfl_\mft$ is not identically zero.
If~\ref{point:obey2} holds, then $\mfo \in D(\mft,N^-\sqcup\hat{\mfl})$, and therefore~\ref{point:obey1} holds.
\end{proof}
\subsection{Truncations}\label{subsec:truncations}
In practice, one works with a truncated version of the space $\cT^\ex$.
We describe here the truncated spaces and projections used to define our fixed point map.

\begin{definition}\label{def of the gamma trunc}
For $\gamma \in \R$, let $\mcT^\ex_{\leq \gamma} \eqdef \{\tau \in \mcT^\ex : |\tau|_+ \leq \gamma\}$. \label{mcT gamma page ref}
Let $\cT^\ex_{\leq \gamma} \leq \cT^\ex$ be the subspace spanned by $\mcT^\ex_{\leq \gamma}$,\label{cT gamma page ref}
and define the projection $\mcQ_{\leq \gamma} : \cT^\ex \to \cT^\ex_{\leq \gamma}$\label{mcQ page ref} which acts as the identity on $\tau \in \mcT^\ex$ if $\tau \in \mcT^\ex_{\leq \gamma}$, and maps $\tau$ to zero otherwise.
We define $\mcT^\ex_{<\gamma}$, $\cT^\ex_{<\gamma}$, and $\mcQ_{<\gamma} : \cT^\ex \to \cT^\ex_{< \gamma}$ similarly.
\end{definition}

We further introduce a truncation map with the important property that it is additive with respect to tree multiplication (which does not hold for the $|\cdot|_+$-degree).
\begin{definition}
For $\tau = T^\mfm_\mff \in \mcT^\ex$ we call $\trunc(\tau) \eqdef |E_T| + |\mpoly|$ the truncation parameter of $\tau$.
For $L \in \N$, define
\[
\mcW_{\leq L} \eqdef \{\tau \in \mcT^\ex : \trunc(\tau) \leq L\}.
\]
Let $\mcb{W}_{\leq L} \subset \cT^\ex$ be the subspace spanned by $\mcW_{\leq L}$, and define the projection $\proj_{\leq L} : \cT^\ex \to \mcb{W}_{\leq L}$ which acts as the identity on $\tau \in \mcT^\ex$ if $\tau \in \mcW_{\leq L}$, and maps $\tau$ to zero otherwise.
We also set
\[
\gamma_L \eqdef \max\{|\tau|_+ :\ \tau \in \mcW_{\leq L}\}\;,
\]
and, for $\alpha \in \R$, set
\begin{equ}
L_\alpha \eqdef \max\{\trunc(\tau) : \tau \in \mcT^\ex, |\tau|_+ \leq \alpha\}.
\end{equ}
\end{definition}
Note that $\mcW_{\leq L}$ is a finite set for any $L \in \N$ and thus $0 \leq \gamma_L < \infty$.
Note also that $L_\alpha < \infty$ since, by subcriticality of the rule $R$, there are only finitely many $\tau \in \mcT^\ex$ for which $|\tau|_+ \leq \alpha$.
It holds that $L_\alpha$ is the smallest natural number for which $\tau \in \mcW_{\leq L_\alpha}$ for all $\tau \in \mcT^\ex$ such that $|\tau|_+ \leq \alpha$. 

Note that $\mcb{W}_{\leq L}$ is closed under the action of $\mfR$ but not in general under the action of the structure group of $\mathscr{T}$.
Finally, we will require the following definition when dealing with renormalised equations.
\begin{definition}\label{def: definition of bar L}
For $L \in \N$, let
\begin{equ}
\bar L \eqdef \max \{L(\tau) : \tau \in \mcT^\ex, \; |\tau|_+ \le \gamma_L \}\;.
\end{equ}
\end{definition}
Note that $L \leq \bar L < \infty$ and that for all $M \in \mfR$
\begin{equ}
M, M^* : \cT^\ex_{\leq \gamma_L} \to \mcb{W}_{\leq \bar L}\;,\label{eq:M maps L to bar L}
\end{equ}
which follows from the fact that $M$ and $M^*$ preserve the $|\cdot|_+$-degree.
\begin{remark}
In the remainder of the paper we will often continue working with the untruncated regularity structure $\mathscr{T}$ and then insert the needed projections into various expressions; these are easily converted to statements on an appropriately truncated regularity structure.
We also remark that any statements we give involving continuity with respect to or convergence of models assume that one has truncated the regularity structure at some level. 
\end{remark}
\subsection{Nonlinearities on trees}\label{subsec:nonlinearities}
Let $\bar\cT^\ex \eqdef \Span{\{\mathbf{X}^{p}:
p \in \N^{d+1}
\}}$\label{poly reg page ref} denote the sector\footnote{A sector is a subspace of $\cT^\ex$ which is stable under the structure group and respects the decomposition into homogeneous subspaces, see~\cite[Def.~2.5]{Regularity}.} of abstract Taylor polynomials in $\cT^\ex$.
For every $\mft \in \mfL_{+}$ we define $\mcT_{\mft}^{\ex}, \widetilde{\mcT}_{\mft}^{\ex} \subset \mcT^{\ex}$
and $\mcb{T}_{\mft}^{\ex}, \widetilde{\mcb{T}}_{\mft}^{\ex} \subset \mcb{T}^{\ex}$\label{tilde mcbT_mft page ref}
\label{mcbT_mft page ref}
via
\begin{equs}[2]
\mcT_{\mft}^{\ex} &\eqdef \{ \tau \in \mcT^{\ex}: \tau = \mcb{I}_{(\mft,0)}[\bar{\tau}]\ \textnormal{for some}\ \bar{\tau} \in \mcT^{\ex}\}\;,&\quad
\mcb{T}_{\mft}^{\ex}
&\eqdef
\bar\cT^\ex \oplus
\Span{ \mcT_{\mft}^{\ex} }\;,
\\
\widetilde{\mcT}_{\mft}^{\ex} &\eqdef \{ \bar{\tau} \in \mcT^{\ex}: \mcb{I}_{(\mft,0)}[\bar{\tau}] \in \mcT^{\ex} \}\;,&
\widetilde{\mcb{T}}_{\mft}^{\ex}
&\eqdef
\Span{\widetilde{\mcT}_{\mft}^{\ex}}\;.
\end{equs}
We note that one also has $\bar\cT^\ex \subset \widetilde{\mcb{T}}_{\mft}^{\ex}$.
The space $\mcb{T}_{\mft}^{\ex}$ contains all ``jets'' used to describe the left hand side of the $\mft$-component of
our equation \eqref{e:SPDE}, while $\widetilde{\mcb{T}}_{\mft}^{\ex}$ contains those used to describe its right hand side.
Thanks to our assumptions on the underlying rule $R$, one has the following lemma, cf.~\cite[Eq.~5.11]{BHZalg}.
Note that we always refer to the $|\cdot|_+$-degree when speaking about the regularity of a sector.
\begin{lemma}\label{lem:sectors of our regularity structure}
For each $\mft \in \mfL_{+}$, $\mcb{T}_{\mft}^{\ex}$, and $\widetilde{\mcb{T}}_{\mft}^{\ex}$ are sectors of $\mathscr{T}$ of respective regularities $\reg(\mft) \wedge 0 $ and $(\reg(\mft)-|\mft|_s) \wedge 0$.\qed
\end{lemma}
We also define $\mcb{H}^{\ex} \eqdef \bigoplus_{\mft \in \mfL_{+}} \mcb{T}_{\mft}^{\ex}$\label{Hex page ref}
and $\widetilde{\mcb{H}}^{\ex} \eqdef \bigoplus_{\mft \in \mfL_{+}} \widetilde{\mcb{T}}_{\mft}^{\ex}$.\label{tilde Hex page ref}

For $p \in \N^{d+1}$ we write $\DD^p$\label{gradient page ref} for the abstract differential operator $\DD^p : \mcb{H}^\ex \to \cT^\ex$ given by $\mcb{I}_{(\mft,0)}[\tau] \mapsto \mcb{I}_{(\mft,p)}[\tau]$, as well as $\mathbf{X}^q \mapsto \frac{q!}{(q-p)!}\mathbf{X}^{q-p}$ if $q \geq p$ and $\mathbf{X}^q \mapsto 0$ otherwise.
Given $U = (U_{\mft})_{\mft \in \mfL_{+}} \in \mcb{H}^{\ex}$ we define $\boldsymbol{U} = (U_{(\mft,p)})_{(\mft,p) \in \mcb{O}} \in \bigoplus_{(\mft,p) \in \mcb{O}} \DD^p \mcb{T}_\mft^{\ex}$ by setting $U_{(\mft,p)}\eqdef \DD^{p} U_{\mft}$.

Writing $\mcb{W} \eqdef \bigoplus_{\mft \in \mfL_{+}} \mcb{T}^{\ex}$, we immediately obtain the following lemma from the implication~\ref{point:obey1}~$\Rightarrow$~\ref{point:obey2} of Proposition~\ref{prop:obeyEquivalence}.
\begin{lemma}\label{lem: nonlinearity on regularity structure} 
Let $F \in \G$.
Write $F\Xi \colon \mcb{W} \to \mcb{W}$ for the map given by 
$U \mapsto \big(\sum_{\mfl \in \mfD_{\mfb}}\bigl(\mathbf{F}^{\mfl}_{\mfb}(U)
\Xi_\mfl\bigr) : \mfb \in \mfL_{+}\big)$ where
\begin{equ}\label{def: lift of smooth functions} 
\mathbf{F}^{\mfl}_{\mfb}(U)
\eqdef
\sum_{
\alpha \in \N^{\mcb{O}}
}
\frac{D^{\alpha}F^{\mfl}_{\mfb}\big(\langle \boldsymbol{U}, \bone \rangle \big)}{\alpha!}
(\boldsymbol{U} - \langle \boldsymbol{U}, \bone \rangle \bone)^{\alpha} \;,
\end{equ}
and, writing $\boldsymbol{U} = (U_{o})_{o \in \mcb{O}}$, we set $\langle \boldsymbol{U}, \mathbf{1} \rangle \eqdef \big(\langle U_{o}, \bone \rangle \big)_{o \in \mcb{O}} \in \R^{\mcb{O}}$. 
Then, for any $\gamma \in \R$, $\mcQ_{ \le \gamma}F\Xi$ maps $\mcb{H}^{\ex}$ to $\widetilde{\mcb{H}}^{\ex}$.
\end{lemma}
\begin{proof}
This follows from the fact that for $\mfb \in \mfL_{+}$, $\mfl \in \mfD_{\mfb}$, and any $\alpha \in \N^{\mcb{O}}$ with $D^{\alpha}F^{\mfl}_{\mfb} \not = 0$, Proposition~\ref{prop:obeyEquivalence} guarantees that $(\boldsymbol{U} - \langle \boldsymbol{U}, \mathbf{1} \rangle \mathbf{1})^{\alpha} \Xi_{\mfl} \in \widetilde{\mcb{T}}^{\ex}_{\mfb}$.
\end{proof}
\subsection{Coherence}\label{subsec:coherence}
For each $F \in \mathring{\G}$ we define $\Upsilon^{F}:\VV \rightarrow \Poly^{\mfL_{+}}$, $\Upsilon^F : \tau \mapsto \Upsilon^{F}[\tau] = (\Upsilon^{F}_{\mft}[\tau])_{\mft \in \mfL_{+}}$ 
by \eqref{eq: first Upsilon}, the only difference being that we now allow
to have $\mfl \in \mfD$. Recall also that we identify $\mcT^\ex$ with a subset of $\VV$.

For $U \in \mcb{H}^{\ex}$ we define $U^{R} \in \widetilde{\mcb{H}}^{\ex}$ by setting, for each $\mft \in \mfL_{+}$, 
\[
U^{R}_{\mft} = \sum_{\tau \in \widetilde{\mcT}_{\mft}^{\ex}} \frac{\langle U, \mcb{I}_{(\mft,0)}[\tau] \rangle}{\langle \tau, \tau \rangle} \tau\;.
\]
Additionally, we define a tuple $\mathbf{u}^{U} = (u_{\alpha}^{U})_{\alpha \in \mcb{O}}$ where $u_{\alpha}^{U} \in \R$ is given by setting $u_{(\mft,p)}^{U} \eqdef \langle \mathbf{X}^{p},U_{\mft}\rangle = \langle \mathbf{1},U_{(\mft,p)} \rangle$.
In this way, every $U \in \mcb{H}^{\ex}$ can be written uniquely as
\begin{equation}\label{eq:generalU}
U_\mft = \sum_{p \in \N^{d+1}} \frac{1}{p!}u^U_{(\mft,p)}\mathbf{X}^p +  \mcb{I}_{(\mft,0)}[U_\mft^R]\;.
\end{equation}

\begin{definition}\label{def: coherent}
We say that $U \in \mcb{H}^{\ex}$ is \textit{coherent} to order $L \in \N$ with $F \in \G$ if, for every $\mft \in \mfL_{+}$ and 
every $\tau$ such that $\mcb{I}_{(\mft,0)}[\tau] \in \mcW_{\leq L+1}$, one has
\[
\langle U_\mft, \mcb{I}_{(\mft,0)}[\tau] \rangle = \Upsilon^F_{\mft}[\tau](\mathbf{u}^U).
\]
\end{definition}
We note the following equivalence.
\begin{lemma}\label{lemma: coherence identity}
Fix $F \in \G$ and $L \in \N$. Consider $U \in \mcb{H}^{\ex}$ of the form~\eqref{eq:generalU}.
Then $U$ is coherent to order $L$ with $F$ if and only if, for every $\mft \in \mfL_{+}$,
\[
\proj_{\leq L} \sum_{\mfl \in \mfD_{\mft}}
F^{\mfl}_{\mft}(U)\Xi_\mfl
=
\proj_{\leq L} U_\mft^R\;.
\]
\end{lemma} 
\begin{proof}
This follows from Lemma~\ref{lemma: coherence identity - big space} below, combined with 
the additivity of tree multiplication with respect to the truncation parameter $\trunc(\tau)$.
\end{proof}
\subsection{The main theorem}\label{subsec: main thm}
We define an action of $\mfR$ on $\mathring{\G}$ written, for $M \in \mfR$, as $F \mapsto MF$ where $MF \in \mathring{\G}$ is defined by setting, for each $\mft \in \mfL_{+}$ and $\mfl \in \mfD_{\mft}$, 
\begin{equs} \label{def:MP}
(M F)_{\mft}^\mfl \eqdef \Upsilon^{F}_{\mft}[M^*\Xi_{\mfl}] =  
\sum_{
\tau \in \mcT^{\ex}}
\frac{
\langle M \tau, \Xi_{\mfl} \rangle}{\langle \tau, \tau \rangle}
\Upsilon_{\mft}^{F}[\tau]
\in \Poly\;.
\end{equs}
We defer the proof of the following lemma to Appendix~\ref{subsec:GPreserveProof}.
\begin{lemma}\label{lem:GPreserve}
Let $F \in \G$ and $M \in \mfR$. Under Assumption~\ref{assump:RregComplete}, it holds that $MF \in \G$.
\end{lemma}
The following lemma, whose proof is deferred to the very end of Section~\ref{subsec:interactionWithRenormG}, implies that the map
$F \mapsto MF$ defines a (left) group action\footnote{One also needs injectivity of the map $F \mapsto (\Upsilon^{F}_{\mft})_{\mft \in \mfL_{+}}$ but this immediately follows from the fact that $F_{\mft}^{\mfl} = \Upsilon^{F}_{\mft}[\Xi_{\mfl}]$ for any $\mft \in \mfL_{+}$, $\mfl \in \mfD$.} of $\mfR$ on $\G$, which is not obvious from 
\eqref{def:MP}.
\begin{lemma}\label{lem: renormalisation of upsilon}
Suppose Assumption~\ref{assump:RregComplete} holds.
Then for all $F \in \G$, $M \in \mfR$, and $\tau \in \mcT^\ex$, it holds that
$
\Upsilon^{F} [M^* \tau] = \Upsilon^{ MF}[\tau]
$.
\end{lemma}
For the rest of this section, we suppose that Assumption~\ref{assump:RregComplete} is in place.
%
\begin{proposition}\label{main prop}
Fix $F \in \G$, $L\in\N$, and $M \in \mfR$.
Suppose $U \in \mcb{H}^{\ex}$ is coherent to order $\bar L$ with $F$, with $\bar L$ as in Definition~\ref{def: definition of bar L}. Then $MU$ is coherent to order $L$ with $MF$. 
\end{proposition}
\begin{proof}
One has, for every $\mft \in \mfL_{+}$ and $\mcb{I}_{(\mft,0)}[\tau] \in \mcW_{\leq L + 1}$
\begin{equs}
\langle MU_{\mft}, \mcb{I}_{(\mft,0)}[\tau] \rangle
&=
\langle U_{\mft}, M^* \mcb{I}_{(\mft,0)}[\tau] \rangle
=
\langle U_{\mft}, \mcb{I}_{(\mft,0)}[M^*\tau] \rangle
\\
&=
\Upsilon^F_\mft[M^* \tau](\mathbf{u}^U)
=
\Upsilon^{MF}_\mft[\tau](\mathbf{u}^{MU})\;.
\end{equs}
Here the third equality uses the definition of coherence and that $M^*\tau \in \mcb{W}_{\leq \bar L}$ (which follows from~\eqref{eq:M maps L to bar L}). 
The fourth equality uses Lemma~\ref{lem: renormalisation of upsilon} and that $\mathbf{u}^U = \mathbf{u}^{MU}$.
\end{proof}

Our main algebraic result can be stated as follows.
\begin{theorem}\label{thm: algebraic main theorem}
Let $F \in \G$, $L \in \N$, and $U \in \mcb{H}^{\ex}$ written as~\eqref{eq:generalU}.
Suppose that $U$ satisfies, for every $\mft \in \mfL_{+}$,
\begin{equ}
\proj_{\leq\bar L} \sum_{\mfl \in \mfD_{\mft}}
F^{\mfl}_{\mft}(U)\Xi_\mfl
=
\proj_{\leq\bar L} U_\mft^R\;,
\end{equ}
where $\bar L$ defined as in Definition~\ref{def: definition of bar L}.
Then $U$ is coherent to order $\bar L$ with $F$, and, for all $M \in \mathfrak{R}$ and $\mft \in \mfL_+$,
\begin{equ}\label{eq:renormalised fixed point}
\proj_{\leq L} M U^R_\mft
=
\proj_{\leq L} \sum_{\mfl \in \mfD_{\mft}}
(MF)^{\mfl}_{\mft}(MU)\Xi_\mfl\;.
\end{equ}
\end{theorem}
\section{Proof of Theorem~\ref{thm: algebraic main theorem}}\label{sec:proof_of_thm}
With Lemma~\ref{lemma: coherence identity} and Proposition~\ref{main prop} at hand, we are ready to give a proof of Theorem~\ref{thm: algebraic main theorem}.
\begin{proof}[of Theorem~\ref{thm: algebraic main theorem}]
Coherence of $U$ to order $\bar L$ with $F$ follows from Lemma~\ref{lemma: coherence identity}.
It follows from Proposition~\ref{main prop} that $MU$ is coherent to order $L$ with $MF$, from which we obtain~\eqref{eq:renormalised fixed point} again by Lemma~\ref{lemma: coherence identity}.
\end{proof}
Lemma~\ref{lemma: coherence identity} and Proposition~\ref{main prop} in turn rely on Lemmas~\ref{lemma: coherence identity - big space} and~\ref{lem: renormalisation of upsilon} respectively.
In the rest of this section, we set up the combinatorial\slash algebraic framework which allows us to prove the latter two lemmas. 
The general strategy of the proof is as follows.
In Section~\ref{sec:SpacesOfTrees} we introduce a space $\BBspan$ consisting of trees with additional information on its polynomial decorations; see Remark~\ref{remark:intuitionBB} for further intuition.
We furthermore introduce a map $Q: \BBspan \to \VVspan$ which discards this additional information, and a map $\mathring\Upsilon^F : \BBspan \to \Poly^{\mfL_{+}}$ which ``lives above'' $\Upsilon^F:\VVspan \to \Poly^{\mfL_{+}}$ in the sense that $\Upsilon^F[\tau] = \mathring\Upsilon^F[Q^*\tau]$.

In Section~\ref{sec: grafting and renormalisation}, we introduce several grafting operators on our spaces of trees and show they possess several important properties.
First, the maps $\mathring\Upsilon^F : \BBspan \to \Poly^{\mfL_+}$, $Q^* : \VVspan \to \BBspan^*$, and $\Upsilon^F : \VVspan \to \Poly^{\mfL_+}$, all become pre-Lie morphisms with respect to these operators (Lemmas~\ref{lem:graftMorphism} and~\ref{lem:QMorphism}, and Corollary~\ref{cor:graftMorphismSmall}).
Second, they are in suitable ``co-interaction'' with the renormalisation group $ \mathfrak{R} $ (Proposition~\ref{prop:co-interaction}).
Third, they allow us to decompose the construction of trees into elementary grafting operations starting from a simple set of generators (Proposition~\ref{prop:generate}).
Together, these facts lead to the proof of Lemma~\ref{lem: renormalisation of upsilon}.

The reason for introducing the space $\BBspan$ is that it greatly facilitates the combinatorics required in the proof of Lemma~\ref{lemma: coherence identity - big space}, as well as the pre-Lie morphism properties in Section~\ref{sec: grafting and renormalisation}.
 
\begin{remark}\label{rem: intro to grafting and pre-lie structures}
The key point in obtaining the renormalised equation relies on the pre-Lie structure associated to the coefficients $ \Upsilon^{F} $ (see \cite{MR2839054} for a survey on pre-Lie algebras).
The presence of edge and polynomial decorations makes this pre-Lie structure somewhat complicated.
In the case of B-series, however, it is quite easy to describe.
Indeed, we just have one grafting operator $  \graft{} $ defined for two trees $ \tau_1 $ and $ \tau_2 $ by
\begin{equ}
\tau_2 \graft{} \tau_1 = \sum_{v \in N_{\tau_2}} \tau_2 \graft{v} \tau_1\;,
\end{equ}
where $ \graft{v} $ means that we attach the tree $ \tau_2 $ to the tree  $ \tau_1 $ by adding a new edge between the root of $ \tau_2  $ and the node $ v $. 
For instance we get in the next example:
\begin{equ}
 \bullet \graft{} \<IXi2> = 2 \<IXi4>  +  \<IXi3>\;\;.
\end{equ}
This grafting operator satisfies the pre-Lie identity for all trees $ \tau_1, \tau_2 $ and $ \tau_3 $
\begin{equ} \label{e:pre_Lie identity}
( \tau_1 \graft{} \tau_2 ) \graft{} \tau_3 - \tau_1 \graft{} (\tau_2 \graft{} \tau_3 ) = (\tau_2 \graft{} \tau_1) \graft{} \tau_3 - \tau_2 \graft{} (\tau_1 \graft{} \tau_3)\;.
\end{equ}
This identity is called pre-Lie because one can derive a Lie bracket from the grafting operator: $ [\tau_1,\tau_2] = \tau_1 \graft{} \tau_2 - \tau_2 \graft{} \tau_1 $.
A vector space endowed with a product $ \graft{} $ satisfying~\eqref{e:pre_Lie identity} is called a pre-Lie algebra.
The identity~\eqref{e:pre_Lie identity} has to be compared with~\eqref{multi pre-Lie} below, where there are several grafting operators coming from the decorations on the edges.
The map $ \Upsilon^{F} $ in the case of B-series is a morphism of pre-Lie algebras between the space of rooted trees and the pre-Lie algebra given by vector fields on $ \mathbf{R}^d $.
Indeed, one has 
\begin{equs}
\Upsilon^{F} [ \bullet \graft{} \<IXi2> ] = 
\Upsilon^{F} [ \bullet] \triangleleft \Upsilon^F[ \<IXi2> ]\;, \quad (F  \triangleleft G) = F \cdot D G\;.
\end{equs}
In the case of SPDEs, the morphism property is given by Corollary~\ref{cor:graftMorphismSmall}.
In~\cite{Chapoton01}, a universal result has been proved, namely that the pre-Lie algebra of rooted trees is freely generated by $ \bullet $ and the grafting operator $ \graft{} $.
Such a result is true for the branched rough path case when the set of generators is replaced by $ \bullet_i $.
In this context, one has also the same morphism property, which is useful for understanding how a map acting on the trees acts on the coefficients associated to them, and thus how to obtain the renormalised equation for rough differential equations in~\cite[Thm~38]{BCFP}.
Indeed, the renormalisation map $ M $ in~\cite{BCFP} is \emph{defined} as the unique pre-Lie morphism translating the generators, which is then verified to agree with the renormalisation maps introduced~\cite{BHZalg}.
In order to repeat the same strategy here, a universal result is given in Proposition~\ref{prop:generate} and the morphism property for the renormalised map $ M $ is proved in Proposition~\ref{prop:co-interaction}.
\end{remark}
\subsection{The space \texorpdfstring{$\BBspan$}{B}}\label{sec:SpacesOfTrees}
We apply the procedure in Section~\ref{sec: gen construction of inner products} to define another space of trees with a specified inner product.
We keep $\mathsf{E}$ as in Section~\ref{Notations for trees} but introduce a new set of node decorations
\[
\nodes'
\eqdef
\Big\{ \Xi_{\mfl} \prod_{i=1}^n \mcI_{(\mft_i,p_i)}[X^{k_i}] :\ \mfl \in \mfD, n \geq 0, (\mft_i,p_i) \in \mcb{O}, k_i \in \N^{d+1}, p_i < k_i \Big\}\;.
\]
where the product over $[n]$ appearing above is treated as commutative. 
The inner product on $\langle \nodes' \rangle$ is defined by
\begin{equation}\label{eq:innerProdBB}
\Big\langle \Xi_{\mfl} \prod_{i=1}^n \mcI_{o_i}[X^{k_i}], \Xi_{\bar\mfl} \prod_{i=1}^{\bar n} \mcI_{\bar o_i}[X^{\bar k_i}] \Big\rangle
\eqdef \delta_{\mfl,\bar\mfl}\delta_{n,\bar n} \sum_{s \in S_n} \prod_{i=1}^n\delta_{o_i,\bar o_{s(i)}} \delta_{k_i,\bar k_{s(i)}} (k_i-p_i)!\;.
\end{equation}
Setting $\decor' \eqdef(\nodes',\edges)$, we then define the set $\BB \eqdef \mcT(\decor')$\label{BB page ref}
and the inner product space $\BBspan \eqdef \mfT(\decor')$ \label{BBspan page ref} as prescribed in Section~\ref{sec: gen construction of inner products}.
In particular, it holds that every tree $\sigma \in \BB$ can be uniquely written as
\begin{equ}\label{eq:generic tree in BB}
\sigma = Y \prod_{j \in J} \mcb{I}_{o_j}[\sigma_j] = 
\Xi_\mfl 
\Big( 
\prod_{i\in I} \mcI_{o_i}[X^{k_i}]\Big)
\Big(
\prod_{j \in J} \mcb{I}_{o_j}[\sigma_j]
\Big)
\end{equ}
where $I$ and $J$ are finite index sets, $\mfl \in \mfD$, $o_i,o_j \in \mcb{O}$, $\sigma_j \in \BB$, and $o_i = (\mft_i,p_i)$ with $p_i<k_i$ (and conversely, any such expression corresponds to a unique tree in $\BB$).
Equivalently, $\BB$ is the set of all decorated trees of the form $ T_{\mff}^{\mfm} $ where $ \mfm : N_T \rightarrow  \nodes$ and $ \mff :
 E_T \rightarrow \edges$ are arbitrary maps.

For every $F \in \mathring{\G}$ we define a linear map $\mathring\Upsilon^F[\cdot]:\BBspan \rightarrow \Poly^{\mfL_{+}}$\label{mathringupsilon page ref} as follows. 
For $\sigma \in \BB$ of the form~\eqref{eq:generic tree in BB} and $\mft \in \mfL_{+}$, we define $\mathring\Upsilon^F_{\mft}[\sigma] \in \Poly$ inductively by setting
\begin{equ}\label{definition: general def of upsilon on big space}
\mathring\Upsilon^F_\mft[\sigma] \eqdef 
\Big(
\prod_{i\in I} \mcb{X}_{(\mft_i,k_i)} 
\Big)
\Big(
\prod_{j\in J} \mathring\Upsilon^F_{\mft_j}[\sigma_j]
\Big)
\Big[
\Bigl(\prod_{i \in I\sqcup J} D_{o_i}\Bigr)F_\mft^\mfl
\Big]
\end{equ}
(the base case being implicitly defined with $J=\emptyset$).

Next we define a linear operator $Q : \langle \mathsf{N}' \rangle \rightarrow \langle \mathsf{N} \rangle$ \label{Q page ref} as follows.
Given $\sigma \in \BB$ as in~\eqref{eq:generic tree in BB} we set
\[
Q\sigma \eqdef \Xi_\mfl \Big(\prod_{i \in I} \mathbf{X}^{k_i-p_i}\Big)
\]
We then have an extension $\mathfrak{T}_{\mathsf{D}',\mathsf{D}}(Q)$ of $Q$ which is a linear map from $\BBspan$ to $\VVspan$. 
We will abuse notation and just write $Q$ instead of $\mathfrak{T}_{\mathsf{D}',\mathsf{D}}(Q)$ for this extension, we will commit such an abuse of notation for other linear operators as well. 
\begin{remark}\label{remark:intuitionBB}
The intuition behind the set $\BB$ is that its trees contain information about which components of the solution (and derivatives thereof) every polynomial term came from.
In other words, a term $\mcI_{(\mft,p)}[X^k]$ at the root of a tree $\sigma \in \BB$ indicates that the expansion of $\DD^p U_\mft$ contributed a polynomial term $\mathbf{X}^{k-p}$ to $\sigma$.
The map $Q: \BB \to \VV$ simply discards this information.
\end{remark}
\begin{lemma}\label{lem:P obeys rule}
Fix $F \in \G$. It holds that $\mathring\Upsilon^F_\mft[\sigma] = 0$ for every $\sigma \in \BB$ and $\mft \in \mfL_+$ for which $\mcb{I}_{(\mft,0)}[Q\sigma] \in \VV \setminus \mcT^\ex$.
\end{lemma}
\begin{proof}
This follows from the implication~\ref{point:obey1}~$\Rightarrow$~\ref{point:obey2} of Proposition~\ref{prop:obeyEquivalence}.
\end{proof}
\subsection{Coherent expansion on \texorpdfstring{$\VVspan$}{V}}
We denote by $Q^* : \VVspan^* \rightarrow \BBspan^*$ the adjoint of $Q$, where $\VVspan^*$ and $\BBspan^*$ are the algebraic duals of $\VVspan$ and $\BBspan$ respectively (identified with the space of series in $\VV$ and $\BB$).
Recall that $\VVspan$ and $\BBspan$ are equipped with an inner product, in particular we identify $\VVspan$ as a subspace of $\VVspan^*$.
With our definitions we have the following lemma which is proved in Appendix~\ref{proof lemma Upsilon sum identity}.
\begin{lemma}\label{lem:UpsilonSum}
For any $\tau \in \VV$, $\mft \in \mfL_{+}$, and $F \in \mathring{\G}$ one has
\begin{equ}\label{eq: Upsilon sum identity}
\Upsilon^{F}_{\mft}[\tau]
=
\mathring{\Upsilon}^{F}_{\mft}[Q^{\ast} \tau].
\end{equ}
\end{lemma}
\begin{remark}
While in principle $Q^*\tau \in \BBspan^*$ is an infinite series, it is easy to see that $\mathring\Upsilon^F_\mft$ is non-zero for only finitely many of its terms. Hence $\mathring\Upsilon^F_\mft[Q^*\tau]$ is a well-defined element of $\Poly$. 
\end{remark}
The following result is used in the proof of Lemma~\ref{lemma: coherence identity}.
\begin{lemma}\label{lemma: coherence identity - big space}
Let $F \in \G$ and $U \in \mcb{H}^{\ex}$ written as~\eqref{eq:generalU}. Then $U$ is coherent to all orders with $F$ if and only if for any $\mft \in \mfL_{+}$,
\begin{equ}\label{eq: coherence relation on big space}
\sum_{\mfl \in \mfD_{\mft}}
F^{\mfl}_{\mft}(U)\Xi_\mfl
=
U_\mft^R\;.
\end{equ}
\end{lemma}
\begin{proof}
Suppose $U$ is coherent to all orders with $F$.
For any $\mft \in \mfL_{+}$ and $\tau \in \mcT^{\ex}$ one can use Lemmas~\ref{lem:P obeys rule} and~\ref{lem:UpsilonSum} along with the fact that $\langle U_\mft, \mcb{I}_{(\mft,0)}[\tau] \rangle \neq 0$ only if $\mcb{I}_{(\mft,0)}[\tau] \in \mcT^\ex$ to see that
\begin{equs}
\langle U_\mft, \mcb{I}_{(\mft,0)}&[\tau] \rangle 
=
\Upsilon_{\mft}^{F}[\tau](\mathbf{u}^{U})
=\mathring{\Upsilon}_{\mft}^{F}[Q^{\ast}\tau](\mathbf{u}^{U})
=
\sum_{\sigma \in \BB}
\mathring\Upsilon^{F}_{\mft}[\sigma](\mathbf{u}^{U}) \frac{ \langle Q \sigma,\tau \rangle}{\langle \sigma, \sigma \rangle}\;.
\end{equs}
Thus $U^{R}_{\mft} = \sum_{\sigma \in \BB}
\mathring\Upsilon^{F}_{\mft}[\sigma](\mathbf{u}^{U}) \frac{Q \sigma}{\langle \sigma, \sigma \rangle}$ and, for each $(\mft,p) \in \mcb{O}$, we can write
\[
U_{(\mft,p)}
=
u_{(\mft,p)} \Xi_{\mathbf{0}} 
+ 
\mathring{U}_{(\mft,p)}
+
Q 
\widehat{U}_{(\mft,p)}
\]
where $\mathring{U}_{(\mft,p)} \in \mcb{T}$ and $\widehat{U}_{(\mft,p)} \in \BBspan$ are given by
\begin{equ}
\mathring{U}_{(\mft,p)}
\eqdef
\sum_{
\substack{
q \in \N^{d+1}\\
q \not = 0
}
}
\frac{u_{(\mft,p+q)}}{q!}
\mathbf{X}^{q}
\enskip
\textnormal{and}
\enskip
\widehat{U}_{(\mft,p)}
\eqdef
\sum_{\sigma \in \BB}
\mathring{\Upsilon}^{F}_{\mft}[\sigma]\frac{\mcb{I}_{(\mft,p)}[\sigma]}{\langle \sigma, \sigma \rangle}\;.
\end{equ}
We introduce some shorthand. 
First we set $\mcA \eqdef (\mcb{O} \times \BB) \sqcup (\mcb{O} \times \N^{d+1} \setminus \{0\})$.
Second, for any $\nu \in \N^{\mcA}$, we define $\bar{\nu} \in \N^{\mcb{O}}$ by setting
\[
\bar{\nu}[(\mfb,p)]
\eqdef
\sum_{\sigma \in \BB} \nu[((\mfb,p),\sigma)]
+
\sum_{q \in \N^{d+1} \setminus \{0\}}
\nu[((\mfb,p),q)]\;.
\] 
Finally, for any $\nu \in \N^{\mcA}$ with $|\nu| < \infty$ we use the shorthand
\begin{equs}
u^{\nu}
\eqdef&
\prod_{
\substack{
(\mfb,p) \in \mcb{O}\\
q \in \N^{d+1} \setminus\{0\}
}
}
\Big(
\frac{u_{(\mfb,p+q)}}{q!}
\Big)^{\nu[((\mfb,p),q)]}\;,
\enskip
\big(\mathring{\Upsilon}^{F}\big)^{\nu}
\eqdef
\prod_{
\substack{
(\mfb,p) \in \mcb{O}\\
\sigma \in \BB
}
}
\mathring{\Upsilon}^{F}_{\mfb}[\sigma]^{\nu[((\mfb,p),\sigma)]}\\
\mcb{I}^{\nu}
\eqdef&
\prod_{
\substack{
(\mfb,p) \in \mcb{O}\\
\sigma \in \BB
}
}
\Big(
\frac{\mcb{I}_{(\mfb,p)}[\sigma]}{\langle \sigma, \sigma \rangle}
\Big)^{\nu[((\mfb,p),\sigma)]}
,\;
\sigma(\nu)
\eqdef
\mcb{I}^{\nu}
\prod_{
\substack{
(\mfb,p) \in \mcb{O}\\
q \in \N^{d+1} \setminus \{0\}
}
}
\Big(
\CI_{(\mfb,p)}[X^{p+q}]
\Big)^{\nu[((\mfb,p),q)]}
\end{equs}
Now fix $\mft \in \mfL_{+}$, then by Taylor expansion one has $\sum_{\mfl \in \mfD_{\mft}}
F^{\mfl}_{\mft}(U)\Xi_\mfl$ is equal to
\begin{equs}
{}
&
\sum_{
\substack{
\mfl \in \mfD_{\mft}\\
\eta \in \N^{\mcb{O}}
}
}
\frac{D^{\eta}F^{\mfl}_{\mft}(\mathbf{u}^{U})\Xi_{\mfl}}
{\eta!}
\Big[
\prod_{(\mfb,p) \in \mcb{O}}
(
\mathring{U}_{(\mfb,p)}
+
Q \widehat{U}_{(\mfb,p)})^{\eta[(\mfb,p)]}
\Big]\\
&=
\sum_{
\substack{
\mfl \in \mfD_{\mft}\\
\eta \in \N^{\mcb{O}}
}
}
\frac{D^{\eta}F^{\mfl}_{\mft}(\mathbf{u}^{U})}
{\eta!}
\sum_{
\substack{
\nu \in \N^{\mcA}\\
\bar{\nu} = \eta
}
}
\frac{\eta!}{\nu!}
u^{\nu}
Q
\big[
\Xi_{\mfl}
(\mathring{\Upsilon}^{F})^{\nu}(\mathbf{u}^{U})
\mcb{I}^{\nu}
\big]\\
&=
\enskip
\sum_{
\substack{
\mfl \in \mfD_{\mft}\\
\nu \in \N^{\mcA}
}
}
\frac{D^{\bar{\nu}}F^{\mfl}_{\mft}(\mathbf{u}^{U})}{\nu!}
u^{\nu}
Q
\big[
\Xi_{\mfl}
(\mathring{\Upsilon}^{F})^{\nu}(\mathbf{u}^{U})
\mcb{I}^{\nu}
\big]\\
&=
\sum_{
\substack{
\mfl \in \mfD_{\mft}\\
\nu \in \N^{\mcA}
}
}
\mathring{\Upsilon}^{F}_{\mft}[\sigma^{\nu}\Xi_{\mfl}](\mathbf{u}^{U}) 
Q
\big[
\sigma(\nu)\Xi_{\mfl}
\big]
\enskip
=
\enskip
\sum_{
\sigma \in \BB
}
\mathring{\Upsilon}^{F}_{\mft}[\sigma](\mathbf{u}^{U}) \frac{Q\sigma}{\langle \sigma, \sigma \rangle}
\end{equs}
Above, in the sums over $\nu$ and $\eta$, we implicitly enforce that $|\nu|$, $|\eta| < \infty$. 

Conversely, suppose that~\eqref{eq: coherence relation on big space} holds.
Let $V \in \mcb{H}^{\ex}$ be the unique element for which $\mathbf{u}^V = \mathbf{u}^U$ and which is coherent to all orders with $F$.
It suffices to show that for $\mft \in \mfL_{+}$ and every $\tau \in \mcT^{\ex}$ one has 
$\langle V_\mft, \mcb{I}_{\mft}(\tau)  \rangle = \langle U_\mft, \mcb{I}_{\mft}(\tau)  \rangle$. We prove this by induction on the number of edges in $\tau$.
The base case when $\tau$ has no edges follows from the definition of $V$. 
Suppose now that $\tau$ has $k \geq 1$ edges and that the claim is true for all $\mfb \in \mfL_{+}$ and all trees with at most $k-1$ edges. 
Then
\[
\langle U_\mft, \mcb{I}_{\mft}(\tau)  \rangle = \langle \sum_{\mfl\in\mfD_{\mft}} F^\mfl_\mft (U) \Xi_\mfl, \tau \rangle = \langle \sum_{\mfl\in\mfD_{\mft}} F^\mfl_\mft (V) \Xi_\mfl, \tau \rangle = \langle V_\mft, \mcb{I}_{\mft}(\tau)  \rangle,
\]
where the first equality follows from~\eqref{eq: coherence relation on big space}, the second from the inductive hypothesis and the fact that the number of edges is additive with respect to tree multiplication, and the third from the coherence of $V$ and the previous part.
\end{proof}
\subsection{Grafting operators and action of the renormalisation group}\label{sec: grafting and renormalisation}
\subsubsection{Grafting operators on \texorpdfstring{$\BBspan$}{B}}\label{subsec:graftB}
\begin{definition}\label{def:graftingLarge}
\begin{enumerate}[label=\upshape(\roman*\upshape)]
\item For $o \in \mcb{O}$, let $\graft{o} : \BBspan \otimes \BBspan \rightarrow \BBspan$\label{graft page ref} be the linear map which, for all $\sigma, \tilde \sigma \in \BB$ with $\sigma$ of the form~\eqref{eq:generic tree in BB}, is given inductively by
\begin{equs}\label{eq:graftDef}
\tilde \sigma \graft{o} \sigma
&\eqdef \tilde \sigma \graft{o}^{\root} \sigma + \tilde \sigma \graft{o}^{\nonroot} \sigma + \tilde \sigma \graft{o}^{\poly} \sigma
\\
&\eqdef
\Xi_\mfl\Big( \prod_{i\in I} \CI_{o_i}[X^{k_i}]\Big) \mcb{I}_o[\tilde\sigma] \Big(\prod_{j\in J} \mcb{I}_{o_j}[\sigma_j]\Big)
\\
&\quad + \Xi_\mfl \Big(\prod_{i\in I} \CI_{o_i}[X^{k_i}] \Big)\sum_{\bar \jmath \in J}\mcb{I}_{o_{\bar \jmath}}[\tilde \sigma \graft{o} 
\sigma_{\bar \jmath}]\Big( \prod_{j\in J\setminus \{\bar \jmath\}} \mcb{I}_{o_j}[\sigma_j]\Big)
\\
&\quad + \Xi_\mfl \sum_{\bar \imath \in I}\delta_{o,(\mft_{\bar\imath},k_{\bar\imath})}  \Big(\prod_{i\in I\setminus\{\bar\imath\}} \CI_{o_i}[X^{k_i}]\Big) \mcb{I}_{o_{\bar\imath}}[\tilde \sigma] \Big(\prod_{j\in J} \mcb{I}_{o_j}[\sigma_j]\Big)\;.
\end{equs}
\item For $l=0,\ldots, d$, define $\uparrow_l : \BBspan \rightarrow \BBspan^*$\label{uparrow page ref}
for $\sigma \in \BB$ of the form~\eqref{eq:generic tree in BB} inductively by
\begin{equs}\label{eq:graftPolyDef}
\uparrow_l \sigma
&\eqdef \Xi_\mfl \sum_{\bar \imath \in I}\mcI_{o_{\bar \imath}}[X^{k_{\bar \imath}+e_l}] \prod_{i \in I\setminus{\bar \imath}} \mcI_{o_i}[X^{k_i}]\prod_{j\in J}\mcb{I}_{o_j}[\sigma_j]
\\
&\quad + \Xi_\mfl \sum_{(\mft,p) \in \mcb{O}}\mcI_{(\mft,p)}[X^{p+e_l}] \prod_{i \in I} \mcI_{o_i}[X^{k_i}]\prod_{j\in J}\mcb{I}_{o_j}[\sigma_j]
\\
&\quad + \Xi_\mfl \prod_{i \in I} \mcI_{o_i}[X^{k_i}]\sum_{\bar\jmath\in J} \mcb{I}_{o_{\bar\jmath}}[\uparrow_l \sigma_{\bar\jmath}]\prod_{j\in J\setminus\{\bar\jmath\}}\mcb{I}_{o_j}[\sigma_j]\;.
\end{equs}
\end{enumerate} 
\end{definition}
Remark that, in the last line of~\eqref{eq:graftDef}, the $\mft_{\bar\imath}$ in the subscript of the Kronecker delta refers to the $\mfL_{+}$-component of $o_{\bar\imath}$, while the $\N^{d+1}$-component of $o_{\bar\imath}$ does not appear.

We give two pictures below which demonstrate these operators. First, for $ \graft{(\mft,p)}$, 
\begin{equs}
\tilde \sigma \graft{(\mft,p)}
\begin{tikzpicture}  [baseline=-0.3cm] 
    \node [dot] (uuul) at (-3, 1) {};
    \node [dot] (uu) at (-2.5, 0) {};
    \node [dot] (uuur) at (-2, 1) {};
    \node [dot,label=360:{\scriptsize $\Xi_{\mfl}\mcI_{(\mft,q)}[X^{p}]\mcI_{(\mft',q')}[X^{k}]$}] (m) at (-2, -1) {};
    \node [dot] (ml) at (-1.5, 0) {};
    \draw (uuul) to (uu);
    \draw (uu) to (uuur);
    \draw (uu) to (m);
    \draw (m) to (ml);
\end{tikzpicture}
&=
\begin{tikzpicture}  [baseline=-0.3cm] 
    \node [dot] (uuul) at (-3, 1) {};
    \node [dot] (uu) at (-2.5, 0) {};
    \node [dot] (uuur) at (-2, 1) {};
    \node [dot,label=360:{\scriptsize $\Xi_{\mfl}\mcI_{(\mft,q)}[X^{p}]\mcI_{(\mft',q')}[X^{k}]$}] (m) at (-2, -1) {};
    \node [dot] (ml) at (-1.5, 0) {};
    \node (new) at (-3.5, 0) {};
    \node [label={$\tilde\sigma$}] (new1) at (-3.5, -.35) {};
    \draw (uuul) to (uu);
    \draw (uu) to (uuur);
    \draw (uu) to (m);
    \draw (m) to (ml);
    \draw (m) to node[sloped,midway,below] {\scriptsize $(\mft,p)$} (new);
\end{tikzpicture}
\\
&
\quad + 
\begin{tikzpicture}  [baseline=-0.3cm] 
    \node [dot] (uuul) at (-3, 1) {};
    \node [dot] (uu) at (-2.5, 0) {};
    \node [dot] (uuur) at (-2, 1) {};
    \node [dot,label=360:{\scriptsize $\Xi_{\mfl}\mcI_{(\mft',q')}[X^{k}]$}] (m) at (-2, -1) {};
    \node [dot] (ml) at (-1.5, 0) {};
    \node (new) at (-3.5, 0) {};
    \node [label={$\tilde\sigma$}] (new1) at (-3.5, -.35) {};
    \draw (uuul) to (uu);
    \draw (uu) to (uuur);
    \draw (uu) to (m);
    \draw (m) to (ml);
    \draw (m) to node[sloped,midway,below] {\scriptsize $(\mft,q)$} (new);
\end{tikzpicture}
\\
&\quad
+
\tilde\sigma \graft{(\mft,p)}^{\nonroot} \sigma
\;.
\end{equs}
Above, we suppose $p \neq k$. The first and second terms correspond to $\graft{(\mft,p)}^{\root}$ and $\graft{(\mft,p)}^{\poly}$ respectively. The symbol $\tilde{\sigma}$ above doesn't represent a node decoration but is a placeholder 
for the tree $\tilde{\sigma}$. Next, for $\uparrow_{l}$, 
\begin{equs}
\uparrow_l
\begin{tikzpicture}  [baseline=-0.3cm] 
    \node [dot] (uuul) at (-3, 1) {};
    \node [dot] (uu) at (-2.5, 0) {};
    \node [dot] (uuur) at (-2, 1) {};
    \node [dot,label=360:{\scriptsize $\Xi_{\mfl}\mcI_{(\mft,q)}[X^{p}]\mcI_{(\mft',q')}[X^{k}]$}] (m) at (-2, -1) {};
    \node [dot] (ml) at (-1.5, 0) {};
    \draw (uuul) to (uu);
    \draw (uu) to (uuur);
    \draw (uu) to (m);
    \draw (m) to (ml);
\end{tikzpicture}
&=
\begin{tikzpicture}  [baseline=-0.3cm] 
    \node [dot] (uuul) at (-3, 1) {};
    \node [dot] (uu) at (-2.5, 0) {};
    \node [dot] (uuur) at (-2, 1) {};
    \node [dot,label=360:{\scriptsize $\Xi_{\mfl}\mcI_{(\mft,q)}[X^{p+e_l}]\mcI_{(\mft',q')}[X^{k}]$}] (m) at (-2, -1) {};
    \node [dot] (ml) at (-1.5, 0) {};
    \draw (uuul) to (uu);
    \draw (uu) to (uuur);
    \draw (uu) to (m);
    \draw (m) to (ml);
\end{tikzpicture}
\\
&\enskip
+
\begin{tikzpicture}  [baseline=-0.3cm] 
    \node [dot] (uuul) at (-3, 1) {};
    \node [dot] (uu) at (-2.5, 0) {};
    \node [dot] (uuur) at (-2, 1) {};
    \node [dot,label=360:{\scriptsize $\Xi_{\mfl}\mcI_{(\mft,q)}[X^{p}]\mcI_{(\mft',q')}[X^{k+e_l}]$}] (m) at (-2, -1) {};
    \node [dot] (ml) at (-1.5, 0) {};
    \draw (uuul) to (uu);
    \draw (uu) to (uuur);
    \draw (uu) to (m);
    \draw (m) to (ml);
\end{tikzpicture}
\\
&\enskip
+ \sum_{(\bar \mft, \bar p) \in \mcb{O}}
\begin{tikzpicture}  [baseline=-0.3cm] 
    \node [dot] (uuul) at (-3, 1) {};
    \node [dot] (uu) at (-2.5, 0) {};
    \node [dot] (uuur) at (-2, 1) {};
    \node [dot,label=360:{\scriptsize $\Xi_{\mfl}\mcI_{(\mft,q)}[X^{p}]\mcI_{(\mft',q')}[X^{k}]\mcI_{(\bar \mft, \bar p)}[X^{ \bar{p} +e_l}]$}] (m) at (-2, -1) {};
    \node [dot] (ml) at (-1.5, 0) {};
    \draw (uuul) to (uu);
    \draw (uu) to (uuur);
    \draw (uu) to (m);
    \draw (m) to (ml);
\end{tikzpicture}
\\
&\quad +
\cdots
\;\;.
\end{equs}
Above, the first two terms correspond to the first term in~\eqref{eq:graftPolyDef}, the third term corresponds to the second term in~\eqref{eq:graftPolyDef}, while $\cdots$ represents the final term in~\eqref{eq:graftPolyDef}.

The motivation for these definitions comes from the following lemma.
For $o \in \mcb{O}$, let $\triangleleft_{o} : \Poly^{\mfL_+} \times \Poly^{\mfL_+} \rightarrow \Poly^{\mfL_+}$ be the bilinear map given by setting, for any $\mfb \in \mfL_{+}$ and $F$, $\bar{F} \in \Poly^{\mfL_{+}}$, 
$
\big( 
F \triangleleft_{o} \bar F 
\big)_{\mfb}
\eqdef F_{\mft}
\cdot 
D_{o} \bar F_{\mfb}
$.
\begin{lemma}\label{lem:graftMorphism}
Let $F \in \mathring{\G}$, $o \in \mcb{O}$, $l\in\{0,\ldots, d\}$, and $\sigma, \tilde \sigma \in \BB$.
It holds that
\begin{equation}\label{eq:graftMorph}
\mathring \Upsilon^F[\tilde \sigma \graft{o} \sigma] 
= 
\mathring \Upsilon^F[\tilde \sigma] \triangleleft_{o} \mathring \Upsilon^F[\sigma]
\end{equation}
and
\begin{equation}\label{eq:graftPolyMorph}
\mathring \Upsilon^F[\uparrow_l \sigma] 
= 
\partial_l \mathring\Upsilon^F[\sigma]\;.
\end{equation}
\end{lemma}

\begin{remark}
Although $\uparrow_l\sigma$ is a series in $\BB$, note that $\mathring\Upsilon^F$ vanishes on all but a finite number of its terms, hence the LHS of~\eqref{eq:graftPolyMorph} is well-defined as an element of $\Poly^{\mfL_+}$.
\end{remark}

\begin{proof}
From the definition~\eqref{definition: general def of upsilon on big space} and the Leibniz rule, we see that $\mathring \Upsilon^F[\tilde \sigma] \triangleleft_{o} \mathring \Upsilon^F[\sigma]$ splits into a sum of three terms.
We immediately see that $\mathring\Upsilon^F[\tilde\sigma\graft{o}^{\root}\sigma]$ and $\mathring\Upsilon^F[\tilde\sigma\graft{o}^{\poly}\sigma]$ from the expression~\eqref{eq:graftDef} match two of the terms from the Leibniz rule.
The term in $\mathring\Upsilon^F[\tilde\sigma\graft{o}^{\nonroot} \sigma ]$ then matches the final term from the Leibniz rule by an induction on the number of edges in $\sigma$, which completes the proof of~\eqref{eq:graftMorph}.
The proof of~\eqref{eq:graftPolyMorph} follows in an identical manner.
\end{proof}

We now provide an expression for the adjoint of $\graft{o}$ and $\uparrow_l$.
For a tree $\sigma = T^\mfm_\mff \in \BB$ and an edge $(x,y) = e \in E_T$, let $\branch^e \sigma\in \BB$ be the subtree of $\sigma$ with node set $N_{\branch^e\sigma} \eqdef \{z \in N_T : z \geq y\}$,
and for which the corresponding decoration maps are given by restrictions of $\mfm$ and $\mff$.
Let $\trunk^e\sigma$ be the subtree of $\sigma$ with node set $N_{\trunk^e\sigma} \eqdef N_T\setminus N_{\branch^e\sigma}$ and decoration map again given by restrictions of $\mfm$ and $\mff$. 
We call $\branch^e \sigma$ and $\trunk^e \sigma$ the branch and trunk respectively of a cut at $e$.
Also, for $(\mft,k) \in \mcb{O}$ and $p \geq k$, we define $\trunk^e_{(\mft,k),p}\sigma := \trunk^e\sigma$ if $p=k$, and otherwise define $\trunk^e_{(\mft,k),p}\sigma \in \BB$ as the tree obtained from $\trunk^e\sigma$ by adding $\mcb{I}_{(\mft,k)}[X^p]$ to the node decoration at $x$.

\begin{lemma}\label{lem:graftAdjoint}
\begin{enumerate}[label=\upshape(\roman*\upshape)]
\item \label{point:adjoint1} For $o = (\mft,p) \in \mcb{O}$, consider the map $\graft{o}^* : \BBspan \rightarrow \BBspan \otimes \BBspan$ given for any $\sigma = T^\mfm_\mff \in \BB$ by
\begin{equation}\label{eq:graftAdjoint1}
\graft{o}^* \sigma \eqdef \sum_{0 \leq k \leq p} \sum_{e\in E_T} \frac{\delta_{\mff(e),(\mft,k)}}{(p-k)!} \branch^{e} \sigma \otimes \trunk^{e}_{(\mft,k),p}\sigma\;.
\end{equation}
Then for all $\sigma_1,\sigma_2,\sigma \in \BB$
\begin{equation}\label{eq:adjointIdentity}
\langle \sigma_1\otimes\sigma_2, \graft{o}^*\sigma \rangle = \langle\sigma_1\graft{o}\sigma_2,\sigma\rangle\;.
\end{equation}
\item \label{point:adjoint2} For $l=0,\ldots, d$, define $\uparrow_l^* : \BBspan \rightarrow \BBspan$ for all $\sigma = T^\mfm_\mff \in \BB$ by
\begin{equation}\label{eq:graftAdjoint2}
\uparrow_l^* \sigma \eqdef \sum_{x \in N_T} \sum_{\substack{\bar\imath \in I\\ k_{\bar\imath}-e_l\geq p_{\bar\imath}}} (k_{\bar\imath}[l]-p_{\bar\imath}[l]) \tilde\sigma\;,
\end{equation}
where in the final sum we denoted
$\mfm(x) = \Xi_{\mfl}\prod_{i \in I}\mcI_{(\mft_{i},p_{i})}[X^{k_i}]$
and $\tilde\sigma \eqdef T^{\tilde\mfm}_\mff\in \BB$ with $\tilde\mfm(y) = \mfm(y)$ for all $y \in N_T\setminus\{x\}$ and
\begin{equ}
\tilde\mfm(x) =
\begin{cases}
\Xi_{\mfl}\mcI_{(\mft_{\bar\imath},p_{\bar\imath})}[X^{k_{\bar\imath}-e_l}]\prod_{i \in I\setminus\{\bar\imath\}}\mcI_{(\mft_i,p_i)}[X^{k_i}] &\textnormal{ if } k_{\bar\imath}-e_l > p_{\bar\imath}\;,
\\
\Xi_{\mfl}\prod_{i \in I\setminus\{\bar\imath\}}\mcI_{(\mft_i,p_i)}[X^{k_i}] &\textnormal{ if } k_{\bar\imath}-e_l = p_{\bar\imath}\;.
\end{cases}
\end{equ}
Then $\langle\bar\sigma,\uparrow_l^*\sigma\rangle = \langle\uparrow_l\bar\sigma,\sigma\rangle$ for all $\sigma,\bar\sigma \in \BB$.\\
\end{enumerate}
\end{lemma}
\begin{proof}
\ref{point:adjoint1} For $\sigma$ given by~\eqref{eq:generic tree in BB}, note that $\graft{o}^*$ admits the inductive form
\begin{equs}
\graft{o}^*\sigma
&= \graft{o}^{*,\root}\sigma + \graft{o}^{*,\nonroot}\sigma
\\
&\eqdef \sum_{0\leq k \leq p}\sum_{\bar\jmath \in J} \frac{\delta_{(\mft,k),o_{\bar\jmath}}}{(p-k)!} \sigma_{\bar\jmath} \otimes Y \mcI_{(\mft,k)}[X^p] \prod_{j \in J\setminus\{\bar\jmath\}} \mcb{I}_{o_j}[\sigma_j]
\\
&\quad + \sum_{\bar\jmath \in J} \sum\sigma_{\bar\jmath}^{(1)} \otimes Y \mcb{I}_{o_{\bar\jmath}}[\sigma_{\bar\jmath}^{(2)}] \prod_{j \in J\setminus\{\bar\jmath\}} \mcb{I}_{o_j}[\sigma_j] \;,
\end{equs}
where we used the shorthand $\graft{o}^* \sigma_{\bar\jmath} = \sum \sigma_{\bar\jmath}^{(1)}\otimes\sigma_{\bar\jmath}^{(2)}$.
For any $\sigma_1,\sigma_2 \in \BB$, it follows from the definition of the inner product~\eqref{eq:innerProdBB} that $\langle \sigma_1 \graft{o}^{\root}\sigma_2,\sigma\rangle$ and $\langle \sigma_1 \graft{o}^{\poly}\sigma_2,\sigma\rangle$ are given by the terms in $\langle \sigma_1 \otimes \sigma_2,\graft{o}^{*,\root}\sigma\rangle$ with $p=k$ and $p<k$ respectively.
It now follows by an induction on the number of edges in $\sigma$ that $\langle \sigma_1 \graft{o}^{\nonroot}\sigma_2,\sigma\rangle = \langle \sigma_1 \otimes \sigma_2,\graft{o}^{*,\nonroot}\sigma\rangle$, which completes the proof of~\eqref{eq:adjointIdentity}.
Point~\ref{point:adjoint2} follows by identical considerations by noting that $\uparrow_l^*$ admits the inductive form
\begin{equs}
\uparrow_l^*\sigma
&= \Xi_\mfl \sum_{\substack{\bar\imath\in I \\ k_{\bar\imath}-e_l \geq p_{\bar\imath}}} (k_{\bar\imath}[l] - p_{\bar\imath}[l])\mcI_{(\mft_{\bar\imath},p_{\bar\imath})}[X^{k_{\bar\imath}-e_l}] \prod_{i \in I\setminus\{\bar\imath\}} \mcI_{o_i}[X^{k_i}]\prod_{j \in J}\mcb{I}_{o_{j}}[\sigma_j]
\\
&\quad + Y\sum_{\bar\jmath\in J}\mcb{I}_{o_{\bar\jmath}}[\uparrow_l^* \sigma_{\bar\jmath}]\prod_{j \in J\setminus\{\bar\jmath\}}\mcb{I}_{o_{j}}[\sigma_j]\;,
\end{equs}
where we used an abuse of notation by assuming that $\mcI_{(\mft_{\bar\imath},p_{\bar\imath})}[X^{k_{\bar\imath}-e_l}]$ is missing in the first sum in the case that $k_{\bar\imath}-e_l=p_{\bar\imath}$.
\end{proof}
\subsubsection{Grafting operators on \texorpdfstring{$\VVspan$}{V}} \label{subsec:graftV}
For a tree $\tau = T^\mfm_\mff \in \VV$ and an edge $(x,y)=e \in E_T$, we define the branch and trunk $\branch^e\tau, \trunk^e\tau \in \VV$ of a cut at $e$ in the identical manner as for $\BB$. 
For a node $x \in N_T$ and $q \in \N^{d+1}$, let $\mfm \pm_x^{q} : N_T \to \mfD \times \Z^{d+1}$ be defined by
\[
\mfm \pm_x^{q}(y) \eqdef
\begin{cases}
\left(\mnoise(x), \mpoly(x) \pm q \right) &\mbox{if } x=y,
\\
\mfm(y) &\mbox{otherwise}.
\end{cases}
\]
$\mfm \pm_x^{q}$ agrees with $\mfm$ at every node of $T$ except $x$ and increases\slash decreases by $q$ the second component of $\mfm(x)$ at $x$.
We extend the notation to
\[
\tau \pm_x^{q} \eqdef \one\{\mpoly(x) \pm q \geq 0\}(T, \mfm\pm_x^{q}, \mff),
\]
where the RHS is understood as an element of $\VVspan$.

We now describe a family of grafting operators $(\hgraft{(\mft,p)})_{(\mft,p) \in \mcb{O}}$ on the space of trees $\VVspan$ for which it holds that $Q^* \hgraft{(\mft,p)} = \graft{(\mft,p)} (Q^* \otimes Q^*)$.
We prefer to define $\hgraft{(\mft,p)}$ in terms of its adjoint.
\begin{definition} \label{def_grafting_explicit}
For $(\mft,p) \in \mcb{O}$, let $\hgraft{(\mft,p)}: \VVspan \otimes \VVspan \rightarrow \VVspan$\label{hgraft page ref} be the unique linear map whose adjoint $\hgraft{(\mft,p)}^* : \VVspan \rightarrow \VVspan \otimes \VVspan$ is given for all $\tau = T^\mfm_\mff \in \VV$ by
\begin{equation}\label{eq:graftAdjointHat}
\hgraft{(\mft,p)}^* \tau \eqdef \sum_{0 \leq k \leq p} \sum_{(x,y) \in E_T} \frac{\mathbbm{1}_{\mff(x,y) = (\mft,k)}}{(p-k)!} \branch^{(x,y)}\tau \otimes \left[ (\trunk^{(x,y)}\tau) +_x^{p-k} \right].
\end{equation}
\end{definition}
In words, the factor in the right tensor of the summands appearing on the RHS of~\eqref{eq:graftAdjointHat} is obtained from $\tau$ by removing the branch $\mcb{I}_{(\mft,k)}[\branch^{(x,y)}\tau]$ and adding $\mathbf{X}^{p-k}$ to the decoration at $x$ (adding nothing if $k=p$). 
We give a pictorial example below in which $p \geq k$ and where we show only the decorations of the edges with type $ \mft $ and the decorations of their incoming nodes. 
\begin{equs}
\hgraft{(\mft,p)}^*
\begin{tikzpicture}  [baseline=0] 
    \node [dot] (uuul) at (-3, 2) {};
    \node [dot] (uu) at (-2.5, 1.25) {};
    \node [dot] (uuur) at (-2, 2) {};
    \node [dot, label=180:{\scriptsize $\Xi_{\mfl}\mathbf{X}^{q}$}] (m) at (-2, 0) {};
    \node [dot] (ml) at (-1.5, 1) {};
    \node [dot] (mr) at (-1, 0.75) {};
    \node [dot] (root) at (-1.25, -1) {};
    \node [dot, label=0:{\scriptsize $\Xi_{\mfl'}\mathbf{X}^{q'}$}] (l) at (-0.75, -0.25) {};
    \node [dot] (new) at (-0.75,.75) {};
    \draw (uuul) to (uu);
    \draw (uu) to (uuur);
    \draw (uu) to node[sloped,midway,above] {\scriptsize $(\mft,k)$} (m);
    \draw (m) to (ml);
    \draw (m) to (mr);
    \draw (m) to (root);
    \draw (root) to (l);
    \draw (l) to node[sloped,midway,below] {\scriptsize $(\mft,p)$} (new);
\end{tikzpicture}
&=
\frac{1}{(p-k)!}
\begin{tikzpicture}  [baseline=0.2cm] 
    \node [dot] (uuul) at (-3, .75) {};
    \node [dot] (uu) at (-2.5, 0) {};
    \node [dot] (uuur) at (-2, .75) {};
    \draw (uuul) to (uu);
    \draw (uu) to (uuur);
\end{tikzpicture}
\otimes
\begin{tikzpicture}  [baseline=0] 
    \node [dot, label=left:{\scriptsize $\Xi_{\mfl}\mathbf{X}^{q + p - k}$}] (m) at (-2, 0) {};
    \node [dot] (ml) at (-1.5, 1) {};
    \node [dot] (mr) at (-1, 0.75) {};
    \node [dot] (root) at (-1.25, -1) {};
    \node [dot, label=0:{\scriptsize $\Xi_{\mfl'}\mathbf{X}^{q'}$}] (l) at (-0.75, -0.25) {};
    \node [dot] (new) at (-0.75,.75) {};
    \draw (m) to (ml);
    \draw (m) to (mr);
    \draw (m) to (root);
    \draw (root) to (l);
    \draw (l) to node[sloped,midway,below] {\scriptsize $(\mft,p)$} (new);
\end{tikzpicture}\\
&
\quad + 
\begin{tikzpicture}[baseline=-0.08cm]
\node [dot] (uu) at (0, 0) {};
\node (empty) at (0, 0) {};
\end{tikzpicture}
\otimes
\begin{tikzpicture}  [baseline=0] 
    \node [dot] (uuul) at (-3, 2) {};
    \node [dot] (uu) at (-2.5, 1.25) {};
    \node [dot] (uuur) at (-2, 2) {};
    \node [dot, label=180:{\scriptsize $\Xi_{\mfl}\mathbf{X}^{q}$}] (m) at (-2, 0) {};
    \node [dot] (ml) at (-1.5, 1) {};
    \node [dot] (mr) at (-1, 0.75) {};
    \node [dot] (root) at (-1.25, -1) {};
    \node [dot, label=above:{\scriptsize $\Xi_{\mfl'}\mathbf{X}^{q'}$}] (l) at (-0.75, -0.25) {};
    \draw (uuul) to (uu);
    \draw (uu) to (uuur);
    \draw (uu) to node[sloped,midway,above] {\scriptsize $(\mft,k)$} (m);
    \draw (m) to (ml);
    \draw (m) to (mr);
    \draw (m) to (root);
    \draw (root) to (l);
\end{tikzpicture}
\;\;.
\end{equs}

\begin{remark}\label{rem:formula for hgraft}
For $(\mft,p) \in \mcb{O}$, one is able to give a precise definition of the grafting operator  $\hgraft{(\mft,p)}$ similar to~\eqref{eq:graftDef}.
Indeed, for $ \tau \in \VV $ and $ \bar \tau = \Xi_\mfl \mathbf{X}^k\Big(\prod_{j\in J} \mcb{I}_{o_j}[\tau_j]\Big) $, we have
\begin{equs}
\tau  \hgraft{(\mft,p)} \bar \tau
& = \sum_{\ell} \binom{ k }{\ell}
\Xi_\mfl \mathbf{X}^{k-\ell}   \mcb{I}_{(\mft,p-\ell)}[\tau] \Big(\prod_{j\in J} \mcb{I}_{o_j}[\tau_j]\Big) 
\\ & +  \sum_{j \in J} \Xi_\mfl \mathbf{X}^k\mcb{I}_{o_j}[\tau  \hgraft{(\mft,p)} \tau_j] \Big(\prod_{k \neq j} \mcb{I}_{o_k}[\tau_k]\Big)\;.
\end{equs}
\end{remark}

\begin{definition} \label{def_increasing_poly}
For $i \in \{0,\ldots, d\}$, let $\hat\uparrow_i : \VVspan \rightarrow \VVspan$\label{hatuparrow page ref} be the unique linear map with adjoint given for all $\tau = T^\mfm_\mff \in \VV$ by
\begin{equation}\label{eq:raiseAdjointHat}
\hat\uparrow_i^* \tau \eqdef \sum_{x \in N_T} \mpoly(x)[i] \tau-_x^{e_i},
\end{equation}
where we write $\mfm(x) = \left(\mnoise(x), \left(\mpoly(x)[0],\ldots, \mpoly(x)[d]\right)\right) \in \mfD \times \N^{d+1}$.
\end{definition}
We give a pictorial example for the above definition.
\begin{equs}
\hat\uparrow_i^*
\begin{tikzpicture}  [baseline=-0.4cm] 
    \node [dot, label=above:$\Xi_{\mfl_2}\mathbf{X}^{q_2}$] (uu) at (-2.75, 0) {};
    \node [dot, label=below:$\Xi_{\mfl_1}\mathbf{X}^{q_1}$] (m) at (-2, -1) {};
    \node [dot, label=above:$\Xi_{\mfl_3}\mathbf{X}^{q_3}$] (ml) at (-1.25, 0) {};
    
    \draw (uu) to (m);
    \draw (m) to (ml);
\end{tikzpicture}
&=
q_1[i]
\begin{tikzpicture}  [baseline=-0.4cm] 
    \node [dot, label=above:$\Xi_{\mfl_2}\mathbf{X}^{q_2}$] (uu) at (-2.75, 0) {};
    \node [dot, label=below:$\Xi_{\mfl_1}\mathbf{X}^{q_1-e_i}$] (m) at (-2, -1) {};
    \node [dot, label=above:$\Xi_{\mfl_3}\mathbf{X}^{q_3}$] (ml) at (-1.25, 0) {};
    
    \draw (uu) to (m);
    \draw (m) to (ml);
\end{tikzpicture}
+
q_2[i]
\begin{tikzpicture}  [baseline=-0.4cm] 
    \node [dot, label=above:$\Xi_{\mfl_2}\mathbf{X}^{q_2-e_i}$] (uu) at (-2.75, 0) {};
    \node [dot, label=below:$\Xi_{\mfl_1}\mathbf{X}^{q_1}$] (m) at (-2, -1) {};
    \node [dot, label=above:$\Xi_{\mfl_3}\mathbf{X}^{q_3}$] (ml) at (-1.25, 0) {};
    
    \draw (uu) to (m);
    \draw (m) to (ml);
\end{tikzpicture}
\\
&\quad +
q_3[i]
\begin{tikzpicture}  [baseline=-0.4cm] 
    \node [dot, label=above:$\Xi_{\mfl_2}\mathbf{X}^{q_2}$] (uu) at (-2.75, 0) {};
    \node [dot, label=below:$\Xi_{\mfl_1}\mathbf{X}^{q_1}$] (m) at (-2, -1) {};
    \node [dot, label=above:$\Xi_{\mfl_3}\mathbf{X}^{q_3-e_i}$] (ml) at (-1.25, 0) {};
    
    \draw (uu) to (m);
    \draw (m) to (ml);
\end{tikzpicture}
\;\;.
\end{equs}
Note that $\graft{o}$ and $\uparrow_l$ extend to well-defined maps $\BBspan^*\otimes\BBspan^* \to \BBspan^*$ and $\BBspan^* \to \BBspan^*$ respectively (this can be seen from Lemma~\ref{lem:graftAdjoint} or directly from the triangular structure of the maps).
\begin{lemma}\label{lem:QMorphism}
Let $(\mft,p) \in \mcb{O}$ and $l \in \{0,\ldots, d\}$.
Then, as maps from $\VVspan\otimes\VVspan$ to $\BBspan^*$,
\begin{equation}\label{eq:QMorphism1}
Q^* \hgraft{(\mft,p)} = \graft{(\mft,p)} (Q^* \otimes Q^*)
\end{equation}
and, as maps from $\VVspan \to \BBspan^*$,
\begin{equation}\label{eq:QMorphism2}
\uparrow_l Q^* = Q^* \hat\uparrow_l.
\end{equation}
\end{lemma}
\begin{proof}
Considering the dual statements, it suffices to show that for all $\sigma \in \BB$,
\begin{equation}\label{eq:QMorphismAd1}
\hgraft{(\mft,p)}^* (Q\sigma) = (Q\otimes Q) \graft{(\mft,p)}^* \sigma,
\end{equation}
and
\begin{equation}\label{eq:QMorphismAd2}
\hat\uparrow_l^* (Q\sigma) = Q (\uparrow_l^* \sigma).
\end{equation}
To show~\eqref{eq:QMorphismAd1}, observe that, by definition of $Q$, there is a bijection between the edges with decoration $(\mft,k)$ in $Q\sigma$ and edges with decoration $(\mft,k)$ in $\sigma$.
Furthermore, for every cut appearing in the sum~\eqref{eq:graftAdjoint1} of Lemma~\ref{lem:graftAdjoint} with corresponding term $\frac{1}{(p-k)!}b \otimes t$, it holds that $\frac{1}{(p-k)!}(Q b) \otimes (Q t)$ is the term appearing from the corresponding cut in~\eqref{eq:graftAdjointHat}, from which~\eqref{eq:QMorphismAd1} follows.

To show~\eqref{eq:QMorphismAd2}, consider a node $x$ in $\sigma$ with decoration $\Xi_\mfl\prod_{i \in I}\mcI_{(\mft_i,p_i)}[X^{k_i}]$.
Denote by $\bar x$ the corresponding node in $Q\sigma$.
Note that the polynomial decoration at $\bar x$ is $k \eqdef \sum_{i \in I}(k_i-p_i)$.
Every term $\tilde\sigma$ of $\uparrow_l^*\sigma$ in the sum~\eqref{eq:graftAdjoint2} then corresponds to a term in $Q\bar\sigma$, up to a combinatorial factor, obtained by lowering the polynomial decoration at $\bar x$ by $e_l$.
It remains to verify that the correct combinatorial factor is obtained.
To this end, the contribution from $x$ to the factor in front of $\hat\uparrow_l^*Q\sigma$ is $k[l]$.
On the other hand, if $\mcI_{(\mft_{\bar\imath},p_{\bar\imath})}[X^{k_{\bar\imath}}]$ at $x$ was lowered first by $e_l$ from $\uparrow_l^*$, and then $Q$ was applied, its contribution to the combinatorial factor becomes $k_{\bar\imath}[l]-p_{\bar\imath}[l]$ (provided $k_{\bar\imath}-p_{\bar\imath} \geq e_l$).
Running over all polynomial decorations at $x$ gives the total combinatorial factor of $\sum_{i \in I}(k_i[l]-p_i[l]) = k[l]$ as desired.
\end{proof}
\begin{corollary}\label{cor:graftMorphismSmall}
For all $F \in \mathring{\G}$, $o \in \mcb{O}$, and $\tau,\bar\tau \in \VV$,
\begin{equation}\label{eq:graftSmall1}
\Upsilon^F[\tau \hgraft{o} \bar\tau] = \Upsilon^F[\tau] \triangleleft_o \Upsilon^F[\bar\tau].
\end{equation}
Furthermore, for all $i =0,\ldots, d$, 
\begin{equation}\label{eq:graftSmall2}
\partial_i \Upsilon^F[\tau] = \Upsilon^F[\hat\uparrow_i \tau]\;.
\end{equation}
\end{corollary}
\begin{proof}
To prove~\eqref{eq:graftSmall1}, observe that, by Lemmas~\ref{lem:UpsilonSum} and~\ref{lem:graftMorphism},
\begin{align*}
\Upsilon^F[\tau] \triangleleft_o \Upsilon^F[\bar\tau]
&= \mathring \Upsilon^F[Q^*\tau] \triangleleft_o \mathring\Upsilon^F[Q^*\bar\tau]
= \mathring\Upsilon^F[(Q^*\tau) \graft{o} (Q^*\bar\tau)]
\\
&= \mathring\Upsilon^F[Q^*(\tau \hgraft{o} \bar\tau)]
= \Upsilon^F[\tau \hgraft{o} \bar\tau],
\end{align*}
where the third equality follows from~\eqref{eq:QMorphism1}. The proof of~\eqref{eq:graftSmall2} follows in the same manner using now~\eqref{eq:QMorphism2}.
\end{proof}
\subsubsection{Interaction with the renormalisation group}\label{subsec:interactionWithRenormG}
An important property of the grafting operators $\hgraft{(\mft,p)}$ is that its adjoint suitably preserves $\mcb{T}^\ex \subset \VVspan$.
\begin{lemma}\label{lem:adjointPreservesT}
For all $o\in\mcb{O}$ and $i\in\{0,\ldots, d\}$, it holds that $\hgraft{o}^*$ (resp. $\hat\uparrow_i$) maps $\mcb{T}^\ex$ to $\mcb{T}^\ex \otimes \mcb{T}^\ex$ (resp. $\mcb{T}^\ex$).
\end{lemma}
\begin{proof}
The claim that $\hat\uparrow_i$ maps $\mcb{T}^\ex$ to $\mcb{T}^\ex$ is obvious.
For $\hgraft{o}^*$, observe that normality of the rule $R$ and the explicit description of the set of trees $\mcT^\ex$~\cite[Lem.~5.25]{BHZalg} imply that $\branch^e\tau \otimes \trunk^e\tau \in \cT^\ex$ for any $\tau \in \mcT^\ex$ and $e \in E_\tau$. The conclusion follows from the expression~\eqref{eq:graftAdjointHat} for $\hgraft{o}^*$.
\end{proof}

Define the linear map $\bgraft{o} : \cT^\ex \otimes \cT^\ex \to \cT^\ex$\label{bgraft page ref} given by $\tau \bgraft{o}\bar\tau \eqdef \pi_{\cT^\ex} (\tau \hgraft{o} \bar\tau)$, where $\pi_{\cT^\ex} : \VVspan \to \cT^\ex$ is the canonical projection. The following is a consequence of Lemma~\ref{lem:adjointPreservesT}.
\begin{corollary}\label{cor:graftMorphT}
It holds that $\bgraft{o}$ is the adjoint of $\hat\curvearrowright^{*}_{o} : \cT^\ex \to \cT^\ex \otimes \cT^\ex$. 
\end{corollary}
The following proposition is proved in Appendix~\ref{subsec:coInteractProofs}.
\begin{proposition}\label{prop:co-interaction}
Let $M \in \mfR$. Then 
\begin{itemize}
\item For $i=0,\ldots, d$, it holds on $\mcT^{\ex}$ that
\begin{equs}[co-interaction polynomial]
M^*\hat\uparrow_i = \hat\uparrow_i M^*.
\end{equs}
\item For all $ \tau,\bar\tau \in \mcT^\ex $ and $o \in \mcb{O}$
\begin{equs}[co-interaction renormalisation]
(M^*\tau) \bgraft{o} (M^*\bar\tau) = M^*(\tau \bgraft{o} \bar \tau).
\end{equs}
\end{itemize}
\end{proposition}
\begin{corollary}\label{cor:XMorphT}
 Let $M \in \mfR$. Then for any $\mft \in \mfL_{+}$, $\mfl \in \mfD$, $k \in \N^{d+1}$:
\begin{equs}[co-interaction poly]
\Upsilon^{MF}_{\mft}[\Xi_\mfl\mathbf{X}^k] = \Upsilon^{F}_{\mft}[M^*(\Xi_\mfl \mathbf{X}^k)].
\end{equs}
\end{corollary}
\begin{proof}
For any $\mft \in \mfL_{+}$, $\mfl \in \mfD$, $k \in \N^{d+1}$,
\[
\Upsilon^{MF}[\Xi_{\mfl}\mathbf{X}^k] = \partial^k \Upsilon^{MF}[\Xi_{\mfl}] = \partial^k \Upsilon^F[M^*\Xi_{\mfl}] = \Upsilon^F[ M^*(\Xi_{\mfl}\mathbf{X}^k)],
\]
where we have used~\eqref{eq:graftSmall2}, the identity $(\hat\uparrow)^k \Xi_\mfl = \Xi_\mfl \mathbf{X}^k$ (where $(\hat\uparrow)^k \eqdef \prod_{i=0}^d (\hat\uparrow_{i})^{k[i]}$, which is well defined due to the commutativity of $\hat\uparrow_i$ and $\hat\uparrow_j$), and \eqref{co-interaction polynomial} which implies $M^*(\hat\uparrow)^k = (\hat\uparrow)^k M^*$.
\end{proof}

Let us write $\bar\mfD \subset \mcT^\ex$\label{bar mfD page ref} for the set of all elements of the form 
$\Xi_\mfl \mathbf{X}^k$ with $k \in \N^{d+1}$ and $\mfl \in \mfD$.
Observe that the grafting operators
$(\hgraft{o})_{o \in \mcb{O}}$
satisfy a pre-Lie type identity:
\begin{equs}\label{pre_lie_identity}
& \left( \tau_1 \hgraft{(\Labhom_{1},p_1)} \tau_2 \right) \hgraft{(\Labhom_{2},p_2)}  
\tau_3 -  \tau_1 \hgraft{(\Labhom_{1},p_1)} \left( \tau_2 \hgraft{(\Labhom_{2},p_2)} \tau_3 \right)  \\
& =
\left( \tau_2  \hgraft{(\Labhom_{2},p_2)}  \tau_1 \right) \hgraft{(\Labhom_{1},p_1)} 
\tau_3 - \tau_2 \hgraft{(\Labhom_{2},p_2)}  \left( \tau_1 \hgraft{(\Labhom_{1},p_1)} \tau_3 \right)\;.
\end{equs}
We next show a universal property of the space $\VVspan$ as a pre-Lie type algebra with the grafting operators $(\hgraft{o})_{o \in \mcb{O}}$, the proof of which will be given in Appendix~\ref{subsec:freelyGenProof}.

\begin{remark}
Each occurrence of a grafting operator $\hgraft{(\Labhom,p)}$ is a linear combination of other grafting operators which are the decorated analogues of those discussed Remark~\ref{rem: intro to grafting and pre-lie structures}. The only difficult part in obtaining~\eqref{pre_lie_identity} is when the grafting occurs at the same node.
For $ \tau_{3} = \Xi_\mfl \mathbf{X}^k $, the identity \eqref{pre_lie_identity} is then equivalent to
\begin{equs}
\sum_{\ell,\ell'} \binom{ k }{\ell'}  \binom{ k-\ell' }{\ell} \Xi_\mfl \mathbf{X}^{k-\ell-\ell'} \mcI_{(\mft_{1},p_{1}-\ell)}[ \tau_1] \mcI_{(\mft_{2},p_{2}-\ell')}[ \tau_2] \\ = \sum_{\ell,\ell'}  \binom{ k }{\ell}  \binom{ k-\ell }{\ell'} \Xi_\mfl \mathbf{X}^{k-\ell-\ell'}\mcI_{(\mft_{1},p_{1}-\ell)}[ \tau_1] \mcI_{(\mft_{2},p_{2}-\ell')}[ \tau_2]\;. 
\end{equs}
Then it is quite straightforward to see that the coefficients of the two previous sums match.
\end{remark}

\begin{proposition}\label{prop:generate}
The space $\VVspan$ is freely generated by the family $(\hgraft{o})_{o \in \mcb{O}}$ with generators $\bar\mfD$. More precisely, consider any vector space $V$ equipped with bilinear operators $(\triangleleft_{\alpha})_{\alpha \in A}$, $\triangleleft_\alpha : V \times V \to V$, which satisfy the pre-Lie identity for all $\alpha,\bar \alpha \in A$ and $x,y,z \in V$,
\begin{equs} \label{multi pre-Lie}
( x \triangleleft_{\alpha} y) \triangleleft_{\bar \alpha} z - x \triangleleft_{\alpha} (y \triangleleft_{\bar \alpha} z ) = (y \triangleleft_{\bar \alpha} x) \triangleleft_{\alpha} z - y \triangleleft_{\bar \alpha} (x \triangleleft_{\alpha} z).
\end{equs}
Then for any map $\Phi : \bar\mfD \to V$ and $\Psi : \mcb{O} \to A$, there exists a unique extension of $\Phi$ to a linear map $\hat \Phi : \VVspan \to V$ which satisfies for all $o \in \mcb{O}$ and $\tau, \bar \tau \in \VVspan$
\[
\hat \Phi(\tau \hgraft{o} \bar \tau) = (\hat \Phi\tau) \triangleleft_{\Psi(o)} (\hat \Phi \bar \tau).
\]
\end{proposition}
\begin{remark}
In what follows, we will only use the fact that $\VVspan$ is generated by $\bar\mfD$ and $(\hgraft{o})_{o\in\mcb{O}}$; we emphasize that this generation is free only to highlight the algebraic structure of $\VVspan$.
\end{remark}
\begin{corollary}\label{cor:generateT}
$\cT^\ex$ is generated by the family $(\bgraft{o})_{o \in \mcb{O}}$ with generators $\bar\mfD$.
\end{corollary}
\begin{proof}[of Lemma~\ref{lem: renormalisation of upsilon}]
By Proposition~\ref{prop:co-interaction} identity \eqref{co-interaction poly}, $\Upsilon^F \circ M^*$ and $\Upsilon^{MF}$ agree on $\bar\mfD$.
Since $MF \in \G$ by Lemma~\ref{lem:GPreserve}, observe that Proposition~\ref{prop:obeyEquivalence} implies that $\Upsilon^{MF}[\tau] = 0$ for all $\tau \in \VV\setminus\mcT^\ex$, and thus $\Upsilon^{MF} \circ \pi_{\cT^\ex} = \Upsilon^{MF}$. Therefore, applying~\eqref{eq:graftSmall1} to $\Upsilon^{MF}$, we see that for all $\tau,\bar\tau \in \mcT^\ex$ and $o \in \mcb{O}$
\[
\Upsilon^{MF}[\tau \bgraft{o} \bar \tau] = \Upsilon^{MF}[\tau] \triangleleft_{o} \Upsilon^{MF}[\bar\tau]\;.
\]
On the other hand, by \eqref{co-interaction renormalisation}, we have
\[
\Upsilon^{F} \bigl[M^* (\tau \bgraft{o} \bar \tau)\bigr] = \left(\Upsilon^{F} [M^* \tau] \right) \triangleleft_{o} \left(\Upsilon^{F} [M^* \bar \tau] \right)\;.
\]
It thus follows from Corollary~\ref{cor:generateT} that $\Upsilon^{MF} = \Upsilon^{F} \circ M^*$ as desired.
\end{proof}
\section{Analytic theory and a generalised Da Prato--Debussche trick}\label{sec: analytic aspects and DPD}
\subsection{Admissible models}\label{subsec: admissible models}
For each $\mft \in \mfL_{+}$ we fix a decomposition
 $G_{\mft} = K_{\mft} + R_{\mft}$\label{Kmft page ref} on $\Lambda \setminus \{0\}$ where 
\begin{itemize}
\item $K_{\mft}(x)$ is supported in the ball $|x|_{\s} \le 1$ and coincides with $G_{\mft}(x)$ whenever $|x|_{\s} \le 1/2$.
\item For a parameter $\gamma \in \R$ to be defined later, and for every polynomial $Q$ on $\Lambda$ of $\s$-degree less than $\gamma + |\mft|_{\s}$, one has
\[
\int_{\Lambda} K_\mft (z) Q(z)\, dz=0\;.
\] 
\item One has $K_{\mft}(t,x) = R_{\mft}(t,x) = 0$ whenever $t < 0$. 
\item $R_{\mft}:\Lambda \rightarrow \R$ is a smooth function and satisfies, for every $k \in \N^{d+1}$, the bound
\[
\sup_{t \ge 0}
\sup_{x \in \T^{d}}
e^{\chi t}
|(D^{k}R_{\mft})(t,x)|
< \infty\; \textnormal{ for some } \chi > 0.
\]
\end{itemize} 
We write $K$ for the tuple $(K_{\mft})_{\mft \in \mfL_{+}}$.
We also write $\Omega_{\infty}$ for the set of all tuples $\xi = (\xi_{\mfl})_{\mfl \in \mfL_{-}}$ where for each $\mfl \in \mfL_{-}$, $\xi_{\mfl}:\Lambda \rightarrow \R$ is a smooth function. 
For $\xi \in \Omega_{\infty}$ we denote by $Z^{\xi}$ the model on $\mathscr{T}$ given by the canonical $K$-admissible lift of $\xi$.
We write $\mathscr{M}_{\infty} $\label{Minfty page ref} for the space of all smooth $K$-admissible models on $\mathscr{T}$. 

We introduce a family of pseudo-metrics on $\mathscr{M}_{\infty}$ indexed by compact $\K \subset \R^{d+1}$ and $\ell \in A = \{|\tau|_{+} : \tau \in \mcT^{\ex} \}$.
Given $(\Pi,\Gamma)$ and $(\bar\Pi,\bar\Gamma)$ one sets $\mathscr{M}_{\infty}$,
\begin{equ}[e:defDistModel]
\$(\Pi,\Gamma) ; (\bar\Pi,\bar\Gamma)\$_{\ell;\K} \eqdef  
\|\Pi - \bar \Pi\|_{\ell;\K} + \|\Gamma - \bar \Gamma\|_{\ell;\K}\;,
\end{equ}
where
\begin{equs}\label{e: def of model seminorms Pi}
\|\Pi - \bar \Pi\|_{\ell;\K}&\eqdef \sup\left\{ 
\frac{
|\big((\Pi_x - \bar \Pi_x) \tau,S^{\lambda}_{\s}\phi \big)|}{\lambda^{\ell}} :
\ 
\begin{array}{c}
x \in \K, \tau \in \mcT^{\ex}_{\ell},\\
\lambda \in (0,1], \phi \in \CB_{x,r}
\end{array}
\right\},\quad
\\ \label{e: def of model seminorms Gamma}
\|\Gamma - \bar \Gamma\|_{\ell;\K}&\eqdef \sup\left\{\frac{\|\Gamma_{xy}\tau -\bar\Gamma_{xy}\tau\|_m}{\|x-y\|_\s^{m-\ell}}:\ 
\begin{array}{c}
x,y \in \K, x\ne y,\\
\tau \in \mcT^{\ex}_{\ell}, m \in [r,\ell) \cap A
\end{array}
\right\}.
\end{equs}
Above, we have used the notation $\mcT^{\ex}_{\ell} \eqdef \{ \tau \in \mcT^{\ex}:\ |\tau|_{+} = \ell\}$, for $a \in \mcb{T}^{\ex}$ of the form $a = \sum_{\tau \in \mcT^{\ex}} a_{\tau} \tau$ and $m \in A$ we set $\|a\|_{m} \eqdef \sup \{ |a_{\tau}|: \tau \in \mcT^{\ex}_{m}\}$, and we set $r \eqdef 1 - \min A$. 
Note that for any fixed $\gamma \ge 0$, the family of pseudo-metrics 
\[
\Big\{ 
\$ 
\bullet ; \bullet 
\$_{\ell,\mathfrak{K}}: 
\ell \in A \cap(-\infty, \gamma], 
\mathfrak{K} \subset \R^{d+1} \textnormal{ compact} 
\Big\}
\]
generates a metric $d_{\gamma}$ on $\mathscr{M}_{\infty}$. 
Denote by $\mathscr{M}_{0}$ \label{M0 page ref} the completion\footnote{The completion does not depend on the choice of $\gamma \ge 0$. This is a consequence of the fact that admissible models are completely determined (in a continuous way) once one knows their restriction to symbols $\tau$ with $|\tau|_{+} \le 0$.} of $\mathscr{M}_{\infty}$ under $d_{\gamma}$.

To prepare for Proposition~\ref{prop:stationary objects} we define a stronger metric on a subset 
of $\mathscr{M}_{0}$. 
\begin{definition}\label{def: L1 norm}
Given $Z,\bar{Z} \in \mathscr{M}_{0}$ we define
\[
\$
Z;\bar{Z}
\$
=
\sup_{n \in \Z} 
\frac{1}{n^{2}+1}
\$Z ; \bar{Z}\$_{\K_{n}}\;,
\]
where 
\begin{equ}\label{compact set def}
\K_{n} \eqdef [n - 1, n + 1] \times \T^{d} \subset \Lambda\;,
\end{equ} 
and for any compact $\K \subset \Lambda$, 
$\$Z ; \bar{Z}\$_{\K}
\eqdef
\max_{
\ell \le 0
}
\$Z ; \bar{Z}\$_{\ell, \K}$. 
We define $\$
Z
\$_{\mathfrak{K}}
=
\$
Z;0
\$_{\mathfrak{K}}$ and $\$
Z
\$
=
\$
Z;0
\$$ analogously by removing the presence of $\bar{\Pi}$ and $\bar{\Gamma}$ in~\eqref{e: def of model seminorms Pi} and~\eqref{e: def of model seminorms Gamma}.
We let $\mathscr{M}_{0,1} \subset \mathscr{M}_{0}$ (resp. $\mathscr{M}_{\infty,1} \subset \mathscr{M}_{\infty}$) denote the collection of $Z \in \mathscr{M}_{0}$  (resp. $Z \in \mathscr{M}_{\infty}$) with $\$Z\$ < \infty $.
\end{definition}
Clearly $ \$\cdot;\cdot\$$ is a stronger metric than $d_{\gamma}$ for any $\gamma \ge 0$ and moreover $\mathscr{M}_{0,1}$ is a complete metric space with respect to $ \$\cdot;\cdot\$$. 

One important fact regarding $\mathfrak{R}$ is the following from \cite[Thm.~6.15]{BHZalg}.
\begin{theorem} \label{partial theorem 6.15}
Any $ M \in \mfR $ defines a map from $ \mathscr{M}_{\infty} $ into itself which associates to a model $ Z = ( \Pi,\Gamma)  $ a renormalised model $ Z^{M} = ( \hat \Pi, \hat \Gamma) $ . This renormalised model satisfies for every $ x \in \Lambda $
 \begin{equ} \label{explicit renormalised model}
 \hat \Pi_x =  \Pi_x M\;.
 \end{equ}
Furthermore, this action of $\RR$ extends to a continuous right action on $\mathscr{M}_0$.
\end{theorem}
\subsection{The space of jets}\label{subsec:space of jets}
We set $P \eqdef \{(t,x) \in \Lambda : t = 0\}$\label{P page ref} (thought of as the singular set of a modelled distribution)
and define $\jets$\label{jets page ref} as the space of all maps $U: \Lambda\setminus P \rightarrow \mcb{H}^{\ex}$
and $\widetilde{\jets}$ the maps $U: \Lambda\setminus P \rightarrow \widetilde{\mcb{H}}^{\ex}$.
For $U\in \jets$ or $U\in \widetilde\jets$, we write $U = (U_{\mft})_{\mft \in \mfL_{+}}$.
For $U\in \jets$ we also write $U_{(\mft,p)} = \DD^{p}U_{\mft}$ and, by applying the definitions in Section~\ref{subsec:nonlinearities} pointwise,
we define $U^{R} \in \widetilde{\jets}$ as well as the tuple of functions $\mathbf{u}^{U} = (u_{\alpha}^{U})_{\alpha \in \mcb{O}}$ (so that $u_{\alpha}^{U}:\Lambda\setminus P \rightarrow \R$ is given by setting $u_{(\mft,p)}^{U}(x) \eqdef \langle \mathbf{X}^{p},U_{\mft}(x)\rangle = \langle \mathbf{1},U_{(\mft,p)}(x) \rangle$).
For any $F \in \mcb{C}_{\mcb{O}}$ we write $F(\mathbf{u}^{U})$ for the real-valued function on $\Lambda\setminus P$  given by $F(\mathbf{u}^{U})(x) \eqdef F(\mathbf{u}^{U}(x))$.

Following Lemma~\ref{lem: nonlinearity on regularity structure}, given $U\in \jets$, we henceforth write $\mcQ_{ \le \gamma}F(U)\Xi \in \widetilde{\jets}$ for the element
obtained by applying the map $\mcQ_{ \le \gamma}F\Xi$ to $U$ pointwise.
Also following Definition~\ref{def: coherent}, we say that $U \in \jets$ is coherent to order $L$ with $F$ if $U(z)$ is for all $z \in \Lambda\setminus P$.

We record the following simple lemma for the canonical model.
\begin{lemma}\label{lem:canonical models commutes with F}
Consider $F \in \G$, $Z^\xi = (\Pi,\Gamma)$ the canonical model built from some $\xi\in\Omega_\infty$, and an element $U = (U_{\mft})_{\mft \in \mfL_{+}} \in \mcb{H}^{\ex}$.
It holds for all $x \in \Lambda$, $\mft \in \mfL_{+}$, and $\mfl=(\hat{\mfl},\mfo) \in \mfD_{\mft}$ that
\begin{equ}
\Pi_x
\left[
\mcQ_{\leq 0} \mathbf{F}^\mfl_\mft(U)\Xi_\mfl
\right](x)
=
F^\mfl_\mft (\boldsymbol{\phi}(x))
\xi_{\hat\mfl}(x)\;,
\end{equ}
where $\mathbf{F}^\mfl_\mft(U)$ is given by~\eqref{def: lift of smooth functions}, $
\xi_{\hat{\mfl}}
\eqdef
\prod_{(\mfb,\mfe) \in \hat{\mfl}}
D^{\mfe}\xi_{\mfb}$, $\phi = (\phi_\mft)_{\mft\in\mfL_+}$ is given by $\phi_\mft \eqdef \Pi_x U_{\mft}$, and $\boldsymbol{\phi}$ is defined as in~\eqref{eq:bold phi def}.
\end{lemma}
\begin{proof}
Note that $\Pi_x(\boldsymbol{U} - \langle\boldsymbol{U},\bone\rangle)^\alpha(x) = 0$ for any $\alpha\in\N^{\mcb{O}}$ such that $\alpha(o) > 0$ for some $o\in\mcb{O}_+$.
Expanding $F^\mfl_\mft$ as in~\eqref{expansion of nonlinearity}, the claim follows from the fact that a polynomial is given exactly by its Taylor expansion, as well as the fact that the canonical model is multiplicative, reduced (i.e., ignores the value of the extended label $\mfo$), and compatible with the abstract gradient $\DD$ (see~\cite[Def.~5.26]{Regularity}).
\end{proof}
\subsection{Modelled distributions}\label{subsec: modelled distributions}
Throughout this subsection we fix $F \in \G$. Our definitions and results, unless explicitly stated, are given with respect to some arbitrary fixed model $Z \in \mathscr{M}_{0}$. 
We often drop dependence on $Z$ from the notation. 

For any sector $V$ of the regularity structure $\mathscr{T}$, and any $\gamma, \eta \in \R$, recall that~\cite[Def.~6.2]{Regularity} defines a corresponding space of singular modelled distributions $\D^{\gamma,\eta}_{P}(V)$\label{Dgamma page ref} with respect to $Z$ over $\Lambda \setminus P$, with values in the sector $V$.
We will often drop the reference to $P$ and $V$ when it is clear from the context.
We will also often write $\D^{\gamma,\eta}_\alpha$ to emphasise that the underlying sector is of regularity $\alpha \leq 0$.

Denote by $\one_+: \Lambda \to \{0,1\}$ the indicator function of the set $\{(t,x) : t > 0\}$, which we canonically identify with an element of $\D^{\infty,\infty}_0$. 
We recall two results from~\cite{Regularity}.
\begin{lemma}\label{lem:abstractIntegral}
Let $\gamma, \alpha, \eta \in \R$, and $\mft \in \mfL_+$.
Suppose that $\gamma, \eta \notin \N $, $\gamma - |\mft|_\s > 0$, and $\eta \wedge \alpha > -\s_{0}+|\mft|_\s$.
Then $\mathcal{K}_{\gamma -|\mft|_\s}^{\mft}$, defined by~\cite[Eq.~5.15]{Regularity} using the kernel $K_\mft$, is a locally Lipschitz map from $\D^{\gamma - | \mft |_{\s},\eta - | \mft |_{\s}}_{(\alpha-|\mft|_\s)\wedge 0}(\widetilde{\cT}^\ex_\mft)$ to $\D^{\gamma,\eta\wedge \alpha\wedge|\mft|_\s}_{\alpha\wedge 0}(\cT^\ex_\mft)$.

Moreover, for all $\kappa \geq 0$ and $T \in (0,1]$, it holds that
\[
\wwnorm{\mcK_{\gamma - |\mft|_\s}^\mft \one_+ f}_{\gamma; \eta \wedge \alpha; T} \lesssim T^{\kappa/\s_{0}}\wwnorm{f}_{\gamma-|\mft|_\s; \eta - | \mft |_{\s} + \kappa; T}
\]
for all $f \in \D^{\gamma - | \mft |_{\s},\eta - | \mft |_{\s} + \kappa}_{(\alpha-|\mft|_\s+\kappa)\wedge 0}(\widetilde\cT^\ex_\mft)$, where the proportionality constant depends on $\wwnorm{Z}_{[-1,2]\times \R^{d}}$.
\end{lemma}
\begin{proof}
For all $ f \in \D^{\gamma - | \mft |_{\s},\eta - | \mft |_{\s}}_{\reg(\mft)-|\mft|_\s}(\widetilde\cT^\ex_\mft) $, it holds by definition of $\widetilde\cT^\ex_\mft$ that $\mathcal{K}_{\gamma - |\mft|_\s}^{\mft} f$ is a function from $\Lambda\setminus P$ to $\cT^\ex_\mft$. The conclusion follows at once from~\cite[Prop.~6.16, Thm.~7.1]{Regularity}.
\end{proof}
Concerning the initial condition, we recall the following result.
\begin{lemma}\label{lem:initialCond}
Let $\alpha \in \R$ such that $\alpha \notin \N$ and $u_0^{\mft} \in \mathcal{C}_{\bar\s}^{\alpha}(\T^{d})$. Then the function
\begin{equs}
v_{\mft}(t,x) \eqdef (G_{\mft} u^{\mft}_0)(t,x) = \int_{\T^{d}} G_{\mft}(t,x-y) u^{\mft}_0(y) dy\;.
\end{equs}
lifts canonically to a singular modelled distribution in $\D^{\gamma,\alpha}(\bar\cT^\ex)$ for all $\gamma > \alpha\vee 0$.
\end{lemma}
\begin{proof}
Identical to the proof of~\cite[Lem.~7.5]{Regularity} upon using~\eqref{eq:initialBound}.
\end{proof}
For $\alpha \in \R$ and $\gamma > 0$, recall the operator $R^\mft_{\gamma} : \mcC_\s^\alpha(\Lambda) \to \D^{\gamma}(\bar\cT^\ex)$ defined by~\cite[Eq.~7.7]{Regularity} using the smooth kernel $R_\mft$.
Let $\mcR$\label{reconstruction} denote the reconstruction operator\footnote{See~\cite[Sec.~6.1]{Regularity}} associated to the model $Z$.

For a choice of $\gamma_\mft,\eta_\mft \in \R$, with $\mft \in \mfL_+$, let us define
\[
\jets^{\gamma,\eta} \eqdef \bigoplus_{\mft \in \mfL_+} \D^{\gamma_\mft,\eta_\mft}(\cT^\ex_{\mft})\;,\label{jets gamma page ref}
\]
which is a subspace of $\jets$ by definition.

To formulate the fixed point map, we introduce the operator
\[
f \mapsto \mcP_{\mft}\one_+f \eqdef (\mcK^\mft_{\gamma_\mft-|\mft|_\s} + R^\mft_{\gamma_\mft}\mcR)\one_+ f\;,
\]
which is a locally Lipschitz map from $\D^{\gamma - | \mft |_{\s},\eta - | \mft |_{\s}}_{(\alpha-|\mft|_\s)\wedge 0}(\widetilde\cT^\ex_\mft)$ to $\D^{\gamma ,\eta \wedge \alpha\wedge|\mft|_\s}_{\alpha\wedge 0}(\cT^\ex_\mft)$ for appropriate $\gamma,\eta, \alpha \in \R$ (see Lemma~\ref{lem:abstractIntegral}).

The direct abstract version of the initial value problem~\eqref{ivp} is given by
\begin{equ}\label{eq: Dgamma fp Naive}
U_{\mft}
=
\mcP_{\mft}
\Big[ 
\one_+ \mcQ_{\leq \gamma_\mft - |\mft|_\s}\Big(\sum_{\mfl \in \mfD_{\mft}}
F^{\mfl}_{\mft}(U)\Xi_{\mfl} \Big) 
\Big]  + G_\mft u_0^{\mft}\;, \quad
\forall \mft \in \mfL_{+}\;.
\end{equ}
In a number of examples, however, one encounters a problem when trying to naively solve~\eqref{eq: Dgamma fp Naive} in $\jets^{\gamma,\eta}$.
The difficulty comes from the fact that some of the terms $\one_+ F^\mfl_\mft(U)\Xi_\mfl$ may take values in a sector 
of regularity $\alpha \leq -\s_{0}$: because of the singularity at $t=0$, the reconstruction operator $\mcR$ 
(and thus the maps $\mcK^\mft_{\gamma_\mft-|\mft|_\s}$ and $\mcP_\mft$) is not a priori well-defined for such 
terms, see~\cite[Prop.~6.9]{Regularity}.
A related difficulty is that (reconstructions of) solutions to~\eqref{eq: Dgamma fp Naive} can be distribution-valued, and thus one needs additional assumptions guaranteeing that they can be evaluated at a fixed time slice;
this is necessary if we wish to restart our fixed point map to obtain a well-posed
notion of maximal solution.
Both of these difficulties already appear in the $\Phi^4_3$ model~\cite[Sec.~9.4]{Regularity}, where they
are dealt with in a somewhat ad hoc manner.
\begin{remark}\label{rem:mcPf}
Note that $\mcP_\mft\one_+ f$ in general makes sense for any singular modelled distribution $f$ for which $\mcR \one_+ f$ can be appropriately defined, see~\cite[Rem.~6.17]{Regularity}; see also~\cite{Mate} where such problems arise on the boundary of the domain.
\end{remark}
\subsection{Renormalised PDEs}\label{subsec: coherence of fp}
%
In the scope of the problems we consider, we wish to apply $\mcP_\mft$ to modelled distributions $\one_+ f \in \D^{\bar\gamma,\bar\eta}_\alpha$ with $\alpha \leq -\s_0$ (as in, e.g., $\Phi^4_d$ with $d \geq 2$);
however it will always be the case that $\bar\eta > -\s_0$ (namely by Assumption~\ref{assump:Schauder1} or~\ref{assump:Schauder2}), so this parameter will not be a problem.
Following Remark~\ref{rem:mcPf}, it suffices to give a canonical definition for $\mcR \one_+ f$ with the expected regularity.

In this subsection, we resolve this issue by assuming that 
the underlying model $Z$ is \emph{smooth} and that 
$\one_+ f \in \D^{\bar\gamma,\bar\eta}_\alpha$ with $\bar\gamma>0$ and $\bar\eta > -\s_0$.
In this case, one can readily see (e.g., by inspecting the proof of~\cite[Prop.~6.9]{Regularity}) that $\mcR \one_+ f$ is canonically defined as a continuous function on $\Lambda\setminus P$ with a blow-up of order $\bar\eta$ at $P$.
As a result, $\mcP_\mft \one_+ f \in \D^{\bar\gamma+|\mft|_\s, (\bar\eta\wedge\alpha)+|\mft|_\s}_{(\alpha+|\mft|_\s)\wedge0}$ is likewise canonically defined.
(In this case, however, $\mcR \one_+ f$ and $\mcP_\mft \one_+ f$ will generally fail to be continuous functions of $f$ and the model!)
In this case, we also note that
\begin{equ}\label{eq:commuteIntegrals}
\mcR \mcP_\mft\one_+ f = G_{\mft} \ast \mcR\one_+ f\;,
\end{equ}
and, for all $x \in \Lambda\setminus P$,
\begin{equ}\label{eq:pointwiseIdentity}
(\mcR \one_+ f)(x) = (\Pi_{x}(\one_+f)(x))(x)\;.
\end{equ}

With these considerations in mind, we can reformulate the purely algebraic result of Theorem~\ref{thm: algebraic main theorem} in the setting of modelled distributions.
\begin{theorem} \label{thm:renormalised_equation}
Fix $F \in \G$, $\xi \in \Omega_{\infty}$, and $M \in \mfR$. 
Let $Z = (\Pi,\Gamma) \in \mathscr{M}_{\infty}$ be the canonical $K$-admissible lift of $\xi$ and 
write $\hat Z = (\hat\Pi,\hat\Gamma)$ for the renormalised model $Z^{M}$ obtained in Theorem~\ref{partial theorem 6.15}. 
For each $\mft \in \mfL_+$, let $\eta_\mft > -\s_0$ and $\gamma_\mft \eqdef \gamma+\reg(\mft)$ for some fixed $\gamma\in\R$.
Suppose that $\gamma_\mft-|\mft|_\s > \gamma_{L}$, with $L \eqdef \overline{L_0}$ and $\gamma_L$ defined as in Section~\ref{subsec:truncations}.

Suppose also that there exists $U \in \jets^{\gamma,\eta}$ with respect to $\hat Z$, defined on an interval $(0,T)$, such that, for each $\mft\in \mfL_+$, there exist $\bar\gamma>0$ and $\bar\eta > -\s_0$ such that $\one_+ \sum_{\mfl\in\mfD_\mft}F^{\mfl}_{\mft}(U)\Xi_{\mfl}$ is an element of $\D^{\bar\gamma,\bar\eta}$.
Suppose finally that $U$ is a solution on $(0,T)$ to the fixed point problem~\eqref{eq: Dgamma fp Naive} with some initial data $u^\mft_0$.

Then for every $\mft \in \mfL_{+}$, the function $u_{\mft} \eqdef \hat{\mcR} U_\mft$ is the unique solution on $(0,T)$
to the stochastic PDE
\[
\d_0 u_{\mft} = 
\mathscr{L}_\mft u_\mft + 
\sum_{(\hat{\mfl},\mfo) \in \mfD_{\mft}}
(MF)^{(\hat{\mfl},\mfo)}_{\mft}(\mathbf{u})
\xi_{\hat{\mfl}}\;,
\]
with initial condition $u^\mft_0$,
where $
\xi_{\hat{\mfl}}
\eqdef
\prod_{(\mfl,\mfe) \in \hat{\mfl}}
D^{\mfe}\xi_{\mfl}$
and the tuple $\mathbf{u} = (u_{o})_{o \in \mcb{O}}$ is given by $u_{(\mfb,q)} \eqdef \partial^{q}u_{\mfb}$. 
\end{theorem}
\begin{proof}
Let $\mft \in \mfL_+$ and consider the expansion of $U_\mft$ as~\eqref{eq:generalU}.
Using the condition $\gamma_\mft-|\mft|_\s > \gamma_L$ to note that $\proj_{\leq L}\mcQ_{\leq \gamma_\mft-|\mft|_\s} = \proj_{\leq L}$, it follows from the definition of $\mcK^\mft_{\gamma_\mft - |\mft|_\s}$ that
\begin{equ}\label{eq:alg fixed point}
\proj_{\leq L} \sum_{\mfl \in \mfD_{\mft}}
F^{\mfl}_{\mft}(U)\Xi_\mfl
=
\proj_{\leq L} U_\mft^R\;.
\end{equ}
By~\eqref{eq:commuteIntegrals}, one has
\begin{equ}
u_{\mft}(x)
=
G_{\mft} 
\ast
\Big[ 
\hat \mcR \one_+
\mcQ_{\leq \gamma_{\mft}-|\mft|_\s}
\sum_{\mfl \in \mfD_{\mft}}
F^{\mfl}_{\mft}(U)\Xi_\mfl
\Big] + G_\mft u^\mft_0\;.
\end{equ}
Since we consider models in $\mathscr{M}_\infty$, we can use~\cite[Rem.~3.15]{Regularity} 
for the term on the right hand side,
which yields for any $x\in\Lambda$
\begin{equs}
\hat\mcR
\Big(
& \one_+
\mcQ_{\leq \gamma_\mft-|\mft|_\s}
\sum_{\mfl \in \mfD_{\mft}}
F^{\mfl}_{\mft}(U)\Xi_{\mfl}
\Big)(x)
=  \hat{\Pi}_x\Big( \one_+(x) \mcQ_{\leq \gamma_\mft-|\mft|_\s}
\sum_{\mfl \in \mfD_{\mft}}
F^{\mfl}_{\mft}(U)(x)\Xi_{\mfl} \Big)(x)
\\
&=  \one_+(x) \hat{\Pi}_x\Big( \proj_{\leq L_0}
\sum_{\mfl \in \mfD_{\mft}}
F^{\mfl}_{\mft}(U)(x)\Xi_{\mfl} \Big)(x)
=  \one_+(x) \hat{\Pi}_x \big(\proj_{\leq L_0}
U^R_\mft(x)
\big)(x)
\\
& =  \one_+(x) \Pi_x \big(\proj_{\leq L_0} M U^R_\mft(x)
\big)(x),
\end{equs}
where the first equality uses~\eqref{eq:pointwiseIdentity}, the second equality uses that $\gamma_\mft - |\mft|_\s > 0$ and that $|\tau|_+ > 0$ for every $\tau \notin \mcW_{\leq L_0}$, and thus $\hat{\Pi}_x(\tau)(x) = 0$, and the third equality uses~\eqref{eq:alg fixed point} and that $L \geq L_0$.
To obtain the final equality, we used the identity $\hat{\Pi}_{x} = \Pi_{x}M$,
combined with the fact that for all $\tau \in \mcT^\ex$ we can write $M\tau = \sum_i \tau_i$ with $|\tau_i|_+ = |\tau|_+$, and thus $ \Pi_x(M\tau)(x) = 0$ for all $\tau \notin \mcW_{\leq L_0}$.
Recalling that $L = \overline{L_0}$ and using again that $|\tau|_+ > 0$ for all $\tau \notin \mcW_{\leq L_0}$, and thus $ \Pi_x(\tau)(x) = 0$, it follows from~\eqref{eq:alg fixed point} and Theorem~\ref{thm: algebraic main theorem} that
\begin{equs}
\Pi_x \big(\proj_{\leq L_0} M U^R_\mft(x)
\big)(x)
&=
\Pi_{x}
\Big(
\sum_{\mfl \in \mfD_{\mft}}
(MF)^{\mfl}_{\mft}(MU(x))
\Xi_{\mfl}
\Big)
(x)
\\
&=
\sum_{(\hat\mfl,\mfo) \in \mfD_{\mft}}
(MF)^{(\hat\mfl,\mfo)}_{\mft}(\mathbf{u}(x))\xi_{\hat\mfl}(x)\;.
\end{equs}
In the first equality we used that $\Pi_x(\proj_{\leq L_0} \tau)(x) = \Pi_x(\tau)(x)$.
For the second equality, suppose that $(MF)^{(\hat\mfl,\mfo)}_\mft$ depends on $\Y_{(\mfb,p)}$, i.e., $D_{(\mfb,p)} (MF)^{(\hat\mfl,\mfo)}_\mft \not \equiv 0$.
Then since $MF$ obeys $R$ by Lemma~\ref{lem:GPreserve}, it holds that $\reg(\mft) < |\mft|_\s + |\Xi_{\hat\mfl}|_- + \reg(\mfb) - |p|_\s$. Since $\gamma_\mft-|\mft|_\s>0$, it follows by definition of $\gamma_\mfb$ that $\gamma_\mfb-|p|_\s>0$.
By~\cite[Prop.5.28]{Regularity} and the fact that $M$ commutes with $\DD^p$, we obtain
\begin{equ}
(\Pi_{x} \DD^p M U_{\mfb}(x))(x)
=
(\hat{\Pi}_{x} \DD^p U_{\mfb}(x))(x)
=
(\hat{\mcR} \DD^p U_\mfb)(x)
 = \partial^p u_{\mfb}(x)\;,
\end{equ}
whence the conclusion follows from Lemma~\ref{lem:canonical models commutes with F}.
\end{proof}
\subsection{Generalised Da Prato--Debussche trick}\label{subsec: gen DPD}
In this subsection, we address the issues discussed at the end of Section~\ref{subsec: modelled distributions} in a way which is stable under taking limits of models.
We do so by finding an appropriate space of modelled distributions and making a few 
additional assumptions on the regularity structure and models, which allow us to perform a 
version of the Da Prato--Debussche trick \cite{MR1941997,DPD2}.
In contrast to Section~\ref{subsec: coherence of fp}, this method does not rely on the smoothness of the underlying model and retains continuity of the fixed point with respect to the model; we reconcile the two viewpoints in Proposition~\ref{prop: underlined U} below.

Consider $\mft \in \mfL_+$.
As in Definition~\ref{def: mft non-vanishing}, we say that $\tau \in \mcT^\ex$ of the form~\eqref{eq:general tree of V} is \emph{$\mft$-non-vanishing for $F$} if $\partial^k \prod_{j=1}^n D_{o_j} F^\mfl_\mft \not\equiv 0$ and $\tau_j$ is $\mft_j$-non-vanishing for every $j=1,\ldots, n$.
Once more, we note that if $\tau$ is $\mft$-non-vanishing, then every subtree\footnote{As before, by a subtree, we mean that $\bar T^{\bar\mfm}_{\bar\mff}$ is a tree whose node and edge sets are subsets of those of $T^{\mfm}_{\mff}$ and whose decorations satisfy $\bar\mff = \mff(e)$ for all $e \in E_{\bar T}$, and $\barmnoise(x) = \mnoise(x)$ and $\barmpoly(x) \leq \mpoly(x)$ for all $x \in N_{\bar T}$.} $\bar\tau$ of $\tau$ with $\rho_{\bar\tau} = \rho_\tau$ is also $\mft$-non-vanishing.
Furthermore, we define
\begin{equs}
\mcT^F_\mft &\eqdef \{\tau \in \widetilde{\mcT}_\mft^\ex : \textnormal{ $\tau$ is $\mft$-non-vanishing}\}\;, 
\quad
\cT^F_\mft \eqdef \Span{\mcT^F_\mft}\;.
\end{equs}
Let $\jets_0 \subset \jets$ denote the subspace of those functions taking values in $\bar\cT^\ex\oplus \bigoplus_{\mft \in \mfL_+} \mcb{I}_{(\mft,0)}[\cT^F_\mft]$.
\begin{remark}
$\mcT^F_\mft$ is in general not invariant under the action of $M\in \mfR$ on $F\in \G$ defined in Section~\ref{sec: alg and main theorem}.
For example, if $F^\mfl_\mft = 0$, then $\Xi_\mfl \notin \mcT^F_\mft$, but it is possible that $(MF)^\mfl_\mft \eqdef \Upsilon^F_\mft[M^*\Xi_\mfl] \neq 0$ and $\Xi_\mfl \in \widetilde\mcT^\ex_\mft$, so that $\Xi_\mfl \in \mcT^{MF}_\mft$ (this can occur naturally when $\mfl = (0,\mfo) \in \mfD_\mft$, i.e., $\Xi_\mfl$ is a ``purely extended decoration'' noise).
However this does not cause issues for proving our main theorem\dash while the notion of $\mft$-non-vanishing is used to ensure that we can solve the fixed point problem associated to $F$, we never try to solve a fixed point problem associated with $MF$.
\end{remark}
\begin{lemma}\label{lem:cTF sector}
Let $\mft\in \mfL_+$.
Then $\bar\cT^\ex\oplus\cT^F_\mft$ and $\bar\cT^\ex\oplus \mcb{I}_{(\mft,0)}[\cT^F_\mft]$ are sectors of $\mathscr{T}$.
Moreover, for every $\mfl \in \mfD_\mft$, $\gamma \in \R$, and $U \in \jets_0$, it holds that $\mcQ_{\leq\gamma}F^\mfl_\mft(U)\Xi_\mfl$ is a map from $\Lambda\setminus P$ to $\bar\cT^\ex \oplus \cT^F_\mft$.
\end{lemma}
\begin{proof}
Consider $\tau \in \mcT^F_\mft$ and write $\Deltap_\ex\mcb{I}_{(\mft,0)}[\tau] = \sum_i \tau^{(1)}\otimes\tau^{(2)}_i$ with $\Deltap_\ex$ defined in~\cite{BHZalg}.
To show that $\bar\cT^\ex\oplus \mcb{I}_{(\mft,0)}[\cT^F_\mft]$ is a sector, it suffices to show that $\tau^{(1)}_i \in \bar\cT^\ex\oplus \mcb{I}_{(\mft,0)}[\cT^F_\mft]$.
If $\tau^{(1)}_i = \mathbf{X}^k$ for some $k \in \N^{d+1}$, this is clear.
Otherwise, we necessarily have $\tau^{(1)}_i = \mcb{I}_{(\mft,0)}[\bar \tau]$, where $\bar\tau \in \widetilde{\mcT}^\ex_\mft$ is a subtree of $\tau$ with $\rho_{\bar\tau} = \rho_\tau$
(indeed, note that the ``driver'' decorations of $\bar\tau$ and $\tau$ necessarily match due to the projection onto the positive trees in the definition of $\Deltap_\ex$).
Since $\tau \in \mcT^F_\mft$ by assumption, it follows that $\bar\tau \in \mcT^F_\mft$, and thus $\bar\cT^\ex\oplus \mcb{I}_{(\mft,0)}[\cT^F_\mft]$.
The argument to show that $\bar\cT^\ex\oplus\cT^F_\mft$ is a sector is similar and simpler.

For the second claim, consider a tree $\tau \in \mcT^\ex$, written as~\eqref{eq:general tree of V}, which appears with a non-zero coefficient in the expansion of $F^\mfl_\mft(U)\Xi_\mfl$.
By Lemma~\ref{lem: nonlinearity on regularity structure}, it suffices to show that $\tau$ is $\mft$-non-vanishing.

We note that since $U \in \jets_0$ it holds that $\tau_i \in \mcT^F_{\mft_i}$ for all $i = 1,\ldots, n$, thus it suffices to show that $\partial^k (\prod_{j=1}^n D_{o_j}) F^\mfl_\mft \not = 0$. 
If $k=0$, this is directly a consequence of the fact that $\tau$ appears with a non-vanishing coefficient.
Similarly, if $k \not = 0$, we know that there must be some $\alpha \in \N^{\mcb{O}}$, $\alpha \not = 0$, such that $D^{\alpha} (\prod_{j=1}^n D_{o_j}) F^\mfl_\mft \not = 0$.
However, this means that $\prod_{j=1}^n D_{o_j} F^\mfl_\mft$ is not a constant and we get the desired result by applying Lemma~\ref{lem:partialK F vanishes}.
\end{proof}
It follows that the natural space in which to solve the fixed point problem~\eqref{eq: Dgamma fp Naive} is
\[
\jets^{\gamma,\eta}_0 \eqdef \jets^{\gamma,\eta} \cap \jets_0 = \bigoplus_{\mft\in \mfL_+} \D^{\gamma_\mft,\eta_\mft}\left(\bar\cT^\ex\oplus \mcb{I}_{(\mft,0)}[\cT^F_\mft]\right).
\]
In order to guarantee that this problem is well-posed we make the following assumption which is natural in view
of the discussion above.
\begin{assumption}\label{assump:planted trees 2}
For every $\mft \in \mfL_+$, every $T = T^{\mfm}_{\mff} \in \widetilde{\mcT}_{\mft}^{\ex}$ which is $\mft$-non-vanishing, and every subtree
$\bar T^{\bar\mfm}_{\bar\mff}$ of $ T $ with $\bar T^{\bar\mfm}_{\bar\mff} \neq T^{\mfm}_{\mff}$ and $\rho_T = \rho_{\bar T}$, one has $|\bar T^{\bar\mfm}_{\bar\mff}|_+ > -(|\mft|_\s \wedge\s_0)$.
\end{assumption}
\begin{remark}
Since every tree in $\widetilde{\mcT}^\ex_\mft$ with a non-zero extended decoration came from contracting a subforest of another tree in $\widetilde{\mcT}^\ex_\mft$ with identically zero extended decorations (see~\cite[Lem.~5.25]{BHZalg}), an equivalent version of Assumption~\ref{assump:planted trees 2} is to replace $|\bar T^{\bar\mfm}_{\bar\mff}|_+ > -(|\mft|_\s \wedge \s_0)$ with $|\bar T^{\bar\mfm}_{\bar\mff}|_- > -(|\mft|_\s \wedge \s_0)$.
\end{remark}
We will see below that Assumption~\ref{assump:planted trees 2} allows us to remove all the planted negative trees in the expansion of the solution to~\eqref{eq: Dgamma fp Naive} and solve for the remainder as a modelled distribution taking values in a function-like sector;
this procedure can be seen as performing the Da Prato--Debussche trick~\cite{DPD2} at the level of modelled distributions.\footnote{While there is similarity to the trick of~\cite{DPD2} our version of the trick plays a different and less central role here: in general the abstract equation we arrive at for our remainder will involve products that,  viewed concretely, are still classically ill-defined.}

For $\mft \in \mfL_+$, define
\begin{equs}[2]
\mcT^F_{\mft,-}
&\eqdef
\mcT^F_\mft \cap \mcT^\ex_{\leq -(|\mft|_\s\wedge\s_0)}\;,
&\quad \cT^F_{\mft,-}
&\eqdef \Span{\mcT^F_{\mft,-}} = \cT^F_\mft \cap \mcb{T}^\ex_{\leq-(|\mft|_\s\wedge\s_0)}\;,
\\
\mcT^F_{\mft,+} &\eqdef \mcT^F_\mft \setminus \mcT^F_{\mft,-}\;,
&\quad
\cT^F_{\mft,+} &\eqdef \Span{ \mcT^F_{\mft,+} }\;.
\end{equs}
We suppose for the remainder of the section that Assumption~\ref{assump:planted trees 2} holds.
\begin{lemma}\label{lem:sectors of our regularity structure2}
Let $\mft \in \mfL_{+}$.
Then $\Gamma \tau = \tau$ for every $\Gamma \in G$ and $\tau \in \mcT^F_{\mft,-}$.
\end{lemma}
\begin{proof}
It suffices to show that $\Deltap_\ex\tau = \tau\otimes 1$ for every $\tau \in \mcT^F_{\mft,-}$.
Writing $\Deltap_\ex\tau = \sum \tau_i^{(1)}\otimes \tau_i^{(2)}$, it holds that $\tau_i^{(1)}$ is a subtree of $\tau$.
Suppose $\tau_i^{(1)}$ is a strict subtree of $\tau$. By Assumption~\ref{assump:planted trees 2}, we have $|\tau_i^{(1)}|_+ \geq -(|\mft|_\s\wedge\s_0)$.
On the other hand, since $|\tau|_+ \leq -(|\mft|_\s\wedge\s_0)$ and $\Deltap_\ex$ preserves the $|\cdot|_+$-degree, 
this would imply that $|\tau^{(2)}_i|_+ < 0$, which is impossible, hence $\Deltap_\ex\tau = \tau\otimes 1$ as desired.
\end{proof}
\begin{lemma}
Let $\mft \in \mfL_{+}$. 
Then $\bar\cT^\ex \oplus \mcb{I}_{(\mft,0)}[\cT^F_{\mft,+}]$ and $\bar\cT^\ex \oplus \cT^F_{\mft,+}$ are sectors of $\mathscr{T}$ of respective regularities $0$ and $-(|\mft|_\s\wedge\s_0)+\kappa$ for some $\kappa > 0$.
\end{lemma}
\begin{proof}
By definition of $\cT^F_{\mft,+}$, we only need to show that $\bar\cT^\ex \oplus \cT^F_{\mft,+}$ and $\bar\cT^\ex \oplus \mcb{I}_{(\mft,0)}[\cT^F_{\mft,+}]$ are sectors.
We only show that the latter is a sector since the argument for the former is identical.

Let $\tau \in \mcT^F_{\mft,+}$.
Writing $\Deltap_\ex\mcb{I}_{(\mft,0)}[\tau] = \sum \tau_i^{(1)}\otimes \tau_i^{(2)}$, it suffices to show that $\tau_i^{(1)} \in \bar\cT^\ex \oplus \mcb{I}_{(\mft,0)}[\cT^F_{\mft,+}]$.
If $\tau_i^{(1)}=\mathbf{X}^k$ for some $k \in \N^{d+1}$, this is clear.
Otherwise $\tau_i^{(1)} = \mcb{I}_{(\mft,0)}[\bar\tau]$ for some subtree $\bar\tau$ of $\tau$ with $\rho_{\bar\tau} = \rho_\tau$.
In particular, $\bar\tau$ is $\mft$-non-vanishing.
If $\bar\tau=\tau$, then evidently $\tau^{(1)}_i \in \mcb{I}_{(\mft,0)}[\cT^F_{\mft,+}]$.
If $\bar\tau\neq\tau$, then, by Assumption~\ref{assump:planted trees 2}, $|\bar\tau|_+ > -(|\mft|_\s\wedge\s_0)$, and so again $\tau^{(1)}_i \in \mcb{I}_{(\mft,0)}[\cT^F_{\mft,+}]$ as desired.
\end{proof}
\begin{lemma}
$\Upsilon^F_\mft[\tau]$ is a constant for every $\mft \in \mfL_+$ and $\tau \in \mcT^F_{\mft,-}$.
\end{lemma}
\begin{proof}
Suppose that $\Upsilon^F_\mft[\tau]$ is not constant for some $\mft \in \mfL_+$ and $\tau \in \mcT^F_{\mft}$.
Consider $\hat\uparrow_i$ from Definition~\ref{def_increasing_poly} below and write $\hat\uparrow_i\tau = \sum_{j}c_j\tau_j$, where $c_j \in \R$, $\tau_j \in \widetilde{\mcT}^\ex_{\mft}$, and $\tau$ is a strict subtree of $\tau_j$ with $\rho_\tau = \rho_{\tau_j}$.
It follows from Lemma~\ref{lem:partialK F vanishes} that $\partial_i \Upsilon^F_\mft[\tau] \not\equiv 0$ for every $i \in \{0,\ldots, d\}$, and thus, by Corollary~\ref{cor:graftMorphismSmall} below, $\Upsilon^F_\mft[\hat\uparrow_i\tau] \not\equiv 0$.
Hence, for some $j$, $\Upsilon^F_\mft[\tau_j] \not\equiv 0$ and thus $\tau_j$ is $\mft$-non-vanishing.
It follows by Assumption~\ref{assump:planted trees 2} that $|\tau|_+ > -(|\mft|_\s\wedge\s_0)$, which concludes the proof.
\end{proof}
Let $\jets_+$ denote the space of functions from $\Lambda\setminus P$ to $\bigoplus_{\mft \in \mfL_+} \bigl(\bar\cT^\ex\oplus \mcb{I}_{(\mft,0)}[\cT^F_{\mft,+}]\bigr)$, noting that $\jets_{+}$ is a subspace of $\jets_0$.
Let $\tilde U \in \jets_0$ denote the unique constant function for which, for every $\mft\in\mfL_+$ and $\tau \in \mcT^\ex$,
\begin{equ}\label{equ: definition of stationary shrubs}
\langle\tilde U_\mft,\tau\rangle =
\begin{cases}
\Upsilon^F_\mft[\tau] &\textnormal{ if } \tau \in \mcT^F_{\mft,-}\;;
\\
0 &\textnormal{ otherwise}\;;
\end{cases}
\end{equ}
(in particular, $\tilde U$ takes values in $\mcb{T}^F_{\mft,-}$).

\begin{lemma}\label{constant sector}
For any $V \in \jets_+$, if one defines $U \in \jets_0$ via $U_\mfb \eqdef V_\mfb + \mcb{I}_{(\mfb,0)}[\tilde U_\mfb]$ for each $\mfb \in \mfL_{+}$, then one has, for every $\mft \in \mfL_+$, 
\[
\mcQ_{\leq -(|\mft|_\s\wedge\s_0)}\sum_{\mfl \in \mfD_{\mft}}F^\mfl_\mft (U) \Xi_\mfl = \tilde U_\mft\;.
\]
\end{lemma}
\begin{proof}
Fix $\mft \in \mfL_{+}$ and $\mfl \in \mfD_{\mft}$. Suppose one has a tree $\tau \in \mcT^\ex$, written as~\eqref{eq:general tree of V}, which appears with a non-zero coefficient in the expansion of $F^\mfl_\mft (U) \Xi_\mfl$.
Note that $\tau \in \mcT^F_\mft$ by Lemma~\ref{lem:cTF sector}.

Suppose that $|\mcb{I}_{(\mft_i,p_i)}[\tau_i]|_+ \geq 0$ for some $i = 1,\ldots, n$.
Let $\bar\tau$ be the strict subtree of $\tau$ formed by removing the branch $\mcb{I}_{(\mft_i,p_i)}[\tau_i]$ from $\tau$.
Then, by Assumption~\ref{assump:planted trees 2}, $|\bar\tau|_+ > -(|\mft|_\s\wedge\s_0)$, which in particular implies that $|\tau|_+ > -(|\mft|_\s\wedge\s_0)$.
Likewise, suppose $k \neq 0$.
Let $\bar\tau$ be the strict subtree of $\tau$ formed by setting the polynomial decoration at the root of $\tau$ to zero.
Then, by Assumption~\ref{assump:planted trees 2}, $|\bar\tau|_+ > -(|\mft|_\s\wedge\s_0)$, which again implies $|\tau|_+ > -(|\mft|_\s\wedge\s_0)$.
Therefore, for every $\mft \in \mfL_{+}$, $\mfl \in \mfD_{\mft}$, and every $\tau$ appearing in the expansion of $F^\mfl_\mft (U) \Xi_\mfl$ with $|\tau|_+ \leq -(|\mft|_\s\wedge\s_0)$, it holds that $\tau$ has no polynomial decoration at the root or branches of non-negative $|\cdot|_+$-degree.
It follows that
\begin{equ}[e:formRHS]
\mcQ_{\leq -(|\mft|_\s\wedge\s_0)} \sum_{\mfl \in \mfD_{\mft}}F^\mfl_\mft(U)\Xi_\mfl
= \mcQ_{\leq -(|\mft|_\s\wedge\s_0)} \sum_{\mfl \in \mfD_{\mft}}F^\mfl_\mft((\mcb{I}_{(\mfb,0)}[\tilde U_\mfb]+\langle V_\mfb,\mathbf{1}\rangle \mathbf{1})_{\mfb\in\mfL_+})\Xi_\mfl\;.
\end{equ}
Since the RHS of~\eqref{e:formRHS} does not depend on the coefficient in $U$ of any tree of positive order, we may assume without loss of generality that $U$ is coherent.
It then follows from Lemma~\ref{lemma: coherence identity} that the LHS is precisely $\tilde U_\mft$.
\end{proof}
In order to work ``at stationarity'' as described in Remark~\ref{rem:stationarity} we want to define, for each $\mft \in \mfL_{+}$ and $\tau \in \mcT^F_{\mft,-}$, $\mcP_{\mft}\tau$ as an appropriately continuous function of the underlying model.
This requires some work since the action of $\mcP_{\mft}$ may not be local.

Throughout this section we have assumed that we have fixed differential operators $\{\mathscr{L}_{\mft}\}_{\mft \in \mfL_{+}}$ as in the beginning of Section~\ref{sec: kernels on torus} and then a truncation of the corresponding Green's functions as described at the beginning of Section~\ref{subsec: admissible models}.
Using an appropriately designed partition of unity, we assume that for each $ \mft \in \mfL_{+}$ we have fixed a decomposition $R_{\mft} = \sum_{m = 0}^{\infty} R_{\mft,m}$ where for each $\mft \in \mfL_{+}$ and $m \in \Z$ one has $R_{\mft, m}$ smooth and supported on $\mathfrak{K}_{m}$ (where $\K_{m}$ was defined in \eqref{compact set def}), and such that for each $k \in \N^{d+1}$ one has, for some $\chi > 0$,
\[
\sup_{m \in \N}
\sup_{z \in \Lambda}
e^{\chi m}
\cdot
|(D^{k}R_{\mft,m})(z)|
<
\infty\;.
\]
With these notations, one has the following straightforward fact.
\begin{lemma}\label{lem:definition of r}
For any $Z \in \mathscr{M}_{0,1}$, $\mft \in \mfL_{+}$, and $\tau \in \mcT^F_{\mft,-}$, 
\begin{equ}\label{remainder of heat kernel}
r^{Z}_{\mft,\tau}
\eqdef
\lim_{N \rightarrow \infty}
\sum_{j = 0}^{N}
R_{\mft,j} \ast (\mcR^{Z}\tau)
\end{equ}
converges in $\CC^{\infty}(\Lambda)$.
Moreover, the map $Z \mapsto r^{Z}_{\mft,\tau}$ is a continuous map from $\mathscr{M}_{0,1}$ into $\CC^{\infty}(\Lambda)$. 
\end{lemma}
\begin{proof}
We fix $\mft \in \mfL_{+}$ and $\tau \in \mcT^F_{\mft,-}$.
By Lemma~\ref{lem:sectors of our regularity structure2}, the structure group acts trivially on $\tau$, and so for any $Z = (\Pi,\Gamma) \in \mathscr{M}_{0}$ one has $(\mcR^{Z}\tau)(\cdot) = \Pi_{z}\tau(\cdot) \in \CC_{\s}^{|\tau|_{+}}$ where $z \in \Lambda$ is arbitrary.

One immediately has the bounds 
\begin{equ}\label{reconstruction bounds for lemma}
\|
\mcR^{Z}\tau\|_{|\tau|_{+},\K} \lesssim \|Z\|_{\K}\; 
\enskip
\textnormal{and}
\enskip
\|
(\mcR^{Z} - \mcR^{\bar{Z}})\tau\|_{|\tau|_{+},\K} \lesssim \|Z ; \bar{Z}\|_{\K}\;,
\end{equ}
uniform in the choice of compact set $\K \subset \Lambda$, and $Z, \bar{Z} \in \mathscr{M}_{0}$. 
The norms on the LHS's of~\eqref{reconstruction bounds for lemma} are those of \eqref{Holder Besov seminorm}.
Below, all of our estimates are uniform in $Z \in \mathscr{M}_{0}$.
It is straightforward to see that one has the bound
\[
|
(\mcR^{Z}\tau)(f)|
\lesssim
\$Z\$_{\K_{n}}
\sup_{
\substack{
k \in \N^{d+1}\\
|k|_{\s} < - |\tau|_{+} + 1
}
}
|D^{k}f(z)|\;,
\] 
uniformly in $n \in \Z$, and over all test functions $f$ supported on $\mathfrak{K}_{n}$.
Therefore, for any $n \in \Z$, $k \in \N^{d+1}$, and uniformly over $z \in \K_{n}$, one has the estimate
\[
\sum_{j = 0}^{N}
|
(D^{k}R_{\mft,j} \ast (\mcR^{Z}\tau))(z)|
\lesssim
\sum_{j = 0}^{N}
e^{-\chi j}( \$Z\$_{\K_{n - j - 1}} + \$Z\$_{\K_{n - j }} + \$Z\$_{\K_{n - j + 1}})\;.
\]
Clearly if $Z \in \mathscr{M}_{0,1}$ the RHS above is absolutely convergent as one takes $N \rightarrow \infty$.
This establishes the convergence of \eqref{remainder of heat kernel} in $\CC^{\infty}(\Lambda)$. 
The statement about continuity follows by using the second bound of \eqref{reconstruction bounds for lemma} as input for the same argument.
\end{proof}
\begin{remark}
Here and in the rest of the section, we use $Z$ as a superscript when we want to stress the dependence of some object on the underlying model $Z$.
\end{remark}
The following result is immediate from Lemmas~\ref{lem:sectors of our regularity structure2} and~\ref{lem:definition of r}.
\begin{proposition}\label{prop:stationary objects} 
Let $Z \in \mathscr{M}_{0,1}$, $\mft \in \mfL_+$, and $\tau \in \mcT^F_{\mft,-}$.
Then the constant function $z \mapsto \mcb{I}_{(\mft,0)}[\tau]$ is an element of $\D^\infty$.
Moreover, there exists a smooth $r^{Z}_{\tau,\mft} \in \mcC^{\infty}(\Lambda)$, which we treat canonically as an element of $\D^\infty(\bar{\mcb{T}}^\ex)$, with the following properties.
\begin{enumerate}
\item Setting,
\begin{equ}
\mcP^{Z}_\mft\tau \eqdef \mcb{I}_{(\mft,0)}[\tau] + r^{Z}_{\tau,\mft} \in \D^{\infty}(\bar\cT^\ex \oplus \mcb{I}_{(\mft,0)}[\cT^F_{\mft,-}])\;,
\end{equ}
the distribution $f^{Z}_{\tau,\mft} \eqdef \mcR^{Z} \mcP^{Z}_\mft \tau \in \mcC^{|\tau|_+ + |\mft|_\s}(\Lambda)$ solves
\[
\partial_t f^{Z}_{\tau,\mft} = \mathscr{L}_\mft f^{Z}_{\tau,\mft} + \mcR^{Z}\tau\;.
\]
\item The map $Z \mapsto f^{Z}_{\tau,\mft}$ is continuous with respect to the metric on $\mathscr{M}_{1,\infty}$.
\end{enumerate}
\end{proposition}
A consequence of Proposition~\ref{prop:stationary objects} is that, for any $Z \in \mathscr{M}_{0,1}$, we can define $\mcP^{Z}_\mft \tilde U_\mft \in \D^\infty$, and thus $\mcP^{Z}\tilde U \in \jets_0$ by setting $(\mcP^{Z}\tilde U)_\mft \eqdef  \mcP^{Z}_\mft \tilde U_\mft$.
Moreover, the map $Z \mapsto \mcP^{Z}\tilde{U}$ is a continuous map from $\mathscr{M}_{0,1}$ to $\bigoplus_{\mft \in \mfL_{+}} \CD^{\infty}$. 

Rather than seeking a solution $U \in \jets$ to~\eqref{eq: Dgamma fp Naive}, we instead treat $U$ as a perturbation of the stationary solution by writing
\begin{equ}\label{def: the continuous guy U}
U_\mft = V_\mft + \mcP_\mft \tilde U_\mft\;,\quad 
\forall \mft \in \mfL_{+}\;,
\end{equ}
where $V_\mft$ is function-like.
More precisely, let us fix
\begin{equs}
\gamma_{\mft} \eqdef \gamma + \reg(\mft)\;, \quad \eta_{\mft} \eqdef \eta + \ireg(\mft)\;,
\end{equs}
for some $\gamma, \eta \in \R$, and define the space
\[
\jets^{\gamma,\eta}_{+} \eqdef \jets_+ \cap \jets_0^{\gamma,\eta} = \bigoplus_{\mft \in \mfL_+} \D^{\gamma_\mft,\eta_\mft}\left(\bar\cT^\ex\oplus\mcb{I}_{(\mft,0)}[\cT^F_{\mft,+}]\right)\;.
\]  
For $\mft \in \mfL_+$, let $\widetilde{\jets}_\mft$ denote the space of all maps $U: \Lambda\setminus P \rightarrow \widetilde{\mcb{T}}_\mft^{\ex}$ (so that $\widetilde{\jets} = \bigoplus_{\mft\in\mfL_+}\widetilde{\jets}_\mft$), and consider the map $H_\mft : \jets \to \widetilde{\jets}_\mft$ given by
\begin{equ}
H_\mft(V) \eqdef \mcQ_{\leq \gamma_\mft - |\mft|_\s} \sum_{\mfl \in \mfD_{\mft}} F^\mfl_\mft(V + \mcP\tilde U) \Xi_\mfl - \tilde U_\mft\;.
\end{equ}
The following lemma makes precise the gain in regularity obtained by considering the remainder $V$.
For $\mft\in\mfL_+$ and $\mfl \in \mfD_{\mft}$, define the quantity $\bar n^\mfl_\mft$ as in~\eqref{eq:nlt def} but with $|\mfl|_\s$ replaced by $\homplus{\mfl}$.
Define further $\bar n_\mft \eqdef \min_{\mfl \in \mfD_{\mft}} n^\mfl_\mft$.
\begin{lemma}\label{lem:DgammaMult}
Let $0 \leq \eta \leq \gamma$ and $\mft \in \mfL_+$.
Then there exists $\kappa_\mft > 0$ sufficiently small, depending only on the rule $R$ and functions $\reg$ and $\ireg$, such that $H_\mft$ is a locally Lipschitz map
\[
\jets^{\gamma,\eta}_+ \to
\D^{\gamma_\mft - |\mft|_\s+ \kappa_\mft, \eta + \bar n_\mft}(\bar\cT^\ex \oplus \cT^F_{\mft,+})\;.
\]
\end{lemma}
%
\begin{proof}
Note that, for every $V \in \jets_+$, $H_\mft(V)$ is indeed a function from $\Lambda\setminus P$ to $\bar\cT^\ex \oplus \cT^F_{\mft,+}$ due to Lemma~\ref{constant sector}.
Fix $\mfl \in \mfD_{\mft}$ and consider a term $\bar F(\Y)\Y^\alpha$ in the expansion~\eqref{expansion of nonlinearity} of $F^\mfl_\mft$.
Write $U \eqdef V+\mcP\tilde U$. Since $\tilde U_\mft \in \D^{\infty,\infty}$, it remains only to show that
\begin{equs}
\mcQ_{\leq \gamma_\mft - |\mft|_\s}\bar F(U)\mathbf{U}^\alpha\Xi_\mfl \in \D^{\gamma_\mft - |\mft|_\s+ \kappa_\mft, \eta + \bar n_\mft}\;.\label{eq:bar F in D gamma}
\end{equs}
Suppose first that $\bar F$ is not identically constant.
Note that, by Lemma~\ref{lem:sectors of our regularity structure}, $\cT^\ex_\mfb$ is a sector of regularity $\reg(\mfb) \wedge 0$. Note also that
\[
\DD^p U_\mfb \in \D^{\gamma+\reg(\mfb,p),\eta_\mfb-|p|_\s}_{\reg(\mfb,p)\wedge 0}(\DD^p\cT^\ex_\mfb)\;.
\]
If $\reg(\mfb,p) > 0$, then the sector $\DD^p\cT^\ex_{\mfb}$ is function-like.
Recall also that in this case $0 \leq \eta \leq \eta_\mfb - |p|_\s \leq \gamma + \reg(\mfb,p)$. 
It follows from~\cite[Prop.~6.13]{Regularity} that
\[
\mcQ_{< \gamma} \bar F(U) \in \D^{\gamma, \eta}_0\;.
\]
Writing
\[
\mathbf{U}^\alpha = \prod_{(\mft,p) \in \alpha} \DD^{p}U_{\mft}, \; \; \DD^{p}U_{\mft} \in \D^{\gamma + \reg(\mft,p), \eta + \ireg(\mft,p)}_{\reg(\mft,p)},
\]
it follows from~\cite[Prop.~6.12]{Regularity} that
\[
\mathbf{U}^\alpha \in \D^{\gamma + \sum_{o \in \alpha}\reg(o), \eta + \sum_{o \in \alpha} \reg(o)\wedge\ireg(o)}_{\sum_{o \in \alpha}\reg(o)}\;.
\]
Finally, note that $\Xi_\mfl \in \D^{\infty,\infty}_{\homplus{\mfl}}$.
Combining everything, we obtain
\[
\mathbf{U}^\alpha\Xi_\mfl\mcQ_{< \gamma} \bar F(U)
\in \D^{\gamma + \homplus{\mfl} + \sum_{o \in \alpha}\reg(o), \eta + \homplus{\mfl} + \sum_{o \in \alpha} \reg(o)\wedge\ireg(o)}_{\homplus{\mfl}+\sum_{o \in \alpha}\reg(o)}\;.
\]
Since $F$ obeys the rule $R$, we can find $\kappa_\mft > 0$ such that
\begin{equs}
\reg(\mft) - |\mft|_\s + \kappa_\mft \leq&  \homplus{\mfl} + \sum_{o \in \alpha}\reg(o)\;.
\end{equs}
By considering the regularity of the relevant sectors (and decreasing $\kappa_\mft$ if necessary), we see that
\[
\mcQ_{\leq \gamma_\mft-|\mft|_\s} \bar F(U)\mathbf{U}^\alpha \Xi_\mfl
=
\mcQ_{< \gamma_\mft-|\mft|_\s + \kappa_\mft} \left[\mathbf{U}^\alpha\Xi_\mfl\mcQ_{< \gamma} \bar F(U)\right]\;,
\]
which proves~\eqref{eq:bar F in D gamma}.

Suppose now that $\bar F$ is identically constant.
Then expanding $(V+\mcP\tilde U)^\alpha\Xi_\mfl$, the term $(\mcP\tilde U)^\alpha\Xi_\mfl$ is an element of $\D^{\infty}$.
On the other hand, using that $V_\mfb \in \D^{\gamma + \reg(\mfb),\eta + \ireg(\mfb)}_0$,
we see that every other term is in $\D^{\gamma + \homplus{\mfl} + \sum_{o \in \alpha} \reg(o), \eta + \bar n_\mft}$, which again proves~\eqref{eq:bar F in D gamma}.
\end{proof}
In light of the above lemmas, it is natural to consider an analogue of Assumption~\ref{assump:Schauder1}.
\begin{assumption}\label{assump:Schauder2}
Assumption~\ref{assump:Schauder1} holds with $n_\mft$ replaced by $\bar n_\mft$.
\end{assumption}
We now take $\gamma$ sufficiently large and $\eta_\mft = \ireg(\mft)$, and write the fixed point problem for the remainder $V$ in the space $\jets^{\gamma,\eta}_+$, with initial condition $v^\mft_s$ at time $s \geq 0$, as
\begin{equ} \label{eq: Dgamma fp}
V_{\mft} = \mcP_\mft \Big[\one_{+}^s H_{\mft}(V) \Big]  + G_\mft v^{\mft}_s\;,\quad 
\forall \mft \in \mfL_{+}\;,
\end{equ}
where $\one_{+}^s$ denotes the indicator function of the set $\{(t,x) \in \Lambda : t > s\}$.
It follows from Lemmas~\ref{lem:abstractIntegral},~\ref{lem:initialCond}, and~\ref{lem:DgammaMult}, as well as Assumption~\ref{assump:Schauder2}, that the fixed point problem~\eqref{eq: Dgamma fp} is well-posed and admits local solutions in the space $\jets^{\gamma,\eta}_{+}$ for any initial condition $(v_0^\mft)_{\mft\in\mfL_+} \in \mcC^{\ireg}$.
Moreover, since $V_\mft$ takes values in a function-like sector, the formulation~\eqref{eq: Dgamma fp} allows us to restart the fixed point to obtain maximal solutions,\footnote{The fact that local solutions can be patched together in a consistent way follows from an argument identical to~\cite[Prop.~7.11]{Regularity}} i.e., up to the blow-up time of $\mcR V$.

Note that we restricted most of our discussion above to one fixed model $Z$.
One can of course extend all the results to obtain continuity properties of the fixed point with respect to the model
(the only extension which doesn't immediately follow from~\cite{Regularity} is Lemma~\ref{lem:DgammaMult}, for which one can use~\cite[Prop.~3.11]{WongZakai}).
We summarise the above discussion along with the remaining necessary results from~\cite{Regularity} in the following theorem.
\begin{theorem}\label{thm: regularity structures results}
Let $\gamma \in \R$, and set $\gamma_\mft \eqdef \gamma+\reg(\mft)$ and $\eta_\mft \eqdef \ireg(\mft)$.
Suppose that Assumptions~\ref{assump:planted trees 2} and~\ref{assump:Schauder2} hold, and that $\gamma_\mft - |\mft|_\s > 0$ and $\gamma_\mft,\eta_\mft \notin \N$ for all $\mft \in \mfL_+$.
Then the following statements hold.
\begin{enumerate}
\item For any model $Z = (\Pi,\Gamma) \in \mathscr{M}_{0,1}$ and periodic initial data $v_0 = (v_0^{\mft})_{\mft \in \mfL_+} \in 
\mathcal{C}^{\ireg}$, the fixed point problem~\eqref{eq: Dgamma fp}
is well posed and admits a local in time solution $V^{Z} \in \jets^{\gamma,\eta}_+$.
\item It holds that $\mcR V \in \CC^{\rem}$, and $V$ is defined on the interval $(0,T[\mcR V])$.
\item The map $(v_0,Z) \mapsto \mcR^Z V^{Z}$ is continuous from $\mathcal{C}^{\ireg} \times \mathscr{M}_{0,1}$ into $\CC^{\rem}$ when $\mathscr{M}_{0,1}$ is equipped with the metric $\$\cdot;\cdot\$$.
\end{enumerate}
\end{theorem}
It remains to connect the remainder $V$ with some abstract fixed point equation to which we can apply Theorem~\ref{thm:renormalised_equation}.
For simplicity, we will only do this in the case where $Z \in \mathscr{M}_{\infty,1}$ so that the reconstruction of all relevant modelled distributions are continuous functions.
Note that, in this case, one can canonically define $\mcP^{Z}_\mft\one_+\tilde U_\mft$, and that the distributions $f_{\tau,\mft}$ from Proposition~\ref{prop:stationary objects} are in fact smooth functions.
In the following result, we implicitly restrict all modelled distributions to the domain $(-\infty,T]\times \T^d$ where $T > 0$ is such that $V^Z$ blows up after time $T$.
\begin{proposition}\label{prop: underlined U} 
Suppose we are in the setting of Theorem~\ref{thm: regularity structures results}.
Let $Z \in \mathscr{M}_{\infty,1}$, $v_0 \in \mcC^{\ireg}$, and consider the functions $U^{Z} \eqdef V^{Z} + \mcP^{Z}\tilde U \in \jets^{\gamma,\eta}$ and $\overline U^Z\eqdef \one_+ U^Z$.
For every $\mft \in \mfL_+$, set
\[
\bar{u}^{Z,\mft}_0 \eqdef \sum_{\tau\in\mcT^F_{\mft,-}} \frac{\Upsilon^F_\mft[\tau]}{\langle \tau,\tau \rangle} f_{\tau,\mft}(0,\cdot) \in \mcC^\infty(\T^d)\;.
\]
It then holds that
\begin{equ}\label{eq:overline U identity}
\overline{U}^{Z}_\mft = V^{Z}_\mft + \mcP^{Z}_\mft\one_+\tilde U_\mft + G_\mft \bar{u}^{Z,\mft}_0\;,\quad \forall \mft \in \mfL_+\;.
\end{equ}
Furthermore, $\overline{U}^{Z}$ solves the fixed point problem
\begin{equ}\label{eq:overline U fp}
\overline{U}^{Z}_\mft = \mcP^{Z}_\mft\Big[\one_+\mcQ_{\leq\gamma_\mft-|\mft|_\s} \sum_{\mfl \in \mfD_\mft}F^\mfl_\mft(\overline{U}^{Z})\Xi_\mfl\Big]
+ G_\mft v_0^\mft + G_\mft\bar{u}_0^{Z,\mft}\;,\quad \forall \mft \in \mfL_+\;.
\end{equ}
In particular, $\overline{U}^{Z}$ falls under the scope of Theorem~\ref{thm:renormalised_equation}.
\end{proposition}

\begin{proof}
To show~\eqref{eq:overline U identity} we show that $\mcP^Z_\mft \one_+ \tilde U + G_\mft \bar u^{Z,\mft}_0 = \one_+ \mcP^Z \tilde U$.
It follows directly from definitions that the two sides agree on all non-polynomial trees, so it suffices to show that their reconstructions coincide (see~\cite[Prop.~3.29]{Regularity}).
However, we see that the reconstruction of either side satisfies the PDE
\[
\partial_t u = \mathscr{L}_\mft u + \mcR^Z \tilde U_\mft
\]
for all $t > 0$ with initial condition given by $u_0 = \bar u^{Z,\mft}_0$, and thus must be equal.

We now check that $\overline{U}^{Z}$ satisfies~\eqref{eq:overline U fp}.
It holds that
\begin{equs}
\overline{U}^{Z}_\mft
&= V^{Z}_\mft + \mcP^{Z}_\mft\one_+\tilde U_\mft + G_\mft\bar{u}_0^{Z,\mft}
\\
&= \mcP^{Z}_\mft\one_+\Big[\mcQ_{\leq\gamma_\mft-|\mft|_\s} \Big(\sum_{\mfl \in\mfD_\mft} F^\mfl_\mft(V^{Z} + \mcP^Z\tilde U)\Xi_\mfl\Big) - \tilde U_\mft\Big] + \mcP_\mft^Z\one_+\tilde U_\mft\\
&\qquad  + G_\mft v^\mft_0 + G_\mft\bar{u}^{Z,\mft}_0
\\
&= \mcP^{Z}_\mft\Big[\one_+\mcQ_{\leq\gamma_\mft-|\mft|_\s}\sum_{\mfl \in\mfD_\mft} F^\mfl_\mft(V^{Z} + \mcP^{Z}\tilde U)\Xi_\mfl \Big] + G_\mft v_0^\mft + G_\mft\bar{u}^{Z,\mft}_0\;.
\end{equs}
It remains to observe that $\one_+F^\mfl_\mft(V^{Z}+\mcP^{Z}\tilde U) = \one_+ F^\mfl_\mft(\overline{U}^{Z})$, which readily follows from the identity~\eqref{eq:overline U identity}.
\end{proof}
\appendix
\section{Additional proofs}\label{sec: proofs of earlier lemmas}
\subsection{Proof of Lemma~\ref{lem:GPreserve}}\label{subsec:GPreserveProof}
\begin{proof}[of Lemma~\ref{lem:GPreserve}]
Let $F \in \G$ and $M \in \mfR$. Suppose that for $\mft \in \mfL_+$, $\mfl=(\hat\mfl,\mfo) \in \mfD_{\mft}$ and $\tau \in \mcT^\ex$ one has $\langle M^* \Xi_\mfl, \tau\rangle \neq 0$. 
Let $\alpha \in \N^{\mcb{O}}$ such that $\mfo \notin D(\mft,\alpha\sqcup\hat{\mfl})$. 
To conclude that $MF \in \G$ it suffices, by Proposition~\ref{prop:obeyEquivalence}, to show that $D^\alpha \Upsilon^F_\mft[\tau] = 0$.

To this end, let us add an additional ``driver'' element $\check\Xi$ and construct the spaces $\check\VV$ and $\check\VVspan$ in the identical manner to $\VV$ and $\VVspan$ but instead using the set $\check\mfD \eqdef \mfD \sqcup \{\check\Xi\}$ in place of $\mfD$.
Note that we can canonically identify $\VVspan$ with a subspace of $\check\VVspan$.
We also set $\Upsilon^F_{\mfb}[\check\Xi] \eqdef 1$ for all $\mfb \in \mfL_+$.

Let us write $\alpha$ as a multi-set $\alpha = \{(\mft_1,p_1),\ldots, (\mft_k,p_k)\}$ for some $k \geq 0$ and $(\mft_j,p_j) \in \mcb{O}$.
Consider the element
\[
\check\tau \eqdef \check\Xi \hgraft{(\mft_1,p_1)} (\check\Xi \hgraft{(\mft_2,p_2)} (\ldots (\check\Xi \hgraft{(\mft_n,p_n)} \tau)\ldots)) \in \check\VVspan. 
\]
Note that in the case $k=0$, one simply has $\check\tau = \tau$.
By~\eqref{eq:graftSmall1} we have
\[
D^\alpha \Upsilon^F_\mft[\tau]
= D_{(\mft_1,p_1)}\ldots D_{(\mft_n,p_n)}\Upsilon^F_\mft[\tau]
= \Upsilon^F_\mft[\check\tau]\;.
\]
Write $\check \tau$ as a sum of trees $\check\tau = \sum_{i=1}^N c_i \check\tau_i$ with $\check\tau_i \in \check\VV$. Observe that due to the choice $\Upsilon^F[\check\Xi] = 1$, for every $\check\tau_i$ with an edge whose two adjacent nodes carry the label $\check\Xi$, it holds that $\Upsilon^F[\check\tau_i] = 0$.

Consider a tree $\check\tau_i$ in which every edge has at most one adjacent node with the label $\check\Xi$. We may identify $\check\tau_i$ with an element $\tau_i \in \VV$ by mapping the label $\check\Xi$ to $1$.
However, by the assumption that the rule $R$ is complete, that $\mfo \notin D(\mft,\alpha\sqcup\hat{\mfl})$, and that $\langle M^*\Xi_\mfl, \tau\rangle \neq 0$, it necessarily holds that $\mcb{I}_{(\mft,0)}[\tau_i] \notin \mcT^\ex$.
Therefore, there exists a non-leaf node in $\mcb{I}_{(\mft,0)}[\tau_i]$ with label $\Xi_{(\bar\mfl,\bar\mfo)}$, incoming edge $\bar\mft$, and a multi-set of outgoing edges $\beta \in \N^{\mcb{O}}$ with $\bar\mfo \notin D(\bar\mft,\beta\sqcup\{\Xi_{\bar\mfl}\})$.
Again by Proposition~\ref{prop:obeyEquivalence}, it holds that $D^\beta F^{\bar\mfl}_{\bar\mft} = 0$, and thus $\Upsilon^F_\mft[\check\tau_i]=0$, which concludes the proof.
\end{proof}
\subsection{Proof of Lemma~\ref{lem:UpsilonSum}}\label{proof lemma Upsilon sum identity}
The key ingredient for establishing the lemma is the following multi-variable generalisation of the Faa di Bruno formula. In order to state this formula we first introduce some more notation. 

We fix some choice of a total order ``$<$'' on the set $\N^{d+1}$ with the property that $0$ is the minimal element.
Then for each $r \in \N$ and $k \in \N^{d+1} \setminus \{0\}$ we define the set 
\[
I(r,k)
\eqdef
\left\{
(\vec{q},\vec{m})
\in (\N^{d+1})^{r} \times (\N^{\mcb{O}} \setminus \{0\})^{r}:\ 
\begin{array}{l}
0 < q_{1} < q_{2} < \cdots < q_{r},\\ 
\sum_{j=1}^{r}|m_{j}| \cdot q_{j} = k
\end{array}
\right\}\;.
\] 
We also set $I(k) \eqdef \bigsqcup_{r=0}^{\infty} I(r,k)$. 
Additionally, for $(\vec{q},\vec{m}) \in I(r,k)$ we use the shorthands $r_{(\vec{q},\vec{m})} = r$ and $m \eqdef \sum_{j=1}^{r} m_{j}$. 
Note that $I(0,k) = \emptyset$ except for the case $k = 0$ when $I(k) = \{(0,0)\}$ with $r_{(0,0)} = 0$. 

We can now state the mentioned Faa Di Bruno formula.
\begin{lemma}\label{lem: Faa Di Bruno}
For any $k \in \N^{d+1}$ and $F \in \Poly$ one has
\begin{equ}\label{Faa di Bruno formula}
\partial^{k}
F
=
k!
\sum_{(\vec{q},\vec{m}) \in I(k)}
\Big[
\prod_{
\substack{
1 \le j \le r_{(\vec{q},\vec{m})}\\
(\mft,p) \in \mcb{O}
}
}
\frac{1}{m_{j}[(\mft,p)]!}
\Big(
\frac{1}{q_{j}!} \Y_{(\mft,p+q_{j})}
\Big)^{m_{j}[(\mft,p)]}
\Big]
\cdot
D^{m}F\;,
\end{equ}
where $m_{j}[(\mft,p)]$ denotes the $(\mft,p)$ component of $m_{j} \in \N^{\mcb{O}}$. 
\end{lemma}
That the above formula really is a ``Faa Di Bruno'' formula is partially obscured by our notation.
One should view the indeterminates $\{\Y_{(\mft,p)}\}_{(\mft,p) \in \mcb{O}}$ as representing a family of smooth functions from $\R^{d+1}$ to $\R$, namely one fixes smooth functions $\{u_{\mft}(z)\}_{\mft \in \mfL_{+}}$ and then the correspondence is given by $\Y_{(\mft,p)} \leftrightarrow \partial_{z}^{p} u_{\mft}(z)$ where $\partial_{z}$ denotes the vector of partial derivatives in the components of $z$. 

Then one has, for $F \in \Poly$, $m \in \N^{\mcb{O}}$, and $k \in \N^{d+1}$, 
\begin{equs}\label{meaning of derivatives}
F(\Y) &\longleftrightarrow F\left((\partial_{z}^{p} u_{\mft}: (\mft,p) \in \mcb{O})\right)\\
D^{m}F(\Y) &\longleftrightarrow 
\Big(\prod_{(\mft,p) \in m} \frac{\partial}{
\partial (\partial^{p} u_{\mft}(z))} 
\Big)F\left((\partial_{z}^{p} u_{\mft}(z): (\mft,p) \in \mcb{O})\right)\\
\partial^{k}F(\Y)
&\longleftrightarrow
\partial_{z}^{k}
F\left((\partial_{z}^{p} u_{\mft}(z): (\mft,p) \in \mcb{O})\right)
\end{equs}
We then compute $\partial^{k}F$ by manipulation of Taylor series (seen as formal power series). 
We first expand the $\partial_{z}^{p} u_{\mft}$ into Taylor series in $z$, insert these Taylor series into the one for $F$ in the variables $\partial_{z}^{p} u_{\mft}$, and then read off the coefficient of $z^{k}$ in the resulting power series.

If one takes the correspondences of \eqref{meaning of derivatives} for granted, then the proof of Lemma~\ref{lem: Faa Di Bruno} is immediate, for completeness we give a careful proof below. 
A more combinatorial proof of the formula can be found in \cite{Ma09}.  

\begin{proof}[of Lemma~\ref{lem: Faa Di Bruno}]
We first claim that it suffices to prove the identity \eqref{Faa di Bruno formula} for the case where the function $F$ is actually a polynomial in the variables $(\Y_{o})_{o \in \mcb{O}}$. 
To see this is the case first note that if $k \in \N^{d+1}$, $x = (x_{o})_{o \in \mcb{O}} \in \R^{\mcb{O}}$ and $F, G \in \Poly$ with $(D^{m}F)(x) = (D^{m}G)(x)$ for every $m \in \N^{\mcb{O}}$ with $|m| \le |k|$ then it follows that $(\partial^{k} F)(x) = (\partial^{k} G)(x)$.

Now suppose that the formula \eqref{Faa di Bruno formula} holds whenever $F \in \Poly$ is a polynomial of $\Y$ and we want to verify it for $G \in \Poly$ and $ k\in \N^{d+1}$ at a point $x \in \R^{\mcb{O}}$. 
The desired claim follows by applying the identity \eqref{Faa di Bruno formula} to
the polynomial $F_x$ given by
\[
F_x(\Y)
\eqdef
\sum_{
\substack{
m \in \N^{\mcb{O}}\\
|m| \le |k|}}
\frac{D^{m}G(x)}{m!}
\Y^{m}\;.
\]

We turn to proving \eqref{lem: Faa Di Bruno} for polynomial $F$. 
In the remainder of this proof we define $z$ and $w$ to be two vectors of mutually commuting indeterminates $(z_{0},\dots,z_{d})$, $(w_{0},\dots,w_{d})$ that will be the variables of our formal power series. 
Given a formal power series $\mathcal{A}(z,w) \eqdef \sum_{j,k \in \N^{d+1}}
A_{j,k} z^{j}w^{k} $ we use the notation $[ \mathcal{A}(z,w);z^{j}w^{k}] = A_{j,k}$. 
We introduce an $\mcb{O}$-indexed family of power series 
\begin{equ}\label{power series for indeterminate}
\Y_{(\mft,p)}(z) 
\eqdef 
 \sum_{q \in \N^{d+1}} \frac{z^{q}}{q!}\Y_{(\mft,p+q)}\;.
\end{equ}
Then each polynomial $F(\Y)$ can be associated to a power series 
\begin{equ}\label{powerseries for F}
F(\Y(z))
=
\sum_{
m \in \N^{\mcb{O}}}
\frac{D^{m}F(\Y)}{m!}
\big(
\Y(z) - \Y
\big)^{m}\;.
\end{equ}
For $k \in \N^{d+1}$ we define $\underline{\partial}^{k}F \eqdef k!\big[F(\Y(z));z^{k}\big]$, by induction one sees that $\underline{\partial}^{k}F = \partial^{k}F$.
For the base cases we clearly have that $\underline{\partial}^{k} = \partial^{k}$ if $|k| \le 1$. 
For the inductive step follows observe that for any $j,k \in \N^{d+1}$, 
\[
\underline{\partial}^{j} \underline{\partial}^{k}F
=
j!\,k!\,\big[F(\Y(z+w));w^{j}z^{k}\big]
=
(j+k)!\big[F(\Y(z));z^{k+j}\big]
=
\underline{\partial}^{j+k}F,
\]
where in the first equality we are using that $F$ is a polynomial and in the second equality we are using the binomial formula. 

All that remains is showing that for $k \not = 0$ the coefficient $k!\big[F(\Y(z));z^{k}\big]$ is given by the RHS of \eqref{Faa di Bruno formula}. When expanding the RHS of \eqref{powerseries for F} the terms that come with a $z^{k}$ are indexed by $I(k)$. 

Namely, one chooses an integer $r > 0$, and then a collection of powers $0 < q_{1} < q_{2} < \dots < q_{r} \in \N^{d+1}$ corresponding to the $q$'s that one will allow oneself to pick out in \eqref{power series for indeterminate} when expanding $(\Y(z) - \Y)^{m}$ for some $m$.
Next, one chooses a tuple $m_{1},\dots,m_{r} \in \N^{\mcb{O}} \setminus \{0\}$ where $m_{j}$ records from which $(\mft,p) \in \mcb{O}$ and with what multiplicity one is drawing out powers of $z^{q_{j}}$. 
To obtain an overall power of $z^{k}$ one has the constraint $k = \sum_{j=1}^{r} |m_{j}|q_{j}$. 
The corresponding $m \in \mcb{N}^{\mcb{O}}$ in the first sum of \eqref{powerseries for F} is given by $m = \sum_{j=1}^{r} m_{j}$. The corresponding coefficient of $z^{k}$ which is contributed is given by the summand on the RHS of \eqref{Faa di Bruno formula}.
\end{proof}
\begin{proof}[of Lemma~\ref{lem:UpsilonSum}]
We prove the statement of the lemma by induction over the number of internal nodes of $\tau$. The base case, 
when $\tau$ is a trivial tree, can be proven in the same way as the inductive step so we immediately 
turn to proving the latter. 

Suppose that $\tau$ is of form~\eqref{eq:general tree of V} with $N$ edges and that the claim has been proved for any $\tilde{\tau} \in \VV$ with fewer than $N$ edges.
Applying Lemma~\ref{lem: adjoint identity} for $Q$ one gets that $Q^{\ast}\tau$ is given by
\begin{equ}
\Xi_{\mfl}\!\!\!
\sum_{(\vec{q},\vec{m}) \in I(k)}
\!\!\Big(
\prod_{
\substack{
1 \le j \le r_{(\vec{q},\vec{m})}\\
(\mft,p) \in \mcb{O}
}
}
\frac{1}{m_{j}[(\mft,p)]!}
\Big(
\frac{\mcI_{(\mft,p)}[X^{p + q_{j}}]}{q_{j}!}
\Big)^{m_{j}[(\mft,p)]}
\Big)
\Big(
\prod_{w=1}^{n}
\mcb{I}_{(\mft_{w},p_{w})}[Q^{\ast}\tau_{w}]
\Big).
\end{equ}
By applying $\mathring{\Upsilon}_{\mft}^{P}[\cdot]$ to the quantity above and applying the inductive hypothesis we see that the RHS of \eqref{eq: Upsilon sum identity} is given by
\begin{equs}\label{work for projection}
\sum_{(\vec{q},\vec{m}) \in I(k)}
&
\Big(
\prod_{
\substack{
1 \le j \le r_{(\vec{q},\vec{m})}\\
(\mft,p) \in \mcb{O}
}
}
\frac{1}{m_{j}[(\mft,q)]!}
\Big(
\frac{\Y_{(\mft,p+q_{j})}}{q_{j}!}
\Big)^{m_{j}[(\mft,p)]}
\Big)
\Big(
\prod_{w=1}^{n}
\Upsilon_{\mft_{w}}^{F}[\tau_{w}]
\Big)
\\
{}
&
\enskip
\cdot
\Big(
D_{m}
\Big(
\prod_{w=1}^{n}
D_{(\mft_{w},p_{w})}
\Big)
F_{\mft}^{\mfl}
\Big)\;.
\end{equs}
The desired result follows by applying Lemma~\ref{lem: Faa Di Bruno}. 
\end{proof}
\subsection{Proof of Proposition~\ref{prop:co-interaction}}\label{subsec:coInteractProofs}

The proof of Proposition~\ref{prop:co-interaction} relies on the next two lemmas. 
We prove them by invoking a more general co-interaction property described in \cite[Thm~3.22]{BHZalg}.
\begin{lemma}\label{lem: grafting cointeraction}
On $\mcb{T}^\ex$, it holds that for $(\mft,p) \in \mcb{O}$
\begin{equs}
\mathcal{M}^{(13)(2)(4)} \left( \Deltam_{\ex} \otimes \Deltam_{\ex} \right) 
\hat\curvearrowright^{*}_{(\mft,p)} = \left( \id \otimes \hat\curvearrowright^{*}_{(\mft,p)} \right) \Deltam_{\ex}.
\end{equs}
\end{lemma}
\begin{lemma}\label{lem: poly cointeraction}
On $\mcb{T}^\ex$, it holds that for any $i \in \{0,\dots,d\}$, 
$\Deltam_{\ex} \hat\uparrow_i^* = \big( \id \otimes \hat\uparrow_i^* \big) \Deltam_{\ex}$.
\end{lemma}
Before proving these lemmas, we recall some notations and the definition of $ \Deltam_{\ex} $.
Let $ \hat\CT^{\ex}_- $ the free commutative algebra generated by $ \cT^{\ex} $. Then we set $ \CT^{\ex}_{-} = \hat \CT^{\ex}_{-} / \mcb{J}_{+} $ where $  \mcb{J}_{+}  $ is the ideal of $ \hat \CT^{\ex}_- $ generated by 
$ \lbrace \tau \in \mcT^{\ex} \; : \; | \tau |_- \geq 0 \rbrace $. The map $ \Deltam_{\ex} :  \CT^{\ex} \rightarrow \CT^{\ex}_- \otimes \CT^{\ex} $ is given for $ T^{\Labn,\Labo}_{\Labe} \in \CT^{\ex} $ by: 
\begin{equs}\label{co-action}
 \Deltam_{\ex}  T^{\Labn,\Labo}_{\Labe} & = 
 \sum_{A \subset T} \sum_{\Labe_A,\Labn_A}  \frac1{\Labe_A!}
\binom{\Labn}{\Labn_A}
 (A,\Labn_A+\pi\Labe_A, \Labo \restr N_A,\Labe \restr E_A) 
 \\& \qquad  \otimes( \CR_A T, [\Labn - \Labn_A]_A, \Labo(A)+[\Labn_A - \pi \Labe_A ]_A  , \Labe + \Labe_A)\;, 
\end{equs}
where
\begin{itemize}
\item For $ C \subset D $ and $ f : D \rightarrow \N^d $, we denote by $ f \restr C$ the restriction of $ f $ to $ C $.
\item The first sum runs over all subgraphs $ A  $ of
$ T $ ($ A $ may be empty). The second sum runs over all  $ \Labn_A : N_A \rightarrow \N^{d+1}  $ and $ \Labe_{A} : \partial(A,T) \rightarrow \N^{d+1} $ where 
$ \partial(A,F) $ denotes the edges in $ E_T \setminus E_A $ that are adjacent to $ N_A $.
\item  We write $ \CR_A T $ for the tree obtained by contracting the connected components of $ A $. This gives an action on the decorations in the sense that for $ f : N_T \rightarrow \N^{d+1}$ such that $ A \subset T $ one has: $ [f]_A(x) = \sum_{x \sim_{A} y} f(y) $ where $ x $ is an equivalence class of $ \sim_A $ and
$ x \sim_A y $ means that $ x $ and $ y $ are connected in $ A $. Moreover, the map $ \Labo(A) $ is defined on $ x $ by: 
\begin{equs}
 \Labo(A)(x) = 
\sum_{y \sim_{A} x} \Labo(y) + \sum_{e \in E_A} \left( \Labhom(e) - \Labe(e) \right).
\end{equs} 
\item For $ f : E_T  \rightarrow \N^{d+1}  $, we set for every $ x \in N_T $, $ (\pi f)(x) = \sum_{e=(x,y) \in E_F} f(x)$.
\end{itemize}
Then one can turn this map into a coproduct 
$ \Deltam_{\ex} : \CT^{\ex}_-  \rightarrow \CT^{\ex}_{-} \otimes \CT^{\ex}_{-}$ and obtain a Hopf algebra for $  \CT^{\ex}_{-}  $ endowed with this coproduct and the forest product, see~\cite[Prop.~5.35]{BHZalg}. Any $ M_{\ell}\in \mfR $ is described by an element $ \ell $ of the character group $\mcb{G}^{\ex}_- $ associated to this Hopf algebra: 
\begin{equs}
M_{\ell} = \left( \ell \otimes \id \right) \Deltam_{\ex}, 
\end{equs}
 where $ \Deltam_{\ex} $ is the co-action defined in \eqref{co-action}. 
Before stating the main co-interaction, we need to recall the definition of another map $ \Delta_2 $ given in~\cite{BHZalg}.
Let $ \hat\CT^{\ex}_+ $ denote the linear span of $ \hat \mcT_+^\ex $, the coloured trees $ (T,\hat T)^{\Labn,\Labo}_{\Labe} $ such that $ \hat T^{-1}(\lbrace 2 \rbrace) = \rho_T $ and $ \Labo(\rho_T) =0 $. If we consider that a vertex $ x $ has the colour $1$ when $ \Labo(x) \neq 0 $ then we can use lighter notations avoiding the notion of a coloured tree and consider that $ \hat T \in \lbrace 0,2\rbrace $.
Hence, elements of
$ \hat\mcT^{\ex}_+ $ are denoted by $ (T,2)^{\Labn,\Labo}_{\Labe} $ and those of $ \mcT^{\ex} $ are denoted by $ T^{\Labn,\Labo}_{\Labe} = (T,0)^{\Labn,\Labo}_{\Labe}  $. 
  Then the map $ \Delta_2 :  \CT^{\ex} \rightarrow \CT^{\ex} \otimes \hat \CT^{\ex}_{+} $ is given for $ T^{\Labn,\Labo}_{\Labe} \in \CT^{\ex} $ by: 
\begin{equs}
 \Delta_2  T^{\Labn,\Labo}_{\Labe} & = 
 \sum_{A \subset T} \sum_{\Labe_A,\Labn_A}  \frac1{\Labe_A!}
\binom{\Labn}{\Labn_A}
 (A,\Labn_A+\pi\Labe_A, \Labo \restr N_A,\Labe \restr E_A) 
 \\& \qquad  \otimes \hat P_2  ( \CR_A T,2, [\Labn - \Labn_A]_A, \Labo(A)+[\Labn_A - \pi \Labe_A ]_A  , \Labe + \Labe_A)\;, 
\end{equs}
where $ \hat P_2 $ sets to zero the
$ \Labo $ decoration at the root.
We define $  \hat\CP^{\ex}_+ \subset \hat \CT_+^{\ex}$ as the subspace of planted trees, i.e., trees having just one edge incident to the root and vanishing node decoration at the root.
In the sequel, we use the co-interaction identity on $ \CT^{\ex} $:
\begin{equs}\label{cointeraction}
\mathcal{M}^{(13)(2)(4)} \left( \Deltam_{\ex} \otimes \Deltam_{\ex} \right) 
\Delta_2 = \left( \id \otimes \Delta_2 \right) \Deltam_{\ex}\;.
\end{equs}
This identity is a consequence of the co-interaction given in \cite[Thm~3.22]{BHZalg}:
\begin{equs} \label{cointeraction0}
\mathcal{M}^{(13)(2)(4)} \left( \Delta_1 \otimes \Delta_1 \right) 
\Delta_2 = \left( \id \otimes \Delta_2 \right) \Delta_1\;.
\end{equs}
We apply $ \left( \mathfrak{p}^{\ex}_- \otimes \id \otimes \id \right) $ to \ref{cointeraction0} in order to obtain \ref{cointeraction}, where $ \mathfrak{p}^{\ex}_-  $ is the projection onto the forest composed of trees with negative degree.
The main idea of the following proofs is to rewrite   $ \hat\curvearrowright^{*}_{(\Labhom,p)} $ and 
$  \hat\uparrow_i^*  $ in terms of $  \Delta_2 $ and some projections which  behave well with $ \Deltam_{\ex} $.   
\begin{proof}[of Lemma~\ref{lem: grafting cointeraction}]
For
$ \hat\curvearrowright^{*}_{(\Labhom,p)} $  from Definition~\ref{def_grafting_explicit}, we have the identity
\begin{equs}
  \hat\curvearrowright^{*}_{(\Labhom,p)} = \mathcal{M}^{(2)(1)}\left( \id \otimes \CR_{2} \circ \mathfrak{p}_{(\Labhom,p)} \circ \Pi_{\hat \CP_+^{\ex}}   \right) \Delta_2
\end{equs}
where
\begin{itemize}
\item $\mathcal{M}^{(2)(1)}(\tau_1 \otimes \tau_2) \eqdef (\tau_2 \otimes \tau_1)$,
\item $ \Pi_{\hat \CP^{\ex}_+} : \hat \CT^{\ex}_+ \rightarrow \hat\CP^{\ex}_+ $ is the projection onto $ \hat\CP^{\ex}_+ $,
\item $ \mathfrak{p}_{(\Labhom,p)}  : \hat \CP^{\ex}_+ \rightarrow \hat\CP^{\ex}_+ $ is the projection onto planted trees with the root edge decorated by $ (\Labhom,k) $ for some $ k \leq p $,
\item $\CR_{2} : \hat \CP^{\ex}_+ \rightarrow \CT^{\ex}  $ acts by removing the edge incident to the root and the color blue at the root.
\end{itemize}
Then it is easy to show that the following identities hold:
\begin{equs}
 \left( \id \otimes \CR_2 \right) \Deltam_{\ex} & = \Deltam_{\ex} \CR_2, \quad \left( \id \otimes \mathfrak{p}_{(\Labhom,p)} \right) \Deltam_{\ex} = \Deltam_{\ex} \mathfrak{p}_{(\Labhom,p)} \; \text{ on }  \hat \CP^{\ex}_+\;,  \\
 \left( \id \otimes \Pi_{\hat \CP^{\ex}_+} \right) \Deltam_{\ex} & = \Deltam_{\ex} \Pi_{\hat \CP^{\ex}_+} \; \text{ on }  \hat \CT^{\ex}_+\;.
\end{equs}
Indeed, the previous projections are linked to the form of the tree at the root. The root is coloured in blue and therefore cannot be touch by $ \Deltam_{\ex} $. 
Then by using these identities, we have
\begin{equs}
& \mathcal{M}^{(13)(2)(4)} \left( \Deltam_{\ex} \otimes \Deltam_{\ex} \right)\hat\curvearrowright^{*}_{(\Labhom,p)} & \\ & = \mathcal{M}^{(13)(2)(4)} \left( \Deltam_{\ex} \otimes \Deltam_{\ex} \right) \mathcal{M}^{(2)(1)} \left( \id \otimes \CR_{2} \circ \mathfrak{p}_{(\Labhom,p)} \circ \Pi_{\hat \CP^{\ex}_+}   \right) \Delta_2
\\ &= \left(\id \otimes \mathcal{M}^{(2)(1)} \right) \mathcal{M}^{(13)(2)(4)} \left( \Deltam_{\ex} \otimes \Deltam_{\ex} \right)  \left( \id \otimes \CR_{2} \circ \mathfrak{p}_{(\Labhom,p)} \circ \Pi_{\hat \CP^{\ex}_+}   \right) \Delta_2
\\
& = \left(\id \otimes \mathcal{M}^{(2)(1)} \right) \mathcal{M}^{(13)(2)(4)} \left( \Deltam_{\ex} \otimes \left( \id \otimes \CR_{2} \circ \mathfrak{p}_{(\Labhom,p)} \circ \Pi_{\hat \CP^{\ex}_+}    \right)\Deltam_{\ex} \right)   \Delta_2 
\\
& = \left(\id \otimes \mathcal{M}^{(2)(1)} \left( \id \otimes \CR_{2} \circ \mathfrak{p}_{(\Labhom,p)} \circ \Pi_{\hat \CP^{\ex}_+}    \right) \right) \mathcal{M}^{(13)(2)(4)} \left( \Deltam_{\ex} \otimes \Deltam_{\ex} \right)   \Delta_2 
\\ &= \left(\id \otimes \mathcal{M}^{(2)(1)} \left( \id \otimes \CR_{2} \circ \mathfrak{p}_{(\Labhom,p)} \circ \Pi_{\hat \CP^{\ex}_+}    \right) \right) \left( \id \otimes \Delta_2 \right) \Deltam_{\ex}
\\ &= \left( \id \otimes  \hgraft{(\Labhom,p)}^* \right) \Deltam_{\ex}.
\end{equs}
\end{proof}
\begin{proof}[Lemma~\ref{lem: poly cointeraction}]
The map $ \hat\uparrow_i^*  $ from Definition~\ref{def_increasing_poly} can be rewritten as
\begin{equs}
\hat\uparrow_i^* = \mathcal{M}^{(1)}\left( \id \otimes \mathfrak{p}_{X_i}\right) \Delta_2
\end{equs}
where $ \mathcal{M}^{(1)}(\tau_1 \otimes \tau_2) = \tau_1 $ and $ \mathfrak{p}_{X_i} : \hat \CT^{\ex}_+ \rightarrow \hat \CT^{\ex}_+ $ is the projection on the tree composed of one node coloured in blue corresponding to $ X_i $. One has $ \left( \id \otimes \mathfrak{p}_{X_i} \right) \Deltam_{\ex} = \mathfrak{p}_{X_i}  $ and by using this identity it follows that
 \begin{equs}
& \left(  \id \otimes \mathcal{M}^{(1)}\left( \id \otimes \mathfrak{p}_{X_i}\right) \right) \mathcal{M}^{(13)(2)(4)} \left( \Deltam_{\ex} \otimes \Deltam_{\ex} \right) \Delta_2 
\\
& = \left(  \id \otimes \mathcal{M}^{(1)} \right) \mathcal{M}^{(13)(2)(4)} \left( \Deltam_{\ex} \otimes \left( \id \otimes \mathfrak{p}_{X_i}\right)\Deltam_{\ex} \right) \Delta_2 \\
& = \left(  \id \otimes \mathcal{M}^{(1)} \right)  \left( \Deltam_{\ex} \otimes  \mathfrak{p}_{X_i} \right) \Delta_2 
\\ & = \Deltam_{\ex} \mathcal{M}^{(1)}\left( \id \otimes \mathfrak{p}_{X_i}\right) \Delta_2
 =  \Deltam_{\ex} \hat\uparrow_i^*.
\end{equs} 
 On the other hand, we obtain: 
\begin{equs}
  \left(  \id \otimes \mathcal{M}^{(1)}\left( \id \otimes \mathfrak{p}_{X_i}\right) \right)\left( \id \otimes \Delta_2 \right)
\Deltam_{\ex} = \big( \id \otimes  \hat\uparrow_i^* \big) \Deltam_{\ex}\;,
\end{equs}
and the claim follows.
\end{proof}
\begin{remark}
Lemmas~\ref{lem: grafting cointeraction} and~\ref{lem: poly cointeraction} can be proven without the use of the strong co-interaction obtained in \cite[Thm.~3.22]{BHZalg}.
The difficult part of the proof is taking care of the binomial coefficients and one can handle this by using the Chu-Vandermonde identity in a more elementary way than in the proof of~\cite[Thm.~3.22]{BHZalg}. Lemma~\ref{lem: poly cointeraction}, for example, only needs the identity $a\binom{a-1}{b} = (a-b)\binom{a}{b}$ for $a,b \in \N$.
\end{remark}
\begin{proof}[Proposition~\ref{prop:co-interaction}]
Let $ \ell \in \mcb{G}_{-}^{\ex} $ and set $M_{\ell} = \left( \ell \otimes \id \right) \Deltam_{\ex} \in \RR$. Lemma~\ref{lem: poly cointeraction} implies that any $i \in \{0,\ldots, d\}$
\begin{equs}
M_{\ell} \hat\uparrow_i^* & = \left( \ell \otimes \id \right) \Deltam_{\ex} \hat\uparrow_i^* 
 = ( \ell \otimes \hat\uparrow_i^* )  \Deltam_{\ex}
 = \hat\uparrow_i^* M_{\ell}\;,
\end{equs}
from which~\eqref{co-interaction polynomial} follows.
Turning to~\eqref{co-interaction renormalisation}, for $o \in \mcb{O}$ one has, by Lemma~\ref{lem: grafting cointeraction},
\begin{equs}
\hgraft{o}^* M_{\ell} & =  \left( \ell \otimes \hgraft{o}^{*} \right) \Deltam_{\ex} \\  & = \left( \ell \otimes \id \otimes \id \right) \mathcal{M}^{(13)(2)(4)} \left( \Deltam_{\ex} \otimes \Deltam_{\ex} \right) 
\hgraft{o}^{*} \\
& =  \left( \ell \otimes \id \otimes \id \right)  \left( \left( \ell \otimes \id \right)\Deltam_{\ex} \otimes \left( \ell \otimes \id \right)\Deltam_{\ex} \right) 
\hgraft{o}^{*} = 
\left( M_{\ell} \otimes  M_{\ell} \right) \hgraft{o}^{*}\;.
\end{equs}
Passing to the adjoint and using Corollary~\ref{cor:graftMorphT} concludes the proof.
\end{proof}
\subsection{Proof of Proposition~\ref{prop:generate}}\label{subsec:freelyGenProof}
\begin{proof}[of Proposition~\ref{prop:generate}]
Our proof follows the one given in 
\cite{Chapoton01} for rooted trees without decorations on the grafting operators.
We first consider decorated variables $ x^{(k)} $, $k \in \N^{d+1}$, and decorated brackets $ (\cdot)_{(\Labhom,p)} $, . 
Let $ \CF^{\ex}(n) $ the vector space given by parenthesised product on these variables indexed by $ \lbrace 1,\ldots,n \rbrace $, using the previous decorated brackets.
For example, a basis for $ \CF^{\ex}(2) $ is given by:
\begin{equs}
( x^{(k_1)}_1 x^{(k_2)}_2 )_{(\mft,p)} ,\quad ( x^{(k_2)}_2 x^{(k_1)}_1)_{(\mft,p)},  \quad  (\Labhom,p) \in  \mcb{O}, \quad k_i \in  \N^{d+1}\;.
\end{equs}
  We set $ \mathcal{P} \mathcal{L}^{\ex} = \CF^{\ex} / (R) $ where $ \CF^\ex = (\CF^{\ex}(n))_{n \geq 1} $ and 
  the equivalence relation $ R $ is generated by the relations
\begin{equs}
r & = ((x^{(k_1)}_1 x^{(k_2)}_2 )_{(\Labhom_{1},p_1)} x^{(k_3)}_3)_{(\Labhom_{2},p_2)} - (x^{(k_1)}_1 (x^{(k_2)}_2 x^{(k_3)}_3 )_{(\Labhom_{2},p_2)} )_{(\Labhom_{1},p_1)} \\
& - 
( (x^{(k_2)}_2 x^{(k_1)}_1 )_{(\Labhom_{2},p_2)}x^{(k_3)}_3 )_{(\Labhom_{1},p_1)} + (x^{(k_2)}_2 (x^{(k_1)}_1 x^{(k_3)}_3 )_{(\Labhom_{1},p_1)} )_{(\Labhom_{2},p_2)}\;.
\end{equs}
Let $ \CR\CT^{\ex}(n)  $ be the linear span of trees with edge decorations in $ \mcb{O} $ and having their nodes labelled by
$ \lbrace 1,\ldots,n \rbrace $. Moreover, they do not have drivers. For $ (\Labhom,p) \in  \mcb{O} $, we define the grafting operator $\hgraft{(\Labhom,p)}$ as a linear map from 
 $ \CR\CT^{\ex}(m) \otimes   \CR\CT^{\ex}(n)    $ into  $ \CR\CT^{\ex}(n+m)  $ by
 \begin{equs}
 \tau  \hgraft{(\mft,p)} \bar \tau
& = \sum_{\ell} \binom{ k_r }{\ell}
 \bullet^{k_{r}-\ell}_{r}   \mcb{I}_{(\mft,p-\ell)}[\tau] \Big(\prod_{j\in J} \mcb{I}_{o_j}[\tau_j]\Big) 
\\ & +   \bullet_{r}^{k_r}\Big(\prod_{j\in J} \mcb{I}_{o_j}[\tau  \hgraft{(\mft,p)} \tau_j]\Big)\;,
 \end{equs}
 where $ \bullet_{r}^{k_r} $ is the rooted tree composed of a single node labelled by $ r $ and decorated by $ k_r $.
Note that the expression of this grafting operator is essentially identical to the one given in Remark~\ref{rem:formula for hgraft}; we commit here an abuse notation by identifying the two operators.

Recall the pre-Lie type identity~\eqref{pre_lie_identity}.
 We define a morphism $ \Phi : \mathcal{P} \mathcal{L}^{\ex} \rightarrow 
\CR\CT^{\ex} $ for the concatenation $ w = (uv)_{(\Labhom,p)} $ of two words $ u $ and $ v$ by setting
\begin{equs}
 \Phi( x^{(k_i)}_i   ) \eqdef \bullet_{i}^{k_i}\;, \quad  \Phi(\left( uv  \right)_{(\Labhom,p)} ) \eqdef \Phi(u) \hgraft{(\Labhom,p)} \Phi(v)\;,
\end{equs} 
The identity~\eqref{pre_lie_identity}  proves that $ \Phi(r) = 0 $, so that this is well-defined. 

We want to construct an inverse 
$ \Psi $ of $ \Phi $. Let us fix a finite set $ I $ and write $ \CR \CT^{\ex}(I) $ for the decorated trees of $ \CR \CT^{\ex} $ labelled with $ I $.
We consider $ \Phi_{I} : \mathcal{P} \mathcal{L}^{\ex}(I) \rightarrow \CR \CT^{\ex}(I) $ the extension of $ \Phi $. We want to define a map $  \Psi_{I} : \CR \CT^{\ex}(I) \rightarrow  \CP \mathcal{L}^{\ex}(I)  $ such that $ \Psi_I \Phi_I = \id $ and $ \Phi_I \Psi_I = \id $. We proceed by induction on the cardinal $ | I | $ of $ I $.
The initialisation with only one element in $ I $ is straightforward. Suppose that the map $ \Psi_I $ is defined for $ |I| \leq n $ and consider  $ I $ such that $ |I| = n+1 $. 
Let $ T \in \CR\CT^{\ex}(I) $ such that the root of $  T $ is labelled by $ r \in I $. 
Using symbolic notation, $ T $ is of the form
\begin{equs}
 T = \bullet^{k_r}_{r} \prod_{i=1}^N \mcb{I}_{(\Labhom_i,p_i)}[\tau_i]\;.  
\end{equs}  
We define the map $ \Psi_I $ by induction on $ N$.
If $ N=1 $, observe that
\begin{equs} \label{grafting_identity}
T = \bullet^{k_r}_{r} \mcb{I}_{(\Labhom,p)}[\tau] = \sum_{\ell} (-1)^{ | \ell |_{\s}} \binom{ k_r }{\ell}  \left( \tau 
\hgraft{(\Labhom,p-\ell)}  \bullet^{k_r-\ell}_{r} \right)\;,
\end{equs}
with the convention that the terms with $ k_r - \ell $ are zero when $ k_r = 0 $.
Then we set  $ \Psi_I(T) = \sum_{\ell} (-1)^{ | \ell |_{\s}} \binom{ k_r }{\ell}  \left( x^{(k_r-\ell)}_r \Psi_I(\tau) \right)_{(\Labhom,p-\ell)} $, where we note that $\Psi_I(\tau)$ is well-defined due to our induction hypothesis on $|I|$.
It is then immediate to verify that $ \Phi_I \Psi_I (T) = T $.
If $ N \geq 2 $, observe that
\begin{equs}
T & =   \sum_{\ell} (-1)^{ | \ell |_{\s}} \binom{ k_r }{\ell}   \;\tau_1 \hgraft{(\Labhom_{1},p_1 - \ell)}   \bullet^{k_r - \ell}_{r}   \prod_{i=2}^{N}  \mcb{I}_{(\Labhom_i,p_i)}[\tau_i]  \\
& -   \sum_{\ell} (-1)^{ | \ell |_{\s}} \binom{ k_r }{\ell} \sum_{j=2}^{N}   \bullet^{k_r- \ell}_{r}   \mcb{I}_{(\Labhom_j,p_j)}[ \tau_1 \hgraft{(\Labhom_1,p_1-\ell)} \tau_j] \prod_{i \neq j }\mcb{I}_{(\Labhom_i,p_i)}[\tau_i]\;.
\end{equs}
Note that intuitively, this represents an ``ungrafting'' of $\tau_1$ from the root of $T$.
Then we define $ \Psi_I(T) $ by
\begin{equs}
& \Psi_I(T)  = \sum_{\ell} (-1)^{ | \ell |_{\s}} \binom{ k_r }{\ell} \Big( \Psi_I(\tau_1)    \Psi_I \Big( \bullet^{k_r- \ell}_{r} \prod_{i=2}^{N} \mcb{I}_{(\Labhom_i,p_i)}[\tau_i]\Big) \Big)_{(\Labhom_{1},p_1 - \ell)}  \\
& - \sum_{\ell} (-1)^{ | \ell |_{\s}} \binom{ k_r }{\ell} \sum_{j=2}^{N}   \Psi_I \Big( \bullet^{k_r-\ell}_{r} \mcb{I}_{(\Labhom_j,p_j)}[ \tau_1 \hgraft{(\Labhom_1,p_1- \ell)} \tau_j] \prod_{i \neq j }\mcb{I}_{(\Labhom_i,p_i)}[\tau_i] \Big).
\end{equs}
One can then verify that $ \Phi_I \Psi_I(T) = T  $ as desired.
Since the tree $ T $ is invariant under permutation of the $ T_i = \mcb{I}_{(\Labhom_i,p_i)}[\tau_i]  $, we need to check that the definition of $ \Psi_I(T) $ does not depend on the subtree we ungraft from the root of $ T $ (in the above, this was taken as $\tau_1$).
We proceed by induction on $N$ and we prove that the order of ungrafting $ \tau_1 $ and $ \tau_2 $ does not matter in the definition of $ \Psi_I(T) $.
The proof follows in exactly the same way as in~\cite{Chapoton01} but we have longer expressions because of the identity~\eqref{grafting_identity}.
We omit the details but note that the relations $ R $, the symmetries, and the pre-Lie identity~\eqref{pre_lie_identity} provide all the necessary ingredients for the verification.

It remains to prove that one has $ \Psi \Phi = \id $. We show by induction on $ N $ that
\begin{equs}
 \Psi(T' \hgraft{(\Labhom,p)} T) = ( \Psi(T') \Psi(T) )_{(\Labhom,p)}\;.
 \end{equs} 
 If we consider a word $ w $ in 
$ \mathcal{P} \mathcal{L}^{\ex} $ then $ w = (uv)_{(\Labhom,p)} $ and we get:
\begin{equs}
\Psi \Phi(w) = \Psi \left( \Phi(u)
\hgraft{(\Labhom,p)} \Phi(v) 
 \right) = \left( \Psi \Phi(u) \Psi \Phi(v) \right)_{(\mft,p)}\;.
\end{equs} 
We conclude by applying the induction hypothesis on $ u $ and $ v $. We obtain Proposition~\ref{prop:generate} by substituting the indexed nodes by the drivers $ \Xi_\mfl$, $\mfl \in \mfD$. 
\end{proof}
\subsection{Proof of Theorem~\ref{thm: THE black box}}\label{sec: proof of blackbox}
\begin{proof}[of Theorem~\ref{thm: THE black box}]
Given $F \in \widetilde{\G}$, we define  $F \in \mathring{\G}$ as follows. 
For any if $\mft \in \mfL_{+}$, if $\mfl \in \mfD_{\mft}$ is not of the form $(\mathbf{0},0)$ or $((\mfb,0),0)$ for some $\mfb \in \mfL_{-}$ then we set $F_{\mft}^{\mfb} \eqdef 0$. 
We then set $F^{(\mathbf{0},0)}_{\mft} \eqdef \hat F^{\mathbf{0}}_{\mft}$ and $F^{((\mfb,0),0)} = \hat F^{\mfb}_{\mft}$.

We now construct a corresponding rule $\hat R$.
If $\mft \in \mfL_{-}$ we set $\hat R(\mft) = \emptyset$. If $\mft \in \mfL_{+}$ then writing expanding for each $\mfb \in \mfL_{-}$ as in \eqref{expansion of nonlinearity}
\[
F_{\mft}^{\mfb}(\Y) 
=
\sum_{j=1}^{m_{\mft,\mfb}}
F_{j,\mft,\mfb}(\Y)
\Y^{\alpha_{j,\mft,\mfb}}\;,
\]
we set
\[
\hat{R}(\mft)
\eqdef
\Big(
\bigcup_{
\substack{
\mfb \in \mfL_{-}\\
1 \le j \le m_{\mft,\mfb}
}
}
\Big\{ 
\alpha \sqcup \beta \sqcup \{(\mfb,0)\}:\  
\begin{array}{c}
\alpha \subset \alpha_{j,\mft,\mfb}\\
\beta \in \N^{\mcb{O}_{+}}
\end{array} 
\Big\}
\Big)
\sqcup
\N^{\mcb{O}_{+}}\;.
\]
Note that $\hat{R}$ is normal and subcritical with respect to $\reg$.
By~\cite[Prop.~5.20]{BHZalg} one can extend $\hat{R}$ to a complete rule $R$ which is again subcritical with respect to $\reg$. 
Finally, it is straightforward to verify that $R$ satisfies Assumption~\ref{assump:RregComplete} and that $F$ obeys $R$. 
We define as in Section~\ref{sec: alg and main theorem} and \cite[Sec.~5]{BHZalg} a regularity structure, space of models, and renormalisation group corresponding to the rule $R$.

Let $Z^{(\rho,\eps)}_{\BPHZ}$ be the random model obtained by taking the BPHZ lift of the noise $\xi^{(\rho,\eps)} \eqdef (\xi_{\mfl}^{(\rho,\eps)})_{\mfl \in \mfL_{-}}$.
Thanks to the assumptions of Theorem~\ref{thm: THE black box} we can apply \cite[Thm.~2.15]{CH} which states that there exists a random element $Z_{\BPHZ}$ of $\mathscr{M}_{0}$, independent of our choice of mollifier $\rho$, such that the random models $Z^{(\rho,\eps)}_{\BPHZ}$ converge in probability to $Z_{\BPHZ}$ in the topology of $\mathscr{M}_{0}$ as $\eps \downarrow 0$.

By stationarity it is clear that the models $Z^{(\rho,\eps)}_{\BPHZ}, Z_{\BPHZ}$ belong to  $ \mathscr{M}_{0,1}$ almost surely and moreover the convergence of statement of \cite[Thm.~2.15]{CH} also implies that $Z^{(\rho,\eps)}_{\BPHZ}$ converge to $Z_{\BPHZ}$ in the topology of $\mathscr{M}_{0,1}$ as $\eps \downarrow 0$. 

Therefore, by Proposition~\ref{prop:stationary objects}, the modelled distributions $\mcP^{Z^{(\rho,\eps)}_{\BPHZ}}\tilde{U}$ are well-defined elements of $\jets^{\gamma,\eta}_{0}$.

We define 
\begin{equ}[e:defSpm]
\mathcal{S}_{\rho,\eps}^{-}(\zeta) \eqdef \mcR^{Z^{(\rho,\eps)}_{\BPHZ}}\mcP^{Z^{(\rho,\eps)}_{\BPHZ}}\tilde{U}
\qquad
\textnormal{and}
\qquad
\mathcal{S}_{\rho,\eps}^{+}(\zeta,\psi) \eqdef \mcR^{Z^{(\rho,\eps)}_{\BPHZ}} V^{Z^{(\rho,\eps)}_{\BPHZ}}(\psi)\;,
\end{equ}
where $\psi \in \mcC^{\ireg}$ and $V^{\bullet}(v_{0})$ is the solution of the fixed point problem~\eqref{eq: Dgamma fp} started at time $s=0$ with initial data $v_{0} \in \mcC^{\ireg}$. 

We also define 
\[
\mathcal{S}_{\rho,\eps}(\zeta,\psi) \eqdef \mcR^{Z^{(\rho,\eps)}_{\BPHZ}}\overline{U}^{Z^{(\rho,\eps)}_{\BPHZ}}\big(\psi - \mathcal{S}_{\rho,\eps}^{-}(\zeta)(0,\cdot)\big)\;,
\]
where $\overline{U}^{Z^{(\rho,\eps)}_{\BPHZ}}(v_{0})$ is defined as in \eqref{eq:overline U identity}. 
The identity \eqref{e:idenSol} is then an immediate consequence of \eqref{eq:overline U identity} and the definitions we chose above.

By Proposition~\ref{prop:stationary objects}, $\mathcal{S}_{\rho,\eps}^{-}$ converges in probability, as $\eps \downarrow 0$, to $\mathcal{S}^{-} \eqdef \mcR^{Z_{\BPHZ}} \mcP^{Z_{\BPHZ}}\tilde{U}$ in the topology of $\CC^{\reg,-}$.  
Theorem~\ref{thm: regularity structures results} implies that  $\mathcal{S}_{\rho,\eps}^{+}$ converges in probability, as $\eps \downarrow 0$, pointwise in its $\CC^{\ireg}$ argument, and in the topology of $\CC^{\rem}$, to $\mathcal{S}^{+} \eqdef \mcR^{Z^{\BPHZ}} V^{Z_{\BPHZ}}$.
This finishes the proof of the convergence of the various solution maps to a limit which is independent of $\varrho$. 
What remains to be verified is that our definition of $\mathcal{S}_{\rho,\eps}$ given here coincides with the earlier definition \eqref{eq: renormalised spde}.

By Proposition~\ref{prop: underlined U}, $\overline{U}^{Z^{(\rho,\eps)}_{\BPHZ}}(v_{0})$ satisfies the fixed point problem \eqref{eq:overline U fp} - when written in differential form the initial data is given by $v_{0} + \bar{u}^{Z^{(\rho,\eps)},\mft}_{0}$. 
Also recall that by definition $\bar{u}^{Z^{(\rho,\eps)},\mft}_{0}(\cdot) = \mathcal{S}_{\rho,\eps}^{-}(\xi)(0,\cdot)$.

We fix initial data $\psi \in \CC^{\ireg}$ and set $\hat{\phi}^{(\rho,\eps)} \eqdef \mathcal{S}_{\rho,\eps}(\xi,\psi)$. 
By combining the observations of the previous paragraph with Theorem~\ref{thm:renormalised_equation} it follows that for every $\eps > 0$ we have that $\hat{\phi}^{(\rho,\eps)}$ is a local solution to the system of equations
\begin{equ}\label{black box theorem - equation}
\hat{\phi}^{(\rho,\eps)}_{\mft} = G_{\mft} 
\ast 
\Big[\one_+ 
\sum_{(\hat{\mfl},\mfo) \in \mfD_{\mft}}
(M^{(\rho,\eps)}_{\BPHZ}F)^{(\hat{\mfl},\mfo)}_{\mft}(\mathbf{\hat{\phi}^{(\rho,\eps)}})
\xi_{\hat{\mfl}}^{(\rho,\eps)}
\Big]
+ G_\mft \psi_\mft,\ \mft \in \mfL_{+}\;.
\end{equ} 
Here $M^{(\rho,\eps)}_{\BPHZ} \in \mfR$ is defined via $M^{(\rho,\eps)}_{\BPHZ} \eqdef (\ell^{(\rho,\eps)}_{\BPHZ} \otimes \id) \Deltam_{\ex}$
where $\ell^{(\rho,\eps)}_{\BPHZ}(\cdot) \eqdef \E( \PPi^{\varrho,\varepsilon} \tilde{\CA}^{\ex}_{-} \cdot )(0)$ and $ \PPi^{(\varrho,\varepsilon)} $ is the canonical lift of $\xi \ast \rho_{\eps}$ defined in~\cite[Sec.~6.2]{BHZalg}. \tabularnewline
In particular $M^{(\rho,\eps)}_{\BPHZ}$ is a deterministic element of $\mathfrak{R}$ such that $M^{(\rho,\eps)}_{\BPHZ}Z_{\can}^{(\rho,\eps)} = Z^{(\rho,\eps)}_{\BPHZ}$.

We now compute $M^{(\rho,\eps)}_{\BPHZ}F$. 
Fix $\mft \in \mfL_{+}$. For any $M \in \mfR$ and $(\hat{\mfl},\mfo) \in \mfD_{\mft}$ with $\hat{\mfl} \not = \emptyset$, we have by Assumption~\ref{assump:remove_noise_positive} that $M^{\ast} \Xi_{(\hat{\mfl},\mfo)} = \Xi_{(\hat{\mfl},\mfo)}$ and consequently, by Definition~\ref{def:MP}, $(MF)_{\mft}^{(\hat{\mfl},\mfo)} = F_{\mft}^{(\hat{\mfl},\mfo)}$. 
Also recall that one always has $M^{\ast} \Xi_{(\emptyset,0)} = \Xi_{(\emptyset,0)}$. 
It follows that
\begin{equs}
\sum_{(\hat{\mfl},\mfo) \in \mfD_{\mft}}
(M^{(\rho,\eps)}_{\BPHZ}F)^{(\hat{\mfl},\mfo)}_{\mft}(\mathbf{\hat{\phi}^{(\rho,\eps)}})
\xi_{\hat{\mfl}}^{(\rho,\eps)}
=
\sum_{\mfl \in \mfL_{-}}&
\hat F_{\mft}^{\mfl}(\mathbf{\hat{\phi}^{(\rho,\eps)}})
\xi_{\mfl}^{(\rho,\eps)}
+
\hat F_{\mft}^{\mathbf{0}}(\mathbf{\hat{\phi}^{(\rho,\eps)}})\\
{}
&
+
\sum_{
\substack{
(\mathbf{0},\mfo) \in \mfD_{\mft}\\
\mfo \not = 0
}
}
\Upsilon^{F}_{\mft}[(M^{(\rho,\eps)}_{\BPHZ})^{\ast} \Xi_{(\mathbf{0},\mfo)}](\mathbf{\hat{\phi}^{(\rho,\eps)}})\;.
\end{equs}
Again, by Assumption~\ref{assump:remove_noise_positive}, we have that
\begin{equ}
\sum_{
\substack{
(\mathbf{0},\mfo) \in \mfD_{\mft}\\
\mfo \not = 0
}
}
\Upsilon^{F}_{\mft}[(M^{(\rho,\eps)}_{\BPHZ})^{\ast} \Xi_{(\mathbf{0},\mfo)}](\mathbf{\hat{\phi}^{(\rho,\eps)}})
= 
\sum_{
\substack{
\tau \in \mcb{T}^{\ex}_{-}\\
\tau \textnormal{ tree}}
}
\ell_{\BPHZ}^{(\rho,\eps)}(\tau)
\frac{\Upsilon^{F}_{\mft}[\tau](\mathbf{\hat{\phi}^{(\rho,\eps)}})}{S(\tau)}\;.
\end{equ}
It follows that the system~\eqref{black box theorem - equation} is the same as the system given in Theorem~\ref{thm: THE black box} if we set $c_{\rho,\eps}[T^{\mfm}_{\mff}] \eqdef \ell_{\BPHZ}^{(\eps)}[T^{\mfm}_{\mff}]$ for each $\mft \in \mfL_+$ and $T^{\mfm}_{\mff} \in \mathring{\mathscr{T}}_{\mft,-}[F]$, where we are using the natural identification of $ \bigcup_{\mft \in \mfL_{+}}\mathring{\mathscr{T}}_{\mft,-}[F]$ with the trees generating $\mcb{T}^{\ex}_{-}$ (here we are using the notation of \cite{BHZalg}).
\end{proof}
\section{Symbolic index}
In this appendix, we collect the most used symbols of the article, together
with their meaning and the page where they were first introduced.

 \begin{center}
\renewcommand{\arraystretch}{1.1}
\begin{longtable}{lll}
\toprule
Symbol & Meaning & Page\\
\midrule
\endfirsthead
\toprule
Symbol & Meaning & Page\\
\midrule
\endhead
\bottomrule
\endfoot
\bottomrule
\endlastfoot
 $\graft{o}$ & Grafting operator on $\BBspan$ & \pageref{graft page ref}\\
 $\hgraft{o}$ & Grafting operator on $\VVspan$ & \pageref{hgraft page ref}\\
 $\bgraft{o}$ & Grafting operator on $\cT^\ex$ & \pageref{bgraft page ref}\\
 $\uparrow_l$ & Polynomial raising operator on $\BBspan$ & \pageref{uparrow page ref}\\
 $\hat\uparrow_l$ & Polynomial raising operator on $\VV$ & \pageref{hatuparrow page ref}\\
 $\BB$ & Subset of $\BBbig$ with restrictions on polynomial nodes & \pageref{BB page ref}\\
 $\BBspan$ & Vector space spanned by $\BB$ & \pageref{BBspan page ref}\\
 $\CC^{\reg,-}$ & Space where distribution-like part of solutions takes values & \pageref{Creg- page ref}\\
 $\mcC^{\rem}$ & Space in which function-like part of solutions takes values & \pageref{Creg page ref}\\
 $\mcC^{\ireg}$ & Space of possible initial conditions & \pageref{Cireg page ref}\\
 $\mcC^{\noise}$ & Space in which noises take values & \pageref{Cnoise page ref}\\
 $\DD$ & Abstract gradient & \pageref{gradient page ref}\\
 $\D^{\gamma,\eta}$ & Space of singular modelled distributions & \pageref{Dgamma page ref}\\
 $\hat\mfD$ & Products of derivatives of noises & \pageref{drivers hat page ref}\\
 $\mfD$ & Set of abstract drivers including extended decorations & \pageref{drivers page ref}\\
 $\mfD_{\mft}$ & Set of elements of $\mfD$ compatible with $\mft \in \mfL_{+}$ & \pageref{drivers for t}\\
 $\bar\mfD$ & The subset $\{\mathbf{X}^k\Xi_\mfl\, : \, k \in \N^{d+1},\;\mfl\in\mfD\} \subset \mcT^\ex$ & \pageref{bar mfD page ref}\\
 $e_a$ & Element in $\N^A$, $a \in A$, defined by $e_a[b] \eqdef \one\{a=b\}$ & \pageref{ei page ref}\\
 $G_\mft$ & Green's function of $\partial_t - \mathscr{L}_\mft$ & \pageref{Gmft page ref}\\
$\mcb{H}^{\ex}$ & $\bigoplus_{\mft \in \mfL_{+}} \mcb{T}_{\mft}^{\ex}$ & \pageref{Hex page ref}\\
$\widetilde{\mcb{H}}^{\ex}$ & $\bigoplus_{\mft \in \mfL_{+}} \widetilde{\mcb{T}}_{\mft}^{\ex}$ & \pageref{tilde Hex page ref}\\
 $K_\mft$ & Truncation of $G_\mft$ & \pageref{Kmft page ref} \\
 $\mathscr{L}_{\mft}$ & Differential operator associated with component $\mft$ & \pageref{Lmft page ref}\\
 $\mfL_{+}$ & Index set for the components of the system of SPDEs & \pageref{set of positive types}\\
 $\mfL_{-}$ & Index set for the rough ``drivers'' in our system of SPDEs & \pageref{set of positive types}\\
 $\Lambda$ & The underlying space-time $[0,\infty) \times \T^{d}$ & \pageref{introduction of Lambda}\\
 $\mathscr{M}_\infty$ & Space of all smooth admissible models on $\mathscr{T}$& \pageref{Minfty page ref}\\
 $\mathscr{M}_0$ & Closure of smooth admissible models & \pageref{M0 page ref}\\
 $\mcb{N}_{+}$ & All node-types in $\mfL_{+} \times \N^{d+1}$ & \pageref{pos node-types}\\
 $\mcb{O}$ & Set indexing the jet of $U$ & \pageref{mcbO page ref}\\
 $P$ & Time $0$ hyperplane & \pageref{P page ref}\\
 $\hat{\mcb{P}}(A)$ & Set of all multi-subsets of $A$. Identified with $\N^{A}$ & \pageref{multiset pageref}\\
$\Poly$ & nonlinear functions of the jet of $U$ & \pageref{poly page ref}\\
$\widetilde{\G}$ & Collection of nonlinearities $(F^\mfl_\mft)_{\mft\in\mfL_+,\mfl\in\mfL_-\sqcup\{\mathbf{0}\}}$ obeying $\reg$ & \pageref{tilde G page ref}\\
$\mathring\G$ & Collection of nonlinearities $(F^\mfl_\mft)_{\mft\in\mfL_+,\mfl\in\mfD_\mft}$ & \pageref{mathring G page ref}\\
$\G$ & Subset of all $F\in\mathring \G$ which obey $R$ & \pageref{G page ref}\\
 $Q$ & Map from $\BB$ to $\VV$ which collapses polynomial decorations & \pageref{Q page ref}\\
 $\mcQ_{\leq \gamma}$ & Natural projection $\cT^\ex \to \cT^\ex_{\leq \gamma}$ & \pageref{mcQ page ref}\\
 $\reg$ & Map $\mfL \sqcup \{\mathbf{0}\} \to \R$ & \pageref{reg page ref}\\
 $\ireg$ & Map $\mfL_+ \to \R$ & \pageref{ireg page ref}\\
 $\mcR$ & Reconstruction operator & \pageref{reconstruction}\\
 $R$ & Rule used to construct a regularity structure & \pageref{rule page ref}\\
 $R_\mft$ & Smooth function such that $G_\mft = K_\mft + R_\mft$ & \pageref{Kmft page ref}\\
 $\mfR$ & Renormalisation group of $\mathscr{T}$ & \pageref{RG page ref}\\
 $\s$ & Space-time scaling & \pageref{space-time scaling introduced}\\
 $S^\varepsilon_\s$ & Scale transformation by $\varepsilon$ around the origin & \pageref{scale transform page ref}\\
 $\mathring{\mathscr{T}}$ & Set of all possible trees (without extended decorations) & \pageref{basic trees page ref} \\
  $\mathring{\mathscr{T}}_{\mft}[F]$ & Subset of $\mathring{\mathscr{T}}$ which are $\mft$-non-vanishing for $F$ & \pageref{basic trees non-vanish page ref} \\
  $\mathring{\mathscr{T}}_{\mft,-}[F]$ & Subset of $\mathring{\mathscr{T}}_{\mft}[F]$ with negative degree & \pageref{basic neg trees page ref} \\
 $\mathscr{T}$ & Regularity structure built from the rule $R$ & \pageref{reg st page ref}\\
 $\mcT^\ex$ & Trees with extended decorations generated by the rule $R$ & \pageref{mcT page ref} \\
 $\mcT^\ex_{\leq \gamma}$ & Set of trees $\tau \in\mcT^\ex$ with $|\tau|_+ \leq \gamma$ & \pageref{mcT gamma page ref}\\
 $\cT^\ex$ & Vector space spanned by $\mcT^\ex$ & \pageref{cT page ref}\\
 $\cT^\ex_{\leq \gamma}$ & Subspace of $\cT^\ex$ spanned by $\mcT^\ex_{\leq \gamma}$ & \pageref{cT gamma page ref}\\
 $\bar\cT^\ex$ & Abstract Taylor polynomials in $\cT^\ex$ & \pageref{poly reg page ref}\\
 $\mcb{T}_{\mft}^{\ex}$ & Sector where $U_\mft$ takes values & \pageref{mcbT_mft page ref} \\
 $\widetilde{\mcb{T}}_{\mft}^{\ex}$ & Sector on which $\mcb{I}_{(\mft,0)}$ is well-defined & \pageref{tilde mcbT_mft page ref} \\
 $\jets$ & Functions from $\Lambda\setminus P$ to $\mcb{H}^{\ex}$ & \pageref{jets page ref} \\
 $\jets^{\gamma,\eta}$ & Direct sum of modelled distribution spaces & \pageref{jets gamma page ref} \\
 $\VV$ & Set of trees which contains $\mcT^\ex$ & \pageref{VV page ref}\\
 $\VVspan$ & Vector space spanned by $\VV$ & \pageref{VVspan page ref}\\
 $\Y$ & Commuting indeterminates representing the jet of $U$ & \pageref{poly reg page ref}\\
 $\Upsilon^F$ & Map into nonlinearities $\Upsilon^F : \VV \to \Poly^{\mfL_+}$ & \pageref{upsilon page ref} \\
 $\mathring\Upsilon^F$ & Map into nonlinearities $\mathring\Upsilon^F : \BBspan \to \Poly^{\mfL_+}$ & \pageref{mathringupsilon page ref} 
 \end{longtable}
 \end{center}
\endappendix
\bibliographystyle{./Martin}
\bibliography{./refs}

\end{document}